\documentclass[10pt,a4paper]{amsart}
\usepackage{verbatim}
\usepackage[dvips]{color}
\usepackage{amssymb}
\usepackage{mdframed}
\usepackage[active]{srcltx}
\usepackage[breakable]{tcolorbox}
\tcbuselibrary{many}

\usepackage[toc]{appendix}
\usepackage[T1]{fontenc}
\DeclareMathOperator{\argmin}{argmin}
\newcommand{\tail}{\textnormal{\texttt{Tail}}}
\usepackage{graphicx}
\usepackage{enumerate}
\usepackage{amsmath,amsfonts,amssymb}
\usepackage{color}
\usepackage{hyperref}
\def\loc{\operatorname{loc}}
\usepackage{cite}
\usepackage{latexsym }
\definecolor{citation}{rgb}{0.11,0.67,0.84}
\definecolor{formula}{rgb}{0.1,0.2,0.6}
\definecolor{url}{rgb}{0.11,0.67,0.84}
\usepackage{pgf,tikz}
\usepackage{mathrsfs}
\usepackage{fancyhdr}
\usepackage{dutchcal}

\newcommand{\medint}{-\kern -,375cm\int}

\newcommand{\medintinrigo}{-\kern -,315cm\int}
\makeatletter
\newcommand{\linethrough}{\mathpalette\@thickbar}
\newcommand{\@thickbar}[2]{{#1\mkern0mu\vbox{
    \sbox\z@{$#1#2\mkern-0.5mu$}%
    \dimen@=\dimexpr\ht\tw@-\ht\z@+2\p@\relax 
    \hrule\@height0.5\p@ 
    \vskip\dimen@
    \box\z@}}
}
\makeatother

\newcommand{\mathstrike}[1]{\ensuremath{\linethrough{#1}}}

\newcommand{\nra}[1]{\mathstrike{\lVert} #1 \rVert}
\newcommand{\snra}[1]{\mathstrike{[} #1 ]}

\newtheorem{theorem}{Theorem}[section]
\newtheorem{lemma}[theorem]{Lemma}
\newtheorem{proposition}[theorem]{Proposition}

\newtheorem{definition}[theorem]{Definition}
\newtheorem{remark}[theorem]{Remark}
\numberwithin{equation}{section}

\usepackage{dsfont}

\newcommand{\reqnomode}{\tagsleft@false}

\usepackage{hyperref}
\hypersetup{hypertexnames=false}

\def\mfh{\mathfrak{h}}

\newcommand\ttB{\textnormal{\texttt{B}}}
\newcommand{\rad}{\textnormal{\texttt{rad}}}

\def\dxy{\,{\rm d}x{\rm d}y}
\def\dyx{\,{\rm d}y{\rm d}x}
\def\dzx{\,{\rm d}z{\rm d}x}
\def\dxyt{\,{\rm d}\ti{x}{\rm d}\ti{y}}
\def\dyxt{\,{\rm d}\ti{y}{\rm d}\ti{x}}
\def\dx{\,{\rm d}x}
\def\MMM{\textnormal{\texttt{M}}}
\def\cchh{c_{\textnormal{\texttt{h}}}}
\def\dz{\,{\rm d}z}

\def\dlam{\,{\rm d}\lambda}

\def\BBB{\textnormal{B}}

\def\dy{\,{\rm d}y}
 
\def \d{\,{\rm d}}
\def \diver{\,{\rm div}}
\def\dist{\,{\rm dist}}

\def\aaat{\textnormal{\texttt{a}}}
\def\bbt{\textnormal{\texttt{b}}}

\def\diam{\,{\rm diam}}

\def\aa{{\bf a}}
\def\aat{\ti{{\bf a}}}

\allowdisplaybreaks
\makeatletter
\DeclareRobustCommand*{\bfseries}{%
  \not@math@alphabet\bfseries\mathbf
  \fontseries\bfdefault\selectfont
  \boldmath
}

\DeclareMathOperator*{\osc}{osc}

\makeatother

\newlength{\defbaselineskip} 
\newcommand{\mint}{\mathop{\int\hskip -1,05em -\, \!\!\!}\nolimits}

\def \diver{\,{\rm div}}

\def\er{\mathbb R}
\def\en{\mathbb N}

\newcommand{\ddim}{\textnormal{dim}_{\mathcal H}}

\newcommand{\ppsi}{\sigma^{2s}\nra{f}_{L^\chi(B_{\sigma})}}
\newcommand{\pplam}{\lambda^{2s}\nra{f}_{L^\chi(B_{\lambda})}}
\newcommand{\ppsix}{\sigma^{2s}\nra{f}_{L^\chi(B_{\sigma}(x))}}

\newcommand{\ppsixo}{\sigma^{2s}\nra{f}_{L^\chi(B_{\sigma}(x_{0}))}}
\newcommand{\pprhoo}{\rr^{2s}\nra{f}_{L^\chi(B_{\rr}(x_{0}))}}
\newcommand{\pptt}{t^{2s}\nra{f}_{L^\chi(B_{t})}}
\newcommand{\pptta}{\lambda^{2s}\nra{f}_{L^1(B_{\lambda})}}
\newcommand{\pptto}{t^{2s}\nra{f}_{L^\chi(B_{t}(x_{0}))}}
\newcommand{\ppttx}{t^{2s}\nra{f}_{L^\chi(B_{t}(x))}}

\newcommand{\ppjj}{\rr_j^{2s}\nra{f}_{L^\chi(B_{j})}}
\newcommand{\eps}{\varepsilon}
\newcommand{\ti}[1]{\tilde{#1}}
\newcommand{\epsb}[1]{\varepsilon_{\textnormal{\texttt{#1}}}}
\newcommand{\rhob}[1]{\rr_{\textnormal{\texttt{#1}}}}

\newcommand{\data}{\textnormal{\texttt{data}}}

\newcommand{\dd}{\mathrm{d}}

\newcommand{\rrr}{\textnormal{\texttt{r}}}

\newcommand{\rr}{\varrho}
\newcommand{\snr}[1]{\lvert #1\rvert}
\newcommand{\nr}[1]{\lVert #1 \rVert}
\newcommand{\rif}[1]{(\ref{#1})}
\newcommand{\tx}[1]{\textnormal{\texttt{#1}}}
\newcommand{\stackleq}[1]{\stackrel{\rif{#1}}{ \leq}}

\def\loc{\operatorname{loc}}

\def\eqn#1$$#2$${\begin{equation}\label#1#2\end{equation}}

\def\Xint#1{\mathchoice
   {\XXint\displaystyle\textstyle{#1}}%
   {\XXint\textstyle\scriptstyle{#1}}%
   {\XXint\scriptstyle\scriptscriptstyle{#1}}%
   {\XXint\scriptscriptstyle\scriptscriptstyle{#1}}%
   \!\int}
\def\XXint#1#2#3{{\setbox0=\hbox{$#1{#2#3}{\int}$}
     \vcenter{\hbox{$#2#3$}}\kern-.5\wd0}}

\def\dashint{\Xint-}

\makeatletter
\@ifundefined{subjclassname@2020}{%
  \@namedef{subjclassname@2020}{\textup{2020} Mathematics Subject Classification}%
}{}
\makeatother

\title{Partial regularity in nonlocal systems I}

\author[De Filippis]{Cristiana De Filippis}  \address{Cristiana De Filippis\\Dipartimento SMFI, Universit\`a di Parma\\ Parco Area delle Scienze 53/A, 43124 Parma, Italy} \email{\url{cristiana.defilippis@unipr.it}}
\author[Mingione]{Giuseppe Mingione}  \address{Giuseppe Mingione\\Dipartimento SMFI, Universit\`a di Parma, Parco Area delle Scienze 53/a, Campus, 43124 Parma, Italy} \email{\url{giuseppe.mingione@unipr.it}}
\author[Nowak]{Simon Nowak}  \address{Simon Nowak\\Fakult\"at f\"ur Mathematik, Universit\"at Bielefeld\\Postfach 100131, 33501 Bielefeld, Germany} \email{\url{simon.nowak@uni-bielefeld.de}}

\begin{document}

to appear in Advances in Mathematics
  \vspace{8mm}

\subjclass[2020]{35B65, 35R11, 31C45} 

\keywords{Nonlocal elliptic systems, Partial regularity, Nonlinear potentials}{\vspace{1mm}}

\begin{abstract}
Solutions to nonlinear vectorial integro-differential equations are regular outside a negligible closed subset whose Hausdorff dimension can be explicitly bounded from above. This subset can be characterized using quantitative, universal energy thresholds for nonlocal excess functionals. The analysis is carried out via the use of nonlinear potentials and allows to derive fine properties of solutions under sharp assumptions on data and kernel coefficients. 
 \end{abstract}

\maketitle

\vspace{-3mm}
 
 \centerline{To Enrico Giusti, in memoriam}

  \vspace{8mm}

\setcounter{tocdepth}{1}
{\small \tableofcontents}

\section{Introduction}
\subsection{Partial regularity: from local to nonlocal}\label{bebe} In this paper we propose a systematic study of partial regularity of solutions to general nonlocal and nonlinear elliptic systems via nonlinear potential theoretic methods. The results cover a wide range of situations, spanning from the most basic ones, where the ingredients are more regular, to sharp fine regularity of solutions, where nonlinear potentials play a crucial role and coefficients are potentially discontinuous. The local story is by now classical. Whereas in the scalar case solutions to linear elliptic equations with measurable coefficients are locally H\"older continuous by De Giorgi-Nash-Moser theory, this is not the case for systems, starting by classical counterexamples of De Giorgi \cite{dgce}
 and Maz'ya \cite{mazya}. The next step is then to look at partial regularity, i.e., determining regularity of solutions outside negligible closed subsets. This viewpoint, linked to the so-called $\eps$-regularity theorems and first developed in the setting of minimal surfaces \cite{deg}, is today the blueprint for the vast majority of regularity results available in the vectorial setting. The first partial regularity results go back to the fundamental works of Giusti \& Miranda  
 \cite{giumi} and Morrey \cite{morrey}, who considered systems of the type
 \eqn{localmodel}
 $$
 -\diver  \left(A(x, u)Du\right)=- D_i (A_{\alpha\beta}^{ij}(x,u) D_j u^\beta)=0 \qquad \mbox{in $\Omega$}
 $$ 
with $i, j\in \{1, \ldots, n\}, \alpha, \beta \in \{1, \ldots, N\}$. Here $\Omega\subset \er^n$ is an open set and $n\geq 2, N\geq 1$ will be two fixed integers (representing the ambient and the target dimensions, respectively).  
In \rif{localmodel} the matrix $A(\cdot)$ has uniformly continuous entries in $\Omega \times \er^N$ and satisfies 
 $$
 \Lambda^{-1}|\xi|^2\leq \langle  A(x, v)\xi,  \xi\rangle\,, \, \qquad |A(x, v)| \leq  \Lambda 
 $$
 for all $x\in \Omega$, $v\in \er^N$, $\xi \in \er^{N\times n}$, where $ \Lambda\geq 1$. In this case it turns out that solutions are locally H\"older continuous outside a closed subset whose Hausdorff dimension is  smaller than $n-2$:
 \eqn{ggresult}
 $$
 u \in C^{0, \beta}_{\loc}(\Omega_u;\er^N) \, \ \  \mbox{ for every $\beta <1$}\,, \qquad \ddim(\Omega\setminus \Omega_u)< n-2\,,
 $$
where $\Omega_u\subset \Omega$ is an open subset called the {\em regular set} of $u$ and its complement $\Omega\setminus \Omega_u$ is the {\em singular set} of $u$. Accordingly, a regular point of $u$ is a point $x_0\in \Omega$ where $u$ is regular - according to the degree of regularity of the context - in a neighbourhood of $x_0$. 
Singularities of solutions might occur even when the entries of $A(\cdot)$ are smooth functions. Indeed, inspired by De Giorgi's famous counterexample \cite{dgce} for  systems of the type $ \diver \, (A(x)Du)=0$, with measurable coefficients $A(\cdot)$, again Giusti \& Miranda \cite{gmce}, \cite[Section 9.1.2]{giagreen} provided a counterexample to regularity of solutions to systems of the specific type $\diver \, (A(u)Du)=0$ with $A(\cdot)$ being this time analytic. There is today a large literature devoted to partial regularity in the vectorial setting, especially in geometric problems from the Calculus of Variations. For this we refer to the classical treatises \cite{giaorange, giagreen} and to the overview paper \cite{min06}. Proof of results as in \rif{ggresult} are in a way or in another achieved by local linearization procedures. Roughly speaking, one starts on a ball $B_{\rr}\equiv B_{\rr}(x_{0})$ where a quantity called excess is small, i.e., 
  \eqn{eccessino}
$$
\nra{u-(u)_{B_{\rr}}}_{L^{2}(B_{\rr})} < \epsb{b}
$$  
holds for a universal number $\epsb{b}>0$, that is, depending only on the ellipticity and regularity properties of the matrix $A(\cdot)$. The one in \rif{eccessino} is a functional bound to measure the oscillations of $u$ by its mean square deviation from the average on $B_{\rr}$. Then one locally compares $u$ to solutions $v$ to linear systems with constant coefficients of the type $\diver\, (A(x_{0}, (u)_{B_{\rr}})Dv)=0$. Thanks to \rif{eccessino}, from the regularity of $v$ one by comparison obtains a suitable decay estimate on the excess (``improvement of flatness") at all scales, i.e., 
  \eqn{eccessino2}
$$
\nra{u-(u)_{B_{\sigma}}}_{L^{2}(B_{\sigma})}\lesssim \sigma^{\beta}\,, \qquad \forall \ \sigma \leq \rr\,.
$$
This eventually leads to local $C^{0, \beta}$-regularity of $u$ in a neighbourhood of $x_{0}$ as \rif{eccessino} is clearly an open condition. The initial requirement in \rif{eccessino} can be then verified only up to a negligible closed subset and this leads to partial regularity. The final outcome comes with an explicit characterization of the regular set, that is 
\eqn{chara}
$$
\Omega_{u}= \{x \in \Omega\, \colon \, \lim_{\sigma\to 0 } \nra{u-(u)_{B_{\sigma}(x)}}_{L^{2}(B_{\sigma}(x))} =0\}\,.
$$
Note that, given a general Sobolev function $u$, the set in \rif{chara} is not open, but it is part of the partial regularity proof to show that it actually is when $u$ solves \rif{localmodel}. 
For the notation used in \rif{eccessino}-\rif{chara} and in the following, we recommend the reader to look at Section \ref{preliminari}. 
 
In this paper, and in its follow-up \cite{follow}, we establish some results and methods concerning partial regularity of solutions to certain general nonlocal elliptic systems. These topics were already treated in certain relevant, specific situations, see for instance the recent, remarkable works of Mazowiecka, Miśkiewicz \& Schikorra \cite{mazo, mazo2} for an overview and references, and the comments in Section \ref{furthersec} below. Here we propose a general approach with methods designed to apply in a variety of nonlinear settings. We remark that our methods are intrinsically nonlocal and do not make use of any extension operator transforming nonlocal problems into degenerate local ones. Among the other things, we shall deliver a full nonlocal analog of the classical results of Giusti \& Miranda and Morrey by considering nonlocal elliptic systems of the type 
	\eqn{nonlocaleqn}
	$$
	-\mathcal{L}_{a}u=f\qquad \mbox{in} \ \ \Omega \subset \mathbb{R}^n
	$$
	where, for some elliptic and bounded coefficient $N\times N$ matrix $a(\cdot)$ and $0<s<1$, the (vector-valued) nonlocal operator $-\mathcal{L}_{a}$ is defined by 
$$
		\langle - \mathcal{L}_{a} u, \varphi \rangle =\int_{\mathbb{R}^n} \int_{\mathbb{R}^n} \langle a(x,y,u(x),u(y))(u(x)-u(y)),\varphi(x)-\varphi(y) \rangle \frac{\dxy}{|x-y|^{n+2s}}\,, 
$$
	for every $\varphi\in C^{\infty}(\er^n;\er^N)$ with compact support in $\Omega$.  
For the rest of the paper $\Omega \subset \er^n$ will always denote a bounded open subset\footnote{Since all the results in this paper are local we can restrict to bounded domains with no loss of generality.}, with $n\geq 2$, while $a\colon \mathbb{R}^{2n}  \times \er^{2N}\to \mathbb{R}^{N\times N}$, $N\geq 1$, is a bounded, Carath\'eodory-regular matrix\footnote{This means, as usual, that $(x,y)\mapsto a(x,y, v,w)$ is measurable for every choice of $v,w\in \er^N$ and that there exists a negligible (with respect to the Lebesgue measure in $\er^{2n}$) set $\mathcal N \subset \er^{2n}$ such that $(v,w) \mapsto a(x,y,v,w)$ is continuous for every choice $(x,y)\not \in \mathcal N$. This guarantees that the composition $a(x,y,v(x),w(y))$ is measurable as long as $v,w$ are measurable maps.}.
The notion of solution to \rif{nonlocaleqn} involves the so-called $\tail$ space \cite{DKP, kokupa, KMS1} 
$$
L^{1}_{2s} := \left\{w\in L^{1}_{\loc}(\er^n;\mathbb{R}^N) \, \colon \, \int_{\er^n} \frac{\snr{w(x)}}{1+\snr{x}^{n+2s}} \dx< \infty\right\}\footnote{It is not difficult to see that $W^{s,2}(\er^n;\er^N)\subset L^2(\er^n;\er^N)\subset L^{1}_{2s}$.}\,.
$$ 
\begin{definition} \label{def:weaksol}
Under the assumption \eqref{bs.1}$_2$ below, with $f \in L^{2_*}(\Omega)$ and $2_{*}:=\frac{2n}{n+2s}$\footnote{In this paper we are also denoting $2^*=2n/(n-2s)$ the fractional Sobolev embedding exponent. Note that $1/2_{*}+1/2^{*}=1$, i.e., $2_{*}$ and $2^{*}$ are conjugate exponents.}, a map $u \in W^{s,2}_{\loc}(\Omega;\mathbb{R}^N) \cap L^{1}_{2s}$ is a weak solution of \eqref{nonlocaleqn} when
\eqn{eqweaksol}
$$
			\int_{\mathbb{R}^n} \int_{\mathbb{R}^n} \langle a(x,y,u(x),u(y))(u(x)-u(y)),\varphi(x)-\varphi(y) \rangle \frac{\dxy}{|x-y|^{n+2s}} = \int_{\Omega} \langle f, \varphi \rangle\dx
$$
		holds for every $\varphi \in W^{s,2}(\er^n;\er^N)$ with compact support in $\Omega$\footnote{Note that \rif{eqweaksol} in particular holds whenever $\varphi$ is of the form $\varphi=\eta v$ with $v \in W^{s,2}_{\loc}(\Omega;\er^{N})$ and $\eta \in C^{\infty}$ has compact support in $\Omega$.}.
	\end{definition}
When considering a solution to $-\mathcal{L}_{a}u=f$, we shall usually refer to Definition \ref{def:weaksol}, and we shall recall this fact from time to time, to avoid ambiguities in the case when solutions to certain nonlocal Dirichlet problems will be considered. It is at this stage worth noting that the energy type condition $u \in W^{s,2}_{\loc}(\Omega;\mathbb{R}^N) \cap L^{1}_{2s}$ makes the duality coupling in \rif{eqweaksol} finite. We now come to the precise hypotheses for \rif{eqweaksol}. We assume that $a(\cdot)$ satisfies the following:
\begin{itemize}
\item There exists $\Lambda\ge 1$ such that
\eqn{bs.1}
$$
\begin{cases}
\Lambda^{-1}\snr{\xi}^{2} \leq \langle a(x,y,v,w)\xi,\xi\rangle \\[4pt]
\qquad \snr{a(x,y,v,w)}\le \Lambda,
\end{cases}
$$
hold for all $x,y\in \mathbb{R}^n$, $v,w,\xi\in \er^N$.
\item Symmetry with respect to all entries, i.e., 
\eqn{bs.2}
$$
a(x,y,v,w)=a(y,x,w,v)
$$
holds for all $x,y\in \mathbb{R}^n$, $v,w\in \er^N$.
\item There exists a modulus of continuity $\omega\colon [0, \infty)\to [0,1]$, i.e. a continuous, concave and non-decreasing function, with $\omega(0)=0$, such that
\eqn{bs.4}
$$
\snr{a(x,y,v_{1},w_{1})-a(x,y,v_{2},w_{2})}\le \Lambda \omega(\snr{v_{1}-v_{2}}+\snr{w_{1}-w_{2}}),
$$
holds for every choice of $x, y \in \er^n$ and of $v_1, v_2, w_1, w_2 \in \er^N$.
\end{itemize}
Assumptions \rif{bs.1}-\rif{bs.2} essentially reproduce those that are typically considered in classical partial regularity results for local systems as in \rif{localmodel}. As for the right-hand side $f$, that we always extend to $\er^n$ letting $f\equiv 0_{\er^N}$ outside $\Omega$, from now on, the initial assumption we make is 
\eqn{bs.5}
$$
f\in L^{\chi}(\er^n;\er^N)\equiv L^{\chi}\qquad \mbox{where} \  \ 2_{*}=\frac{2n}{n+2s}< \chi < \frac{n}{2s} 
\footnote{Of course the real assumption here is that $\chi > 2_*$. The  restriction 
$\chi<n/(2s)$ 
is only used to choose a convenient subcritical working exponent for the potential-theoretic and Marcinkiewicz estimates, for instance when considering the potentials $\mathbf{I}^{f}_{2s,\chi}$ and $\mathbf{I}^{f}_{2s-1, \chi}$ and applying the criterion in \rif{qualo2}. Such a restriction can  be made without loss of generality as $f\equiv 0$ outside $\Omega$ and $\Omega$ is bounded. }\,.
$$
Summarizing, for the rest of the paper $s\in (0,1)$ will always denote the fixed number linked to the differentiation degree of the operator in \rif{nonlocaleqn} and $\chi$ the number appearing in \rif{bs.5}. In order to shorten the notation concerning the dependence on the various constants, we shall abbreviate as follows
\eqn{idati}
$$
\data := (n,N,s,\Lambda, \chi)
$$
i.e., the set of parameters defining the ellipticity properties of the operator in $-\mathcal{L}_{a}$ together with the basic integrability exponent of $f$. Very often, for  readability, a constant depending on a strict subset of the parameters in \rif{idati} will be still indicated as to depend on the whole set $\data$. This for instance will happen in those cases where $\nr{f}_{L^\chi}$ does not play a role or when $f\equiv 0$, when the dependence on $\data$ will just indicate a dependence on $n,N,s,\Lambda$.  

We shall derive a number of partial regularity results for solutions to \rif{nonlocaleqn}, under various and optimal assumptions on the matrix coefficients $a(\cdot)$. Among  our results we single out the following one, that can be recovered from Theorem \ref{ureg5} below, and which is the exact analogue of the classical Giusti \& Miranda's result \cite{giumi} reported in \rif{ggresult}, including the characterization of the singular set in \eqref{chara}. 
\begin{theorem}[Nonlocal partial regularity - toy model case]\label{ureg0}
Under assumptions \eqref{bs.1}-\eqref{bs.2}, let $u$ be a weak solution to \eqref{nonlocaleqn} and let $\Omega_u$ be defined in \eqref{chara} (as in the local case). If $a(\cdot)$ is uniformly continuous\footnote{Here by uniform continuity we mean the existence of a modulus of continuity $\tilde{\omega}(\cdot)$ as in \rif{bs.4} and such that 
$$
\snr{a(x_1,y_1,v_{1},w_{1})-a(x_2,y_2,v_{2},w_{2})}\le \Lambda \tilde{ \omega}(\snr{x_{1}-x_{2}}+\snr{y_{1}-y_{2}}+\snr{v_{1}-v_{2}}+\snr{w_{1}-w_{2}}),
$$
for every choice of $x_1, x_2, y_1,y_2 \in \er^n$ and $v_1, v_2, w_1,w_2 \in \er^N$.} in $\er^{2n}\times \er^{2N}$ and $f\in L^{\infty}$, then $\Omega_u$ is open and there exists a positive number $\gamma \equiv \gamma (\data)\leq n-2s$ such that 
\eqn{partclassic}
 $$
 \begin{cases}
 \, u \in C^{0, \beta}_{\loc}(\Omega_u;\er^N) \,  \mbox{ for every $\beta <\min\{1, 2s\}$}\\
\,  \ddim(\Omega\setminus \Omega_u)\leq  n-2s-\gamma\,.
 \end{cases}
 $$
\end{theorem}
\section{Results} 
Theorem \ref{ureg0} is a sample of the regularity results obtained in this paper. In fact, in the process, we shall update the standard partial regularity approach developed in classical papers and treatises like \cite{giaorange, giagreen, giusti}
 with more recent advances on fine regularity of solutions to local problems, as for instance detailed in \cite{by, dqc, ds, kumi}. Specifically, we shall give a number of results bound to sharply describe the kind of partial regularity of solutions one obtains in conjunction with precise assumptions on coefficients and data $f$. This will be done via the use of nonlinear potential theoretic methods. The outcome will be a rather complete picture sharply describing how regularity of external ingredients affects that of solutions, of course modulo a singular set, which is unavoidable in the vectorial setting. This will be in full analogy both with the results available in the local case, as described for instance in \cite[Corollary 1]{kumig} and \cite{KMSvec}, see also \cite{kmjfa}, and in the nonlocal one \cite{bls, KMS1, sn, sn1}. More comments on such aspects will be given along the way. 

As mentioned, the techniques we use incorporate information from recent advances in Nonlinear Potential Theory \cite{dqc, kumig, kumi,KMSvec}.  For this we need Havin-Mazya-Wolff potentials \cite{HM}, that in this setting will be represented by the Riesz type potential
\eqn{diadico}
$$
\mathbf{I}^{f}_{2s,\chi}(x,\sigma) := \int_{0}^{\sigma} \lambda^{2s}\nra{f}_{L^\chi(B_{\lambda}(x))}  \, \frac{\dlam}{\lambda} \gtrsim  \sum_{i\geq 1} \left(\frac{\sigma}{2^i}\right)^{2s}\nra{f}_{L^\chi(B_{\sigma/2^i}(x))}
$$
where $\chi$ is as in \rif{bs.5}. 
See \rif{shalldenote} for notation here. Note that, when formally taking $\chi=1$, the potential $\mathbf{I}^{f}_{2s,\chi}$ coincides with the classical (truncated) Riesz potential $\mathbf{I}^{f}_{2s}$ \cite{kumig} whose kernel is equivalent to the fundamental solution of the fractional Laplacean $-\Delta^{s} $. Moreover, as a simple consequence of H\"older inequality, we have 
\eqn{dueto}
$$
\mathbf{I}^{f}_{2s}(x,\sigma) =\mathbf{I}^{f}_{2s,1}(x,\sigma)= \int_{0}^{\sigma} \pptta \, \frac{\dlam}{\lambda} \leq \mathbf{I}^{f}_{2s,\chi}(x,\sigma) \,.
$$
Due to \rif{dueto}, the potential $\mathbf{I}^{f}_{2s,\chi}$ essentially shares the same scaling and mapping properties of the classical Riesz potential $\mathbf{I}^{f}_{2s}$ provided $f \in L^{\chi}$, see \cite{dmn, dmn2,  kumig}. Indeed, such potentials provide a precise description of the fine properties of solutions to nonlinear elliptic equations in a similar way the Riesz potentials with integer index describe the fine properties of usual harmonic functions. We refer to \cite{finnish, kumig} for an account of the main results and for the relevant literature and to \cite{dqc, dmn, kumig, kumi} for their use in regularity theory of nonlinear, potentially degenerate equations. By their use we can indeed reformulate in the present setting several of the fine properties criteria from classical linear and Nonlinear Potential Theory; see for instance Theorem \ref{ureg3} below.

We shall consider various types of regularity on coefficients $a(\cdot)$, i.e., the regularity of the partial map $(x,y)\mapsto a(x,y,\cdot)$. In particular, we shall use suitable analogs of the concept of BMO/VMO regularity of coefficients following for instance \cite[Section 1.1]{sn}. Compare what follows with the classical definitions of BMO/VMO-functions in \rif{bmodef}-\rif{vmodef} below. 
\begin{definition}\label{deltadef}
		Let $\delta, r_0>0$.
		\begin{itemize}
			\item We say that $a(\cdot)$ is $\delta$-vanishing in the ball $B_{r}(x_{0})\subset \Omega$, iff
			\eqn{gestire} $$ h_{a}(x_{0}, r):=\sup_{0<\sigma \leq r; v,w\in \er^N} \, \mint_{B_{\sigma}(x_{0})}\mint_{B_{\sigma}(x_{0})}\snr{a(x,y,v,w)-a_{B_{\sigma}(x_{0})}(v,w)}\dxy\leq \delta  \,,$$
			where  
$$a_{B_{\sigma}(x_{0})}(v,w):=\mint_{B_{\sigma}(x_{0})}\mint_{B_{\sigma}(x_{0})}a(x,y,v,w)\dxy\,.$$
\item We say that $a(\cdot)$ is $(\delta,r_{0})$-\textnormal{BMO} in $\Omega$, if $a(\cdot)$ is $\delta$-vanishing in every ball $B_r(x_{0})\subset \Omega$ with $r \leq r_{0}$. 
						\item We say that $a(\cdot)$ is \textnormal{VMO} in $\Omega$, if for every $\delta>0$ there exists $r_0\equiv r_0 (\delta)>0$ such that $a(\cdot)$ is $(\delta,r_{0})$-\textnormal{BMO} in $\Omega$. Equivalently, if
						$$\lim_{r\to 0}\, h_{a}(x_{0}, r)=0 \quad  \mbox{uniformly with respect to $x_{0}$, with $B_{r}(x_0)\subset \Omega$} \,.$$
		\end{itemize}
	\end{definition}
Observe that a uniformly continuous matrix $a(\cdot)$ in the sense used in Theorem \ref{ureg0} is automatically VMO-regular. Note also that if $a(\cdot)$ is $\delta$-vanishing in the ball $B_{r}(x_{0})\subset \Omega$, then $a(\cdot)$ is also $\delta$-vanishing in every ball $B_{\sigma}(x_{0})$ with $\sigma \leq r$. An object whose features we shall sistematically exploit is a suitable notion of excess functional. In the local case \cite{giaorange, giumi}, this is traditionally a functional of the type in \rif{eccessino}. 
In our case long-range interactions play a crucial role in the regularization process of solutions. In order to control them along the iteration processes we need to adopt the following extended notion of excess functional, already identified and used in \cite{KMS1}. For this we first need the notion of $\tail$. For $w\in L^{1}_{\loc}(\er^n;\er^N)$ the $\tail$ of $w$ with respect to the ball $B_{\rr}(x)\subset \er^n$ is defined \cite{BK, DKP, kokupa, KMS1, silve} as 
\eqn{lacoda}
$$
\tail(w;B_{\rr}(x)):= \rr^{2s}\int_{\er^n\setminus B_{\rr}(x)} \frac{\snr{w(y)}}{\snr{y-x}^{n+2s}} \dy \,.
$$
In this respect note that 
$
w \in L^{1}_{2s}$ implies that $\tail(w;B_{\rr}(x))$ is finite for every ball $B_{\rr}(x)$. 
\begin{definition}[Nonlocal Excess]\label{leccesso}
Given $w \in L^{2}(B_{\rr};\mathbb{R}^N)\cap L^{1}_{2s}$, where $B_{\rr}\equiv B_{\rr}(x_{0})\subset \er^n$ is a ball, the excess functional $ \tx{E}_{w}(x_{0},\rr)$ is defined by
\eqn{exc}
 $$
 \tx{E}_{w}(x_{0},\rr):= \sqrt{\nra{w-(w)_{B_{\rr}}}_{L^{2}(B_{\rr})}^2+\tail(w-(w)_{B_{\rr}};B_{\rr})^2}\,.
 $$
When $s>1/2$, let $\ell$ be an affine map; the (affine) excess functional $\tx{E}_{w}(\ell;x_{0},\rr)$ is instead defined by 
\eqn{excl}
$$
\tx{E}_{w}(\ell;x_{0},\rr):=\sqrt{\nra{w-\ell}_{L^{2}(B_{\rr})}^2+\tail(w-\ell;B_{\rr})^2} \,.
$$
\end{definition}
We remark that $w \in L^{2}(B_{\rr};\mathbb{R}^N)\cap L^{1}_{2s}$ guarantees that $ \tx{E}_{w}(x_{0},\rr)$ is always finite and that $\tx{E}_{w}(\ell;x_{0},\rr)$ is always finite when in addition we also require that $s>1/2$; for this see Section \ref{affinetail} and \rif{tritri}$_2$ below. We are now ready to describe our main results, that will be presented according to increasing degrees of regularity for solutions. They will be in perfect accordance with both the optimal regularity results known in the local case  \cite{kumig} and with those available in the scalar nonlocal case for nonlinear problems \cite{bls}. The  statements span from zero order oscillations integral estimates, Theorems \ref{ureg1}-\ref{ureg2}, to maximal gradient H\"older continuity, Theorem \ref{gradreg3}, passing through delicate borderline cases of partial gradient continuity as in Theorem \ref{gradreg2}. In particular, Theorem \ref{ureg5} below contains, under more general assumptions, the nonlocal version of Giusti \& Miranda's and Morrey's original results, that is Theorem \ref{ureg0} above. To aid the reader, we will separate the results into two categories. The first category addresses the oscillatory properties of the solution $u$ and are valid for any value $s \in (0,1)$ of the differentiability order of the operator $-\mathcal{L}_{a}$. The second category focuses on the oscillatory behavior of the gradient  $Du$, where we require $s >1/2$ in alignment with the nonlinear potential theory approach developed in this work. See also \cite{dnowak, KMS1, kns} for potential estimates in the scalar case. 
\subsection{Oscillations and potential estimates for $u$} The starting point of our results is concerned with  the weakest  regularity assertion on the oscillations of $u$, that is BMO-regularity. For this we recall that a map $w\in L^{1}(\mathcal A;\er^N)$ defined on an open subset $\mathcal A \subset \er^n$ is BMO-regular in $\mathcal A$, i.e., $w\in \textnormal{BMO} (\mathcal A;\er^N)$, iff 
\eqn{bmodef}
$$ 
\sup_{B\subset \mathcal A} \, \nra{w-(w)_{B}}_{L^1(B)} < \infty\,.
$$
In \rif{bmodef} $B$ denotes the generic ball contained in $\mathcal A$; the local variant $\textnormal{BMO}_{\loc}(\mathcal A;\er^N)$ is defined in the obvious fashion.  Accordingly, $w$ is VMO-regular in $\mathcal A$, i.e., $w\in \textnormal{VMO} (\mathcal A;\er^N)$, iff
\eqn{vmodef}
$$
\lim_{\sigma \to 0}\sup_{B\subset \mathcal A, \snr{B} \leq \sigma} \, \nra{w-(w)_{B}}_{L^1(B)} =0\,, 
$$
again with the local variant defined in the usual fashion. BMO and VMO regularity are classical concepts in modern analysis and they considerably impact both Harmonic Analysis and the analysis of partial differential equations. They were introduced in \cite{john} and \cite{sarason}, respectively. 
\begin{theorem}[Partial BMO regularity]\label{ureg1}
Under assumptions \eqref{bs.1}-\eqref{bs.5}, let $u$ be a weak solution to \eqref{nonlocaleqn}. There exist $\epsb{b}\equiv \epsb{b}(\data,\omega(\cdot)) \in (0,1)$, $\delta \equiv \delta(\data) \in (0,1)$ such that if $a(\cdot)$ is $(\delta,r_{0})$-\textnormal{BMO} in $\Omega$ for some $r_{0}>0$ and
\eqn{f.bmo}
$$ 
\sup_{B_{\sigma} \Subset \Omega; \sigma \leq \sigma_{0}} \ppsi<\epsb{b}
$$
for some $\sigma_{0} >0$, then there exists an open set $\Omega_{u}\subset \Omega$, and a positive number $\gamma \equiv \gamma (\data)\leq n-2s$, such that $
u\in \textnormal{BMO}_{\loc}(\Omega_{u};\er^N)$ and $ \ddim(\Omega\setminus \Omega_u)\leq n-2s-\gamma.
$
\end{theorem}
\begin{theorem}[Partial VMO regularity]\label{ureg2} Under assumptions \eqref{bs.1}-\eqref{bs.5}, let $u$ be a weak solution to \eqref{nonlocaleqn}. There exists $\delta \equiv \delta(\data) \in (0,1)$ such that if $a(\cdot)$ is $(\delta,r_{0})$-\textnormal{BMO} in $\Omega$ for some $r_{0}>0$ and
\eqn{f.vmo}
$$
\lim_{\sigma\to 0}\, \ppsix=0\qquad \mbox{uniformly with respect to $x\in \bar{\Omega}$}\,,
$$
then there exists an open set $\Omega_{u}\subset \Omega$, and a positive number $\gamma \equiv \gamma (\data)\leq n-2s$, such that $
u\in \textnormal{VMO}_{\loc}(\Omega_{u};\er^N) $ and $\ddim(\Omega\setminus \Omega_u)\leq n-2s-\gamma.$
Specifically, 
\eqn{vazero}
$$
\lim_{\sigma \to 0}  \tx{E}_{u}(x,\sigma) =0 \quad \mbox{locally uniformly in $\Omega_{u}$}\,,
$$
where $\tx{E}_{u}$ has been defined in \eqref{exc}. 
\end{theorem}
\begin{theorem}[Oscillations bounds and precise representative]\label{ureg3} Under assumptions \eqref{bs.1}-\eqref{bs.5}, let $u$ be a weak solution to \eqref{nonlocaleqn} and $B_{\rr}(x_0)\Subset \Omega$ be a ball. There exist $\epsb{b}\equiv \epsb{b}(\data,\omega(\cdot))\in (0,1)$, $\delta \equiv \delta(\data) \in (0,1)$ and $R=R(\data) \geq 1$ such that if $a(\cdot)$ is $\delta$-vanishing in the ball $B_{R\rr}(x_{0})\subset \Omega$ and the smallness conditions
\eqn{osc.t0}
$$
 \tx{E}_u(x_{0},\rr)< \epsb{b} \qquad \mbox{and}\qquad \sup_{\sigma \leq \rr} \ppsixo<\epsb{b}
 $$
are verified, then 
the local oscillation estimate
\eqn{oscosc}
$$
\snr{u(x_{0})-(u)_{B_{\sigma}(x_{0})}}\le c\, \tx{E}_{u}(x_{0},\sigma)+c\, \mathbf{I}^{f}_{2s,\chi}(x_{0},\sigma)
$$ 
holds whenever $\sigma \leq \rr$ and $B_{\sigma}\equiv B_{\sigma}(x_{0})$, where $c\equiv c (\data)$, provided the right-hand side is finite. In \eqref{oscosc}, $u(x_0)$ denotes the precise representative of $u$ at $x_0$, i.e., 
\eqn{lebp}
$$
u(x_{0}):=\lim_{\sigma\to 0}(u)_{B_{\sigma}(x_{0})}
$$
whose existence is indeed implied by the finiteness of $\mathbf{I}^{f}_{2s,\chi}(x_{0},\sigma)$. 
\end{theorem}
Theorems \ref{ureg1} and \ref{ureg2} closely parallel the results established in the local case \cite{kumig, KMSvec}, with the notable exception of the singular set $\Omega\setminus \Omega_{u}$, a distinctive feature specific to the vectorial setting. The numbers $\epsb{b}, \delta$ are the same throughout Theorems \ref{ureg1}-\ref{ureg3} 
 and should be considered as universal thresholds for the formation of singularities as well as the number $R$ from Theorem \ref{ureg3}. They only depend on the operator, but are independent of the solution considered. As it will be clear from the proofs, such numbers serve to fix three smallness conditions, i.e., 
\eqn{linearizzanti}
$$ 
\begin{cases}
 \displaystyle \tx{E}_u(x_{0},\rr)< \epsb{b} \\[2pt]
   \displaystyle \mbox{$a(\cdot)$ is $\delta$-vanishing in $B_{R\rr}(x_{0})$}\\[2pt]
  \displaystyle \sup_{\sigma \leq \rr} \ppsixo<\epsb{b}
\end{cases}
$$
under which the system in \rif{nonlocaleqn} looks as a homogenous constant coefficients system in the ball $B_{\rr}(x_{0})$.  
These in fact determine the regular set $\Omega_{u}$, where higher regularity of $u$ eventually holds under natural assumptions. See Section \ref{parametri} for a more precise quantitative description. 
The occurrence of \rif{linearizzanti} allows us to locally linearize \eqref{nonlocaleqn} in the ball $B_{\rr}{(x_{0})}$ towards comparison estimates with solutions with constant coefficients nonlocal systems. This is in analogy with the local case shortly described in Section \ref{bebe}. 
It is interesting to note an additional peculiar nonlocal phenomenon displayed in \rif{linearizzanti}$_2$. In order to get estimate \rif{oscosc} in a fixed ball $B_{\rr}{(x_{0})}$, one has to prescribe small $\delta$-oscillations of $a(\cdot)$ in the sense of \rif{gestire} in the larger ball $B_{R\rr}{(x_{0})}$. This is essential in order to treat certain nonlocal interaction terms as lower order terms and make them negligible. Note also that the smallness oscillation condition on coefficients \rif{linearizzanti}$_2$ is now naturally formulated in a BMO-fashion, rather than in $L^\infty$, i.e., coefficients are not required to be necessarily continuous. Here we adopt this principle via the $(\delta,r_{0})$-\textnormal{BMO} condition in Definition \ref{deltadef} and \rif{linearizzanti}$_2$. In this respect, let us mention that assuming a control on the oscillations of coefficients is necessary. In the general vectorial, local case, singularities may occur on a dense subset when coefficients are merely measurable \cite{johnmaly}.

Theorem \ref{ureg3} presents a coherent analogue of the description of precise representatives from classical Potential Theory. Indeed, both in the linear and in the nonlinear local case, solutions to equations of the type $-\diver\, a(Du)=f$ with $\partial a(\cdot)\approx \mathbb I_n$, admit a precise representative at a point $x_{0}$ provided the standard Riesz potential $\mathbf{I}^{f}_{2}(x_{0},\sigma)$ is finite; see \cite[Theorem 16]{kumig} and \cite[Theorem 1.2]{KMS1}. In this setting we have essentially two differences. First, the classical Riesz potential $\mathbf{I}^{f}_{2s}$ is replaced by the nonlinear potential $\mathbf{I}^{f}_{2s, \chi}$. The second is that, on the contrary to the standard scalar case, both local and nonlocal, we assume unavoidable smallness conditions in \rif{osc.t0}. Also note that  assuming the finiteness of $\mathbf{I}^{f}_{2s, \chi}(x_{0}, \cdot)$ guarantees that \rif{linearizzanti}$_3$ is satisfied at a certain scale. In fact, since 
\eqn{azero}
$$ 
\mbox{$\mathbf{I}^{f}_{2s,\chi}(x_{0},\rr) < \infty$ for some $\rr$} \Longrightarrow \lim_{\sigma \to 0}\, \mathbf{I}^{f}_{2s,\chi}(x_{0},\sigma)=0\,,
$$
as a consequence of the absolute continuity of the integral, then by \rif{diadico} we have that all the terms of the type $\ppsixo$ are smaller than $\epsb{b}$ provided we choose the scale $B_{\sigma}$ small enough. This leads to a possible alternative formulation of Theorem \ref{ureg3} where the fixed scale $B_{\rr}$ at which $\epsb{b}$ is small is not a priori prescribed and therefore \rif{oscosc} holds for some small enough but unknown $\ti{\varrho}\leq \varrho$. The advantage is of course that the second condition in \rif{osc.t0} can be dropped from the assumptions as when proving \rif{oscosc} we can always assume that the right-hand side is finite. An example of this is visible in the next
\begin{theorem}[Partial continuity]\label{ureg4}
Under assumptions \eqref{bs.1}-\eqref{bs.5}, let $u$ be a weak solution to  \eqref{nonlocaleqn}. There exists $\delta \equiv \delta(\data) \in (0,1)$ such that if $a(\cdot)$ is $(\delta,r_{0})$-\textnormal{BMO} in $\Omega$ for some $r_0>0$, and if
\eqn{unic}
$$
\lim_{\sigma\to 0}\, \mathbf{I}^{f}_{2s,\chi}(x,\sigma)=0\quad \mbox{uniformly with respect to $x\in \bar{\Omega}$}\,,
$$
then there exists an open set $\Omega_{u}\subset \Omega$, and a positive number $\gamma \equiv \gamma (\data)\leq n-2s$, such that
$
u$ is continuous in $\Omega_{u}$ and   $\ddim(\Omega\setminus \Omega_u)\leq n-2s-\gamma.
$
\end{theorem}
A particularly neat criterion implied by Theorem \ref{ureg4} can be given in terms of Lorentz spaces $L(d,1)$. The Lorentz space 
$L(d,1)(\mathcal A)$, with $1\leq d <\infty$ and $\mathcal A \subset \er^n$ being a measurable set, 
 is defined as the set of measurable maps  $f\colon \er^n \to \er^N$ such that 
$$
\nr{f}_{d,1;\mathcal A} := \int_0^{\infty}   \snr{\{|f|> \lambda\}\cap \mathcal A}^{1/d} \dd\lambda < \infty\,.
$$
In regularity theory, these are the most convenient and natural spaces in order to formulate sharp results for solutions. Using \rif{qualo2} below, and recalling \rif{bs.5}, at this point we have (recall that $\Omega$ has finite measure)
 \eqn{sarto1}
$$
f\in L^p\,, \ p> \frac{n}{2s}  \Longrightarrow f\in L\left(\frac{n}{2s},1\right) \Longrightarrow \eqref{unic} \Longrightarrow \mbox{$u$   is continuous in $\Omega_{u}$}\,.
 $$
The continuity criterion in Theorem \ref{ureg4}, and its corollary \rif{sarto1}, are the sharp analogue of the ones known in the classical local case \cite[Theorem 1]{kumig} and in the scalar nonlocal one \cite[Corollary 1.2]{KMS1}.  For the next result we need the notion of  Marcinkiewicz space $\mathcal  M^d \equiv L(d, \infty)$; its definition prescribes that a measurable map $f\colon \er^n \to \er^N$ belongs to $\mathcal  M^d(\mathcal A)$, $d\geq 1$, $\mathcal A \subset \er^n$ being a measurable set, iff
\eqn{mardef}
$$
\nr{f}_{\mathcal M^{d}(\mathcal A)}^d
:=\sup_{\lambda\geq 0} \, \lambda^d \snr{\{|f|> \lambda\}\cap \mathcal A} <\infty\,.
$$ 
\begin{theorem}[Sharp partial H\"older regularity]\label{ureg5}
Under assumptions \eqref{bs.1}-\eqref{bs.5}, let $u$ be a weak solution to  \eqref{nonlocaleqn}, and assume that 
\eqn{asbeta}
$$
f \in \mathcal M^{\frac{n}{2s-\beta}}\,, \qquad 0< \beta <\min\left\{2s,1\right\} \,.
$$
There exists $\delta \equiv \delta(\data,\beta) \in (0,1)$ such that if $a(\cdot)$ is $(\delta,r_{0})$-\textnormal{BMO} in $\Omega$ for some $r_{0}>0$,
then there exists an open set $\Omega_{u}\subset \Omega$, and a positive number $\gamma \equiv \gamma (\data)\leq n-2s$, such that $
u\in C^{0,\beta}_{\loc}(\Omega_{u};\er^N)$ and $ \ddim(\Omega\setminus \Omega_u)\leq n-2s-\gamma.
$
Moreover, for every $x_{0}\in \Omega_{u}$, there exist positive radii $r_{x_{0}}, \rr_{x_{0}}$, with $r_{x_{0}}\leq \rr_{x_{0}}/8$, such that $B_{r_{x_{0}}}(x_{0})\subset \Omega_u$ 
and 
\eqn{holdest}
$$
[u]_{0,\beta;B_{r_{x_{0}}}(x_{0})}\le c\rr_{x_{0}}^{-\beta}\tx{E}_{u}(x_{0},\rr_{x_{0}})+c \nr{f}_{\mathcal M^{\frac{n}{2s-\beta}}(B_{\rr_{x_{0}}}(x_{0}))}
$$
holds, where $c\equiv c(\data,\beta)$.
\end{theorem}
The local estimate \rif{holdest} around regular points $x_{0}\in \Omega_{u}$ is analogous to the one that holds everywhere in the scalar case found in \cite[Theorem 1.4]{bls}, \cite[Theorem 1.3]{fall} and \cite[Theorem 4.2]{feros}, under the slightly stronger integrability condition $f \in L^{\frac{n}{2s-\beta}}$. Needless to say,  a main point in Theorem \ref{ureg5} is the direct and sharp connection between the integrability of $f$ in \rif{asbeta} and the rate of local H\"older continuity of the solution, via the exponent $\beta$. Indeed, the requirement \rif{asbeta} cannot be weakened \cite[Examples 1.5-1.6]{bls}. Estimate \rif{holdest} is in turn a corollary of a more general estimate involving fractional maximal operators, i.e., 
\eqn{holdest-anc}
$$
[u]_{0,\beta;B_{r_{x_{0}}}(x_{0})}\le c\rr_{x_{0}}^{-\beta}\tx{E}_{u}(x_{0},\rr_{x_{0}})+c \nr{\MMM^{2s-\beta}_{\rr_{x_{0}}/2, \chi}(f)}_{L^{\infty}(B_{r_{x_0}}(x_{0}))}
$$
where
\eqn{defimax}
$$ \MMM^{2s-\beta}_{\rr_{x_{0}}/2, \chi}(f)(x):= \sup_{ \sigma \leq \rr_{x_{0}}/2 }\, \sigma^{2s-\beta}\nra{f}_{L^\chi(B_{\sigma}(x))}\;.$$
\subsection{Oscillations and potential estimates for the gradient} When passing to gradient regularity we use the nonlinear potential $\mathbf{I}^{f}_{2s-1,\chi}$ to obtain sharp partial gradient regularity results, provided, of course, $s>1/2$, an  assumption that we shall keep from now on. Needless to say, $a(\cdot)$ must satisfy stronger, yet necessary regularity assumptions. Specifically, we assume that $a(\cdot)$ satisfies \eqref{bs.4} with the explicit modulus of continuity 
\eqn{modco1}
$$\omega(r)=  \min\{r^\alpha,1\}, \qquad \alpha \in (0,2s-1)\,.$$
Moreover, we assume that for all $v,w \in \mathbb{R}^N$, and any ball $B_r(x_{0}) \Subset \Omega$, $a(\cdot)$ satisfies
\eqn{modco2}
$$ \sup_{x,y \in B_{r}(x_{0})} |a(x,y,v,w)-a(x_{0},x_{0},v,w)| \leq \Lambda \omega_{\star}(r) $$
for some modulus of continuity $\omega_{\star}(\cdot)$ whose features will be determined later on. Note that, as standard from classical Schauder estimates, when asking for gradient regularity results, such as for instance gradient H\"older continuity, additional assumptions as \rif{modco1}-\rif{modco2} are necessary. We present the results in increasing order of regularity.
\begin{theorem}[Partial Lipschitz regularity] \label{gradreg1}
	Under assumptions \eqref{bs.1}-\eqref{bs.5} and \eqref{modco1}-\eqref{modco2}, and with $s>1/2$, let $u$ be a weak solution to  \eqref{nonlocaleqn}. If 
	\eqn{assilip} 
	$$\int_0 \omega_{\star}(\lambda)\frac{\dlam}{\lambda} < \infty $$
	and
		\eqn{assilipanc} 
	$$\quad \sup_{x\in \Omega}\, \mathbf{I}^{f}_{2s-1,\chi}(x, 1)<\infty\,,$$
	then there exists an open set $\Omega_{u}\subset \Omega$, and a positive number $\gamma \equiv \gamma (\data)\leq n-2s$, such that
$
u\in C^{0,1}_{\loc}(\Omega_{u};\er^N) $ and $  \ddim(\Omega\setminus \Omega_u)\leq n-2s-\gamma. 
$ 
\end{theorem}  
Theorem \ref{gradreg1} can be considered as a partial regularity analog of the nonlocal potential estimates recently obtained in \cite{dnowak, kns} both in the linear and in the nonlinear case, and extending the local ones originally proved in \cite{min11}.  In our setting we need a new treatment that in fact allows us to work under the integral condition on $\omega_{\star}(\cdot)$ in \rif{assilip}. This is the well-known Dini continuity of  coefficients and it is a necessary and sufficient condition for gradient boundedness already in the local case \cite{mazyadini}. Assumption \rif{assilipanc} implies that 
	$$\lim_{\sigma\to 0}\, \mathbf{I}^{f}_{2s-1,\chi}(x,\sigma)=0\ \quad \mbox{for every  $x\in \Omega$}\,.$$
Strengthening this into a locally uniform convergence leads to upgrade gradient boundedness to gradient continuity and to reach the delicate borderline case of Schauder estimates. 
\begin{theorem}[Partial gradient continuity under Dini coefficients] \label{gradreg2}
	Under assumptions \eqref{bs.1}-\eqref{bs.5} and \eqref{modco1}-\eqref{modco2}, and with $s>1/2$, let $u$ be a weak solution to \eqref{nonlocaleqn}. If \eqref{assilip}  holds together with 
	\eqn{assilipul}
	$$\lim_{\sigma\to 0}\, \mathbf{I}^{f}_{2s-1,\chi}(x,\sigma)=0\ \ \mbox{uniformly with respect to $x\in \bar{\Omega}$}\,,$$
	then there exists an open set $\Omega_{u}\subset \Omega$, and a positive number $\gamma \equiv \gamma (\data)\leq n-2s$, such that
$Du$ is continuous in $\Omega_{u}$ and $\ddim(\Omega\setminus \Omega_u)\leq n-2s-\gamma$.
\end{theorem}
Theorem \ref{gradreg2} is the natural gradient version of Theorem \ref{ureg4}. Indeed, again invoking \rif{qualo2}, we have 
$$
f\in L^p\,, \ p> \frac{n}{2s-1}  \Longrightarrow f\in L\left(\frac{n}{2s-1},1\right) \Longrightarrow \eqref{assilipul} \Longrightarrow \mbox{$Du$   is continuous in $\Omega_{u}$}\,.
$$
Note that formally taking $s=1$ we have a partial regularity version of a well-known theorem of Stein claiming that solutions to $\Delta u\in L(n,1)$ are of class $C^1$ \cite{stein}. As for coefficients, Theorem \ref{gradreg2} perfectly parallels the results available in the local case, where both a uniform convergence of potentials and the Dini continuity of coefficients are sufficient conditions for gradient continuity. See for instance \cite{kumig, KMdini, yanyanli} and related references.

Finally, considering H\"older continuous coefficients, so that \rif{modco1} is automatically satisfied, leads to a partial regularity analog of classical Schauder estimates, which is again optimal with respect to data (see also \cite{feros, KiWe} for the scalar case). 
\begin{theorem}[Partial Schauder] \label{gradreg3}
	Under assumptions \eqref{bs.1}-\eqref{bs.5} and \eqref{modco1}-\eqref{modco2}, and with $s>1/2$, let $u$ be a weak solution to  \eqref{nonlocaleqn}. If  
	\eqn{assilip2}
	$$  \omega_{\star}(r) \lesssim r^{\alpha} \quad \mbox{and} \quad f\in \mathcal M^{\frac{n}{2s-1-\alpha}}\,,$$
	then there exists an open set $\Omega_{u}\subset \Omega$, and a positive number $\gamma \equiv \gamma (\data)\leq n-2s$, such that
$
u\in C^{1,\alpha}_{\loc}(\Omega_{u};\er^N)$ and $\ddim(\Omega\setminus \Omega_u)\leq n-2s-\gamma$.
\end{theorem}
As for the regular set $\Omega_u$, the criterion for membership of a point is the same for Theorems \ref{gradreg1}-\ref{gradreg3} and coincides with the one determined in Theorem \ref{ureg5}, for the choice $\beta=1/2$. More in general, the regular set $\Omega_{u}$ can be characterized by means of certain energy thresholds and scales. It is worth remarking that such thresholds essentially depend on the operator $a(\cdot)$ and the ingredient $f$, but are otherwise independent of the solution in question. 
\begin{theorem}[Regular set and universal energy thresholds] \label{singth} For each of Theorems \ref{ureg1}-\ref{ureg2} and \ref{ureg4}-\ref{gradreg2}
there exist positive numbers 
$\epsb{b}, \rhob{b}$ and $ R\geq 1$, that are independent of the solution $u$, such that
\eqn{ilsingolare}
$$
 \Omega_{u} := 
\left\{
x \in \Omega\, \colon\,  \mbox{there exists $\rr\in (0, \rhob{b})$ such that $\tx{E}_{u}(x,\rr)<\epsb{b}$, with $B_{R \rr }(x)\Subset \Omega$}\right\}.
$$
\end{theorem} 
Sets of the type in \rif{ilsingolare} are always open and when $u$ is a solution as consider in Theorem \ref{singth} their complement has Hausdorff dimension strictly less than $n-2s$; see Section \ref{rs.sec} for details. A thorough discussion of the choice of the threshold quantities $\epsb{b}, \rhob{b}, R$ starts from Section \ref{parametri} and in fact the proof of Theorem \ref{singth} can be deduced by following the proofs of the preceding theorems. 
Summarizing, in the case of Theorems \ref{ureg1}-\ref{ureg4} the number $\epsb{b}$ depends only on $\data$ and $ \omega(\cdot)$, while $\rhob{b}$ additionally depends on the decay rate of a certain norm of $f$ and on the oscillations of $a(\cdot)$. Additional dependence on $\beta$ occurs in the case of Theorems \ref{ureg0} and \ref{ureg5}. After this point, the regular set $\Omega_{u}$ for Theorems \ref{gradreg1}-\ref{gradreg2} remains essentially the one determined for Theorem \ref{ureg5}. 

\subsection{Everywhere regularity. Dini continuous coefficients}\label{everywhere} When switching to simpler, linear systems of the type
\eqn{eqweaksol-linear}
$$
			\int_{\mathbb{R}^n} \int_{\mathbb{R}^n} \langle a(x,y)(u(x)-u(y)),\varphi(x)-\varphi(y) \rangle \frac{\dxy}{|x-y|^{n+2s}} = \int_{\Omega} \langle f, \varphi \rangle\dx
$$
our approach leads to everywhere interior results under optimal assumptions. Specifically, 
all the results of Theorems \ref{ureg1}-\ref{gradreg3} hold with $\Omega_{u}=\Omega$, i.e., the singular set is empty. In this respect, our results provide a full analog of classical Campanato's theory for systems \cite{cam1, cam2, giusti}. 
This relies on the basic observation that, when no $u$-dependence occurs in the coefficients $a(\cdot)$, then the smallness assumptions of the type \rif{linearizzanti}$_1$ become useless and the linearization process at the heart  of partial regularity only relies on the verification of \rif{linearizzanti}$_{2,3}$. These are in turn only linked to the so-called ``external ingredients", that is, $a(\cdot)$ and $f$, that do not give rise to any singular set. For this, see the definition of the regular set $\Omega_{u}$ in Section \ref{parametri}. In addition to this, all the local estimates that are implicit in our partial regularity proofs, and that depend on the singular set as for instance \rif{holdest}, become unconditional, everywhere estimates in the interior. Indeed, note that estimates as \rif{holdest} exhibit a  hidden dependence on the solution $u$ via $r_{x_0}$, which is in fact determined via the verification of conditions as 
\rif{linearizzanti}$_1$. Again, when no dependence on the solution $u$ is present in the coefficients the corresponding local estimates hold in every interior ball. 
An example of such an outcome is for instance the following result, which mirrors Theorems \ref{gradreg1}-\ref{gradreg3}:
\begin{theorem}[Linear systems with Dini coefficients]\label{ppaolo}
	Under assumptions \eqref{bs.1}-\eqref{bs.2}, \eqref{bs.5} and \eqref{modco2}-\eqref{assilipanc}, and with $s>1/2$, let $u$ be a weak solution to \eqref{eqweaksol-linear}. Then $u$ is locally $C^{0,1}$-regular in $\Omega$. Moreover, 
	\eqn{finale-linear}
$$  [u]_{0,1;B_{r/8}} \leq  c \, r^{-1}\nra{u}_{L^{2}(B_{r})}+ c\, r^{-1}\tail(u;B_{r})  + c\,  \nr{\mathbf{I}^{f}_{2s-1,\chi}(\cdot, 1)}_{L^\infty(B_{r/2})}
$$
holds for every ball $B_r \Subset \Omega$, with  $c\equiv c(\data,\omega_{\star}(\cdot))$. If \eqref{assilipanc} is reinforced in \eqref{assilipul}, then $Du$ is continuous in $\Omega$; finally if \eqref{assilip2} is assumed, then $u\in C^{1,\alpha}_{\loc}(\Omega;\er^N)$. 
\end{theorem}
Theorem \ref{ppaolo} appears to be new already in the scalar case. It is the optimal nonlocal counterpart of classical local results for linear and nonlinear equations involving Dini continuous coefficients \cite{dudini, hawi, kovats, kumig, KMdini, yanyanli}.  See also the recent and interesting papers on Dini continuous coefficients in the setting of fully nonlinear nonlocal equations \cite{dong1, dong2}. 
Comments on how to get results for systems of the type in \rif{eqweaksol-linear} from those for general systems of type in \rif{eqweaksol} are in Section \ref{indicazioni} below. 

\subsection{Technical approach} In order to obtain Theorems \ref{ureg1}-\ref{singth} we develop and combine a certain number of tools. We believe that some of these are of independent interest and will be useful in the future in other nonlocal settings where partial regularity will be involved. This will already occur in \cite{follow}, where different classes of systems and further cases will be treated. We briefly summarize some of these and other main points in the following.

\subsubsection{A direct Campanato's theory for nonlocal systems.} While solutions to linear nonlocal problems with constant coefficients are regular by more classical tools, the quantitative version of such results is a more delicate aspect. Specifically, we need a suitable set of a priori regularity estimates for solutions in terms of the excess functionals from Definition \ref{leccesso}. The relevant statement is in Proposition \ref{phr}. This contains a decay estimate for solutions to homogeneous systems with constant coefficients that is a natural counterpart of those proved in Campanato's classic papers \cite{cam0, cam00, cam1, cam2}. The proof is direct and short, and makes no use of lifting and extension procedures. Based on this, we can then also deliver the promised analog of Campanato's theory for linear systems of the type \rif{eqweaksol-linear}, including borderline cases.

\subsubsection{Harmonic type nonlocal approximation and blow-up} Our proof of partial H\"older continuity is essentially a blow-up one, see Section \ref{blowupsec}, but avoids any contradiction argument, keeping a direct control on the constants used. These are, in principle, quantifiable at every stage. Such an approach is in spirit close to De Giorgi's original arguments from the regularity theory of minimal surfaces \cite{deg}. The classical Harmonic type approximation lemma of De Giorgi provides a standard paradigm for proving partial regularity assertions; see also \cite{simon1, simon2}. In the local case, it basically asserts that for every $\eps>0$ there exists a universal $\delta \equiv \delta (\eps)$ such that if a map $u \in W^{1,2}(\texttt{B}_1)$ satisfies $\nr{Du}_{L^2(\texttt{B}_1)}\lesssim 1$ and 
\eqn{harmonic}
$$
\left|\int_{\texttt{B}_{1}}\langle Du, D\varphi\rangle \dx\right| \leq  
\delta[\varphi]_{0,1;\texttt{B}_{1}}
$$  
for every $\varphi\in C^{1}_{0}(\texttt{B}_{1})$, 
then there exists  $h\in W^{1,2}(\texttt{B}_{1})$ such that
$$
\nr{h-u}_{L^2(\texttt{B}_{1})} \leq \eps \,, \qquad -\Delta h=0\,, \qquad \nr{Dh}_{L^2(\texttt{B}_1)}\lesssim 1\,.
$$
Several extensions, also covering nonlinear cases, have been proposed  \cite{dumi, dsv,dust, KMSvec}. In  Lemma \ref{shar} below we introduce a fractional analog of such tools; the proof is still direct and avoids contradiction arguments, allowing for a direct control of the constants. The nonlocal version differs from the one we could expect looking at the local case \rif{harmonic}. Indeed, while the H\"older seminorm $[\varphi]_{0, s;\texttt{B}_{1}}$ would be a natural replacement of $[\varphi]_{0,1;\texttt{B}_{1}}$, actually a larger quantity is needed in the nonlocal reformulation of \rif{harmonic}, namely $[\varphi]_{0, t;\texttt{B}_{1}}$ for some $t>s$. This makes a weaker assumption and therefore a stronger formulation of the lemma. It is a fact ultimately linked to the fact that $C^{0, s}$ does not embed in $W^{s,2}$ due to the presence of a singular kernel in the Gagliardo seminorm \rif{gagliardo}, making a duality inequality as \rif{harmonic} ineffective. In other words, making a tentative naive estimation as 
\eqn{unbalance}
$$
\left|\int_{\mathbb{R}^{n}}\int_{\mathbb{R}^{n}}\langle v(x)-v(y),\varphi(x)-\varphi(y)\rangle\frac{\dxy}{\snr{x-y}^{n+2s}}\right|\le [v]_{s,2;\er^n} [\varphi]_{0,s;\er^n}
$$
for every $\varphi \in C^{0,s}(\er^n;\er^N)$ vanishing outside $\ttB_1$ to model a nonlocal analog of \rif{harmonic}, would not give a condition which is compatible with the integrability properties of the kernel $1/|x-y|^{n+2s}$. The lack of pairing occurring in \rif{unbalance} can be overcome thanks to certain self-improving properties recently discovered in \cite{KMS}, that are in fact due to the same presence of the singular kernel; the two things naturally rebalance. The $s$-harmonic approximation needs a few additional tools to be implemented. A main tool is a H\"older counterpart of the classical Lipschitz truncation lemmas \cite{dms}. Roughly speaking, the result allows to replace a $W^{t,p}$-regular function with a $t$-H\"older continuous function outside a set which can be taken small, see Proposition \ref{htrunc}. This is in turn achieved by analysing the level sets of certain fractional sharp maximal operators originally introduced by Calder\'on \& Scott in \cite{calderon}. See Section \ref{truncsec} below.

\subsubsection{Improvement of flatness} This is in Proposition \ref{cor.1} and it is the basic tool to get the subsequent partial regularity statements. It involves the excess in \rif{exc} and its flavour is classical. When $f\equiv 0$ Proposition \ref{cor.1} prescribes that if the oscillations of $a(\cdot)$ are small enough in an appropriately larger ball $B_{R\rr}(x_{0})$, $R\geq 1$, then once the nonlocal excess is small at one scale $B_{\rr}(x_{0})$ then it keeps small at every scale and starts decaying, i.e., 
$$
\tx{E}_{u}(x_{0},\rr)<  \epsb{b} \Longrightarrow \tx{E}_{u}(x_{0},\sigma) \lesssim_{\beta}\sigma^{\beta}
$$
whenever $\beta < \min\{1, 2s\}$ and $0< \sigma \leq \rr$. This is obviously the natural analog to flatness improvement displayed in
\rif{eccessino}-\rif{eccessino2} for the local case. The full version of the statement \rif{exx.5} takes into account the presence of the right-hand side $f$ via an additional smallness condition as the second inequality in \rif{exx.42}. The proof of Proposition \ref{cor.1} is based on a nonlocal linearization technique in turn relying on the application of the $s$-Harmonic Approximation Lemma  \ref{shar}. This is done in Section \ref{sec6}. These are precisely the points where the smallness conditions in \rif{linearizzanti} are fixed. A summary of all the numerical thresholds for linearization is then accordingly given in Section \ref{parametri}, where the regular set $\Omega_{u}$ is determined. 

\subsection{Some existing literature and further developments} \label{furthersec} In the setting of nonlocal problems, partial regularity results have been proved in certain special situations and using ad hoc approaches, very often relying on the so-called Caffarelli-Silvestre extension procedure \cite{caff}. This allows us to lift the problem in $\er^{n+1}$, transforming it into a local, degenerate one with controlled weight; this method fits linear operators-driven equations. For instance, in \cite{colombo1} the authors provided Caffarelli-Kohn-Nirenberg type partial regularity results for hyperdissipative variants of the Navier-Stokes system. The proof goes via $\eps$-regularity theorems and takes advantage of extension procedures. Such results extend those of \cite{katz} where the analysis uses Littlewood-Paley decompositions, again an approach linked to linear  operators. Another set of partial regularity theorems is available when turning to geometric constrained variational problems and, most notably, those involving fractional harmonic maps, starting by the fundamental work \cite{dalio1}. These rely on Fourier analysis and Littlewood-Paley decompositions; see also \cite{schi1, schi2} for important progress. In \cite{millot1, millot2, roberts} partial regularity results and singular sets dimension estimates were obtained for manifold constrained problems. Proofs again rely on extension methods, at least partially.

\vspace{3mm}

{\bf Acknowledgments.} This work is supported by the European Research Council, through the ERC StG project NEW, nr.~101220121 and by the University of Parma through the action ``Bando di Ateneo per la ricerca''. Simon Nowak is supported by the Deutsche Forschungsgemeinschaft (DFG, German Research Foundation) - SFB 1283/2 2021 - 317210226. 
The authors thank Carlo Alberto Antonini for his remarks on a preliminary draft of the paper and the  referees for the attention dedicated to the paper.

\section{Preliminaries}\label{preliminari}
\subsection{Notation}  
In this paper we denote by $c$ a general, finite constant such that $c\geq 1$, which, as usual, may change from line to line. The same will happen with constants playing the same role like $c_*,  \tilde c$ and so on. Relevant dependencies on parameters will be as usual emphasized by putting them in parentheses. By  $\texttt{x} \lesssim \texttt{y}$, with $\texttt{x},\texttt{y}$ being two non-negative real numbers, we mean that  $\texttt{x}\leq c\,  \texttt{y}$ holds for a universal constant $c$, i.e., a constant at most depending on the ambient dimension $n$. In less frequent cases, the constant will depend on a fixed set of parameters that will be clear from the context.  In case we want to emphasize such parameters, when for instance the constant $c$ depends on, to say, $\gamma, \sigma$, we shall denote  $\texttt{x} \lesssim_{\gamma, \sigma} \texttt{y}$. We shall write $\texttt{x} \approx_{\gamma, \sigma} \texttt{y}$ when both $\texttt{x} \lesssim_{\gamma, \sigma} \texttt{y}$ and $\texttt{y} \lesssim_{\gamma, \sigma} \texttt{x}$ occur. In particular, we shall write $\texttt{x} \lesssim 1$ provided there exists an absolute constant $c$, i.e., depending at most on  $n$, such that $\texttt{x} \leq c$, and $\texttt{x} \approx 1$ when $1/c \leq \texttt{x} \leq c$. A similar meaning occurs when using the notation $\texttt{x} \approx_{\gamma, \sigma}  1$. We shall always abbreviate $0_{\er^n}\equiv 0$. With $x_{0} \in \er^n$ and $r>0$, we denote 
\eqn{zerocenter}
$$
B_r(x_{0}):= \{x \in \er^n  :   |x-x_{0}|< r\}, \quad \ttB_{r} \equiv \ttB_{r}(0):= \{x \in \er^n  :   |x|< r\}\,.
$$
We shall omit denoting the center, i.e., abbreviating $B_r \equiv B_r(x_{0})$ when no ambiguity will arise; this will often be the case when various balls in the same context share the same center. When no need to specify center and radius will occur we shall denote a generic  ball of $\er^n$ by $\BBB$ and we shall denote by $x_\BBB$ its center. Moreover, with $\BBB$ being a given ball with radius $r$ and $\gamma$ being a positive number, we denote by $\gamma \BBB$ the concentric ball with radius $\gamma r$ and by $\BBB/\gamma \equiv (1/\gamma){\BBB}$. Finally, with $\BBB\equiv B_r(x_{0})\subset \er^n$ being a fixed ball, we shall denote
\eqn{ballnotation}
$$
\mathcal B_r(x_{0}) :=B_r(x_{0}) \times B_r(x_{0}) \subset \er^{2n} \,, \qquad \mathcal B(\BBB):=\BBB  \times \BBB  \,.
$$
Also in this case, when no confusion on the identity of the center of the ball shall arise, we shall abbreviate $ \mathcal B\equiv \mathcal B_r \equiv \mathcal B_{r}(x_{0})$ and $\mathcal B\equiv \mathcal B(\BBB)$, respectively. In the following we denote by $\langle \cdot, \cdot \rangle$ the standard scalar product in $\er^N$; as usual, $\en$ denotes the set of positive integer numbers, and $\en_0=\en\cup \{0\}$ the set of non-negative integers. In denoting several function spaces like $L^{p}(\Omega), W^{1,p}(\Omega)$, we shall denote the vector valued version by $L^{p}(\Omega,\er^k), W^{s,p}(\Omega,\er^k)$ in the case the maps considered take values in $\er^k$, $k\in \mathbb{N}$; when confusion will not arise, we shall abbreviate $L^{p}(\Omega;\er^k)\equiv L^{p}(\Omega)$ and the like. With $\mathcal  A \subset \er^{n}$ being a measurable subset such that  $0<|\mathcal A|<\infty$, and $w \colon \mathcal  A \to \er^{k}$, $k\geq 1$, being an integrable map, we denote by
$$(w)_{\mathcal  A}:=\frac{1}{\snr{\mathcal  A}}\int_{\mathcal  A}w\dx:= \mint_{\mathcal  A}w\dx$$ its integral average. With $\alpha \in (0,1]$
$$
[w]_{0,\alpha;\mathcal  A}:= \sup_{x,y\in \mathcal  A; x\neq y}\frac{\snr{w(x)-w(y)}}{\snr{x-y}^\alpha}
$$
is the usual $\alpha$-H\"older seminorm of $w$ in $\mathcal  A$ and, with $t\in (0,1)$ and $p\geq 1$
\eqn{gagliardo}
$$
[w]_{t,p;\mathcal A} :=\left(\int_{\mathcal A}\int_{\mathcal A}\frac{\snr{w(x)-w(y)}^{p}}{\snr{x-y}^{n+pt}}\dxy\right)^{1/p}
$$
defines the usual Gagliardo seminorm related to the fractional space $W^{t,p}(\mathcal A)$, where this time $\mathcal A$ is an open subset. 
Accordingly, we shall denote 
\eqn{shalldenote}
$$
\nra{w}_{L^p(\mathcal A)} := \left(\mint_{\mathcal A}\snr{w}^p\dx\right)^{1/p} \,, \qquad 
\snra{w}_{t,p;\mathcal A} := \left(\frac{1}{\snr{\mathcal A}} \int_{\mathcal A}\int_{\mathcal A}\frac{\snr{w(x)-w(y)}^{p}}{\snr{x-y}^{n+pt}}\dxy\right)^{1/p}\,.
$$
The Gagliardo norm of a fractional Sobolev function $w \in W^{t,p}(\mathcal A)$ is defined, as usual, by 
$\nr{w}_{W^{t,p}(\mathcal A)}:= \nr{w}_{L^{p}(\mathcal A)}+[w]_{t,p;\mathcal A}$. An inequality involving Marcinkiewicz norms defined in \rif{mardef} is
\eqn{marhol}
$$
\nra{w}_{L^{p}(\mathcal A)}\lesssim_{d}  \frac{\snr{\mathcal A}^{-1/d}}{(d-p)^{1/p}}\nr{w}_{\mathcal{M}^{d}(\mathcal A)}
$$
valid whenever $1\leq p < d$ and $\mathcal A\subset \er^n$ is a measurable subset with finite measure  (see for instance \cite[Lemma 5.1]{kumi} for a proof). We now recall a series of standard inequalities involving the quantities in \rif{shalldenote}. We shall often use the following Poincar\'e-Friedrichs inequality (see for instance \cite[(2.6)]{bls}):
\eqn{poincf}
$$
\|w\|_{L^p(\ti{\Omega})} \lesssim_{n,t,p}|\ti{\Omega}|^{t/n} 
[w]_{t,p; \er^n}
$$
valid whenever $w\in W^{t,p}(\er^n)$ is such that $w\equiv 0$ outside $\ti{\Omega}$, $0<t<1 \leq p$ and $\ti{\Omega}$ is a bounded domain. This is related to the classical fractional Sobolev inequality
\eqn{poincf2}
$$
\nr{w}_{L^{p^*_t}(\er^n)} \lesssim_{n,t,p} [w]_{t,p;\er^n} \,,\qquad p^*_t:= \frac{np}{n-pt} 
$$
that holds whenever $w$ is compactly supported \cite[Theorem 6.5]{dpv} and $pt<n$; recall that here we are denoting $2^*=2^*_s$. 
We shall also use the following classical fractional Poincar\'e inequality \cite[(4.2)]{min03}:
\eqn{poinfrac}
$$
\nra{w-(w)_{\BBB}}_{L^{p}(\BBB)} \lesssim_{n,t,p}  \snr{\BBB}^{t/n}  \snra{w}_{t,p;\BBB}\,.
$$
Similarly, again when $pt<n$, the fractional Sobolev-Poincar\'e inequality \cite[Lemma 2.4]{sn} asserts
\eqn{sobpoinfrac}
$$
\nra{w-(w)_{\BBB}}_{L^{p^*_t}(\BBB)} \lesssim_{n,t,p}  \snr{\BBB}^{t/n}  \snra{w}_{t,p;\BBB}\,.
$$
Less familiar forms of such inequalities are in the following:
\begin{lemma} Let $\BBB\subset \er^n$ be a ball and $w \in W^{t,p}(\BBB)$, $0< t < 1 \leq p$, be such that $
\snr{\BBB} \leq c_{0} \snr{\{w = 0 \}\cap \BBB}
$ 
is satisfied for some $c_{0}\geq 1$. Then 
\eqn{poin1}
$$
\nra{w}_{L^{p}(\BBB)} \lesssim_{n,t,p,c_{0}}  \snr{\BBB}^{t/n}  \snra{w}_{t,p;\BBB}\,.
$$
If $pt<n$ and $p^*_t$ is as in \eqref{poincf2}, then 
\eqn{poin2}
$$
\nra{w}_{L^{p^*_t}(\BBB)} \lesssim_{n,t,p,c_{0}}  \snr{\BBB}^{t/n}  \snra{w}_{t,p;\BBB}
$$
and therefore
\eqn{compa}
$$
\nra{w}_{L^{2}(\BBB)}\lesssim_{n,t,p,c_{0}}  \snr{\BBB}^{t/n}   \snra{w}_{t,p;\BBB}\,, \qquad p \geq \frac{2n}{n+2t}\,.
$$
\end{lemma}
\begin{proof} 
Estimate \rif{poin1} is \cite[(2.2)]{brascoparini} and \cite[Lemma 7]{kokupa}. For \rif{poin2}, via a standard scaling argument it is sufficient to prove it when $\BBB\equiv \ttB_1$. Then, (fractional) Sobolev embedding yields $\nra{w}_{L^{p^*_t}(\ttB_1)} \lesssim \nra{w}_{L^{p}(\ttB_1)}+   \snra{w}_{t,p;\ttB_1}$ and the assertion follows by \rif{poin1}. For \rif{compa} when $p\geq 2$ we simply estimate $\nra{w}_{L^{2}(\BBB)}\leq \nra{w}_{L^{p}(\BBB)}$ and conclude directly via \rif{poin1}. When $1<p<2$, note that $pt <n$ and the lower bound on $p$ in \rif{compa} implies $p^*_t \geq 2$. Therefore $\nra{w}_{L^{2}(\BBB)}\leq \nra{w}_{L^{p^*_t}(\BBB)}$ and this time conclude using \rif{poin2}. \end{proof}
With $w \in L^{p}(\BBB;\er^N)$, $p\geq 1$ and $\BBB \subset \er^n$ being a ball, we shall often use the following elementary property:
\eqn{minav}
$$
\nra{w-(w)_{\BBB}}_{L^{p}(\BBB)}\le 2\nra{w-z}_{L^{p}(\BBB)}\qquad \mbox{for all} \ \ z\in \er^N\,.
$$ 
In the above inequality we can drop the constant $2$ when $p=2$. 
Using this it is not difficult to prove, again for every $p\geq 1$, that
\eqn{scav}
$$
B_{\rr}(x_{1})\subseteq B_{\rrr}(x_{2}) \ \Longrightarrow \ \nra{w-(w)_{B_{\rr}(x_{1})}}_{L^{p}(B_{\rr}(x_{1}))}\lesssim_{n,p}\left(\frac{\rrr}{\rr}\right)^{n/p}\nra{w-(w)_{B_{\rrr}(x_{2})}}_{L^{p}(B_{\rrr}(x_{2}))}\,.
$$
\begin{lemma}\label{dedicata}
Let $w\in L^{1}_{2s}$. If $B_{\rr}\Subset B_{\rrr}\subset \mathbb{R}^{n}$ are concentric balls, then 
\begin{flalign}\label{scatail}
\tail(w-(w)_{B_{\rr}};B_{\rr}) &\lesssim_{n,s} \left(\frac{\rr}{\rrr}\right)^{2s}\tail(w-(w)_{B_{\rrr}};B_{\rrr})+\left(\frac{\rr}{\rrr}\right)^{2s}\nra{w-(w)_{B_{\rrr}}}_{L^1(B_{\rrr})}\nonumber \\  
&\qquad +\int_{\rr}^{\rrr}\left(\frac{\rr}{\lambda}\right)^{2s}\nra{w-(w)_{B_{\lambda}}}_{L^1(B_{\lambda})}\frac{\dlam}{\lambda}
\end{flalign}
and 
\begin{flalign}\label{scatailancora}
	\tail(w;B_{\rr}) &\lesssim_{n,s} \left(\frac{\rr}{\rrr}\right)^{2s}\tail(w;B_{\rrr}) + \left(\frac{\rr}{\rrr}\right)^{2s} \nra{w}_{L^1(B_{\rrr})} +\int_{\rr}^{\rrr}\left(\frac{\rr}{\lambda}\right)^{2s}\nra{w}_{L^1(B_{\lambda})}\frac{\dlam}{\lambda}\,,
\end{flalign}
where the $\tail$ has been defined in \eqref{lacoda}. 
Moreover, if $B_{\rr}(x_{1})\subset B_{\rrr}(x_{2})$ are two not necessarily concentric balls, then
\begin{flalign}\label{scatail.1}
\tail(w-(w)_{B_{\rr}(x_{1})};B_{\rr}(x_{1}))&\lesssim_{n,s} \left(\frac{\rr}{\rrr}\right)^{2s}\left(\frac{\rrr}{\rrr-\snr{x_{1}-x_{2}}}\right)^{n+2s}\tail(w-(w)_{B_{\rrr}(x_{2})};B_{\rrr}(x_{2}))\nonumber \\
& \qquad +\left(\frac{\rrr}{\rr}\right)^{n}\nra{w-(w)_{B_{\rrr}(x_{2})}}_{L^{1}(B_{\rrr}(x_{2}))}\,. 
\end{flalign}
\end{lemma}
\begin{proof}
Estimate \rif{scatail} can be obtained as a small variant of the proof of \cite[Lemma 2.4, (3)]{KMS1}. For \rif{scatailancora}, with $x_{0}$ indicating the center of $B_{\rrr}$, we first split as
\eqn{naivesplit}
	$$ \tail(w;B_{\rr}) = \left(\frac{\rr}{\rrr}\right)^{2s}\tail(w;B_{\rrr}) + \rr^{2s} \int_{B_\rrr \setminus B_\rr} \frac{|w(y)|}{|y-x_{0}|^{n+2s}}\dy\,. $$
We then continue splitting in annuli the last term  in the previous display. Fix $\texttt{t} \in (1/4, 1/2]$ and determine 
$k\in \en_0$ such that $\texttt{t}^{k+1}\rrr< \rr \leq  \texttt{t}^{k}\rrr$; then we have 
	\begin{align*}
		\rr^{2s} \int_{B_\rrr \setminus B_\rr} \frac{|w(y)|}{\snr{y-x_{0}}^{n+2s}} \dy &\leq \rr^{2s} \int_{B_\rrr \setminus B_{\texttt{t}^{k+1}\rrr}} \frac{|w(y)|}{\snr{y-x_{0}}^{n+2s}} \dy\\
		& =  \rr^{2s} \sum_{i=0}^k \int_{B_{\texttt{t}^i \rrr} \setminus B_{\texttt{t}^{i+1} \rrr}} \frac{|w(y)|}{\snr{y-x_{0}}^{n+2s}} \dy  \lesssim   \sum_{i=0}^k \texttt{t}^{2s(k-i)} \nra{w}_{L^1(B_{\texttt{t}^{i} \rrr})}\nonumber  \,. \nonumber
	\end{align*}
In the case $k=0$ \rif{scatailancora} follows using this last inequality and \rif{naivesplit}. Otherwise we continue to estimate as follows:
\begin{flalign}
		\rr^{2s} \int_{B_\rrr \setminus B_\rr} \frac{|w(y)|}{\snr{y-x_{0}}^{n+2s}} \dy 
		& \lesssim   \texttt{t}^{2sk} \nra{w}_{L^1(B_{\rrr})}+  \sum_{1\leq i\leq k} \int^{\texttt{t}^{i-1}\rrr}_{\texttt{t}^{i}\rrr} \texttt{t}^{2s(k-i)} \nra{w}_{L^1(B_{\lambda})}\, \frac{\dlam}{\lambda}   \nonumber\\
		& \lesssim  \left(\frac{\rr}{\rrr} \right )^{2s} \nra{w}_{L^1(B_{\rrr})} +  \sum_{1\leq i\leq k} \int^{\texttt{t}^{i-1}\rrr}_{\texttt{t}^{i}\rrr} \left(\frac{\rr}{\lambda} \right )^{2s} \nra{w}_{L^1(B_{\lambda})}\, \frac{\dlam}{\lambda}  \nonumber\\
		& \lesssim_{n,s}  \left(\frac{\rr}{\rrr} \right )^{2s} \nra{w}_{L^1(B_{\rrr})}+  \int_{\rr}^{\rrr}\left(\frac{\rr}{\lambda}\right)^{2s}\nra{w}_{L^1(B_{\lambda})}\frac{\dlam}{\lambda} \,. \nonumber
	\end{flalign}
Matching the content of the last display with the one of \rif{naivesplit} completes the proof of \rif{scatailancora}. Finally, the proof of \rif{scatail.1} follows \cite[Lemma 2.3]{bls} with minor modifications, and we report it for the sake of completeness. 
Observe that if $y\in \er^n\setminus B_{\rrr}(x_{2})$, then we have
\eqn{esco1}
$$
 \snr{y-x_{1}} \geq  \snr{y-x_2}- \snr{x_{1}-x_2}   \geq \frac{\rrr- \snr{x_{1}-x_2}}{\rrr} \snr{y-x_2}\,.
$$
We decompose  
\begin{flalign*}
\tail(w-(w)_{B_{\rr}(x_{1})};B_{\rr}(x_{1})) &\leq  \rr^{2s} \int_{\er^n\setminus B_{\rrr}(x_2)}
\frac{\snr{w-(w)_{B_{\rrr}(x_2)}}}{\snr{y-x_{1}}^{n+2s}} \dy
\\ & \quad +\rr^{2s} \int_{ B_{\rrr}(x_2)\setminus B_{\rr}(x_{1})}
\frac{\snr{w-(w)_{B_{\rrr}(x_2)}}}{\snr{y-x_{1}}^{n+2s}} \dy\\
&\quad  + \rr^{2s} \int_{\er^n\setminus B_{\rr}(x_{1})}
\frac{\snr{(w)_{B_{\rr}(x_{1})}-(w)_{B_{\rrr}(x_2)}}}{\snr{y-x_{1}}^{n+2s}} \dy\\ & =: I_1+I_2+I_3\,.
\end{flalign*}
By \rif{esco1} we have 
$$
I_1 \leq \left(\frac{\rrr}{\rrr- \snr{x_{1}-x_2}}\right)^{n+2s}\left(\frac{\rr}{\rrr}\right)^{2s}\tail(w-(w)_{B_{\rrr}(x_{2})};B_{\rrr}(x_{2}))\,.
$$
Then, using also \rif{scav}, we obtain
$$
I_2+I_3 \lesssim_{n,s} \left(\frac{\rrr}{\rr}\right)^{n}\nra{w-(w)_{B_{\rrr}(x_{2})}}_{L^{1}(B_{\rrr}(x_{2}))}
$$
and the proof of \rif{scatail.1} follows combining the content of the last three displays. 
\end{proof}
\begin{remark}[Miscellanea of consequences]{\em In the following, when manipulating the excess functional $\tx{E}_{w}$ defined in \rif{excl}, we shall often use the following obvious equivalence:
\begin{flalign}
\notag \nra{w-\ell}_{L^{2}(B_{\rr})}+\tail(w-\ell;B_{\rr}) &\leq \sqrt{2}\tx{E}_{w}(\ell;x_{0},\rr)\\
& \leq \sqrt{2}\left( \nra{w-\ell}_{L^{2}(B_{\rr})}+\tail(w-\ell;B_{\rr})\right)\,.\label{ovvia}
\end{flalign}
A similar relation obviously holds for $ \tx{E}_{w}(x_{0},\rr)$ defined in \rif{exc}. Lemma \ref{dedicata} implies 
\eqn{scataildopoff}
$$
\begin{cases}
\tx{E}_{w} (x_{0}, \texttt{t}  r)\lesssim_{n,s} \texttt{t}^{-n/2} \tx{E}_{w} (x_{0} , r)\,, \qquad  \forall \ \texttt{t} \in (0,1]\\
\tx{E}_{w} (x, \texttt{t}r)\lesssim_{n,s} \texttt{t}^{-n} \tx{E}_{w} (x_{0} , r)\,, \qquad \forall \ x \in B_{(1-\texttt{t})r}(x_0)\,,  \texttt{t} \in (0,1)
\end{cases}
$$
provided 
 $w \in L^2(B_{r}(x_{0});\er^N)\cap L^{1}_{2s}$. Inequality \rif{scataildopoff}$_2$ is a direct consequence of \rif{scav} and \rif{scatail.1}; for 
 \rif{scataildopoff}$_1$ we apply  \rif{scav} repeatedly in conjunction with \rif{scatail} in order to get 
\begin{flalign*}
\tx{E}_{w} (x_{0}, \texttt{t}  r)&\lesssim\left(\texttt{t}^{-n/2}+\texttt{t}^{2s}\right)\nra{w-(w)_{B_{r}(x_0)}}_{L^{2}(B_{r}(x_0))}+ \texttt{t}^{2s}\tail(w-(w)_{B_{r}(x_0)};B_{r}(x_0))\nonumber \\
&\quad + \int_{\texttt{t}r}^{r}\left(\frac{\texttt{t}r}{\lambda}\right)^{2s}\nra{w-(w)_{B_{\lambda}(x_0)}}_{L^{2}(B_{\lambda}(x_0))}\frac{\dlam}{\lambda}\nonumber\\ &
 \lesssim\texttt{t}^{-n/2}\tx{E}_{w} (x_{0} , r)+ \texttt{t}^{2s }r^{2s+n/2}\int_{\texttt{t}r}^{r}\lambda^{-2s-n/2}\frac{\dlam}{\lambda}\,\nra{w-(w)_{B_{r}(x_0)}}_{L^{2}(B_{r}(x_0))}\\
 &  \lesssim\texttt{t}^{-n/2}\tx{E}_{w} (x_{0} , r) \,.
\end{flalign*}
A simpler form of \rif{scatail.1} is 
\eqn{scatailggg}
$$
\tail(w;B_{\rr}(x_{1}))\lesssim_{n,s} \left(\frac{\rr}{\rrr}\right)^{2s}\left(\frac{\rrr}{\rrr-\snr{x_{1}-x_{2}}}\right)^{n+2s}\tail(w;B_{\rrr}(x_{2}))+\left(\frac{\rrr}{\rr}\right)^{n}\nra{w}_{L^{1}(B_{\rrr}(x_{2}))}
$$
that again holds provided $B_{\rr}(x_{1})\subset B_{\rrr}(x_{2})$ and whose proof can be obtained as the one of \rif{scatail.1}. We conclude pointing out an elementary fact we shall repeatedly use in the rest of the paper. Assume $\BBB\Subset \ti{\BBB} \subset \er^n$ are two balls centred at $x_{\BBB} \in \er^n$. Then 
\eqn{uset}
$$
\frac{\snr{y-x_{\BBB}}}{\snr{y-x}}\le 1+\frac{\rad(\BBB)}{\dist(\BBB,\partial \ti{\BBB})}
$$
holds whenever $y\in \er^n\setminus \ti{\BBB}$ and $x\in \bar{ \BBB}$, where $\rad(\BBB)$ denotes the radius of $\BBB$.  
}
\end{remark}
\subsection{Affine maps and the \tail}\label{affinetail} Let  $\ell(\ti{x}):= \bbt(\ti{x}-x)+\aaat$, $\bbt \in \mathbb{R}^{N \times n}$, $\aaat \in \mathbb{R}^N$, be an arbitrary affine map, with $\ti{x}, x\in \er^n$; direct computations show that, whenever $\rr>0$
\eqn{tritri}
$$
\begin{cases}
\displaystyle \nra{\ell}_{L^2(B_{\rr}(x))} \approx_{n}  \snr{\bbt} \rr + \snr{\aaat} \\[5pt]
\displaystyle 
\tail(\ell;B_{\rr}(x)) \lesssim_{n} \frac{\snr{\bbt} \rr}{2s-1} + \frac{\snr{\aaat}}{2s}  \lesssim_{n,s} \snr{\bbt} \rr + \snr{\aaat} \,, \quad \mbox{provided $s>1/2$ \ (\mbox{$s>0$ if $\bbt=0$})}\,.
\end{cases}
$$
The upper bound in \rif{tritri}$_2$ is a direct consequence of the definitions. For the reader's convenience, we spend a few words on the proof of \rif{tritri}$_1$. This follows either by a direct computation or by the following argument, that works for every $p\geq 1$. 
First note that 
$$  \big(\mint_{\texttt{B}_1} |\mathfrak{b}\ti{x} +\mathfrak{a}|^p \d\ti{x}\big)^{1/p}\approx |\mathfrak{b}|+|\mathfrak{a}|$$ for every $(\mathfrak{b},\mathfrak{a})\in \er^{N\times n}\times \er^N,$ as the first quantity defines a norm on $\er^{N\times n}\times \er^N$ (all norms are equivalent in finite dimensions). By scaling it follows 
$$
 \nra{\ell}_{L^p(B_{\rr}(x))} =\left( \mint_{\ttB_{1}} | \rr \bbt\tilde x+\aaat|^p \d\ti{x}\right)^{1/p}   \approx  \snr{\bbt} \rr + \snr{\aaat}\,.
$$
\begin{proposition}\label{affinemente}
Let $w \in L^{2}(\BBB;\mathbb{R}^N)\cap L^{1}_{2s}$ with $s>1/2$, where $\BBB\subset \er^n$ is a ball, $\rr>0$ is a fixed number and let $\{\ell_k\}$ be a sequence of affine maps such that $\ell_k(\ti{x}):=\bbt_k(\ti{x}-x_k) +\aaat_k$, and such that $\bbt_k\to \bbt $ in $\er^{N\times n}$, $\aaat_k\to \aaat $, $x_k \to x$ in $\er^n$ and $B_\rr(x_k), B_{\rr}(x) \subset \BBB$. We have 
\eqn{affini1}
$$ \lim_{k\to \infty} \, \tx{E}_{w}(\ell_k;x_k,\rr) =  \tx{E}_{w}(\ell;x,\rr)$$
where $\ell(\ti{x}):=\bbt(\ti{x}-x) +\aaat$. 
In particular, it holds that 
\eqn{affini2}
$$
\begin{cases}
\displaystyle \lim_{k\to \infty}\,   \tx{E}_{w}(x_k,\rr) =  \tx{E}_{w}(x,\rr)\quad (\mbox{holds in the full range $0<s<1$}) \\[7pt]
\displaystyle \lim_{k\to \infty} \, \tx{E}_{w}(\ell;x_k,\rr) =  \tx{E}_{w}(\ell;x,\rr) \ \ \mbox{for every affine map $\ell$.}
 \end{cases}
 $$
\end{proposition}
\begin{proof}
The assumptions and \rif{tritri}$_1$ imply that $\nra{\ell_k-\ell}_{L^{2}(B_{\rr}(x_k))}\to 0$, so that the absolute continuity of Lebesgue integral yields that
$ \nra{w-\ell_k}_{L^{2}(B_{\rr}(x_k))}\to \nra{w-\ell}_{L^{2}(B_{\rr}(x))}$ and therefore we only need to prove
\eqn{continua}
$$
\lim_{k \to \infty}\tail(w-\ell_k;B_{\rr}(x_k))=\tail(w-\ell;B_{\rr}(x))\,.
$$
For this, in the rest of the proof we shall assume with no loss of generality that $\snr{x_k-x}< \rr/2$, and write
\begin{flalign}
\notag &\left|\tail(w-\ell_k;B_{\rr}(x_k))-\tail(w-\ell;B_{\rr}(x)) \right|\nonumber \\
\notag & \qquad \leq c \rr^{2s} \left|
\int_{\er^n\setminus B_{\rr}(x_k)} \frac{\snr{w(y)-\ell_k(y)}}{\snr{y-x_k}^{n+2s}}\dy- \int_{\er^n\setminus B_{\rr}(x_k)} \frac{\snr{w(y)-\ell(y)}}{\snr{y-x_k}^{n+2s}}\dy \right|\\
\notag & \qquad  \quad  \ +c \rr^{2s} \left|
\int_{\er^n\setminus B_{\rr}(x_k)} \frac{\snr{w(y)-\ell(y)}}{\snr{y-x_k}^{n+2s}}\dy-\int_{\er^n\setminus B_{\rr}(x_k)} \frac{\snr{w(y)-\ell(y)}}{\snr{y-x}^{n+2s}}\dy \right|\\
\notag & \qquad  \quad  \ +c \rr^{2s} \left|
\int_{\er^n\setminus B_{\rr}(x_k)} \frac{\snr{w(y)-\ell(y)}}{\snr{y-x}^{n+2s}}\dy-\int_{\er^n\setminus B_{\rr}(x)} \frac{\snr{w(y)-\ell(y)}}{\snr{y-x}^{n+2s}}\dy \right|\\
& \qquad =: \textnormal{T}_{1} +\textnormal{T}_{2}+ \textnormal{T}_{3}\,.\label{continua2}
\end{flalign}
Obviously, using triangle inequality and \rif{tritri}$_2$ we find
\eqn{tt1}
$$
\textnormal{T}_{1} \lesssim_{n,s} \tail(\ell_k-\ell;B_{\rr}(x_k)) \lesssim  \snr{\bbt_k-\bbt} \rr + \snr{\bbt}\snr{x_k-x}+\snr{\aaat_k-\aaat} \Longrightarrow \textnormal{T}_{1}\to 0\,.
$$
In order to estimate $\textnormal{T}_{2}$, note that $\snr{x_k-x}\leq \rr/2$ and triangle inequality imply
\eqn{imim}
$$ y \in \er^n\setminus B_{\rr}(x_k) \Longrightarrow \snr{y-x}\geq  \rr/2\,,$$
and, yet another application of triangle inequality gives also
$$
\frac{\snr{y-x_k}}{\snr{y-x}} \leq 1 + \frac{2\snr{x_k-x}}{\rr}  \qquad \mbox{and}\qquad \frac{\snr{y-x}}{\snr{y-x_k}} \leq 1 + \frac{\snr{x_k-x}}{\rr} 
$$
and, in any case, for $k$ large enough depending on $\rr$ but not on $y$
$$ 
\frac{\snr{y-x_k}}{\snr{y-x}} \approx 1 \,.
$$
The last two inequalities yield, again for $k\equiv k(\rr)$ large enough
\begin{flalign*}
\left|\frac{1}{\snr{y-x_k}^{n+2s}}-\frac{1}{\snr{y-x}^{n+2s}}\right| & \lesssim \frac{1}{\snr{y-x_k}^{n+2s}}
\left[\left(1 + \frac{2\snr{x_k-x}}{\rr}\right)^{n+2s}-1\right]\\
&\lesssim \frac{\snr{x_k-x}}{\rr} \frac{1}{\snr{y-x_k}^{n+2s}}\,.
\end{flalign*}
Using this together with \rif{scatailggg} (applied with $\rrr=2\rr$) and \rif{tritri}$_2$, and yet recalling that $\snr{x_k-x}\leq \rr/2$, we bound 
\begin{flalign}
\notag \textnormal{T}_{2} &\lesssim \frac{\snr{x_k-x}}{\rr}  \rr^{2s} 
\int_{\er^n\setminus B_{\rr}(x_k)} \frac{\snr{w(y)-\ell(y)}}{\snr{y-x_k}^{n+2s}}\dy\\
& \lesssim  
 \frac{\snr{x_k-x}}{\rr}   \left(\tail(w;B_{\rr}(x_k))+\tail(\ell;B_{\rr}(x_k))\right)
\notag\\
 & \lesssim 
 \frac{\snr{x_k-x}}{\rr}   \left(\tail(w;B_{2\rr}(x))+\snr{\bbt}\rr+\snr{\aaat}+ \nra{w}_{L^1(B_{2\rr}(x))}\right)
 \Longrightarrow  \textnormal{T}_{2}\to 0\,.\label{tt2}
\end{flalign}
Finally, again recalling \rif{imim}, we observe that $y \in B_{\rr}(x_k)\triangle B_{\rr}(x)$ implies that $\snr{y-x}\geq \rr/2$ and therefore, also using that $(\er^n\setminus B_{\rr}(x_k) ) \triangle (\er^n\setminus B_{\rr}(x) ) = B_{\rr}(x_k)\triangle B_{\rr}(x)$, we can estimate
\begin{flalign}
\notag \textnormal{T}_{3}&\lesssim  \rr^{2s}\int_{B_{\rr}(x_k)\triangle B_{\rr}(x)}\frac{\snr{w(y)-\ell(y)}}{\snr{y-x}^{n+2s}}\dy \\& \lesssim \rr^{-n}\int_{B_{\rr}(x_k)\triangle B_{\rr}(x)} \snr{w}\dy +\rr^{-n}\snr{B_{\rr}(x_k)\triangle B_{\rr}(x)}\left(\snr{\bbt}\rr +\snr{\aaat}\right)\Longrightarrow \textnormal{T}_{3}\to 0\,.\label{tt3}
\end{flalign}
Obviously, \rif{tt1}, \rif{tt2} and \rif{tt3} together with \rif{continua2} imply \rif{continua}. This completes the proof of \rif{affini1}. As for \rif{affini2}$_1$, this follows observing that $\tx{E}_{w}(x_k,\rr)= \tx{E}_{w}(\ell_k;x_k,\rr)$ where $\ell_k\equiv (w)_{B_{\rr}(x_k)}$ and obviously $(w)_{B_{\rr}(x_k)}\to (w)_{B_{\rr}(x)}$; note that in this case the above argument, and especially \rif{tt1}$_2$, does not require the restriction $s>1/2$ as specified in \rif{tritri}$_2$. Finally, for \rif{affini2}$_2$, with $\ell(\ti{x})=\bbt(\ti{x}-x) +\aaat$, $\bbt \in \mathbb{R}^{N \times n}$, $\aaat \in \mathbb{R}^N$, the assertion follows from \rif{affini1} writing $\ell(\ti{x})=\ell_k(\ti{x}):=\bbt(\ti{x}-x_k) +\aaat + \bbt(x_k-x)$. 
\end{proof}
\begin{remark}[Minimizing the excess]\label{esisteminimo}\em{Again with $w \in L^{2}(B_{\rr}(x);\mathbb{R}^N)\cap L^{1}_{2s}$ and $s>1/2$, we are now interested in the fact that the minimization problem
\eqn{unicaff}
$$  \tx{E}_{w}(\ell_x;x,\rr):=  \inf_{\ell \textnormal{ affine}} \tx{E}_{w}(\ell;x,\rr)$$ 
is uniquely solvable, i.e., there exists a unique affine map $\ell_x$ attaining equality in \rif{unicaff}. We shall use that, needless to say, solving \rif{unicaff} is the same as solving $  \tx{E}_{w}(\ell_x;x,\rr)^2:=  \inf_{\ell \textnormal{ affine}} \tx{E}_{w}(\ell;x,\rr)^2$. For this, take any affine map $\ell$, that we write in the form $\ell(\ti{x}):=\bbt(\ti{x}-x)+\aaat$, $\bbt\in \er^{N\times n}, \aaat \in \er^N$. We identify such maps with couples $(\bbt, \aaat)\in \er^{N\times n}\times \er^{N}$ so that minimization in \rif{unicaff} is equivalent 
to find $\ell_x(x)\in \er^{N}$ and $D\ell_x\in \er^{N\times n}$ such that 
$$(D\ell_x, \ell_x(x))\in  \argmin_{(\bbt, \aaat)\in \er^{N\times n}\times \er^{N}}\, \texttt{f}(\bbt, \aaat)\,,$$ 
where
$$\texttt{f}(\bbt, \aaat):=\texttt{f}_{\textnormal{i}}(\bbt, \aaat) + \texttt{f}_{\textnormal{e}}(\bbt, \aaat)  :=\nra{w-\ell}_{L^{2}(B_{\rr}(x))}^2+\tail(w-\ell;B_{\rr}(x))^2= \tx{E}_{w}(\ell;x,\rr)^2 $$ for every $(\bbt, \aaat)\in \er^{N\times n}\times \er^{N}$. This function is coercive in the sense that $\texttt{f}(\bbt, \aaat)\to \infty$ when $\snr{\aaat}+\snr{\bbt}\to \infty$. Indeed, use \rif{tritri}$_1$ in the form 
\eqn{bowie}
$$
\snr{\bbt} \rr + \snr{\aaat} \lesssim \nra{\ell}_{L^2(B_{\rr}(x))} \lesssim  \sqrt{\texttt{f}_{\textnormal{i}}(\bbt, \aaat)} +\nra{w}_{L^2(B_{\rr}(x))} \lesssim  \sqrt{\texttt{f} (\bbt, \aaat)}+\nra{w}_{L^2(B_{\rr}(x))}\,.
$$
Finally, it is also strictly convex as both $\texttt{f}_{\textnormal{i}}$ and $\texttt{f}_{\textnormal{e}}$ are convex and $\texttt{f}_{\textnormal{i}}$ is strictly convex. Existence and uniqueness for \rif{unicaff} now follow by standard Direct Methods in finite dimensional vector spaces (recall that every finite convex function on a finite-dimensional
space is continuous).}
\end{remark}
\begin{proposition}\label{continues} Let $w$ as in Proposition \ref{affinemente}, and, whenever $B_{\rr}(x)\subset \BBB$, denote by $\ell_{x}$ the unique affine map such that 
\eqref{unicaff} holds. 
Then the map $x \mapsto D\ell_{x}$ is continuous. 
\end{proposition}
\begin{proof} Let $\{x_k\}\subset \er^n$ be a sequence such that $B_{\rr}(x_k)\subset \BBB$ and $x_k \to x$ and denote $\ell_{x_k}(\ti{x}) := \bbt_k(\ti{x}-x_k)+\aaat_k$ for suitable $\bbt_k \in \mathbb{R}^{N \times n}$, $\aaat_k \in \mathbb{R}^N$ such that 
\eqn{affini3}
$$  \tx{E}_{w}(\ell_{x_k};x_{k},\rr)=  \inf_{\ell \textnormal{ affine}} \tx{E}_{w}(\ell;x_{k},\rr)\,. $$
 By minimality it is easy to see that the sequences $\{\aaat_k\} \subset \er^N$ and $\{\bbt_k\} \subset \er^{N\times n}$ are bounded. Indeed, testing the right-hand side in \rif{affini3} with $\ell\equiv 0$ and using \rif{tritri}$_1$ as in \rif{bowie} to estimate the left-hand side we obtain
\begin{flalign*}
\snr{\bbt_k} \rr + \snr{\aaat_k}  \lesssim \nra{\ell_{x_k}}_{L^2(B_{\rr}(x_k))}
&  \leq    \tx{E}_{w}(\ell_{x_k};x_{k},\rr) +\nra{w}_{L^2(B_{\rr}(x_k))}\\& \leq  \tx{E}_{w}(0;x_k,\rr)+\nra{w}_{L^2(B_{\rr}(x_k))}
\le 2\tx{E}_{w}(0;x_k,\rr).
\end{flalign*}
By \rif{affini2}$_2$ with $\ell\equiv 0$, we have
$
\tx{E}_{w}(0;x_k,\rr)\to \tx{E}_{w}(0;x,\rr),
$
and therefore the sequences $\{\aaat_k\}$ and $\{\bbt_k\}$ are bounded. In order to prove that $\aaat_k \to \ell_x(x)$ and $\bbt_k \to D\ell_x$ it will be sufficient to prove that every converging subsequence of $\{\aaat_k\}$ and $\{\bbt_k\}$ converges to $\ell_x(x)$ and $D\ell_x$, respectively. 
Let us take a subsequence along which both converge, still denoted by $\{\aaat_k\}, \{\bbt_k\}$, such that $\aaat_k \to \aaat \in \er^N$ and $\bbt_k \to \bbt \in \er^{N\times n}$; \rif{affini1} implies $\tx{E}_{w}(\ell_{x_k};x_k,\rr)\to \tx{E}_{w}(\ell;x,\rr)$ with $\ell(\ti{x}):=\bbt(\ti{x}-x)+\aaat$. On the other hand by minimality of $\ell_{x_k}$ we have  $\tx{E}_{w}(\ell_{x_k};x_k,\rr)\leq \tx{E}_{w}(\ell_{x};x_k,\rr)$ and again by \rif{affini2} it is $\tx{E}_{w}(\ell_{x};x_k,\rr)\to \tx{E}_{w}(\ell_x;x,\rr)$. It follows $\tx{E}_{w}(\ell;x,\rr) \leq \tx{E}_{w}(\ell_x;x,\rr)$ and therefore $\ell=\ell_x$ by uniqueness in \rif{unicaff}. The proof is complete. 
\end{proof}
\subsection{A Caccioppoli type inequality} We present a standard Caccioppoli-type inequality for solutions to \eqref{nonlocaleqn}. While this inequality has appeared in various forms in the literature, it is typically formulated with different operators, often under additional assumptions on solutions, and in any case always for scalar solutions \cite{byunnon, DKP, KMS1, meng}. Since we are dealing with  the vectorial case, we are going to provide a full proof.
\begin{lemma}\label{prop:cacc}
Under assumptions \eqref{bs.1},\eqref{bs.2}, and \eqref{bs.5}, let $\BBB\equiv B_{\rrr}(x_{0})\Subset \Omega$ be a ball, and $u$ be a weak solution to \eqref{nonlocaleqn}. Then, for all $0<\sigma<1$, the Caccioppoli inequality
\begin{flalign}\label{cacc}
&\notag \sigma^{n}\snra{u}_{s,2;\sigma \BBB}^2\\
&\le \frac{c}{(1-\sigma)^{2n+2s}\rrr^{2s}}\left(\nra{u-u_0}_{L^{2}(\BBB)}^2+\tail(u-u_0;\BBB)\nra{u-u_0}_{L^{1}(\BBB)}\right)+c \rrr^{2s}\nra{f}_{L^{2_*}(\BBB)}^2
\end{flalign}
holds with $c\equiv c(n,N,s,\Lambda)$ for every $u_0 \in \er^N$. As a consequence, it holds that 
\eqn{cacce}
$$
\snra{u}_{s,2;\BBB/2} \leq  c\rrr^{-s} \tx{E}_{u}(x_{0},\rrr)+c \rrr^{s}\nra{f}_{L^{2_*}(\BBB)}\,.
$$
\end{lemma}
\begin{proof}
We can always assume that $u_0\equiv 0$. Indeed, define $\ti{a}(x,y,v,w):= 
a(x,y,u_0+v,u_0+w)$ for every $x, y \in \er^n$, $v,w\in \er^N$. It follows that $\ti{a}(\cdot)$ still satisfies  \eqref{bs.1}, \eqref{bs.2} and that $u-u_0$ solves \eqref{nonlocaleqn} with $a(\cdot)$ replaced by $\ti{a}(\cdot)$. In the following all the considered balls will be concentric with $\BBB$ and therefore centred at $x_0$. We introduce parameters $\sigma\in (0,1)$, and $\sigma\rrr\le \tau_{1}<\tau_{2}\le \rrr$, set $\ti{\tau}_{2}:=(3\tau_{2}+\tau_{1})/4$, $\ti{\tau}_{1}:=(\tau_{2}+\tau_{1})/2$ so that $\tau_1 < \ti{\tau}_1 <  \ti{\tau}_2< \tau_2$, and let $\eta\in C^{\infty}_0(B_{\ti{\tau}_2})$ be a cut-off function such that $\mathds{1}_{B_{\ti{\tau}_{1}}}\le \eta\le \mathds{1}_{B_{\ti{\tau}_{2}}}$ and $\nr{D\eta}_{L^{\infty}(\BBB)}\lesssim 1/(\tau_{2}-\tau_{1})$.  We test \rif{eqweaksol} by $\varphi:=\eta^{2}(u-(u)_{B_{\tau_{2}}})$, thereby obtaining, by means of \rif{bs.2}
 \begin{flalign*}
\mbox{(I)}&:= \int_{\BBB}\langle f,\varphi\rangle\dx=\int_{B_{\tau_{2}}}\int_{B_{\tau_{2}}}\langle a(x,y,u(x),u(y))(u(x)-u(y)),\varphi(x)-\varphi(y)\rangle \frac{\dxy}{\snr{x-y}^{n+2s}}\nonumber \\
&\qquad +2\int_{\mathbb{R}^{n}\setminus B_{\tau_{2}}}\int_{B_{\tau_{2}}}\langle a(x,y,u(x),u(y))(u(x)-u(y)),\varphi(x)-\varphi(y)\rangle \frac{\dxy}{\snr{x-y}^{n+2s}}\nonumber \\
&=:\mbox{(II)}+\mbox{(III)}.
 \end{flalign*}
To estimate $\mbox{(II)}$, we observe that
$$
\varphi(x)-\varphi(y)=\frac{1}{2}(\eta(x)^{2}+\eta(y)^{2})(u(x)-u(y))+\frac{1}{2}(\eta(x)^{2}-\eta(y)^{2})\left(u(x)+u(y)-2(u)_{B_{\tau_{2}}}\right)
$$
and  therefore
\begin{flalign*}
\mbox{(II)}&=\frac{1}{2}\int_{B_{\tau_{2}}}\int_{B_{\tau_{2}}}(\eta(x)^{2}+\eta(y)^{2})\langle a(x,y,u(x),u(y))(u(x)-u(y)),u(x)-u(y)\rangle \frac{\dxy}{\snr{x-y}^{n+2s}}\nonumber \\
&\quad +\frac{1}{2}\int_{B_{\tau_{2}}}\int_{B_{\tau_{2}}}(\eta(x)^{2}-\eta(y)^{2})\times \\ & \hspace{2cm}\times\langle a(x,y,u(x),u(y))(u(x)-u(y)),u(x)+u(y)-2(u)_{B_{\tau_{2}}} \rangle \frac{\dxy}{\snr{x-y}^{n+2s}}\nonumber \\
&\ge\frac{1}{2\Lambda}\int_{B_{\tau_{2}}}\int_{B_{\tau_{2}}}(\eta(x)^{2}+\eta(y)^{2})\frac{\snr{u(x)-u(y)}^{2}}{\snr{x-y}^{n+2s}}\dxy\nonumber \\
&\quad -\Lambda\int_{B_{\tau_{2}}}\int_{B_{\tau_{2}}}(\eta(x)+\eta(y))\snr{\eta(x)-\eta(y)}\snr{u(x)-u(y)}\times\\
& \quad \qquad \qquad   \times \max\left\{\snr{u(x)-(u)_{B_{\tau_{2}}}},\snr{u(y)-(u)_{B_{\tau_{2}}}}\right\}\frac{\dxy}{\snr{x-y}^{n+2s}}\nonumber \\
&\ge\frac{1}{4\Lambda}\int_{B_{\tau_{2}}}\int_{B_{\tau_{2}}}(\eta(x)^{2}+\eta(y)^{2})\frac{\snr{u(x)-u(y)}^{2}}{\snr{x-y}^{n+2s}}\dxy\nonumber \\
&\quad \quad -c\int_{B_{\tau_{2}}}\int_{B_{\tau_{2}}}\max\left\{\snr{u(x)-(u)_{B_{\tau_{2}}}}^{2},\snr{u(y)-(u)_{B_{\tau_{2}}}}^{2}\right\}\snr{\eta(x)-\eta(y)}^{2}\frac{\dxy}{\snr{x-y}^{n+2s}}\nonumber\\
&\ge \frac{1}{4\Lambda}\int_{B_{\tau_{1}}}\int_{B_{\tau_{1}}}\frac{\snr{u(x)-u(y)}^{2}}{\snr{x-y}^{n+2s}}\dxy-\frac{c\rrr^{2(1-s)}}{(\tau_{2}-\tau_{1})^{2}}\int_{B_{\tau_{2}}}\snr{u(x)-(u)_{B_{\tau_{2}}}}^{2}\dx\,,
\end{flalign*}
with $c\equiv c(n,s,\Lambda)$. Note that we argued as follows:
\begin{flalign*}
&\int_{B_{\tau_{2}}}\int_{B_{\tau_{2}}}\max\left\{\snr{u(x)-(u)_{B_{\tau_{2}}}}^{2},\snr{u(y)-(u)_{B_{\tau_{2}}}}^{2}\right\}\snr{\eta(x)-\eta(y)}^{2}\frac{\dxy}{\snr{x-y}^{n+2s}}\\ 
& \quad \leq \frac{c(n)}{(\tau_2-\tau_1)^{2}} \int_{B_{\tau_{2}}}\int_{B_{\tau_{2}}} \snr{u(x)-(u)_{B_{\tau_{2}}}}^{2}\frac{\dxy}{\snr{x-y}^{n+2(s-1)}}\\
& \quad \leq  \frac{c(n)}{(\tau_2-\tau_1)^{2}}\int_{B_{\tau_{2}}}\snr{u(x)-(u)_{B_{\tau_{2}}}}^{2}\dx\int_{\texttt{B}_{2\tau_2}}\frac{\dz}{\snr{z}^{n+2(s-1)}}\\ & \quad \leq \frac{c\tau_{2}^{2(1-s)}}{(1-s)(\tau_{2}-\tau_{1})^{2}}\int_{B_{\tau_{2}}}\snr{u(x)-(u)_{B_{\tau_{2}}}}^{2}\dx\,.
\end{flalign*}
Concerning $\mbox{(III)}$, note that, with $y \in \er^n\setminus B_{\tau_2}$ and $x \in B_{\ti{\tau}_2}$, as in \rif{esco1} we have
\eqn{esco2}
$$
 \snr{y-x} \geq  \frac{\tau_2- \snr{x-x_{0}}}{\tau_2} \snr{y-x_{0}}>
  \frac{\tau_2- \ti{\tau}_2}{\tau_2} \snr{y-x_{0}}\approx   \frac{\tau_2- \tau_1}{\tau_2} \snr{y-x_{0}}
$$
so that, estimating $\tail(u-(u)_{B_{\tau_{2}}};B_{\tau_{2}}) \lesssim \tail(u;B_{\tau_{2}})+\snr{(u)_{B_{\tau_{2}}}}$, we have
\begin{eqnarray*}
\snr{\mbox{(III)}}&\le &c\int_{\mathbb{R}^{n}\setminus B_{\tau_{2}}}\int_{B_{\ti{\tau}_{2}}}\eta(x)^{2}\max\left\{\snr{u(x)-(u)_{B_{\tau_{2}}}},\snr{u(y)-(u)_{B_{\tau_{2}}}}\right\}\times \\ && 
\hspace{20mm}\times \snr{u(x)-(u)_{B_{\tau_{2}}}}\frac{\dxy}{\snr{x-y}^{n+2s}}\nonumber \\
&\stackleq{esco2} &c\tau_{2}^{-2s}\left(\frac{\tau_{2}}{\tau_{2}-\tau_{1}}\right)^{n+2s}\tail(u-(u)_{B_{\tau_{2}}};B_{\tau_{2}}) \int_{B_{\tau_{2}}}\snr{u-(u)_{B_{\tau_{2}}}}\dx\nonumber \\
&&\quad +c\tau_{2}^{-2s}\left(\frac{\tau_{2}}{\tau_{2}-\tau_{1}}\right)^{n+2s}\int_{B_{\tau_{2}}}\snr{u-(u)_{B_{\tau_{2}}}}^{2}\dx\nonumber \\
&\stackleq{minav} &\frac{c\tau_{2}^n}{(\tau_{2}-\tau_{1})^{n+2s}}\left(\tail(u;B_{\tau_{2}}) \int_{B_{\tau_{2}}}\snr{u}\dx\nonumber +\mint_{B_{\tau_2}}\snr{u}\dx\int_{B_{\tau_{2}}}\snr{u}\dx+\int_{B_{\tau_{2}}}\snr{u}^{2}\dx\right)\nonumber \\
&\stackrel{\mbox{Jensen}}{\leq} & \frac{c\tau_{2}^n}{(\tau_{2}-\tau_{1})^{n+2s}}\left(\tail(u;B_{\tau_{2}}) \int_{B_{\tau_{2}}}\snr{u}\dx\nonumber+\int_{B_{\rrr}}\snr{u}^{2}\dx\right)\nonumber\\
&\stackleq{scatailggg} & \frac{c\rrr^n}{(\tau_{2}-\tau_{1})^{n+2s}}\left(\tail(u;B_{\rrr}) \int_{B_{\rrr}}\snr{u}\dx\nonumber+\frac{\rrr^n}{\tau_2^n}\mint_{B_{\rrr}}\snr{u}\dx\int_{B_{\tau_{2}}}\snr{u}\dx+\int_{B_{\rrr}}\snr{u}^{2}\dx\right)\nonumber\\
&\leq & \frac{c\rrr^{2n}}{(\tau_{2}-\tau_{1})^{2n+2s}}\left(\tail(u;\BBB) \int_{\BBB}\snr{u}\dx\nonumber+\int_{\BBB}\snr{u}^{2}\dx\right)
\end{eqnarray*}
for $c\equiv c(n,s,\Lambda)$. Finally, via H\"older and Young inequalities, and \rif{sobpoinfrac}, we find \begin{flalign*}
\snr{\mbox{(I)}}&\le \int_{B_{\ti{\tau}_{2}}}\eta^{2}\snr{f}\snr{u-(u)_{B_{\tau_{2}}}}\dx\nonumber \\
&\le\snr{B_{\tau_{2}}}\left(\mint_{B_{\tau_{2}}}\snr{u-(u)_{B_{\tau_{2}}}}^{2^{*}}\dx\right)^{\frac{1}{2^{*}}}\left(\mint_{B_{\tau_{2}}}\eta^{\frac{2n}{n+2s}}\snr{f}^{\frac{2n}{n+2s}}\dx\right)^{\frac{n+2s}{2n}}\nonumber \\
&\le c\tau_{2}^{s}\snr{B_{\tau_{2}}}\left(\mint_{B_{\tau_{2}}}\int_{B_{\tau_{2}}}\frac{\snr{u(x)-u(y)}^{2}}{\snr{x-y}^{n+2s}}\dxy\right)^{1/2}\left(\mint_{B_{\tau_{2}}}\snr{f}^{2_*}\dx\right)^{1/2_*}\nonumber \\
&\le\varepsilon\int_{B_{\tau_{2}}}\int_{B_{\tau_{2}}}\frac{\snr{u(x)-u(y)}^{2}}{\snr{x-y}^{n+2s}}\dxy+\frac{c\tau_{2}^{n+2s}}{\varepsilon}\left(\mint_{B_{\tau_{2}}}\snr{f}^{2_*}\dx\right)^{2/2_*}
\end{flalign*}
with $c\equiv c(n,N,s)$ and $\varepsilon\in (0,1)$ to be chosen. Merging the content of the above displays, using \eqref{minav} again, and choosing $\varepsilon\equiv \eps(n,N,s,\Lambda)$ sufficiently small we obtain
\begin{flalign*}
\int_{B_{\tau_{1}}}\int_{B_{\tau_{1}}}\frac{\snr{u(x)-u(y)}^{2}}{\snr{x-y}^{n+2s}}\dxy  &\le\frac{1}{2}\int_{B_{\tau_{2}}}\int_{B_{\tau_{2}}}\frac{\snr{u(x)-u(y)}^{2}}{\snr{x-y}^{n+2s}}\dxy+\frac{c\rrr^{2(1-s)}}{(\tau_{2}-\tau_{1})^{2}}
\int_{\BBB}\snr{u}^2\dx\nonumber \\
&\qquad +\frac{c\rrr^{2n}}{(\tau_{2}-\tau_{1})^{2n+2s}}\left(\tail(u;\BBB) \int_{\BBB}\snr{u}\dx\nonumber+\int_{\BBB}\snr{u}^{2}\dx\right)\\ &\qquad +c\rrr^{n+2s} \left(\mint_{\BBB}\snr{f}^{2_*}\dx\right)^{2/2_*}
\end{flalign*}
for $c\equiv c(n,N,s,\Lambda)$. Lemma \ref{iterlem} below applied to $h(\tau)\equiv [u]_{s,2;B_{\tau}}^2$ yields \eqref{cacc} (with $u_0\equiv 0$), and the proof is complete.
\end{proof}
\begin{lemma}\label{iterlem}
Let $h\colon [t,s]\to [0, \infty)$ be a bounded function, and let $\tx{a},\tx{b},\tx{c},  \gamma, \gamma_*$ be non-negative numbers. Assume that the inequality 
$
h(\tau_1)\le  h(\tau_2)/2+(\tau_2-\tau_1)^{-\gamma}\tx{a}+(\tau_2-\tau_1)^{-\gamma_*}\tx{b}+\tx{c}
$
holds whenever $t\le \tau_1<\tau_2\le s$. Then $
h(t)\lesssim_{\gamma, \gamma_*} (s-t)^{-\gamma}\tx{a}+ (s-t)^{-\gamma_*}\tx{b}+\tx{c}
$ holds too. 
\end{lemma}
See \cite[Lemma 6.1]{giusti} for the proof. 
\subsection{A convergence result based on Lorentz spaces} This is a straightforward consequence of some arguments hidden in \cite[Section 4]{dmn}. For this we consider the potential 
$$
\mathbf{I}^{f}_{\gamma,\chi}(x,\sigma) := \int_{0}^{\sigma} \lambda^{\gamma}\nra{f}_{L^\chi(B_{\lambda}(x))}  \, \frac{\dlam}{\lambda}\,,$$
where $\gamma >0$ is such that $n > \gamma \chi$. Following \cite[Section 4]{dmn}, we identify $\mathbf{I}^{f}_{\gamma,\chi}(x,\sigma)=\mathbf{P}^{\chi, 1}_{\chi,\gamma}(f;x,\sigma)$ in the notation of \cite[Section 4]{dmn}. From the proof of \cite[Lemma 4.1]{dmn} it follows that 
\eqn{qualo}
$$
\mathbf{I}^{f}_{\gamma,\chi}(x,\sigma) \lesssim \int_0^{\snr{\ttB_{1}}\sigma^n} \left[\lambda^{\chi\gamma/n}(|f|^{\chi})^{**}(\lambda)\right]^{1/\chi} \, \frac{\dlam}{\lambda}
$$
holds whenever $B_{\sigma}(x)\subset \er^n$ is a ball, where 
$$(|f|^{\chi})^{**} (\lambda)=\frac1{\lambda} \int_0^{\lambda} (|f|^{\chi})^{*}(r)\d r$$ and $(|f|^{\chi})^{*}$ is the non-increasing rearrangement of $|f|^{\chi}$. Note that, as $n > \gamma \chi$, the quantity on the right-hand side of \rif{qualo} is finite provided $|f|^{\chi}$ belongs to the Lorentz space $L(n/(\chi\gamma),1/\chi)$, which is the same as requiring $f\in L(n/\gamma, 1)$; see \cite[Section 4]{dmn} for details and relevant definitions. From \rif{qualo} we conclude that 
\eqn{qualo2}
$$
\frac{n}{\chi\gamma}>1\,, \ \  f \in L(n/\gamma,1) \Longrightarrow 
\lim_{\sigma\to 0}\, \mathbf{I}^{f}_{\gamma,\chi}(x,\sigma)=0\ \  \mbox{uniformly with respect to $x$}\,.
$$
\subsection{Fractional sharp maximal operators}\label{truncsec0}
The sharp fractional maximal operator of a map $w \in L^1_{\loc}(\er^n;\er^N)$ is for every $x\in \er^n$ defined as follows:
\eqn{massimale}
$$
\textnormal{M}^{\#,t}_{\rrr}(w;x):=\sup_{0 <\rr \leq \rrr}\,  \rr^{-t} \mint_{B_{\rr}(x)} \snr{w - (w)_{B_{\rr}(x)}}\dy \,, \quad 0\leq t \leq 1\,.
$$
Note that the choice $t=0$ gives a localization of the classical Fefferman-Stein sharp maximal operator. Such an operator has been first considered by Calder\'on \& Scott in \cite{calderon}; see also \cite{des} for its properties. It can be used to prove a local H\"older estimate for $w$ in the sense of the following
\begin{lemma}\label{campmax} Let $w\in L^1_{\loc}(\er^n;\er^N)$; if $x\in \er^n$ is such that $\textnormal{M}^{\#,t}_{\rrr}(w;x)$ is finite for some $\rrr>0$ and $0< t \leq 1$, then the limit 
\eqn{preciso}
$$
w(x):=\lim_{\sigma \to 0}\, (w)_{B_{\sigma}(x)} 
$$
exists and thereby defines the precise representative of $w$ at the point $x$. Moreover, for every $\eps \in (0,1)$ the inequality
\eqn{stimalfa}
$$
|w(x)-w(y)|\lesssim_n \frac{1}{\eps^n t}\left[\textnormal{M}^{\#,t}_{\rrr}(w;x)+\textnormal{M}^{\#,t}_{\rrr}(w;y)\right]|x-y|^{t}
$$
holds whenever $x, y \in \er^n$ are such that $\ (1+\eps)\snr{x-y}\leq \rrr$ and the right-hand side is finite. In particular, if $\BBB \subset \er^n$ is a ball with radius $\rr$, then 
\eqn{stimalfa2}
$$
[w]_{0,t;\BBB} \lesssim_{n, \eps}  \frac{1}{t}\nr{\textnormal{M}^{\#,t}_{\rrr}(w;\cdot)}_{L^{\infty}(\BBB)}$$
holds provided $\rrr\geq 2(1+\eps)\rr$ \footnote{In the spirit of the lemma, i.e., considering also the fine behaviour of $w$, inequality \rif{stimalfa} deserves a comment. In the case $ \sup_{\BBB}\, \textnormal{M}^{\#,t}_{\rrr}(w;\cdot)$ is finite then \rif{stimalfa} holds for every choice of $x,y$. In the case the right-hand side of \rif{stimalfa} is finite, then \rif{stimalfa} holds a.e., and then $w$ admits a H\"older continuous representative in $\BBB$.}. 
\end{lemma}
A proof can be obtained making minor modifications in the one of \cite[Proposition 1]{kumig}. Lemma \ref{campmax} in fact implies the usual Campanato-Meyers integral characterization of H\"older continuity and its proof indeed uses similar arguments. See for instance the proof of Proposition \ref{cdg}. We remark that a more detailed look at the proof of Lemma \ref{campmax} leads to the improved inequality
$$
|w(x)-w(y)|\lesssim_n \frac{1}{ t}\left[(1+\eps^{-n})\textnormal{M}^{\#,t}_{(1+\eps)\snr{x-y}}(w;x)+\textnormal{M}^{\#,t}_{\eps\snr{x-y}}(w;y)\right]|x-y|^{t}
$$
which is valid whenever $x, y \in \er^n$. 

\section{Linear systems with constant coefficients}\label{linearsec}
The goal of this section is to prove a priori estimates for solutions to nonlocal systems with constant coefficients, and in particular to prove suitable quantitative smoothness results matching  partial regularity techniques.
We consider weak solutions to systems of the type
$-\mathcal{L}_{a_{0}}h=g$, 
where $a_{0}\in \mathbb{R}^{N\times N}$ is a constant matrix such that
\eqn{bs.1bis}
$$
\Lambda^{-1}\snr{\xi}^{2} \leq \langle a_{0}\xi,\xi\rangle\qquad \mbox{and}\qquad \snr{a_{0}}\le \Lambda\,,
$$
hold for all $\xi\in \er^N$. We start by looking at Dirichlet problems of the type
\eqn{basicsolve}
$$
\begin{cases} 
\ -\mathcal{L}_{a_{0}}h=g\quad &\mbox{in} \ \ \Omega\\
\ h=v\quad &\mbox{in} \ \ \mathbb{R}^{n}\setminus \Omega\,,
\end{cases}
$$ 
where $\Omega \subset \er^n$ is as usual a bounded domain. We here refer to the setting of  \cite[Section 3.1]{kokupa} and \cite[Proposition 2.12]{bls}, whose proofs easily adapt to the vectorial case considered in \rif{basicsolve}. For this it is convenient to introduce a suitable function space setting. With $\Omega \Subset \ti{\Omega}$ being two open and bounded sets and $v\in L^{1}_{2s}$, $0<t<1\leq p$, we define the nonlocal Dirichlet class 
\eqn{asinfor}
$$
\mathbb{X}^{t,p}_{v}(\Omega ,\ti{\Omega}) := \left\{w\in W^{t,p}(\ti{\Omega};\er^N) \cap L^1_{2s}\, \colon \, \mbox{$w\equiv v$ \, a.e. in $\mathbb{R}^{n}\setminus \Omega$}\right\}\,.
$$
Obviously, $\mathbb{X}^{t,p}_{v}(\Omega ,\ti{\Omega})$ is non-empty when $v \in W^{t,p}(\ti{\Omega}; \er^N)\cap L^1_{2s}$. We also denote $\mathbb{X}^{t,p}_{v}(\Omega, \ti{\Omega})\equiv \mathbb{X}^{t,p}_{0}(\Omega, \ti{\Omega})$ when $v\equiv 0$; note that $h \in \mathbb{X}^{t,p}_{v}(\Omega, \ti{\Omega})$ iff $h-v\in \mathbb{X}^{t,p}_{0}(\Omega, \ti{\Omega})$ when  $v \in W^{t,p}(\ti{\Omega}; \er^N)$. 
\begin{lemma}\label{esiste} If $v\in W^{s,2}(\ti{\Omega};\er^N)\cap L^1_{2s}$ and $g\in L^{q}(\Omega;\er^N)$ for some $q \geq 2_*$, there exists a unique solution $h\in \mathbb{X}^{s,2}_{v}(\Omega, \ti{\Omega})$ to \eqref{basicsolve} in the sense that
\eqn{eqweaksol22}
$$
			\int_{\mathbb{R}^n} \int_{\mathbb{R}^n} \langle a_{0}(h(x)-h(y)),\varphi(x)-\varphi(y) \rangle \frac{\dxy}{|x-y|^{n+2s}} = \int_{\Omega} \langle g, \varphi \rangle\dx
$$
holds whenever $\varphi \in \mathbb{X}^{s,2}_{0}(\Omega, \ti{\Omega})$. 
\end{lemma}
See \cite[Proposition 2.12]{bls} (and \cite{kokupa}) for related proofs. In the present vector-valued linear setting, the proof adapts straightforwardly or more directly from the Lax-Milgram theorem.  
\begin{remark}\label{remarkino}\emph{Consider open subsets $\Omega \Subset \ti{\Omega} $ as in \rif{asinfor} and a map $w\in W^{t,p}(\ti{\Omega};\er^N)$ with $p\geq 1$ and $t\in (0,1)$ such that $w\equiv 0$ a.e.  in $\ti{\Omega}\setminus \Omega$. Then, letting $w\equiv 0$ a.e. in $\er^n\setminus \Omega$  we obtain an extension of $w$ to $\er^n$ such that $w \in W^{t,p}(\mathbb{R}^{n};\er^N)$ with 
\eqn{ht.10}
$$
[w]_{t,p;\mathbb{R}^{n}}^{p}\le [w]_{t,p;\ti{\Omega}}^{p}+\frac{c\nr{w}_{L^{p}(\Omega)}^{p}}{\dist(\Omega,\er^n\setminus \ti{\Omega})^{tp}}\,,
$$
where $c\equiv c(n,N,t,p)$; see for instance \cite[Lemma 2.11]{bls} or \cite[Lemma 5.1]{dpv}.
In particular, this applies to maps $w\in \mathbb{X}^{s,2}_{0}(\Omega ,\ti{\Omega})$. 
It follows that in \rif{eqweaksol22} it is equivalent to require that the equality holds whenever $\varphi \in W^{s,2}(\er^n;\er^N)$ is such that $\varphi \equiv 0$ outside $\Omega$ and that  every solution to \rif{basicsolve} is a (local) solution to $-\mathcal{L}_{a_{0}}h=g$ in the sense of Definition \ref{def:weaksol}. Moreover, $w \in W^{t,p}(\er^n;\er^N) \Longrightarrow \mathbb{X}^{t,p}_{w}(\Omega, \ti{\Omega}) \subset W^{t,p}(\er^n;\er^N).$ 
}
\end{remark} 
\begin{lemma}\label{lemma3.4}
Let $v \in W^{s,2}_{\loc}(\Omega;\mathbb{R}^N) \cap L^{1}_{2s}$ be a weak solution to $-\mathcal{L}_{a_{0}}v=g$ in the sense of Definition \ref{def:weaksol}, with $g \in L^{2_*}(\Omega;\er^N)$. Take concentric balls $B\Subset \tilde{B} \Subset \Omega$ and define $h \in \mathbb{X}_{v}^{s,2}(B,\tilde{B})$ as the solution to
\eqn{solveswhat}
$$
\begin{cases}
\ -\mathcal{L}_{a_{0}}h=0\quad &\mbox{in} \ \ B\\
\ h=v\quad &\mbox{in} \ \ \mathbb{R}^{n}\setminus B
\end{cases}
$$
in the sense of \eqref{eqweaksol22}. 
Then, whenever $d>2_{*}$ it holds that 
\eqn{marri}
$$
\begin{cases}
[v-h]_{s,2;\er^n} \leq c \snr{B}^{1/2_{*}-1/d}\nr{g}_{\mathcal M^d(B)}\\[6pt]
\nra{v-h}_{L^2(B)} \leq c\snr{B}^{2s/n-1/d}\nr{g}_{\mathcal M^d(B)}
\end{cases}
$$
where $c\equiv c (n,N, \Lambda, s,d)$. 
The quantity $\nr{g}_{\mathcal M^d(B)}$ is defined in \eqref{mardef}. 
\end{lemma}
\begin{proof} Solvability of \rif{solveswhat} follows by Lemma \ref{esiste}. Note that $v-h \in \mathbb{X}^{s,2}_{0}(B, \tilde{B})$ and therefore Remark \ref{remarkino} implies $v-h \in W^{s,2}(\er^n;\er^N)$. As for the estimates, the derivation is the same proof as \cite{bls} using the ellipticity of $a_0$. Indeed, for \rif{marri}$_1$ we follow \cite[Lemma 3.4]{bls} taking $p=2$ and $q=2_*$ (see the fourth display of \cite[Page 807]{bls}). This yields, together with \rif{shalldenote} 
$$[v-h]_{s,2;\er^n}\lesssim \snr{B}^{1/2_*}\nra{g}_{L^{2_{*}}(B)}\lesssim \snr{B}^{1/2_{*}-1/d}\nr{g}_{\mathcal M^d(B)}\,,$$ that is \rif{marri}$_1$. For \rif{marri}$_2$ we use \rif{poincf} with $w=v-h$ and $p=2$ and then \rif{marri}$_1$.
\end{proof}
We now consider weak solutions to homogeneous systems with constant coefficients of the type 
\eqn{3.1.0}
$$
-\mathcal{L}_{a_{0}}h=0\qquad \mbox{in} \ \ \BBB\,,
$$
where $a_{0}\in \mathbb{R}^{N\times N}$ satisfies \rif{bs.1bis} and $\BBB\subset \er^n$ is a ball. We give a direct and self-contained proof of a Campanato type estimate that does not rely in any way on lifting operators or other regularity theorems.  As such, it might be useful in situations where such tools cannot be employed, such as nonlinear problems. The analysis goes through via certain difference quotients methods in fractional spaces, and extends that made for instance in \cite{cozzi}. For this, we shall use finite difference operators $\tau_{\xi}^{k}$, with $\xi \in \er^n$ and $k \in \mathbb N_0$, following \cite[(1.384)]{Triebel3}. For $w \colon \mathcal A \to \er^N$ and $\mathcal A \subset \er^n$ being a non-empty set, we define the operators
\eqn{enfasi}
$$
\tau_{\xi}^{k}w \equiv \tau_{\xi, \mathcal A}^{k}w\,, \quad k \in \en_0
$$
as follows. We denote 
$$\mathcal A_{k,\xi}:= \{ x \in \mathcal A \colon x+i\xi \in \mathcal A \textnormal{ for every } i=1,\ldots,k\}\,.$$ For every $k\geq 1$,  we define,
inductively, the operators 
$$
\begin{cases}
\tau_{\xi}^{k+1}w(x):= \tau_{\xi}(\tau_{\xi}^{k} w)(x) \quad \mbox{for every $x\in \mathcal A_{k+1,\xi}$}, \\[3pt]
\tau_{\xi}w(x)\equiv \tau_{\xi}^1w(x):= w(x+\xi) -w(x)\quad \mbox{for every $x\in \mathcal A_{1,\xi}$},\\[3pt]
\tau_{\xi}^0w(x):=w(x)\quad \mbox{for every $x\in  \mathcal A=:\mathcal A_{0,\xi}$}   \,.
\end{cases}
$$
Instead, we let $\tau_{\xi}^{k}w(x)=0$ if $x \not \in \mathcal A_{k,\xi}$ and $k\geq 1$. Note that, as emphasized in \rif{enfasi}, such a definition depends on the set $\mathcal A$ and we shall simply abbreviate the notation as in \rif{enfasi} when this will be clear from the context. 
\begin{proposition}[Nonlocal Campanato I]\label{phr}
Let $h\in W^{s,2}(\BBB;\er^N)\cap L^{1}_{2s}$ be a weak solution to \eqref{3.1.0}. Then $h\in C^{\infty}_{\loc}(\BBB;\er^N)$. Moreover
\eqn{eq:higherreg0}
$$
		\snr{\BBB}^{k/n}\nr{D^k(h-h_0)}_{L^\infty(\BBB/2)} \leq c \left( \nra{h-h_{0}}_{L^{2}(\BBB)}+\tail(h-h_{0};\BBB)\right)
$$
	holds whenever $k\geq 0$ is an integer and $h_{0}\in \er^N$, where $c\equiv c (n,N,s,\Lambda,k)$.
\end{proposition}
\begin{proof} By replacing $h$ with $h-h_0$, which is still a solution, we can assume that $h_{0}=0$. Moreover, upon scaling, we can assume that $\BBB\equiv \ttB_1$ (see also the proof of Proposition \ref{psp} below, Step 1). We prove by induction that for every $k \in \en_0$ 
	\eqn{eq:highsobest}
	$$
	\sup_{0<\snr{\xi}\leq 
\frac{1-\sigma}{8^{k}10}}|\xi|^{-sk}\nr{\tau_{\xi}^{k} h}_{L^2(\ttB_{(1+\sigma)/2})} \le c \nr{h}_{L^{2}(\ttB_{1})} +c\, \tail(h;\ttB_{1}) 
	$$
holds for every $\sigma \in [0,1)$, with $c\equiv c(n,N,s,\Lambda,k,\sigma)$, where, as in \rif{enfasi}, here and in the following we are abbreviating $\tau_{\xi}^{k} \equiv \tau_{\xi, \er^n}^{k}$. To proceed, note that \rif{eq:highsobest} is trivial for $k=0$. Next, assume that \eqref{eq:highsobest} holds for all $k\leq m$ for some $m \in \mathbb{N}_0$ and let us prove it holds for $k=m+1$. Therefore from now on we fix $\sigma \in [0, 1)$ and $\xi \in \er^n$ such that 
\eqn{piccolino}
$$0<\snr{\xi}\leq   \frac{1-\sigma}{8^{m+1}10}$$ and prove that 
\eqn{passetto}
$$ |\xi|^{-s(m+1)}\nr{\tau_{\xi}^{m+1} h}_{L^2(\ttB_{(1+\sigma)/2})} \le c \nr{h}_{L^{2}(\ttB_{1})} +c\, \tail(h;\ttB_{1}) \,.
	$$ 
We take $\eta \in C_{0}^\infty(\ttB_{(6+\sigma)/7})$ such that 
\eqn{moremore0}
$$\mathds{1}_{\ttB_{(5+\sigma)/6}}\leq \eta\leq \mathds{1}_{\ttB_{(6+\sigma)/7}}\quad \mbox{and}\quad \snr{D^{k}\eta}\lesssim \frac1{(1-\sigma)^k}\, \ \mbox{ for every $k \leq m+1$}\,.$$ Then $w:=\eta h \in W^{s,2}(\er^n;\mathbb{R}^N) \cap L^1_{2s}$ solves 
\eqn{solves}
$$-\mathcal{L}_{a_{0}} w = g\quad \mbox{in}\ \ \ttB_{(4+\sigma)/5}\,,$$ in the sense of Definition \ref{def:weaksol},  where
\eqn{solves2}
$$ g(x):=\mathds{1}_{\ttB_{(4+\sigma)/5}}(x)2\int_{\mathbb{R}^n \setminus \ttB_{(5+\sigma)/6}}(1-\eta(y))a_{0}h(y)\frac{\dy}{\snr{x-y}^{n+2s}}.$$
For completeness we briefly recall the argument. Whenever $\varphi \in W^{s,2}(\er^n;\er^N)$ has compact support in $\ttB_{(4+\sigma)/5}$, it is 
\begin{flalign*}
\hspace{-5mm} \langle - \mathcal{L}_{a_0} w, \varphi \rangle &=\int_{\mathbb{R}^n} \int_{\mathbb{R}^n} \langle a_{0}(\eta(x)h(x)-\eta(y)h(y)),\varphi(x)-\varphi(y) \rangle \frac{\dyx}{|x-y|^{n+2s}}  \\
&  = \int_{\ttB_{(5+\sigma)/6}} \int_{\ttB_{(5+\sigma)/6}} \langle a_{0}(h(x)-h(y)),\varphi(x)-\varphi(y) \rangle \frac{\dyx}{|x-y|^{n+2s}}\\
& \quad 
+2\int_{\ttB_{(5+\sigma)/6}} \int_{\er^n\setminus \ttB_{(5+\sigma)/6}} \langle a_{0}(h(x)-\eta(y)h(y)),\varphi(x)-\varphi(y) \rangle \frac{\dyx}{|x-y|^{n+2s}}\\
& = \int_{\ttB_{(5+\sigma)/6}} \int_{\ttB_{(5+\sigma)/6}} \langle a_{0}(h(x)-h(y)),\varphi(x)-\varphi(y) \rangle \frac{\dyx}{|x-y|^{n+2s}}\\
& \quad  
+2\int_{\ttB_{(5+\sigma)/6}} \int_{\er^n\setminus \ttB_{(5+\sigma)/6}} \langle a_{0}(h(x)-h(y)),\varphi(x)-\varphi(y) \rangle \frac{\dyx}{|x-y|^{n+2s}}\\
& \quad +2\int_{\ttB_{(5+\sigma)/6}} \int_{\er^n\setminus \ttB_{(5+\sigma)/6}} \langle a_{0}(1-\eta(y))h(y),\varphi(x)-\varphi(y) \rangle \frac{\dyx}{|x-y|^{n+2s}}\\
&  = \langle - \mathcal{L}_{a_0} h, \varphi \rangle +2\int_{\ttB_{(5+\sigma)/6}} \int_{\er^n\setminus \ttB_{(5+\sigma)/6}} \langle a_{0}(1-\eta(y))h(y),\varphi(x)\rangle \frac{\dyx}{|x-y|^{n+2s}}\\
&   =\int_{\ttB_{(4+\sigma)/5}} \langle g,\varphi\rangle\dx
\end{flalign*}
and \rif{solves} is proved. 
From \rif{solves2} it follows that $g \in L^{\infty}(\er^n;\er^N)\cap C^\infty(\ttB_{(4+\sigma)/5};\mathbb{R}^N)$ and that, for any multiindex $\tx{b}$ 
\begin{eqnarray}\label{dg}
\notag \sup_{x \in \ttB_{(4+\sigma)/5}} |D^{\tx{b}} g(x)| & \leq  &c \sup_{x \in \ttB_{(4+\sigma)/5}} \int_{\mathbb{R}^n \setminus \ttB_{(5+\sigma)/6}} \frac{|a_{0}||h(y)|}{|x-y|^{n+2s+|\tx{b}|}}\dy \\ 
\notag & \leq  & \frac{c}{(1-\sigma)^{|\tx{b}|}} \sup_{x \in \ttB_{(4+\sigma)/5}} \int_{\mathbb{R}^n \setminus \ttB_{(5+\sigma)/6}} \frac{|h(y)|}{|x-y|^{n+2s}}\dy \\
\notag
& \stackleq{uset} & \frac{c}{(1-\sigma)^{n+2s+|\tx{b}|}}\, \tail(h;\ttB_{1/2})\\
& \stackleq{scatailggg}  & c\,  \nr{h}_{L^{2}(\ttB_1)}+c\, \tail(h;\ttB_1) 
\end{eqnarray}
for $c\equiv c(n,s,\Lambda,\sigma,\tx{b})$. As a simple consequence 
of \rif{piccolino} and of \rif{solves} we have  
\eqn{solvesagain}
$$
-\mathcal{L}_{a_{0}}  (\snr{\xi}^{-ms}\tau_{\xi}^m w)   =\snr{\xi}^{-ms}\tau_{\xi}^{m} g \qquad \mbox{in $\ttB_{(3+\sigma)/4}$}\,.
$$
Indeed we can first prove by induction that $-\mathcal{L}_{a_{0}}  (\tau_{\xi}^k w)   =\tau_{\xi}^{k} g$ holds in $\ttB_{(4+\sigma)/5-k\snr{\xi}}$ whenever $(4+\sigma)/5-k\snr{\xi}>0$.  Then we observe that \rif{piccolino} implies $\ttB_{(3+\sigma)/4}\Subset \ttB_{(4+\sigma)/5-m\snr{\xi}}\subset \ttB_{(4+\sigma)/5-k\snr{\xi}}$ for every $k \leq m$.
By \rif{solvesagain} we can apply \rif{cacc}  to $u\equiv \snr{\xi}^{-ms}\tau_{\xi}^m w$, and using Young and H\"older inequalities this gives 
\begin{flalign}
\notag |\xi|^{-sm}\nr{\tau_{\xi}^{m} w}_{W^{s,2}(\ttB_{(2+\sigma)/3})} & \leq 
c|\xi|^{-sm}\nr{\tau_{\xi}^{m} w}_{L^2(\ttB_{(3+\sigma)/4})}+c\, \tail\left(|\xi|^{-sm}\tau_\xi^{m} w;\ttB_{(3+\sigma)/4}\right)\noindent \\
& \quad + c|\xi|^{-sm}\nr{\tau_{\xi}^{m} g}_{L^2(\ttB_{(3+\sigma)/4})}\label{caccibeck}
\end{flalign}
with $c\equiv c(n,N,s,\Lambda, \sigma)$. We now estimate the right-hand side in \rif{caccibeck} and for this we shall use the Leibniz rule for finite differences, i.e., 
\eqn{differenzona}
$$
\tau_{\xi}^{m} w(\cdot)= \sum_{0\leq k\leq m}\binom{m}{k} \tau^{m-k}_{\xi}[\eta(\cdot +k\xi)]\tau^{k}_{\xi}h(\cdot)\,.
$$
See for instance \cite[Section 30, (10)]{jordan}. 
Moreover, 
\eqn{moremore}
$$
\snr{\tau^{m-k}_{\xi}\eta(\cdot +k\xi)} \lesssim (1-\sigma)^{k-m}\snr{\xi}^{m-k}
$$ holds by repeatedly using Mean Value Theorem and \rif{moremore0}. Now take $\sigma_* \in (0,1)$ such that 
$(7+\sigma)/8 =(1+\sigma_*)/2$ and observe that \rif{piccolino} implies $\snr{\xi} \leq  (1-\sigma_*)/(8^m 10)= (1-\sigma)/(8^m 40)$ so that, keeping also \rif{eq:highsobest} in mind, the induction assumption \rif{eq:highsobest} yields 
$$
 |\xi|^{-sk}\nr{\tau_{\xi}^{k} h}_{L^{2}(\ttB_{(7+\sigma)/8})} = 
  |\xi|^{-sk}\nr{\tau_{\xi}^{k} h}_{L^{2}(\ttB_{(1+\sigma_*)/2})} \leq c\nr{h}_{L^{2}(\ttB_{1})} +c\, \tail(h;\ttB_{1})
$$
for every $k \leq m$. 
This and \rif{differenzona}-\rif{moremore} allow us to conclude with 
\begin{flalign}
 \notag  |\xi|^{-sm}\nr{\tau_{\xi}^{m} w}_{L^{2}(\ttB_{(7+\sigma)/8})} & \leq c  |\xi|^{-sm}
  \sum_{0\leq k\leq m}\binom{m}{k} \nr{\tau^{m-k}_{\xi}[\eta(\cdot +k\xi)]\tau^{k}_{\xi}h}_{L^{2}(\ttB_{(7+\sigma)/8})} \\
  & \leq c  \sum_{0\leq k\leq m} \snr{\xi}^{(1-s)(m-k)} |\xi|^{-sk}\nr{\tau^{k}_{\xi}h}_{L^{2}(\ttB_{(7+\sigma)/8})}\notag \\
  &\leq c\nr{h}_{L^{2}(  \ttB_{1})} +c\, \tail(h;\ttB_{1})\label{stanno}\,.
\end{flalign}
Again by \rif{piccolino}, and by \rif{differenzona}$_1$ we have that $\tau_{\xi}^{m} w$ is supported in $\ttB_{(7+\sigma)/8}$ and therefore we find
 \begin{flalign*}
 \tail\left(|\xi|^{-sm}\tau_{\xi}^{m} w;\ttB_{(3+\sigma)/4}\right) & \leq  c|\xi|^{-sm}\nr{\tau_{\xi}^{m} w}_{L^{2}(\er^n)} \\
 &=c|\xi|^{-sm}\nr{\tau_{\xi}^{m} w}_{L^{2}(\ttB_{(7+\sigma)/8})}  \stackleq{stanno} c\nr{h}_{L^{2}(\ttB_{1})} +c\, \tail(h;\ttB_{1})\,.
 \end{flalign*}
Using \rif{piccolino}, Mean Value Theorem and finally \rif{dg}, we find
$$
 |\xi|^{-sm}\nr{\tau_{\xi}^{m} g}_{L^2(\ttB_{(3+\sigma)/4})}   \leq c 
 \max_{|\tx{b}|=  m}\,  \nr{D^\tx{b} g}_{L^{\infty}(\ttB_{(4+\sigma)/5})}\leq  c\,  \nr{h}_{L^{2}(\ttB_1)}+c\, \tail(h;\ttB_1) \,.
$$
 Using the content of the last three displays in \rif{caccibeck} we conclude with
 $$
 |\xi|^{-sm}\nr{\tau_{\xi}^{m} w}_{W^{s,2}(\ttB_{(2+\sigma)/3})} \leq c\,  \nr{h}_{L^{2}(\ttB_1)}+c\, \tail(h;\ttB_1)\,.
 $$
Applying Proposition \ref{prop:embedding0} below to $v\equiv |\xi|^{-sm}\tau_{\xi}^{m} w$ yields
$$
|\xi|^{-s(m+1)}\nr{\tau_{\xi}^{m+1} w}_{L^2(\ttB_{(1+\sigma)/2})} \le c  |\xi|^{-sm}\nr{\tau_{\xi}^{m} w}_{W^{s,2}(\ttB_{(2+\sigma)/3})}\,,
$$
with $c\equiv c(n,N,s,\Lambda, \sigma, m)$. By \rif{piccolino} and \rif{moremore0} we have that $\eta(\cdot +k\xi)\equiv 1$ on $\ttB_{(1+\sigma)/2}$ for every $k\leq m+1$ so that \rif{differenzona}, applied with $m+1$ instead of $m$, implies $\tau_{\xi}^{m+1} w \equiv \tau_{\xi}^{m+1} h$ on $\ttB_{(1+\sigma)/2}$ and therefore from the previous display we finally  deduce \rif{passetto}. This completes the proof of the induction step and therefore the validity of \rif{eq:highsobest} for every $k \in \en$. In turn, a trivial consequence of \rif{eq:highsobest} is that 
\begin{flalign}
\notag \sup_{0<\snr{\xi}\leq 
\frac{1-\sigma}{8^{l}10}}|\xi|^{-\beta}\nr{\tau_{\xi, \ttB_{1/2}}^{l} h}_{L^2(\ttB_{1/2})} &\leq 	\sup_{0<\snr{\xi}\leq 
\frac{1-\sigma}{8^{l}10}}|\xi|^{-\beta}\nr{\tau_{\xi}^{l} h}_{L^2(\ttB_{(1+\sigma)/2})}\notag \\ & \le c \nr{h}_{L^{2}(\ttB_{1})} +c\, \tail(h;\ttB_{1}) \label{inlb}
\end{flalign}
holds with $c\equiv c(n,N,s,\Lambda,l,\sigma)$, whenever $l \in \en$ and $0 \leq \beta \leq sl$. The assertion now follows using Proposition \ref{prop:embedding} below for a suitable choice of $l, \beta$ in \rif{inlb}. 
\end{proof}
\begin{proposition} \label{prop:embedding0} 
 Let $\rr < \rrr$ and $v \in W^{s,2}(\ttB_{\rrr})$. Then 
 $$
 |\xi|^{-s}\nr{ v(\cdot +\xi)-v(\cdot)}_{L^2(\ttB_{\rr})} \leq c \nr{v}_{W^{s,2}(\ttB_{\rrr})}
 $$
 holds for every $\xi\in \er^n $ such that $\snr{\xi}< (\rrr-\rr)/2$, where $c\equiv c (n,s,\rrr-\rr, \rrr)$\,.
\end{proposition}
\begin{proof} This is \cite[Proposition 2.6, (2.11)]{bl}. 
\end{proof}
We now pass to establish Proposition \ref{prop:embedding} below, which has already been used in the proof of Proposition \ref{phr}. To proceed, we need to briefly discuss a few function spaces. The H\"older-Zygmund space $\Lambda^{t}(\er^n;\er^N)$, $t>0$, is defined as the space of those maps $w\colon \er^n\to \er^N$ such that 
$$
\nr{w}_{\Lambda^t(\er^n)}:= 
\nr{w}_{L^\infty(\er^n)}+\sup_{0<\snr{\xi}\leq 
		1}|\xi|^{-t}\nr{\tau_{\xi}^{l} w}_{L^\infty(\er^n)} < \infty\,,
$$
where $l$ is any integer such that $l>t$ (the resulting space is invariant with the choice of $l$). See for instance \cite{steinbook} and \cite[Section 1.2.1 (iv)]{Triebel3}. This space coincides with the Besov space $B^{t}_{\infty,\infty}(\er^n;\er^N)$ \cite[(1.10)]{Triebel3}. When $t\not\in \en$, it also coincides with the classical H\"older space $C^{[t], \{t\}}(\er^n)$ and a quantity equivalent to $\nr{w}_{\Lambda^t(\er^n)}$ is given by the usual 
$$\sum_{0\leq k\leq [t]} \nr{D^{k}w}_{L^\infty(\er^n)} +[D^{[t]} w]_{0, \{t\};\er^n}\,.$$ See for instance \cite[Section 1.5.2]{Triebel3} and \cite{steinbook}.  
\begin{proposition} \label{prop:embedding}
	Let $p \in (1,\infty)$, $\beta>0$, $\xi_0 >0$ and fix an integer $l > \beta$. Moreover, suppose that $v \in L^p(\Omega)$ for a Lipschitz regular, bounded domain $\Omega \subset \mathbb{R}^n$. If 
	\begin{equation} \label{eq:rel}
	\beta-n/p > k
	\end{equation}
	for some $k \in \en_0$, then
	\begin{equation} \label{eq:emb}
	\nr{D^k v}_{L^\infty(\Omega)} \leq c \nr{v}_{L^p(\Omega)}+c\sup_{0<\snr{\xi}\leq 
		\xi_0}|\xi|^{-\beta}\nr{\tau_{\xi, \Omega}^{l} v}_{L^p(\Omega)} 
	\end{equation}
	where 
	$c$ depends only on $n,\beta,l,p,k,\xi_0,\Omega$.
\end{proposition}
\begin{proof}
	Note that, in order to prove \rif{eq:emb}, in view of the strict inequality in \rif{eq:rel} and upon suitably decreasing the value of $\beta$, we can assume without loss of generality that $\beta-n/p$ is not an integer. We use results from \cite{des2, dispa, cornelia}, also reported in \cite[Theorem 1.118]{Triebel3}; see also \cite[Theorem 4.4.2.1]{TriebelI} and the discussion in  \cite[Remark 1.119]{Triebel3}. A norm in the Nikolskii/Besov space $B^\beta_{p,\infty}(\Omega)$ is given by
	\eqn{aran1}
$$  \nr{v}_{L^p(\Omega)}+ \sup_{0<\texttt{t}\leq 1} \,\texttt{t}^{-\beta} \sup_{0<\snr{\xi}\leq \texttt{t}} \, 
		\nr{\tau_{\xi, \Omega}^{l} v}_{L^p(\Omega)}$$		
which is equal to
\eqn{aran2}
$$\nr{v}_{L^p(\Omega)}+\sup_{0<\snr{\xi}\leq 
		1}|\xi|^{-\beta}\nr{\tau_{\xi,\Omega}^{l} v}_{L^p(\Omega)}\,.\footnote{The quantity in \rif{aran1} is obviously larger than the one in \rif{aran2}. For the opposite inequality, fix $\texttt{t}, \eps\in (0,1)$ and find $\xi_\texttt{t}$ such that $0<\snr{\xi_\texttt{t}}\leq \texttt{t}$ and 
$$
 \sup_{\snr{\xi}\leq \texttt{t}}  \, \nr{\tau_{\xi,\Omega}^{l} v}_{L^p(\Omega)}  <  \nr{\tau_{\xi_\texttt{t}}^{l} v}_{L^p(\Omega_{l,\xi_\texttt{t}})} + \eps \texttt{t}^{\beta} \leq  \texttt{t}^\beta\sup_{0<\snr{\xi}\leq 
		1}|\xi|^{-\beta}\nr{\tau_{\xi,\Omega}^{l} v}_{L^p(\Omega)} +  \eps\texttt{t}^{\beta}\,.
$$
Then divide by $\texttt{t}^\beta$, take the sup with respect to $\texttt{t}$ and let $\eps \to 0$.}$$
In turn, this last quantity is easily seen to be equivalent to the one in parentheses in \rif{eq:emb} via a constant that also depends on $\snr{\xi_0}$ but is independent of $v$.  	
Via standard extension operators the space $B^\beta_{p,\infty}(\Omega)$ embeds continuously in $B^\beta_{p,\infty}(\mathbb{R}^n)$; see for instance \cite[Theorem 6.1]{des2}, \cite[Theorem 2.9]{cornelia}, \cite{dispa} and \cite[Theorem 1.105]{Triebel3}. In turn, by \cite[Theorem 2.5]{Triebel4} we have the continuous embedding $B^\beta_{p,\infty}(\mathbb{R}^n) \hookrightarrow B^{\beta-n/p}_{\infty,\infty}(\mathbb{R}^n)$, and ultimately, the embedding $B^\beta_{p,\infty}(\Omega) \hookrightarrow B^{\beta-n/p}_{\infty,\infty}(\er^n)$; moreover, $B^{\beta-n/p}_{\infty,\infty}(\er^n)$ coincides with the H\"older-Zygmund space $\Lambda^{\beta-n/p}(\er^n)$ by the discussion made immediately before Proposition \ref{prop:embedding}. In turn, using the fact that $\beta-n/p$ is not an integer, the same discussion implies that this last space coincides with $C^{[\beta-n/p],\{\beta-n/p\}}(\er^n)$, that obviously embeds into $C^k(\Omega)$ in view of \eqref{eq:rel}. Summarizing, we have the continuous embeddings $B^\beta_{p,\infty}(\Omega) \hookrightarrow B^{\beta-n/p}_{\infty,\infty}(\er^n) = \Lambda^{\beta-n/p}(\er^n) \hookrightarrow C^k(\Omega),$
	that imply \eqref{eq:emb} in view of the equivalence of norms discussed at the beginning of the proof. 
\end{proof}
Proposition \ref{phr} can be coupled with perturbation arguments to deduce quantitative oscillation estimates for weak solutions to nonlocal systems with constant coefficients and a sufficiently integrable forcing term. The result is in Proposition \ref{cdg} below. We remark that the integrability assumption \rif{integrino} below on the right-hand side $g$ of \rif{dg.1} is sharp with respect to the result obtained in \cite{bls}. 
\begin{proposition}[Nonlocal Campanato II]\label{cdg}
Let $v\in W^{s,2}(\BBB;\er^N)\cap L^{1}_{2s}$ be a weak solution to
\eqn{dg.1}
$$
-\mathcal{L}_{a_{0}}v=g\qquad \mbox{in} \ \ \BBB\equiv B_{\rrr}(x_\BBB)\,,\ \ \rrr\leq 1\,,
$$
where
\eqn{integrino}
$$g\in \mathcal M^{\frac{n}{2s-\beta}}(\BBB;\er^N)\,, \quad 0< \beta < \min\left\{2s,1\right\} \,.$$
Then $v\in C^{0,\beta}_{\loc}(\BBB;\er^N)$, and for every $\gamma \in (0,1)$ there exists a constant $c\equiv c (n,N,s,\Lambda,\beta, \gamma)$ such that 
\eqn{dg.2}
$$
\osc_{\rr \BBB}v\le c\left(\frac{\rr}{\rrr}\right)^{\beta} \tx{E}_{v}(x_{\BBB},\rrr)+c\nr{g}_{\mathcal M^{\frac{n}{2s-\beta}}(\BBB)}\rr^{\beta}
$$
holds for any $\rr \leq \gamma \rrr$ together with
\eqn{dg.2bis}
$$
[v]_{0,\beta;\gamma \BBB} \leq c  \rrr^{-\beta} \tx{E}_{v}(x_{\BBB},\rrr)+c\nr{g}_{\mathcal M^{\frac{n}{2s-\beta}}(\BBB)}\,.
$$
Moreover 
\eqn{dg.77}
$$
 \tx{E}_{v}(x_{\BBB},\rr)\le c\left(\frac{\rr}{\rrr}\right)^{\beta} \tx{E}_{v}(x_{\BBB},\rrr)+c\nr{g}_{\mathcal M^{\frac{n}{2s-\beta}}(\BBB)} \rr^{\beta} 
$$
holds whenever $0 < \rr \leq \rrr$, with $c\equiv c (n,N,s,\Lambda,\beta)$. The excess quantity $\tx{E}_{v}(x_{\BBB},\cdot)$ is defined in \eqref{exc}. 
\end{proposition}
\begin{proof} In the following we denote $\beta_m:=\min\{2s,1\}$. Let $x_0 \in \BBB$ and take $r$ such that $0<r<\rrr-\snr{x_0-x_\BBB}$ that implies $B_{r}(x_{0})\Subset \BBB$.   In the following, unless otherwise specified, all the balls but $\BBB$ will be centred at $x_0$ while $\BBB$ will always abbreviate $B_{\rrr}(x_\BBB)$. Consider $r_{*}$ such that $B_{r}\equiv B_{r}(x_{0})\Subset B_{r_{*}}(x_{0})\equiv B_{r_{*}}\subset   \BBB$ and by Lemma \ref{esiste} define $h\in\mathbb{X}_{v}^{s,2}(B_{r},B_{r_{*}})$ as the solution to
$$
\begin{cases}
\ -\mathcal{L}_{a_{0}}h=0\quad &\mbox{in} \ \ B_{r}\\
\ h=v\quad &\mbox{in} \ \ \mathbb{R}^{n}\setminus B_{r}
\end{cases}
$$
in the sense of \rif{eqweaksol22}. Lemma \ref{lemma3.4} applies with $d=n/(2s-\beta)$, giving
\eqn{dg.3}
$$
\nra{v-h}_{L^{2}(B_{r})}\le cr^{\beta}\nr{g}_{\mathcal M^{\frac{n}{2s-\beta}}(B_r)}
$$
with $c \equiv c (n,N,s,\Lambda,\beta)$. Thanks to this and Jensen's inequality, we obtain
\begin{flalign}
\notag \tail(h-(h)_{B_{r}};B_{r}) & = \tail(v-(h)_{B_{r}};B_{r})\\*
 & \leq \tail(v-(v)_{B_{r}};B_{r}) + c\nra{v-h}_{L^{2}(B_{r})}\notag \\*
& \leq  \tail(v-(v)_{B_{r}};B_{r}) + cr^{\beta}\nr{g}_{\mathcal M^{\frac{n}{2s-\beta}}(B_r)}\,. \label{dg.3bis}
\end{flalign}
We fix $\texttt{t}\in (0,2^{-4})$; recalling Proposition \ref{phr} and using \rif{minav} repeatedly we have
\begin{eqnarray}
\nra{v-(v)_{B_{\texttt{t} r}}}_{L^{2}(B_{\texttt{t} r})}&\le&c\nra{h-(h)_{B_{\texttt{t} r}}}_{L^{2}(B_{\texttt{t} r})}+c\nra{v-h}_{L^{2}(B_{\texttt{t} r})}\nonumber \\
&\stackrel{\eqref{dg.3}}{\le}&c\texttt{t} r\nr{Dh}_{L^{\infty}(B_{r/2})}+c\texttt{t}^{-n/2}r^{\beta}\nr{g}_{\mathcal M^{\frac{n}{2s-\beta}}(B_r)}\nonumber \\
&\stackrel{\eqref{eq:higherreg0}}{\le}&
c\texttt{t}\nra{h-(h)_{B_{r}}}_{L^{2}(B_{r})}+c\texttt{t}\tail(h-(h)_{B_{r}};B_{r}) \nonumber \\*
&&+c\texttt{t}^{-n/2}r^{\beta}\nr{g}_{\mathcal M^{\frac{n}{2s-\beta}}(B_r)}\nonumber \\*
&\stackrel{\eqref{dg.3},\eqref{dg.3bis}}{\le}&c\texttt{t} \left(\nra{v-(v)_{B_{r}}}_{L^{2}(B_{r})}
+\tail(v-(v)_{B_{r}};B_{r})\right)\nonumber \\ && +c\texttt{t}^{-n/2}r^{\beta}\nr{g}_{\mathcal M^{\frac{n}{2s-\beta}}(B_r)},\label{disdis}
\end{eqnarray}
for $c\equiv c(n,N,s,\Lambda,\beta)$. Furthermore, using \rif{minav} and estimating similarly to \rif{disdis} but this time on $B_{\lambda}$, $\lambda \leq r/4$
\begin{eqnarray*}
\tail(v-(v)_{B_{\texttt{t} r}};B_{\texttt{t} r})&\stackrel{\eqref{scatail}}{\le}&c\texttt{t}^{2s}\tail(v-(v)_{B_{r}};B_{r})+c\texttt{t}^{2s}\nra{v-(v)_{B_{r}}}_{L^{2}(B_{r})}\nonumber \\
&&+c\int_{\texttt{t} r}^{r}\left(\frac{\texttt{t} r}{\lambda}\right)^{2s}\nra{v-(v)_{B_{\lambda}}}_{L^{2}(B_{\lambda})}\frac{\dlam}{\lambda}\nonumber \\
&\leq &c\texttt{t}^{2s}\tail(v-(v)_{B_{r}};B_{r})+c\texttt{t}^{2s}\nra{v-(v)_{B_{r}}}_{L^{2}(B_{r})}\nonumber \\
&&+c\int_{\texttt{t} r}^{r/4}\left(\frac{\texttt{t} r}{\lambda}\right)^{2s}\nra{v-(v)_{B_{\lambda}}}_{L^{2}(B_{\lambda})}\frac{\dlam}{\lambda}\nonumber \\
&\le&c\texttt{t}^{2s}\tail(v-(v)_{B_{r}};B_{r})+c\texttt{t}^{2s}\nra{v-(v)_{B_{r}}}_{L^{2}(B_{r})}\nonumber \\
&&+c\texttt{t}^{2s}\nra{v-h}_{L^{2}(B_{r})}\int_{\texttt{t} r}^{r}\left(\frac{ r}{\lambda}\right)^{2s+n/2}\frac{\dlam}{\lambda}\nonumber \\
&&+c\texttt{t}^{2s}r\nr{Dh}_{L^{\infty}(B_{r/2})}\int_{\texttt{t} r}^{r}\left(\frac{ r}{\lambda}\right)^{2s-1}\frac{\dlam}{\lambda}\nonumber \\
&\stackrel{\eqref{dg.3}}{\le}&c\texttt{t}^{2s}\tail(v-(v)_{B_{r}};B_{r})+c\texttt{t}^{2s}\nra{v-(v)_{B_{r}}}_{L^{2}(B_{r})}\nonumber \\
&&+c\texttt{t}^{-n/2}r^{\beta}\nr{g}_{\mathcal M^{\frac{n}{2s-\beta}}(B_r)}+c\texttt{t}^{\beta_m}\log (1/\texttt{t}) r\nr{Dh}_{L^{\infty}(B_{r/2})}\nonumber
 \\ 
&\stackrel{\eqref{eq:higherreg0}}{\le}&c\texttt{t}^{2s}\tail(v-(v)_{B_{r}};B_{r})+c\texttt{t}^{2s}\nra{v-(v)_{B_{r}}}_{L^{2}(B_{r})}\nonumber \\
&&+c\texttt{t}^{-n/2}r^{\beta}\nr{g}_{\mathcal M^{\frac{n}{2s-\beta}}(B_r)}+c\texttt{t}^{\beta_m}\log (1/\texttt{t}) \nra{h-(h)_{B_{r}}}_{L^{2}(B_{r})}\nonumber \\
&&+c\texttt{t}^{\beta_m}\log (1/\texttt{t})\tail(h-(h)_{B_{r}};B_{r})\nonumber \\
&\stackrel{\eqref{dg.3}, \eqref{dg.3bis}}{\le}&c\texttt{t}^{\beta_m}\log (1/\texttt{t})\left(\nra{v-(v)_{B_{r}}}_{L^{2}(B_{r})}+  \tail(v-(v)_{B_{r}};B_{r})\right)\nonumber \\
&&+c\texttt{t}^{-n/2}r^{\beta}\nr{g}_{\mathcal M^{\frac{n}{2s-\beta}}(B_r)}
\end{eqnarray*}
with $c\equiv c(n,N,s,\Lambda,\beta)$. \footnote{Note that in the above chain of inequalities, using \eqref{minav} we have estimated
\begin{flalign*}
\int_{\texttt{t} r}^{r}\left(\frac{\texttt{t} r}{\lambda}\right)^{2s}\nra{v-(v)_{B_{\lambda}}}_{L^{2}(B_{\lambda})}\frac{\dlam}{\lambda} & = \int_{\texttt{t} r}^{r/4}[\ldots]\frac{\dlam}{\lambda}+ \int_{r/4}^{r}[\ldots]\frac{\dlam}{\lambda}
\\  &\lesssim \int_{\texttt{t} r}^{r/4}[\ldots]\frac{\dlam}{\lambda}+\texttt{t}^{2s}\nra{v-(v)_{B_{r}}}_{L^{2}(B_{r})}. 
\end{flalign*}
Moreover, we have estimated 
\[
\texttt{t}^{2s}r\int_{\texttt{t}r}^{r}\left(\frac{r}{\lambda}\right)^{2s-1}\,\frac{d\lambda}{\lambda}
= \begin{cases}
\texttt{t}^{2s}\frac{(1-\texttt{t}^{1-2s})}{1-2s} r , & s\not=\frac12\\ \\
\texttt{t}\left(\log\frac1{\tt t} \right)r, & s=\frac12
\end{cases}
\quad \le  \texttt{t} ^{\min\{2s,1\}}\left(\log\frac1{\texttt{t}}\right)r. 
\]
This also shows that in the proof of \rif{dg.2}-\rif{dg.77} there is no blow-up of the constants when $s$ moves around $1/2$. 
} Recalling \rif{ovvia}, 
and merging the displays from \rif{disdis} on, we arrive at
$$
 \tx{E}_{v}(x_{0},\texttt{t} r)\le c_*\texttt{t}^{\beta_m}\log (1/\texttt{t})\tx{E}_{v}(x_{0},r)+c_*\texttt{t}^{-n/2}\nr{g}_{\mathcal M^{\frac{n}{2s-\beta}}(\BBB)}r^{\beta}
$$ 
with 
$c_* \equiv c_* (n,N,s,\Lambda,\beta)$.  We now consider 
$\beta_1:=(\beta + \beta_m)/2$, so that $\beta<\beta_1 < \beta_m$, and then we determine $\texttt{t}<1/2^4$ such that
\eqn{thistt}
$$
c_*\texttt{t}^{\beta_m-\beta_1}\log (1/\texttt{t})< 1/2 \ \Longrightarrow \ \texttt{t}\equiv \texttt{t}(n,N,s,\Lambda,\beta),
$$
so that we gain 
$$
 \tx{E}_{v}(x_{0},\texttt{t} r)\le\texttt{t}^{\beta_1}  \tx{E}_{v}(x_{0},r)+c_{1}\nr{g}_{\mathcal M^{\frac{n}{2s-\beta}}(\BBB)}r^{\beta}
$$
where $c_1\equiv c_1(n,N,s,\Lambda,\beta)$. Iterating the previous inequality, by induction 
$$
 \tx{E}_{v}(x_{0},\texttt{t}^{k+1} r) \leq \texttt{t}^{(k+1)\beta_1}\tx{E}_{v}(x_{0},r)+c_{1}\nr{g}_{\mathcal M^{\frac{n}{2s-\beta}}(\BBB)}\texttt{t}^{k\beta}r^{\beta} \sum_{0\leq j\leq k} \texttt{t}^{j(\beta_1-\beta)}
$$
is seen to hold for every $k \in \en_0$, that readily yields 
\eqn{dg.6}
$$
 \tx{E}_{v}(x_{0},\texttt{t}^{k+1} r) \leq c \texttt{t}^{(k+1)\beta} \left(  \tx{E}_{v}(x_{0},r)+\nr{g}_{\mathcal M^{\frac{n}{2s-\beta}}(\BBB)} r^{\beta}\right)
$$
where $c\equiv c(n,N,s,\Lambda,\beta)$ (for this keep in mind the dependence of $\texttt{t}$ in \rif{thistt}). 
 This implies that 
\eqn{dg.7}
$$
 \tx{E}_{v}(x_{0},\rr)\le c\left(\frac{\rr}{r}\right)^{\beta} \left(  \tx{E}_{v}(x_{0},r)+\nr{g}_{\mathcal M^{\frac{n}{2s-\beta}}(\BBB)} r^{\beta}\right)$$
holds whenever $0<\rr\le r< \rrr-\snr{x_0-x_{\BBB}}$, again with $c\equiv c(n,N,s,\Lambda,\beta)$. The derivation of \rif{dg.7} from \rif{dg.6} follows via a variation of a standard interpolation argument that we briefly recall for completeness. 
By \rif{scataildopoff}$_1$
\eqn{scataildopo2}
$$
 \tx{E}_{v}(x_{0},\sigma_1  r )\lesssim_{n,s} \left(\frac{\sigma_2}{\sigma_1}\right)^{n/2}  \tx{E}_{v}(x_{0},\sigma_2 r  )
$$
holds whenever $0 < \sigma_1 < \sigma_2 <1$, from which \rif{dg.7} follows immediately when $\texttt{t}r \leq \rr\leq r$, recalling that $\texttt{t}\equiv \texttt{t}(n,N,s,\Lambda,\beta)$ as in \rif{thistt}. To proceed with the proof of \rif{dg.7} for the case $\rr < \texttt{t}r$, we identify an integer $k\geq 0$ such that $\texttt{t}^{k+2}r \leq \varrho < \texttt{t}^{k+1}r$ so that, using again \rif{scataildopo2}, we find 
$$ \tx{E}_{v}(x_{0},\varrho) \leq c\texttt{t}^{-n/2}  \tx{E}_{v}(x_{0},\texttt{t}^{k+1} r) \leq c \texttt{t}^{(k+1)\beta} \left(  \tx{E}_{v}(x_{0},r)+\nr{g}_{\mathcal M^{\frac{n}{2s-\beta}}(\BBB)} r^{\beta}\right)\,,$$ 
from which \rif{dg.7} follows using that $\rr/r\approx_{\texttt{t}} \texttt{t}^{k+1}$ and the dependence of $\texttt{t}$ described in \rif{thistt}. Taking $x_0=x_\BBB$ in \rif{dg.7} and letting $r \to \rrr$ yields \rif{dg.77}. We now pass to the proof of \rif{dg.2}-\rif{dg.2bis}. We apply the above arguments to any $x_0\in \gamma \BBB$ with $r= (1-\gamma)\rrr$. Applying \rif{dg.7} together with \rif{scataildopoff}$_2$ (with $\texttt{t}\equiv 1-\gamma$, $x_0\equiv x_{\BBB}$) we deduce that
$$
 \tx{E}_{v}(x_{0},\rr)\le \frac{c}{(1-\gamma)^{n+\beta}}\left(\frac{\rr}{\rrr}\right)^{\beta}   \tx{E}_{v}(x_{\BBB},\rrr)+c\nr{g}_{\mathcal M^{\frac{n}{2s-\beta}}(\BBB)} \rr^{\beta}
$$
holds whenever $x_{0}\in B_{\gamma \rrr}(x_{\BBB})$ and $\rr \leq  (1-\gamma)\rrr$, that in turn implies 
$$
\sup_{0<\rr\leq  (1-\gamma)\rrr}\sup_{x_{0}\in B_{\gamma \rrr}(x_{\BBB})}\rr^{-\beta}\nra{v-(v)_{B_{\rr}(x_0)}}_{L^{2}(B_{\rr}(x_0))}\le \frac{c\rrr^{-\beta}}{(1-\gamma)^{n+\beta}} \tx{E}_{v}(x_{\BBB},\rrr)+c\nr{g}_{\mathcal M^{\frac{n}{2s-\beta}}(\BBB)}\,.
$$
Recalling the definition in \rif{massimale}, the content of the last display and H\"older inequality give
$$
\nr{\textnormal{M}^{\#,\beta}_{(1-\gamma)\rrr}(v;\cdot)}_{L^\infty(\gamma \BBB)}\le \frac{c\rrr^{-\beta}}{(1-\gamma)^{n+\beta}} \tx{E}_{v}(x_{\BBB},\rrr)+c\nr{g}_{\mathcal M^{\frac{n}{2s-\beta}}(\BBB)}\,.
$$
Lemma \ref{campmax} now implies that 
$$
|v(x)-v(y)| \le c\left[ \frac{\rrr^{-\beta}}{(1-\gamma)^{n+\beta}} \tx{E}_{v}(x_{\BBB},\rrr)+\nr{g}_{\mathcal M^{\frac{n}{2s-\beta}}(\BBB)}\right]\snr{x-y}^{\beta}
$$
holds whenever $x,y\in \gamma\BBB$ are such that  $\snr{x-y}\leq (1-\gamma)\rrr/2$. Using a standard chain argument, from this we conclude that 
$$
|v(x)-v(y)| \le c\left[\frac{\rrr^{-\beta}}{(1-\gamma)^{n+1}} \tx{E}_{v}(x_{\BBB},\rrr)+ \frac{1}{(1-\gamma)^{1-\beta}}\nr{g}_{\mathcal M^{\frac{n}{2s-\beta}}(\BBB)}\right]\snr{x-y}^{\beta}
$$
this time holds for every choice of $x,y\in \gamma\BBB$, where $c\equiv c(n,N,s,\Lambda,\beta)$. This last inequality is easily seen to imply \rif{dg.2}-\rif{dg.2bis}. 
\end{proof}
\section{Higher differentiability and integrability}\label{higherhigher}
Here we revisit some methods from \cite{KMS} in order to obtain various higher differentiability and integrability results, that will be useful both here and in future work \cite{follow}. 

\subsection{Boundary differentiability and integrability}\label{albordo} 
We start by deriving global higher differentiability results for solutions $w\in W^{s,2}(\er^n;\er^N)$ to nonlocal Dirichlet problems of the type
\begin{flalign}\label{sh.30}
\begin{cases}
\ -\mathcal{L}_{\ti{a}}w=-\mathcal{L}_{-\ti{a}}v+g\quad &\mbox{in} \ \ \Omega\\
\ w=0\quad &\mbox{in} \ \ \mathbb{R}^{n}\setminus \Omega
\end{cases}
\end{flalign}
where $\ti{a}\colon \mathbb{R}^n \times \mathbb{R}^n\to \mathbb{R}^{N\times N}$ is a measurable  matrix field satisfying  \eqref{bs.1}-\eqref{bs.2} (obviously recasted for the case there is no dependence on $(v,w)$) and $\Omega$ is a Lipschitz-regular, bounded domain. Keeping in mind the content of Remark \ref{remarkino}, we have 
\begin{definition}\label{defidir}
With $v\in W^{s,2}  (\er^n;\er^N) $ and $g\in L^{2_*}(\Omega;\er^N)$, a map $w\in W^{s,2}(\er^n;\er^N)$, $w \equiv 0$ in $\er^n\setminus \Omega$, is a weak solution to  
\eqref{sh.30} provided 
		\begin{flalign} 
		\notag
			\noindent &\int_{\mathbb{R}^n} \int_{\mathbb{R}^n} \langle \ti{a}(x,y)(w(x)-w(y)),\varphi(x)-\varphi(y) \rangle \frac{\dxy}{|x-y|^{n+2s}} \\
			& \quad =-\int_{\mathbb{R}^n} \int_{\mathbb{R}^n} \langle \ti{a}(x,y)(v(x)-v(y)),\varphi(x)-\varphi(y) \rangle \frac{\dxy}{|x-y|^{n+2s}}  + \int_{\Omega} \langle g, \varphi \rangle\dx \label{debolebordo}
		\end{flalign}
		holds for every $\varphi \in W^{s,2}(\er^n;\er^N)$ such that $\varphi \equiv 0$ in $\er^n\setminus \Omega$.
\end{definition}
\begin{theorem}\label{maggiore} Let $w\in W^{s,2}(\er^n;\er^N)$ be a weak solution to \eqref{sh.30} in the sense of Definition \ref{defidir} and assume that $v\in W^{s+\delta_{0},2} (\er^n;\er^N) $ and $g\in L^{\infty}(\er^n;\er^N),$ where $0 < s<s +\delta_{0}< 1$. Then there exist numbers $t$ and $p$, with $s<t<1$ and $p>2$, both depending on $n, N,s, \Lambda, \delta_{0}$ and $[\Omega]_{0,1}$, such that 
\eqn{maggiore2}
$$
[w]_{t,p;B} \leq c[w]_{s,2;\er^n}+c[v]_{s+\delta_0,2;\er^n}+ c \nr{g}_{L^{\infty}(2  B)} 
$$
holds for every ball $B \subset \er^n$, where $c \equiv c (n, N,s, \Lambda, \delta_{0}, [\Omega]_{0,1}, \snr{B})\geq 1$.  The dependence on $\snr{B}$ in the last constant can be replaced by a dependence on $c_{0}\geq 1$, where $1/c_{0} \leq \snr{B} \leq c_{0}$. 
\end{theorem} 
The idea for proving Theorem \ref{maggiore} is essentially to show that the arguments developed in \cite{KMS} can be applied here once a suitable Caccioppoli type inequality is shown to hold. 
In the rest of this section we employ the notation and the conventions used in \cite{KMS} (adapted to the vectorial case). For this we shall assume without loss of generality that $\delta_{0}< s/40$. We further denote $\gamma:= s+\delta_{0}/2$. As in \cite[Section 3A]{KMS}, it is possible to find $0< \delta_1 <s/(4n)$ and $\ti{p}$ with $2n/(n+\delta_{0})<\ti{p}<2$
such that
\eqn{relazione}
$$
s+\delta_{0} -\frac{n}{2}= \gamma(1+\delta_1) - \frac{n}{\ti{p}(1+\delta_1)}\,, \qquad \gamma(1+\delta_1) < s+\delta_{0}
$$
so that the classical fractional embedding theorem \cite[Theorem 14.22]{leoni} gives
\eqn{immersione}
$$
[v]_{\gamma(1+\delta_1), \ti{p}(1+\delta_1);\er^n}
 \le c
[v]_{s+\delta_{0}, 2;\er^n}\,,
$$
where the constant $c$ depends only on $n,N,s,\delta_0,\delta_1, \ti{p}$. 
With $B\subset \er^n$ being a ball, as usual we denote $\mathcal B =  B\times B$ and, for $x, y \in \er^n, x\neq y$ and $0< \eps<1$, 
\eqn{meaning}
$$
\begin{cases}
\displaystyle U_{\eps}(x,y) := \frac{\snr{w(x)-w(y)}}{\snr{x-y}^{s+\eps}}\,, \qquad G_{\eps}(x,y) := \frac{\snr{v(x)-v(y)}}{\snr{x-y}^{\gamma+2\eps/\ti{p}}}\\[12pt]
\displaystyle \textnormal{d}\mu_{\eps}(x,y):= \frac{\textnormal{d}\mathcal L^{2n}(x,y)}{\snr{x-y}^{n-2\eps}}\,, \qquad F(x,y):= \snr{g(x)}\,,
\end{cases}
$$
where the exponent $\ti{p}$ appears in \eqref{relazione}. Such definitions parallel those in \cite[(4-5)]{KMS}. 
\begin{lemma}\label{selfie} In the setting of Theorem \ref{maggiore}, and with the notation in \eqref{meaning}, we have that 
\begin{flalign}
\nonumber 
\left(\mint_{\mathcal B} U_{\eps}^2\d\mu_{\eps} \right)^{1/2}&\leq   \frac{c}{\sigma \eps^{1/q-1/2}} \left( \mint_{2\mathcal B} U_{\eps}^q\d\mu_{\eps}\right)^{1/q} \\ \notag & \quad  + \frac{\sigma}{ \eps^{1/q-1/2}} 
\sum_{k\geq 1} 2^{-k(s-\eps)}\left(\mint_{2^k\mathcal B} U_{\eps}^q \d \mu_{\eps}\right)^{1/q}\\ 
 & \quad   +  \frac{c[\mu(\mathcal B)]^{\eta}}{\eps^{1/2_*-1/2}}\left(\mint_{2\mathcal B} F^{2_*}\d\mu_{\eps}\right)^{1/2_*}\notag \\
 &\quad +\frac{c[\mu(\mathcal B)]^{\theta}}{\eps^{1/\ti{p}-1/2}}\sum_{k\geq 1} 2^{-(2s-\gamma-2\eps/\ti{p})k}\left(\mint_{2^k\mathcal  B} G_{\eps}^{\ti{p}}\d\mu_{\eps} \right)^{1/\ti{p}}\label{reverse}
\end{flalign}
 holds whenever $B\subset \er^n $ is a ball, $2^k\mathcal B=2^kB\times 2^kB$ for $k\geq 0$, $0<\sigma <1$, where
 \eqn{ilq} 
 $$
 \begin{cases}
 \displaystyle  0 < \eps < \min\{s/4, 1-s\}\,, \qquad 
 2_*= \frac{2n}{n+2s}< q := \frac{2n+4\eps}{n+2s +2\eps}< 2 \\[13pt]
  \displaystyle 
 \eta:= \frac{s-\eps}{n+2\eps}\,, \quad \theta:=\frac{\gamma-s +\eps(2/\ti{p}-1)}{n+2\eps}=
 \frac{\delta_{0}/2 +\eps(2/\ti{p}-1)}{n+2\eps}\,.
 \end{cases}
 $$
The constant $c$ in \eqref{reverse} depends only on $n,N,s, \Lambda, \delta_0, [\partial \Omega]_{0,1}$. Moreover, the right-hand side of \eqref{reverse} is always finite. 
\end{lemma}
Note that in \rif{reverse} we are obviously denoting
$$
\mint_{2^k\mathcal B} W\d\mu_{\eps} = \frac{1}{\mu_{\eps}(2^k\mathcal B)} \int_{2^k\mathcal B} W\d\mu_{\eps} 
$$
as soon as $W$ is integrable with respect to the measure $\mu_{\eps}$. As in \cite[(4.1)]{KMS}, for every ball $B \subset \er^n$, it is 
\eqn{addiaddi}
$$
\mu_\eps(B\times B) = \frac{c(n)}{\eps} |B|^{1+2\eps/n} \,.
$$
Prior to the proof of Lemma \ref{selfie}, we report the following one, which is an analog of \cite[Lemma 4.3]{KMS}. It can be proved using \rif{compa} instead of \cite[(4.11)]{KMS}.
\begin{lemma}\label{fracci}  Let $B_0\subset \er^n$ be a ball and $w \in W^{s,2}(B_0)$ be such that 
\eqn{poinc0}
$$
\snr{B_0} \leq c_{0} \snr{\{w= 0 \}\cap B_0}
$$ 
is satisfied. If  $\eps$ and $q$ are as in \eqref{ilq}, 
then 
\eqn{perfect}
$$
\mint_{B_0} |w|^2 \dx \leq \frac{c|B_0|^{2(s +\eps)/n}}{\eps^{2/q}}  \left(\mint_{B_0 \times B_0} U_{\eps}^q\d\mu_{\eps} \right)^{2/q} 
$$
holds for a constant $c$ depending only on $n,N,s,s-\eps$ and $c_{0}$. 
\end{lemma} 
\begin{proof}[Proof of Lemma \ref{selfie}] The proof is a close adaptation of \cite[Proposition 4.4]{KMS}. Let us denote $B\equiv B(x_{0}, \rrr)$; we distinguish three different situations. The first is when $3B/2\subset \Omega$. In this case we can test \rif{debolebordo} by $\varphi= \psi^2 (w -(w)_{2B})$ where $\psi \in C^{\infty}_0(3B/2)$ (compactly supported in $3B/2$), $\psi \equiv 1$ on $B$ and $\snr{D\psi}\lesssim 1/\rrr$ and \rif{reverse} follows exactly as in \cite[Proposition 4.4]{KMS}.  The second case is when $3B/2\subset \er^n \setminus \Omega$ and here there is actually nothing to prove as the left-hand side of \rif{reverse} vanishes. The remaining case is when $3B/2$ touches both $\Omega$ and its complement. It follows that condition \rif{poinc0} is satisfied by $2^kB\equiv B_0$ for every integer $k\geq 1$, with $c_{0}\equiv c_{0}(n, [\partial \Omega]_{0,1})$ being independent of $k$ so that  \rif{perfect} implies
 \eqn{perfect2}
 $$ 
\left(\mint_{2^kB} |w|^2 \dx\right)^{1/2} \leq \frac{c(2^k\rrr)^{s +\eps}}{\eps^{1/q}}  \left(\mint_{2^k\mathcal B} U_{\eps}^q\d\mu_{\eps} \right)^{1/q} 
$$
whenever $k \geq 1$ with $c\equiv c (n,N,s, [\partial \Omega]_{0,1})$. 
  Testing \rif{debolebordo} by $\varphi= \psi^2 w$, we proceed as in \cite[Theorem 3.2]{KMS}, and using the ellipticity of $\tilde a(\cdot)$ in  \eqref{bs.1}-\eqref{bs.2} to estimate as in \cite[first display at page 69]{KMS}, we find 
 \begin{flalign}
   \notag \textnormal{I} &:= \int_{2B} \mint_{2B}  \frac{|\psi(x)w(x)-\psi(y)w(y)|^2}{|x-y|^{n+2s}}  \dxy\\
  \notag &  \leq \frac{c}{\rrr^{2s} } \mint_{2B} \snr{w}^2 \dx +c  \int_{\er^n \setminus 2B} \frac{\snr{w(y)}}{|y-x_{0}|^{n+2s}}\dy 
\mint_{2B} \snr{w}\dx +c \rrr^{2s}\left(\mint_{2B} |g|^{2_*}\dx\right)^{2/2_*}\\ \notag 
  & \qquad + c\rrr^{\delta_{0}}\left[\sum_{k\geq 1} 2^{-(2s-\gamma) k} \left(\int_{2^{k}B}\mint_{2^{k}B} \frac{ |v(x)-v(y)|^{\ti{p}}}{|x-y|^{n+ \ti{p}\gamma}}\dxy\right)^{1/\ti{p}}\right]^{2} \notag \\
  &   =: \textnormal{II} + \textnormal{III} +\textnormal{IV}+ \textnormal{V}\label{perfect22}, \end{flalign}
that holds for a constant $c \equiv c (n,N,\Lambda, s, [\partial \Omega]_{0,1})$. In order to estimate the terms I,...,V we proceed similarly to \cite[pages 79-80]{KMS} and therefore using also \rif{perfect2} with $k=1$, we find 
$$
\frac{\rrr^{2\eps}}{\eps}
\mint_{\mathcal B} U_{\eps}^2\d\mu_{\eps} \leq  c \textnormal{I}  \,, \qquad \textnormal{II}  \leq   \frac{c\rrr^{2\eps}}{\eps^{2/q}}  \left(\mint_{2\mathcal B} U_{\eps}^q \d\mu_{\eps} \right)^{2/q}\,.
$$
Coming to \textnormal{III}, via annuli decomposition, we have
\begin{align*}
 & \int_{\er^n \setminus 2B} \frac{\snr{w}}{|y-x_{0}|^{n+2s}}\dy  
 \leq  c\sum_{k\geq 2} (2^{k}\rrr)^{-2s} \mint_{2^{k} B }\snr{w}\dy\\ &\qquad  \leq c\sum_{k\geq 2} (2^{k}\rrr)^{-2s} \left(\mint_{2^{k} B }\snr{w}^2\dy\right)^{1/2} \leq \frac{c\rrr^{-s+\eps}}{\eps^{1/q}}\sum_{k\geq 2} 2^{-k(s-\eps)}\left(\mint_{2^k \mathcal B} U_{\eps}^q \d\mu_{\eps} \right)^{1/q} 
\end{align*}
where \rif{perfect2} has been again used in the last line. 
Yet another application of \rif{perfect2}, this time with $k=1$, and using Young's inequality, leads us to 
\begin{flalign*}
\textnormal{III} &\leq  \frac{c \rrr^{2\eps}}{\eps^{2/q}}  \left(\mint_{2\mathcal B} U_{\eps}^q \d\mu_{\eps} \right)^{1/q}\sum_{k\geq 2} 2^{-k(s-\eps)}
 \left(\mint_{2^k\mathcal B} U_{\eps}^q \d\mu_{\eps}\right)^{1/q} \\
&\leq   \frac{c \rrr^{2\eps}}{ \sigma^2 \eps^{2/q}}  \left(\mint_{2\mathcal B} U_{\eps}^q \d\mu_{\eps} \right)^{2/q}
 + \frac{\sigma^2 \rrr^{2\eps}}{  \eps^{2/q}} \left[\sum_{k\geq 1} 2^{-k(s-\eps)}\left(\mint_{2^k\mathcal B} U_{\eps}^q \d\mu_{\eps} 
 \right)^{1/q}\right]^2
\end{flalign*}
for every $\sigma \in (0,1)$. 
Again arguing as in \cite[pages 80-81]{KMS} we have 
$$
\textnormal{IV} \leq  \frac{c\rrr^{2s}}{\eps^{2/2_*}}\left(\mint_{2\mathcal B} F^{2_*}\d\mu_{\eps}\right)^{2/2_*}
$$
and 
$$
\textnormal{V}  \leq \frac{c \rrr^{2(\gamma-s+2\eps/\ti{p})}}{\eps^{2/\ti{p}}}\left[\sum_{k\geq 1}2^{-(2s-\gamma-2\eps/\ti{p})k}\left(\mint_{2^k\mathcal  B} G_{\eps}^{\ti{p}}\d\mu_{\eps} \right)^{1/\ti{p}}\right]^2\,.
$$
Connecting the estimates found for $\textnormal{I}, \ldots, \textnormal{V}$ leads to \rif{reverse}, recalling that $$\mu_{\eps}(2^k\mathcal B)\approx \frac{(2^k\rrr)^{n+2\eps}}{\eps}$$ by \rif{addiaddi}. The finiteness of the terms in the r.h.s. of \rif{reverse} can be deduced as in \cite[Section 4C]{KMS}; see also the estimation of the terms T$_1$ and T$_3$ in the proof of Theorem \ref{maggiore} below. \end{proof}
\begin{proof}[Proof of Theorem \ref{maggiore}] Following the Gehring type approach introduced in \cite{KMS}, Theorem \ref{maggiore} is now just a consequence of the reverse type inequality on diagonal balls \rif{reverse}; no information on the fact that $w$ solves an equation is needed other than \rif{reverse}.  
 Indeed, using Lemma \ref{selfie} we can now apply the methods from the proof of \cite[Theorem 6.1]{KMS} and determine constants 
 $\eps \in (0,\min\{s/4, 1-s\})$ and $\delta\in (0,1)$, $c\geq 1$, 
all depending on $n, N,s, \Lambda, \delta_{0}, [\partial \Omega]_{0,1}$, such that  
\eqn{iltttt}
$$ t := s+\frac{\eps\delta }{2+\delta} <1$$
and 
 \begin{flalign*}
\nonumber 
\left(\mint_{\mathcal B} U_{\eps}^{2+\delta} \d\mu_{\eps} \right)^{1/(2+\delta)}&\leq  c
\sum_{k\geq 1} 2^{-k(s-\eps)}\left(\mint_{2^k\mathcal B} U_{\eps}^2\d \mu_{\eps}\right)^{1/2} \\  &  \quad   + c \rrr^{\delta_{0}/2+\eps(2/\ti{p}-1)}\left(\mint_{2\mathcal  B} G_{\eps}^{\ti{p}(1+\delta_1)} \d\mu_{\eps} \right)^{\frac{1}{\ti{p}(1+\delta_1)}} \nonumber \\
&\notag  \quad  + \rrr^{\delta_{0}/2+\eps(2/\ti{p}-1)}\sum_{k\geq 1} 2^{-(s-\delta_{0}/2-2\eps/\ti{p})k}\left(\mint_{2^k\mathcal  B} G_{\eps}^{\ti{p}}\d\mu_{\eps} \right)^{1/\ti{p}}\\ & \qquad + c\rrr^{s-\eps}\nr{g}_{L^{\infty}(2B)} \\
&=:\textnormal{T}_{1}+\textnormal{T}_{2}+\textnormal{T}_{3}+  c\rrr^{s-\eps}\nr{g}_{L^{\infty}(2B)} 
\end{flalign*}
holds whenever $B\subset \er^n$ is a ball with radius $\rrr$ and, as usual, $\mathcal B = B\times B$. Now observe that 
$$
 \textnormal{T}_{1} \leq c \rrr^{-n/2-\eps} [w]_{s,2;\er^n}\sum_{k\geq 1} 2^{-k(s+n/2)}\leq c \rrr^{-n/2-\eps}  [w]_{s,2;\er^n}
$$
To estimate the remaining terms, observe that, since $\eps < s/4$
\begin{flalign*}
& \left(\mint_{2^k\mathcal  B} G_{\eps}^{\ti{p}(1+\delta_1)}\d\mu_{\eps} \right)^{\frac{1}{\ti{p}(1+\delta_1)}}\\
 & \quad \leq  c (2^k\rrr)^{-\frac{n+2\eps}{\ti{p}(1+\delta_1)}}  \left(\int_{2^kB} \int_{2^kB} \frac{\snr{v(x)-v(y)}^{\ti{p}(1+\delta_1)}}{\snr{x-y}^{\gamma\ti{p}(1+\delta_1)+2\eps\delta_1}}\frac{\dxy}{\snr{x-y}^n} \right)^{\frac 1{\ti{p}(1+\delta_1)}}\\
& \quad  \leq  c (2^k\rrr)^{-\frac{n}{\ti{p}(1+\delta_1)}-2\eps/\ti{p}+\gamma\delta_1}
[v]_{\gamma(1+\delta_1),\ti{p}(1+\delta_1);2^kB} \\
 &\ \  \stackleq{immersione}  c (2^k\rrr)^{-\frac{n}{\ti{p}(1+\delta_1)}-2\eps/\ti{p}+\gamma\delta_1}[v]_{s+\delta_{0}, 2;\er^n}
\end{flalign*}
holds for every $k \in \en_0$. 
It therefore follows, also using H\"older's inequality, that
\begin{flalign*}
 \textnormal{T}_{2}  + \textnormal{T}_{3}  & \leq  c \rrr^{\delta_{0}/2+\eps(2/\ti{p}-1)} \sum_{k\geq 1} 2^{-(s-\delta_{0}/2-2\eps/\ti{p})k}\left(\mint_{2^k\mathcal  B} G_{\eps}^{\ti{p}(1+\delta_1)}\d\mu_{\eps} \right)^{\frac{1}{\ti{p}(1+\delta_1)}}\\
& \leq c  \rrr^{\delta_{0}/2+\gamma\delta_1-\frac{n}{\ti{p}(1+\delta_1)}-\eps} [v]_{s+\delta_{0}, 2;\er^n} \sum_{k\geq 1} 2^{-\left(s-\delta_{0}/2+\frac{n}{\ti{p}(1+\delta_1)}-\gamma\delta_1\right)k}\\
& \leq c  \rrr^{\delta_{0}/2+\gamma\delta_1-\frac{n}{\ti{p}(1+\delta_1)}-\eps} [v]_{s+\delta_{0}, 2;\er^n}
\end{flalign*}
holds for suitable constant $c$ depending only on  on $n,N, s, \Lambda, \delta_{0}, \eps, [\Omega]_{0, 1}$. Observe that the last series converges due to the restriction we have taken on the size of both $\delta_0$ and $\delta_1$. Recalling that \rif{meaning}$_1$ and \rif{addiaddi} imply
$$
\mint_{\mathcal B} U_{\eps}^{2+\delta}\d\mu_{\eps}= \frac{c\eps}{\rrr^{n+2\eps}}\int_{\mathcal B} U_{\eps}^{2+\delta}\d\mu_{\eps}=\frac{c\eps}{\rrr^{n+2\eps}} [w]_{t, 2+\delta;B}^{2+\delta}, \qquad c\equiv c(n), 
$$
merging the content of the last five displays yields \rif{maggiore2} with $p=2+\delta$  and $t$ as in \rif{iltttt}, and with the prescribed dependence of the constants. 
\end{proof}
\subsection{Higher Sobolev differentiability}
Here we state a localized form of the main result from \cite{KMS}, that is
\begin{theorem}\label{lamaggiore} 
Under assumptions \eqref{bs.1}-\eqref{bs.2}, let $w\in W^{s,2}_{\loc}(B;\er^N)\cap L^{1}_{2s}$ be a solution to  $$ -\mathcal{L}_{a}w= g\in L^{2_*+\delta_1}(B;\er^N)\,, \quad \mbox{in the ball $B\equiv B_{\rrr}\subset \er^n$},\ \rrr\leq 1,\ \delta_1 >0\,,$$  in the sense of Definition \ref{def:weaksol}. There exist $\delta_0 \in (0, 1-s)$, $t \in (s,1)$ and $p>2$, all depending on $n,N,s,\Lambda, \delta_1$, 
such that $w\in W^{s+\delta_0,2}(B/2;\er^N)\cap W^{t,p}(B/2;\er^N)$. Moreover, 
\eqn{stimapp}
$$ \rrr^{s+\delta_{0}} \snra{w}_{s+\delta_{0}, 2;B/2} + \rrr^{t} \snra{w}_{t, p;B/2}\leq c 
\rrr^{s} \snra{w}_{s, 2;B}+c\tail(w;2B)+c\rrr^{2s}\nra{g}_{L^{2_*+\delta_1}(B)}
$$
holds with $c \equiv c (n,N,s,\Lambda, \delta_1)$.
\end{theorem}
The main differences with the original result in \cite{KMS} are essentially three, none being really crucial. The first, already encountered in Section \ref{albordo}, is that we are dealing with vector valued solutions. This type of results are anyway invariant when passing to the vectorial case as they are only based on the use nonlocal Caccioppoli inequalities of the type in \rif{prop:cacc}, for which assumptions \rif{bs.1}-\rif{bs.2} are sufficient. The second is that in \cite{KMS} solutions in the entire $\er^n$ are considered. The third being that solutions are here considered in the class  $w\in W^{s,2}_{\loc}(B;\er^N)\cap L^{1}_{2s}$ while in \cite{KMS} they were supposed to belong to $W^{s,2}(\er^n;\er^N)$. 

The version stated in Theorem \ref{lamaggiore} can be inferred for instance from \cite{byunnon}, that actually covers a much more complex situation (nonuniformly elliptic, nonlocal operators). The result in Theorem \ref{lamaggiore} can be inferred from the proof of \cite[Theorem 1.1]{byunnon} once we take $p=q=2$ and $b(\cdot)\equiv 0$ in \cite{byunnon}. Estimate \rif{stimapp} then follows from \cite[(1.5)]{byunnon} with the same choice of the parameters. Specifically, the starting point in \cite{byunnon} is in fact \cite[Lemma 3.1]{byunnon}, that in the current setting is precisely \rif{perfect22}. Although boundedness is assumed in the statement of \cite[Theorem 1.1]{byunnon}, in the special case 
$p=q=2$ and 
$b(\cdot)\equiv 0$, as pointed out in \cite[Remark 1]{byunnon}, the proof shows that such an assumption is not actually needed. Such main modifications actually occur in \cite[Lemma 3.4]{byunnon} after which the proof proceeds similarly to \cite{KMS} with the use of the $\tail$ to localize terms done for instance in \cite{meng}. Note that related localization arguments leading to statement similar to that of Theorem \ref{lamaggiore} are also present in \cite[Theorem 4.3]{sn}. At this stage, the modified version of \cite[Theorem 1.1]{byunnon} provides us with the estimate 
$$
\rrr^{t} \snra{w}_{t, p;B/2}\leq c 
\rrr^{s} \snra{w}_{s, 2;B}+c\tail(w;2B)+c\rrr^{2s}\nra{g}_{L^{2_*+\delta_1}(B)}
$$
for some $t>s$ and $p>2$ as in the statement Theorem \ref{lamaggiore}. The full conclusion of \rif{stimapp} now follows using H\"older's inequality in the form $ \rrr^{s+\delta_{0}} \snra{w}_{s+\delta_{0}, 2;B/2} \lesssim \rrr^{t} \snra{w}_{t, p;B/2}$, which holds as long as $p>2$ and $s+\delta_0< t$.

\section{\texorpdfstring{$s$-Harmonic approximation}{s-Harmonic approximation}}
In this section we consider constant coefficients systems of the type 
\eqn{costante}
$$
-\mathcal{L}_{a_{0}}w=g
$$
under assumptions \rif{bs.1bis}. Then we show that if a map is weakly close to be a solution to \rif{costante}, then it is actually quantitatively close in norm to a real solution to \rif{costante}. This is in Lemma \ref{shar} below and the proof involves a few preliminary tools developed in Section \ref{truncsec} below. 
\subsection{H\"older truncation}\label{truncsec}
Here we prove a H\"older truncation lemma, that can be seen as a nonlocal counterpart of Lipschitz truncation type lemmas very often used in the literature. See for instance \cite{dms,dsv} and references therein. A key point is the use of maximal operators described in \rif{truncsec0} and in particular of estimate \rif{stimalfa}. We first need a Hardy-Littlewood type result. 
\begin{lemma}\label{l.2}
Let $t\in (0,1)$, $1\leq p<\infty$, $\rrr>0$ and $w\in W^{t,p}(\mathbb{R}^{n};\er^N)$. Then 
\eqn{ht.16}
$$
\snr{\{x\in \mathbb{R}^{n}\colon \textnormal{M}^{\#,t}_{\rrr}(w;x)>\lambda\}}\lesssim_{n,t,p} \frac{[w]_{t,p;\mathbb{R}^{n}}^{p}}{\lambda^{p}}
$$
holds for every $\lambda >0$. 
\end{lemma}
\begin{proof}
We shall prove that 
\eqn{ht.16M}
$$
\snr{\{x\in \ttB_{\kappa}\colon \textnormal{M}^{\#,t}_{\rrr}(w;x)>\lambda\}}\leq \frac{c[w]_{t,p;\mathbb{R}^{n}}^{p}}{\lambda^{p}}
$$
holds whenever $\kappa>0$, for a constant $c\equiv c (n,t,p)$ being independent of $\kappa$. At this point \rif{ht.16} will follow letting $\kappa\to \infty$ in \rif{ht.16M}. Set $A_{t, \kappa}(w;\lambda):=\{x\in \ttB_{\kappa}\colon \textnormal{M}^{\#,t}_{\rrr}(w;x)>\lambda\}$. Jensen's inequality  implies
\begin{flalign}\label{ht.1}
A_{t,\kappa}(w;\lambda)\subset \left\{x\in \ttB_{\kappa}\colon \sup_{0<\rr\leq \rrr}\rr^{-t p}\mint_{B_{\rr}(x)}\snr{w(y)-(w)_{B_{\rr}(x)}}^{p}\dy>\lambda^{p}\right\}=:A_{t,\kappa,p}(w;\lambda)\,.
\end{flalign}
For every $x\in A_{t,\kappa,p}(w;\lambda)$, there exists $\rr_{x}\in (0,\rrr]$ such that
\begin{flalign}\label{ht.2}
\frac{1}{\lambda^{p}\rr_{x}^{pt}}\int_{B_{\rr_{x}}(x)}\snr{w(y)-(w)_{B_{\rr_{x}}(x)}}^{p}\dy>\snr{B_{\rr_{x}}(x)}.
\end{flalign}
Applying Besicovitch covering lemma to the family $\{B_{\rr_x}(x)\}$, we extract a finite number $\mathfrak{n}$, depending only on $n$, of countable families of mutually disjoint 
 balls $\{B_{i, j}\}_{j\geq 1}$, $i \in \{1, \ldots, \mathfrak n\}$, such that 
 \eqn{pippi}
 $$
  A_{t,\kappa,p}(w;\lambda)\subset \bigcup_{i\leq \mathfrak{n}}\bigcup_{j\geq 1}\overline {B_{i, j}}\,.
 $$
We then conclude with \rif{ht.16M} as follows:
\begin{eqnarray*}
\snr{A_{t,\kappa}(w;\lambda)}&\stackrel{\eqref{ht.1}}{\le}&\snr{ A_{t,\kappa,p}(w;\lambda)}\\ &\stackrel{\eqref{ht.2},\eqref{pippi}}{\leq} &\frac{c}{\lambda^{p}}\sum_{i\leq \mathfrak n}\sum_{j\geq 1}\snr{B_{i, j}}^{-pt/n}\int_{B_{i, j}}\snr{w(x)-(w)_{B_{i,j}}}^{p}\dx\nonumber \\
&\stackleq{poinfrac}&\frac{c}{\lambda^{p}}\sum_{i\leq \mathfrak n}\sum_{j\geq 1}[w]_{t,p;B_{i,j}}^{p}\\ &\leq & \frac{\ti{c}}{\lambda^{p}}\sum_{i\leq \mathfrak n}[w]_{t,p;\er^n}^p\\ &=& \frac{\mathfrak{n}\ti{c}[w]_{t,p;\mathbb{R}^{n}}^{p}}{\lambda^{p}}\\\ &=&\frac{c[w]_{t,p;\mathbb{R}^{n}}^{p}}{\lambda^{p}}
\end{eqnarray*}
with $c\equiv c(n,t,p)$. Note that in the previous display we have used the following general property: 
$$
\sum_{j\geq 1}\,  [w]_{t,p;\mathcal A_j}^p \leq   [w]_{t,p; \cup_j \mathcal A_j}^p
$$
holds whenever $\{\mathcal A_i\}$ is a countable family of mutually disjoint subsets.
 \end{proof}
\begin{proposition}\label{htrunc}
Let $\Omega\subset \mathbb{R}^{n}$ be an open, bounded Lipschitz domain, $\ti{\Omega}$ be an open, bounded set such that $\Omega \Subset \ti{\Omega}$, and let $w\in \mathbb{X}_{0}^{t,p}(\Omega,\ti{\Omega})$ with $p\geq 1$ and $t\in (0,1)$. For every $\lambda>0$ there exists $w_{\lambda}\in C^{0,t}(\er^n;\er^N)$ such that 
\eqn{ht.18}
$$
\begin{cases}
\, \left\{x\in \mathbb{R}^{n}\colon w_{\lambda}(x)\not =w(x)\right\} \subset \Omega\\[5pt]
\, \displaystyle [w_{\lambda}]_{0,t;\er^n}\le c\lambda\\ 
 \displaystyle
\, \snr{\{x\in \mathbb{R}^{n}\colon w_{\lambda}(x)\not =w(x)\}}\le \frac{c}{\lambda^{p}}\left([w]_{t,p;\ti{\Omega}}^{p}+\frac{\nr{w}_{L^{p}(\Omega)}^{p}}{\dist(\Omega,\er^n\setminus \ti{\Omega})^{pt}}\right)\\[9pt]
\,w_{\lambda} \in \mathbb{X}_{0}^{\ti{t},p}(\Omega,\ti{\Omega})\,, \quad  \mbox{for every positive  $\ti{t} <t$}
\end{cases}
$$
where $c\equiv c(n,N,t,p,\Omega)$.
\end{proposition} 
\begin{proof} By Remark \ref{remarkino} it follows that $w \in W^{t,p}(\mathbb{R}^{n};\er^N)$ and \rif{ht.10} holds, so that, in particular, $w\equiv 0$ a.e. in $\er^n\setminus \Omega$. Fix $\rrr>0$ and consider $x \in \er^n$ such that $\textnormal{M}^{\#,t}_{\rrr}(w;x)$ is finite. Recalling the meaning of \rif{preciso} there exists the precise representative $w(x)$. Setting $\sigma_{i}:=\sigma/2^{i}$, $i \in \en_0$, with $\sigma\in (0, \rrr]$, we control, for $k\geq 1$
\begin{flalign*}
\snr{(w)_{B_{\sigma_{k}}(x)}-(w)_{B_{\sigma}(x)}}&\le \sum_{i=0}^{k-1}\snr{(w)_{B_{\sigma_{i+1}}(x)}-(w)_{B_{\sigma_{i}}(x)}}\nonumber \\
& \leq 2^n \sum_{i=0}^{k-1}\mint_{B_{\sigma_{i}}(x)}\snr{w-(w)_{B_{\sigma_i}(x)}}\dy\nonumber \\
& = 2^n \sigma^{t}\sum_{i=0}^{k-1}2^{-t i}\sigma_{i}^{-t}\mint_{B_{\sigma_{i}}(x)}\snr{w-(w)_{B_{\sigma_{i}}(x)}}\dy\nonumber \\
&\leq 2^n  \sum_{i=0}^{\infty}2^{-t i} \, \sigma^{t}\textnormal{M}^{\#,t}_{\rrr}(w;x)\\ & \approx_{n,t}   \sigma^{t}\textnormal{M}^{\#,t}_{\rrr}(w;x)\,.
\end{flalign*}
Letting $k\to \infty$ in the above display yields
\eqn{ht.6}
$$
\snr{w(x)-(w)_{B_{\sigma}(x)}}\lesssim_{n, t} \sigma^{t}\textnormal{M}^{\#,t}_{\rrr}(w;x) \, , \qquad \mbox{for every $  \sigma \in (0,\rrr)$}\,.
$$
From now on we shall follow and modify some arguments from \cite{dms}. We take  $\rrr:=8\diam (\Omega)$ and define
$$
H_{\lambda}:= \left\{x\in \mathbb{R}^{n}\colon \textnormal{M}^{\#,t}_{\rrr}(w;x)\leq \lambda\right\}
$$
so that we give meaning to $w(x)$ as the precise representative of $w$ at $x$ whenever $x\in H_{\lambda}$.  
 Now, let $x\in H_{\lambda}\cap \Omega$ and set $\sigma:=2\dist(x,\mathbb{R}^{n}\setminus \Omega)< \rrr$. Keeping in mind that $\Omega$ is bounded and $\partial\Omega$ is Lipschitz regular, then
\eqn{ht.17}
$$
\frac{\snr{B_{\sigma}(x)}}{\snr{B_{\sigma}(x)\cap (\mathbb{R}^{n}\setminus \Omega)}}\le c_{\Omega},
$$
where $c_{\Omega}$ depends on $\Omega$, we have (recall that $w\equiv 0$ in $\mathbb{R}^{n}\setminus \Omega$)
\begin{eqnarray} 
\notag \mint_{B_{\sigma}(x)}\snr{w(y)-(w)_{B_{\sigma}(x)}}\dy&\ge &\frac{1}{\snr{B_{\sigma}(x)}}\int_{(\mathbb{R}^{n}\setminus \Omega)\cap B_{\sigma}(x)}\snr{w(y)-(w)_{B_{\sigma}(x)}}\dy\nonumber \\
\notag &= &\frac{\snr{(\mathbb{R}^{n}\setminus \Omega)\cap B_{\sigma}(x)}}{\snr{B_{\sigma}(x)}}\snr{(w)_{B_{\sigma}(x)}}\\
&\stackrel{\rif{ht.17}}{\ge}  &\frac{1}{c_{\Omega}}\snr{(w)_{B_{\sigma}(x)}}\,.\label{mediona0}
\end{eqnarray}
It follows that
$
\snr{(w)_{B_{\sigma}(x)}}\leq c\sigma^{t}\textnormal{M}^{\#,t}_{\rrr}(w;x)
$
for $c\equiv c(n,\Omega)$, and, from \rif{ht.6} and triangle inequality, that
\eqn{mediona}
$$
\snr{w(x)}\leq c\sigma^{t}\textnormal{M}^{\#,t}_{\rrr}(w;x)\le c\sigma^{t}\lambda\leq  c\dist(x,\mathbb{R}^{n}\setminus \Omega)^t  \lambda \qquad \mbox{for all} \ x\in H_{\lambda}\cap \Omega
$$
where $c\equiv c(n,t,\Omega)$. When $x \in H_\lambda\cap \partial \Omega$  the argument in \rif{mediona0} works for any $\sigma \in (0,\rrr)$, thereby yielding $\snr{w(x)}\le c\sigma^{t}\lambda$ for every $\sigma>0$, so that we conclude with $w(x)=0$. This allows to pointwise (re)define without ambiguities $w(x):=0$ whenever $x\in \er^n\setminus \Omega$ (note that if $x \in \er^n\setminus \bar{\Omega}$, it automatically follows that the precise representative of $w$ at $x$ is $0$ independently of the way we redefine $w$ up to  a negligible set). We have therefore pointwise defined $w$ in $H_{\lambda}\cup (\mathbb{R}^{n}\setminus \Omega)$. We now want to prove that 
\eqn{ht.5}
$$
x_{1},x_{2}\in H_{\lambda}  \ \Longrightarrow \  \snr{w(x_{1})-w(x_{2})} \le c\lambda\snr{x_{1}-x_{2}}^{t}
$$
holds with $c\equiv c(n,t,\Omega)$. To proceed, we assume that $x_1\neq x_2$ and we distinguish three different cases. The first case is when $x_{1},x_{2}\in H_{\lambda}\setminus \Omega$; this is trivial as $w\equiv 0$ in $\mathbb{R}^{n}\setminus \Omega$. The second case is when $x_{1},x_{2}\in H_{\lambda}\cap \Omega$, so that $\snr{x_1-x_2}< \rrr/4=2\diam (\Omega)$. In this situation \rif{stimalfa} gives 
$$
\snr{w(x_{1})-w(x_{2})}\le c\left(\textnormal{M}^{\#,t}_{\rrr}(w;x_{1})+\textnormal{M}^{\#,t}_{\rrr}(w;x_{2})\right)\snr{x_{1}-x_{2}}^{t}\le c\lambda\snr{x_{1}-x_{2}}^{t}
$$
and \eqref{ht.5} follows. It remains to examine the third case, i.e., when either  $x_{1}\in H_{\lambda}\cap \Omega$, $x_{2}\in H_{\lambda}\setminus \Omega$ or $x_{2}\in H_{\lambda}\cap \Omega$, $x_{1}\in H_{\lambda}\setminus \Omega$ happen. By symmetry we confine ourselves to $x_{1}\in H_{\lambda}\cap \Omega$, $x_{2}\in H_{\lambda}\setminus \Omega$. We then have, using that $w(x_2)=0$ and \rif{mediona}, 
$$
\snr{w(x_{1})-w(x_{2})}=\snr{w(x_{1})}\le c\dist(x_1,\mathbb{R}^{n}\setminus \Omega)^t  \lambda \le c\lambda\snr{x_{1}-x_{2}}^{t}.
$$
This completes the proof of \rif{ht.5}. Observing that the argument of the third case works whenever $x_{1}\in H_{\lambda}\cap \Omega$, $x_2 \in \er^n\setminus \Omega$, we have in fact proved that 
$$
\snr{w(x_{1})-w(x_{2})}\le c\lambda\snr{x_{1}-x_{2}}^{t}\qquad \mbox{for all} \ \ x_{1},x_{2}\in \Sigma_{\lambda}:=H_{\lambda}\cup (\mathbb{R}^{n}\setminus \Omega),
$$
and  $c\equiv c(n,t,p,\Omega)$. Also recalling \rif{mediona}, by  \cite[Theorem 15.1]{emmanuele} (applied to each component of $w$) we find a globally $t$-H\"older continuous map $w_{\lambda}\colon \er^n\to \er^N$ such that 
$$
\begin{cases}
\, w_{\lambda}\equiv w \ \ \mbox{in} \ \ \Sigma_{\lambda}\ \Longrightarrow \ \left\{x\colon w_{\lambda}(x)\not =w(x)\right\}\subset  \Omega \\[3pt]
\, \nr{w_{\lambda}}_{L^{\infty}(\er^n)}\leq c \rrr^t \lambda \\[3pt] \, [w_{\lambda}]_{0,t;\er^n} \leq c\lambda
\end{cases}
$$
for a new constant $c\equiv c(n,N,t,\Omega)$, so that \rif{ht.18}$_{1,2}$ follow. As for \rif{ht.18}$_{3}$, note that
$$\{x\in \mathbb{R}^{n}\colon w_{\lambda}(x)\not =w(x)\}\subset \Omega\cap \left\{x\in \mathbb{R}^{n}\colon \textnormal{M}^{\#,t}_{\rrr}(w;x)>\lambda\right\}$$
and therefore 
$$
 \left|\left\{x\in \mathbb{R}^{n}\colon w_{\lambda}(x)\not =w(x) \right\}\right|\leq  \left|\left\{x\in \mathbb{R}^{n}\colon \textnormal{M}^{\#,t}_{\rrr}(w;x)>\lambda\right\}\right|\stackrel{\eqref{ht.16}}{\le}\frac{c[w]_{t,p;\mathbb{R}^{n}}^{p}}{\lambda^{p}}
$$
with $c\equiv c(n,N,t,p,\Omega)$, so that \rif{ht.18}$_3$ follows using \rif{ht.10} to estimate the right-hand side in the above display. Obviously, for every positive  $\ti{t}<t$, we have $w_{\lambda}\in W^{\ti{t},p}(\ti{\Omega};\er^N)$ and $w_{\lambda}\equiv 0$ outside $\Omega$ and therefore $w_{\lambda}\in \mathbb{X}^{\ti{t},p}_{0}(\Omega,\ti{\Omega})$. 
\end{proof}
\subsection{$s$-Harmonic Approximation Lemma} This is
\begin{lemma}\label{shar}
Let $a_{0}\in \mathbb{R}^{N\times N}$ be such that \eqref{bs.1bis} holds and let $\delta_0$ be a  number such that $0 <s <s+\delta_{0}<1$; let $B\Subset \ti{B}\Subset \mathbb{R}^{n}$ be concentric balls with radii belonging to  $(1/32,1)$ and such that $\dist(B,\partial \ti{B}):=\textnormal{\texttt{d}}>0$. Let $v\in W^{s,2}(\er^n;\er^N)$, $g\in L^{\infty}(\mathbb{R}^{n};\er^N)$ be such that 
\eqn{sh.0}
$$
\nr{v}_{W^{s,2}(\mathbb{R}^{n})}+\nr{v}_{W^{s+\delta_{0},2}(\er^n)}+ \tail(v;\ti{B}) + \nr{g}_{L^{\infty}(\mathbb{R}^{n})}\le c_{0}
$$
for some $c_{0}>0$. There exist numbers $t \in (s,1)$, $p>2$, both depending only on $n,N,s,\Lambda, \delta_{0}$, and a constant $\cchh\equiv \cchh(n,N,s,\Lambda,c_{0},\delta_{0}, \textnormal{\texttt{d}})\geq 1$ such that, if for $\eps\in (0,1)$
\eqn{sh.1}
$$
\left|\int_{\mathbb{R}^{n}}\int_{\mathbb{R}^{n}}\langle a_{0}(v(x)-v(y)),\varphi(x)-\varphi(y)\rangle\frac{\dxy}{\snr{x-y}^{n+2s}}-\int_{B}\langle g,\varphi\rangle\dx\right|\le\left(\frac{\eps}{\cchh}\right)^{\frac{2p}{p-2}} [\varphi]_{0,t;\er^n}
$$
holds for all $\varphi\in C^{0,t}(\er^n;\er^N)$ vanishing outside $B$, and if $h\in\mathbb{X}^{s,2}_{v}(B,\ti{B})$ solves
\eqn{sh.2}
$$
\begin{cases}
\ -\mathcal{L}_{a_{0}}h=g\quad &\mbox{in} \ \ B\\
\ h=v\quad &\mbox{in} \ \ \mathbb{R}^{n}\setminus B\\
\end{cases}
$$
in the sense of \eqref{eqweaksol22}, then 
\eqn{sh.5}
$$
\nr{h}_{W^{s,2}(\ti{B})}
+\tail(h-(h)_{B};B)\le c\equiv c(n,N,s,\Lambda, c_{0})
$$
holds together with
\eqn{sh.5.1}
$$
\begin{cases}
[h-v]_{s,2;\ti{B}}<\varepsilon\\[6pt]
 \nr{h-v}_{L^2(B)}\leq \frac{c(n,s)}{\sqrt{\textnormal{\texttt{d}}}}\eps\,.
 \end{cases}
$$
\end{lemma}
\begin{proof} In the rest of the proof we assume that \rif{sh.1} is verified with a certain constant $\cchh\geq 1$; the value of $\cchh$ will be then determined at the very end of the proof, as a function of the parameters $n,N,s,\Lambda,c_{0},\delta_{0},\textnormal{\texttt{d}}$. 
Observe that $v\in W^{s,2}(\er^n;\er^N)$ implies $v \in L^1_{2s}$ and therefore Lemma \ref{esiste} ensures the existence of a unique weak solution $h\in \mathbb{X}^{s,2}_{v}(B,\ti{B})$ to \rif{sh.2} in the sense of \eqref{eqweaksol22}. By Remark \ref{remarkino} it follows that  $h \in W^{s,2}(\er^n;\er^N)$. Set $w:=h-v\in \mathbb{X}^{s,2}_{0}(B,\ti{B})\cap W^{s,2}(\er^n;\er^N)$, and note that 
\begin{flalign}\label{sh.3}
\begin{cases}
\ -\mathcal{L}_{a_{0}}w=-\mathcal{L}_{-a_{0}}v+g\quad &\mbox{in} \ \ B\\
\ w=0\quad &\mbox{in} \ \ \mathbb{R}^{n}\setminus B
\end{cases}
\end{flalign}
in the sense of Definition \ref{defidir} (keep in mind Remark \ref{remarkino}). We next apply  Theorem \ref{maggiore} to \rif{sh.3}, that yields $t \in (s,1)$, $p >2$, both depending only on $n, N,s, \Lambda, \delta_{0}$, such that \rif{maggiore2} holds. This determines the numbers $t,p$ from the statement of Lemma \ref{shar}. Note that $w$ is an admissible test function in \eqref{sh.3} so that using Young inequality and \eqref{poincf} gives 
\eqn{sh.4}
$$
[w]_{s,2;\er^n}\le c [v]_{s,2;\mathbb{R}^{n}}+c\nr{g}_{L^{\infty}(2B)} \ \stackrel{\eqref{sh.0}}{\Longrightarrow} \ \nr{h}_{W^{s,2}(\mathbb{R}^{n})}+\nr{w}_{W^{s,2}(\mathbb{R}^{n})}\le c\,,
$$
with $c\equiv c(n,N,\Lambda,c_{0})$. For the $\tail$ term  we have
\begin{eqnarray*}
\tail(h-(h)_{B};B)&\stackrel{\eqref{sh.2}}{=}&\tail(v-(h)_{B};B)\le c\,  \tail(v-(v)_{B};B)+c\nr{v-h}_{L^{2}(B)}\nonumber \\
&\stackrel{\eqref{poincf}, \eqref{scatailggg}}{\leq} &c\, \tail(v;\ti{B})+c\nr{v}_{L^{2}(\ti{B})}+c[w]_{s,2; \er^n}\\
& \stackrel{\eqref{sh.0},\eqref{sh.4}}{\le}& c
\end{eqnarray*}
with $c\equiv c(n,N,s,\Lambda,c_{0})$.  This completes the proof of \rif{sh.5}.  This, together with \rif{maggiore2}, \rif {sh.0} and \eqref{sh.4}, gives 
\eqn{sh.10pre}
$$
[w]_{t,p;\ti{B}} \leq c (n,N,s,\Lambda, c_0, \delta_{0})\,.
$$
Next, observe that
$$
\nr{w}_{L^p(\ti{B})} \leq c  \nr{w-(w)_{\tilde {B}}}_{L^p(\ti{B})}+c \nr{w}_{L^1(\ti{B})} \stackrel{\eqref{poinfrac},\eqref{sh.4}}{\leq}
c[w]_{t,p;\ti{B}} +  c \stackleq{sh.10pre} c 
$$
and we therefore conclude with
\eqn{sh.10}
$$
\nr{w}_{W^{t,p}(\ti{B})} \leq c\,,
$$
again with $c \equiv c (n,N,s,\Lambda, c_0, \delta_{0})$. To proceed, we fix $\lambda \geq 1$ and apply Proposition \ref{htrunc} to $w$, thereby obtaining 
$w_{\lambda} \in \mathbb{X}^{s,2}_{0}(B,\ti{B})\cap C^{0,t}(\er^n;\er^N)$ satisfying 
\rif {ht.18} with the exponents $p>2$, $t>s$ we just found, and $\Omega\equiv B$, $\ti{\Omega} \equiv \ti{B}$, and such that
\eqn{livelli}
$$
\begin{cases}
\,\snr{\mathbb{R}^{n}\setminus \Sigma} \leq c\lambda^{-p}\,, \quad \mathbb{R}^{n}\setminus \Sigma\subset B\\[3pt]
\,\mbox{$w_{\lambda}\equiv w\equiv 0$ outside $B$}\,, \quad  \nr{w_{\lambda}}_{L^\infty(\er^n)} +[w_{\lambda}]_{0,t;\er^n} \le c\lambda\\[3pt]
\, \Sigma:=\left\{x\in \mathbb{R}^{n}\colon w_{\lambda}(x)=w(x)\right\}
\end{cases}
$$
for $c \equiv c (n,N,s,\Lambda, c_0, \delta_{0},\textnormal{\texttt{d}})$, and we have used also \rif{sh.10}. Note that $w_{\lambda} \in \mathbb{X}^{s,2}_{0}(B,\ti{B})$ follows easily by the fact that $w_{\lambda} \in \mathbb{X}^{\tilde t,p}_{0}(B,\ti{B})$ for every $\tilde t < t$ by \rif{ht.18}$_4$ and the fact that $w_{\lambda}$ vanishes outside $B$. Testing \rif{sh.3} by $w_{\lambda}$ yields
\begin{flalign*}
&\int_{\mathbb{R}^{n}}\int_{\mathbb{R}^{n}}\langle a_{0}(w(x)-w(y)),w_{\lambda}(x)-w_{\lambda}(y)\rangle\frac{\dxy}{\snr{x-y}^{n+2s}}\nonumber \\
&\qquad  =-\int_{\mathbb{R}^{n}}\int_{\mathbb{R}^{n}}\langle a_{0}(v(x)-v(y)),w_{\lambda}(x)-w_{\lambda}(y)\rangle\frac{\dxy}{\snr{x-y}^{n+2s}}+\int_{B}\langle g,w_{\lambda}\rangle\dx,
\end{flalign*}
so that 
we can rewrite
\begin{flalign*}
\mbox{(I)}&:=\int_{\Sigma}\int_{\Sigma}\langle a_{0}(w(x)-w(y)),w(x)-w(y)\rangle\frac{\dxy}{\snr{x-y}^{n+2s}}\nonumber \\
&=\int_{\Sigma}\int_{\Sigma}\langle a_{0}(w(x)-w(y)),w_{\lambda}(x)-w_{\lambda}(y)\rangle\frac{\dxy}{\snr{x-y}^{n+2s}}\nonumber \\
&=-\left(\int_{\mathbb{R}^{n}}\int_{\mathbb{R}^{n}}\langle a_{0}(v(x)-v(y)),w_{\lambda}(x)-w_{\lambda}(y)\rangle\frac{\dxy}{\snr{x-y}^{n+2s}}-\int_{B}\langle g,w_{\lambda}\rangle\dx\right)\nonumber \\
&\qquad -2\int_{\Sigma}\int_{\mathbb{R}^{n}\setminus \Sigma}\langle a_{0}(w(x)-w(y)),w_{\lambda}(x)-w_{\lambda}(y)\rangle\frac{\dxy}{\snr{x-y}^{n+2s}}\nonumber \\
&\qquad -\int_{\mathbb{R}^{n}\setminus\Sigma}\int_{\mathbb{R}^{n}\setminus \Sigma}\langle a_{0}(w(x)-w(y)),w_{\lambda}(x)-w_{\lambda}(y)\rangle\frac{\dxy}{\snr{x-y}^{n+2s}}\nonumber \\&=: \mbox{(II)}+\mbox{(III)}+\mbox{(IV)}\,.
\end{flalign*}
We then estimate 
$$
\snr{\mbox{(II)}}\stackrel{\eqref{sh.1}}{\le}\left(\frac{\eps}{\cchh}\right)^{\frac{2p}{p-2}} [w_{\lambda}]_{0,t;\er^n}\stackleq{livelli} c\left(\frac{\eps}{\cchh}\right)^{\frac{2p}{p-2}}\lambda 
$$
for $c \equiv c (n,N,s,\Lambda, c_0, \delta_{0},\textnormal{\texttt{d}})$. In order to treat the remaining terms, using also H\"older's inequality, we note that
\begin{eqnarray}
\notag \int_{ \ti{B}}\int_{\mathbb{R}^{n}\setminus \Sigma}\frac{\snr{w(x)-w(y)}}{\snr{x-y}^{n+2s-t}}\dxy
 & \leq & c [w]_{t,p;\ti{B}}  \left(\int_{\mathbb{R}^{n}\setminus \Sigma}\int_{\ti{B}}\snr{x-y}^{-n+\frac{2p(t-s)}{p-1}}\dyx\right)^{1-1/p}\\ \notag 
 & \leq & c [w]_{t,p;\ti{B}}  \left(\int_{\mathbb{R}^{n}\setminus \Sigma}\int_{\ttB_{2}}\snr{z}^{-n+\frac{2p(t-s)}{p-1}}\dzx\right)^{1-1/p}\\ \notag 
  & = & c [w]_{t,p;\ti{B}} \snr{\mathbb{R}^{n}\setminus \Sigma}^{1-1/p} \left(\int_{\ttB_{2}}\snr{z}^{-n+\frac{2p(t-s)}{p-1}}\dz\right)^{1-1/p}\\
  & = & \frac{c}{(t-s)^{1-1/p}} [w]_{t,p;\ti{B}} \snr{\mathbb{R}^{n}\setminus \Sigma}^{1-1/p}\notag \\
  & \stackleq{sh.10} & c \snr{\mathbb{R}^{n}\setminus \Sigma}^{1-1/p} \notag \\ &\stackleq{livelli}  &c\lambda^{1-p}
  \label{elemy}
\end{eqnarray}
with $c \equiv c (n,N,s,\Lambda, c_0, \delta_{0},\textnormal{\texttt{d}})$.
Then, with $x_B$ denoting the centre of both $B$ and $\ti{B}$, and recalling that $\mathbb{R}^{n}\setminus \Sigma\subset B$, we have
\begin{eqnarray*}
\snr{\mbox{(III)}}&\stackrel{\eqref{bs.1bis}}{\le}&c \int_{\Sigma\cap \ti{B}}\int_{\mathbb{R}^{n}\setminus \Sigma}\snr{w(x)-w(y)}\snr{w_{\lambda}(x)-w_{\lambda}(y)}\frac{\dxy}{\snr{x-y}^{n+2s}}\nonumber \\
&&\quad +c\int_{\Sigma\setminus \ti{B}}\int_{\mathbb{R}^{n}\setminus \Sigma} \snr{w(x)-w(y)}\snr{w_{\lambda}(x)-w_{\lambda}(y)}\frac{\dxy}{\snr{x-y}^{n+2s}}\nonumber \\
&\stackrel{\eqref{uset}}{\le}&c[w_{\lambda}]_{0,t;\er^n}\int_{\Sigma\cap \ti{B}}\int_{\mathbb{R}^{n}\setminus \Sigma}\frac{\snr{w(x)-w(y)}}{\snr{x-y}^{n+2s-t}}\dxy\nonumber \\
&&\quad +c\nr{w_\lambda}_{L^\infty(\er^n)}\int_{\mathbb{R}^{n}\setminus \Sigma}\snr{w}\dx  \int_{\er^n\setminus \ti{B}}\frac{\dy}{\snr{y-x_{B}}^{n+2s}} \nonumber \\
&\stackrel{\eqref{livelli}, \eqref{elemy}}{\le}&c  \lambda^{2-p} +c\nr{w}_{L^{2}(B)}\lambda\snr{\mathbb{R}^{n}\setminus \Sigma}^{1/2}\nonumber\\
&\stackrel{\eqref{sh.10}, \eqref{livelli}}{\leq}& c\lambda^{2-p}+c \lambda^{1-p/2}\\ & \leq  & c \lambda^{1-p/2}
\end{eqnarray*}
where  it is $c \equiv c (n,N,s,\Lambda, c_0, \delta_{0},\textnormal{\texttt{d}})$. 
Similarly,
\begin{eqnarray*}
\snr{\mbox{(IV)}}&\stackrel{\eqref{bs.1bis}}{\le}&c\int_{\mathbb{R}^{n}\setminus\Sigma}\int_{\mathbb{R}^{n}\setminus \Sigma}\snr{w(x)-w(y)}\snr{w_{\lambda}(x)-w_{\lambda}(y)}\frac{\dxy}{\snr{x-y}^{n+2s}}\\
& \stackrel{\eqref{livelli}}{\le} &c\lambda\int_{\ti{B}}\int_{\mathbb{R}^{n}\setminus \Sigma}\frac{\snr{w(x)-w(y)}}{\snr{x-y}^{n+2s-t}}\dxy\nonumber \\&\stackrel{\eqref{elemy}}{\le}  &c   \lambda^{2-p}\\ & \leq & c \lambda^{1-p/2}
\end{eqnarray*}
with $c \equiv c (n,N,s,\Lambda, c_0, \delta_{0},\textnormal{\texttt{d}})$. Merging the content of the previous displays, we obtain
\eqn{sh.9}
$$
\mbox{(I)} \leq c \left(\frac{\eps}{\cchh}\right)^{\frac{2p}{p-2}}\lambda+ c\lambda^{1-p/2}
$$
again with $c \equiv c (n,N,s,\Lambda, c_0, \delta_{0},\textnormal{\texttt{d}})$. 
Estimating as for \rif{elemy}, we have
\begin{eqnarray*}
[w]_{s,2;\ti{B}}^2
&=&\int_{\ti{B}\cap \Sigma}\int_{\ti{B}\cap \Sigma}\frac{\snr{w(x)-w(y)}^{2}}{\snr{x-y}^{n+2s}}\dxy+2\int_{\ti{B}\setminus \Sigma}\int_{\ti{B}\cap \Sigma}\frac{\snr{w(x)-w(y)}^{2}}{\snr{x-y}^{n+2s}}\dxy\nonumber \\
&&+\int_{\ti{B}\setminus \Sigma}\int_{\ti{B}\setminus \Sigma}\frac{\snr{w(x)-w(y)}^{2}}{\snr{x-y}^{n+2s}}\dxy\nonumber \\
&\leq &c\mbox{(I)} +3\int_{\ti{B}}\int_{\er^n\setminus \Sigma}\frac{\snr{w(x)-w(y)}^{2}}{\snr{x-y}^{n+2s}}\dxy\nonumber \\
&\leq &c\mbox{(I)} +3[w]_{t,p;\ti{B}}^{2}\left(\int_{\er^n\setminus \Sigma}\int_{\ti{B}}\snr{x-y}^{-n+\frac{2p(t-s)}{p-2}}\dxy\right)^{1-2/p}\nonumber \\
&\stackrel{\eqref{sh.10}}{\le}&c\mbox{(I)} +c\snr{\mathbb{R}^{n}\setminus \Sigma}^{1-2/p}\left(\int_{\ttB_{2}}\snr{z}^{-n+\frac{2p(t-s)}{p-2}}\dz\right)^{1-2/p}\\
&\stackleq{livelli} & c\mbox{(I)} +c \lambda^{2-p}\\ &\leq & c \mbox{(I)} +c \lambda^{1-p/2}\end{eqnarray*}
for $c\equiv c(n,N,s,\Lambda,c_{0},\delta_{0},\textnormal{\texttt{d}})$. Using \rif{sh.9} in the last inequality we conclude with 
$$
[h-v]_{s,2;\ti{B}}^2=[w]_{s,2;\ti{B}}^2 \leq c_*\left(\frac{\eps}{\cchh}\right)^{\frac{2p}{p-2}}\lambda+ c_*\lambda^{1-p/2}
$$
for $c_*\equiv c_*(n,N,s,\Lambda,c_{0},\delta_{0},\textnormal{\texttt{d}})\geq 1$. Taking $\lambda := (\eps/\cchh)^{4/(2-p)}$ and $\cchh := 2\sqrt{c_*}$ finally yields inequality in \rif{sh.5.1}$_1$, while \rif{sh.5.1}$_2$ follows from the \rif{sh.5.1}$_1$ and \cite[(2.2)]{brascoparini}. 
\end{proof}
\begin{remark}\label{semplice}{\em By a standard scaling argument in \rif{sh.1} we can always assume that $[\varphi]_{0,t;\er^n}\leq 1$.}
\end{remark}

 \section{Linearization and excess decay}\label{sec6}
For the rest of Section \ref{sec6} we fix an arbitrary ball $B_{\rr}(x_{0})\Subset \Omega$. Unless otherwise stated, all the other balls will be centred at $x_{0}$ but those denoted by $\ttB_{\sigma}$ that, according to the notation in \rif{zerocenter}, are instead centred at the origin. We shall analyze the occurrence of two different situations. The first is the {\em small potential regime}, corresponding to
\eqn{exx}
$$
\tx{E}_{u}(x_{0},\rr)<\varepsilon_{*}\qquad \mbox{and}\qquad \pprhoo<\varepsilon^{*}\tx{E}_{u}(x_{0},\rr)\,,
$$
the second being the {\em large potential regime}, i.e., when
\eqn{lp}
$$
\pprhoo\ge\varepsilon^{*}\tx{E}_{u}(x_{0},\rr)
$$
holds. 
The positive constants $\varepsilon_{*},\varepsilon^{*}\in (0,1)$ will be determined in a few lines with a suitable dependence on a certain choice of relevant parameters. Next to the fixed ball $B_{\rr}(x_0)$ we shall consider the enlarged, concentric ball $B_{R\rr}(x_{0})$, where $R\geq 1$ is an absolute constant whose exact value will be determined in due course of the proof, essentially as a function of $\data$. We will assume $B_{R\rr}(x_{0})\Subset \Omega$ and that  $a(\cdot)$ is $\delta$-vanishing in $B_{R\rr}(x_{0})$, in the sense of Definition \ref{deltadef}, again for some $\delta >0$ to be fixed in  the course of the proof as a function of $\data$. 
\subsection{Excess decay in the two regimes}\label{sec.sp} We start with 
\begin{proposition}[Excess decay in the small potential regime]\label{psp}
Under assumptions \eqref{bs.1}-\eqref{bs.5}, let $u$ be a weak solution to \eqref{nonlocaleqn} and $\gamma_{0}$ such that
\eqn{gammazero}
$$
0 < \gamma_{0} < \min\{2s,1\}\,.
$$
There exist threshold parameters
\eqn{sceltona}
$$
\begin{cases}
\,\varepsilon_{*}\equiv \varepsilon_{*}(\data,\omega(\cdot),\gamma_{0})\in (0,1)\\[4pt]\, \varepsilon^{*}\equiv \varepsilon^{*}(\data,\gamma_{0})\in (0,1)\\[4pt]
\,R \equiv R(\data,\gamma_{0}) \geq 1\\[4pt]
\, 
\delta \equiv \delta(\data,\gamma_{0})\in (0,1)\\[4pt]\, 
\tau\equiv \tau(\data,\gamma_{0})\in (0,1/64)
\end{cases}
$$ such that if \eqref{exx} holds in $B_{\rr}(x_{0}) \Subset \Omega$ and $a(\cdot)$ is $\delta$-vanishing in $B_{R\rr}(x_{0})\Subset \Omega$, then
\eqn{exx.2}
$$\tx{E}_{u}(x_{0},\tau\rr)\leq  \frac{\tau^{\gamma_{0}}}{2}\tx{E}_{u}(x_{0},\rr)\,.$$
\end{proposition}
\begin{proof}
This goes in five steps. Under assumption \eqref{exx}, necessarily $\tx{E}_{u}(x_{0},\rr)>0$.
For completeness, note that if $\tx{E}_{u}(x_{0},\rr)=0$, then
\eqref{exx.2} trivially follows from \rif{scataildopoff}$_1$. We shall assume \rif{sceltona} with the values of the parameters determined in the course of the proof with the dependence on the constants specified in \rif{sceltona}. 
\subsubsection{Step 1: Blow-up} \label{blowupsec}                                                                                                                                           
For $x \in \er^n $  we define the blow-up maps
\begin{flalign}\label{bs.4.0}
\left\{
\begin{array}{c}
\displaystyle
u_{\rr}(x):=\frac{u(x_{0}+\rr x)-(u)_{B_{\rr}(x_{0})}}{\tx{E}_{u}(x_{0},\rr)},\quad \quad  f_{\rr}(x):=\frac{\rr^{2s}f(x_{0}+\rr x)}{\tx{E}_{u}(x_{0},\rr)}\\[15pt]\displaystyle
a_{\rr}(x,y,v,w):=a\left(x_{0}+\rr x,x_{0}+\rr y, (u)_{B_{\rr}(x_{0})}+\tx{E}_{u}(x_{0},\rr)v,(u)_{B_{\rr}(x_{0})}+\tx{E}_{u}(x_{0},\rr)w\right),
\end{array}
\right.
\end{flalign} 
and observe that $a_{\rr}(\cdot)$ still satisfies \eqref{bs.1}-\eqref{bs.2}, i.e., ellipticity and symmetry. Also note that 
\eqn{mediette}
$$(u(x_{0}+\rr \ \cdot \ ))_{\ttB_{\sigma} }= (u)_{B_{\sigma \rr}(x_{0})}\,, \qquad 
\tail(u(x_{0}+\rr \ \cdot \ )-\xi;\ttB_{\sigma}  )=\tail(u-\xi;\ttB_{\sigma\rr}(x_0))$$ 
for every $\xi \in \er^N$ and $\sigma>0$. Moreover, assuming that $a(\cdot)$ is $\delta$-vanishing in $B_{R\rr}(x_{0})\Subset \Omega$ in the sense of Definition \ref{deltadef}, and letting 
\eqn{avw}
$$
a_{\rr, \ttB_{\sigma}}(v,w):=\mint_{\ttB_{\sigma}}\mint_{\ttB_{\sigma}}a_{\rr}(x,y,v, w)\dxy $$
for $v,w \in \er^N$, we note that it follows that 
$$a_{\rr, \ttB_{\sigma}}(v,w)=a_{B_{\sigma\rr}(x_{0})}\left((u)_{B_{\rr}(x_{0})}+\tx{E}_{u}(x_{0},\rr)v,(u)_{B_{\rr}(x_{0})}+\tx{E}_{u}(x_{0},\rr)w\right)$$
and
\eqn{smally}
$$  \mint_{\ttB_{\sigma}}\mint_{\ttB_{\sigma}}\snr{a_{\rr}(x,y,v,w)-a_{\rr, \ttB_{\sigma}}(v,w)}\dxy\leq \delta \,, \qquad \mbox{for every $\sigma \leq R$}\,.  $$
Moreover, \eqref{bs.4} holds with 
\eqn{defiom}
$$\omega_{\tx{E}}(t):= \omega (\tx{E}_{u}(x_{0},\rr) \tx{t})\,, \qquad \tx{t} \geq 0$$ 
replacing $\omega(\cdot)$, i.e., 
\eqn{bs.4bis}
$$
\snr{a_{\rr}(x,y,v_{1},w_{1})-a_{\rr}(x,y,v_{2},w_{2})}\le \Lambda \omega_{\tx{E}}(\snr{v_{1}-v_{2}}+\snr{w_{1}-w_{2}})
$$
holds whenever  $v_1, v_2, w_1, w_2 \in \er^N$ and $x, y \in \er^n$. 
Letting $\Omega_{\rr}:=\{x\in \mathbb{R}^{n}\colon x_{0}+\rr x\in \Omega\}= (\Omega-x_0)/\rr$, a simple scaling argument shows that $u_{\rr}\in W^{s,2}_{\loc}(\Omega_{\rr};\er^N)\cap L^{1}_{2s}$ is a weak solution to 
\eqn{riscalata}
$$-\mathcal{L}_{a_{\rr}}u_{\rr}=f_{\rr}\, , \qquad \mbox{in $\Omega_{\rr}$}\,,$$ and that $\ttB_{1} \subset \ttB_{R}\Subset \Omega_{\rr}$. Moreover, we have
\eqn{sc.0}
$$
\begin{cases}
\displaystyle
(u_{\rr})_{\ttB_{1} }=0, \qquad  \nra{u_{\rr}}_{L^{2}(\ttB_{1} )}^2+\tail(u_{\rr};\ttB_{1} )^2= 1,\qquad  \nra{f_{\rr}}_{L^{\chi}(\ttB_{1} )}\le \varepsilon^{*}\,,\\[10pt]
[u_{\rr}]_{s,2;\ttB_{4/5} }\le c\equiv c(\data) , \qquad  [u_{\rr}]_{s+\delta_0,2;\ttB_{3/4} } \le c\equiv c(\data)\,,
\end{cases}
$$
where $\delta_0\in (0, 1-s)$ depends only on $\data$. 
Indeed, \rif{sc.0}$_1$ follows directly from \rif{bs.4.0}-\rif{mediette}, and recalling \rif{exx} as far as the terms involving $f$ is concerned. In order to get the first inequality in \rif{sc.0}$_2$ it is sufficient to apply Lemma \ref{prop:cacc} to $u_{\rr}$ and use the information in \rif{sc.0}$_1$. Finally, once the first inequality in \rif{sc.0}$_2$ is achieved, the second one follows by Theorem \ref{lamaggiore} and a standard covering argument. 
\subsubsection{Step 2: Localization} This is in the following:
\begin{lemma}\label{cflem}
Let 
\eqn{cutty}
$$
\begin{cases}
\ti{u}_{\rr}:=\eta u_{\rr}   \\[3pt]
\mbox{$\eta \in C^{\infty}(\er^n)$, \  $\mathds{1}_{\ttB_{1/4} }\le \eta \le \mathds{1}_{\ttB_{1/2} }$ \ and \ $\nr{D\eta }_{L^{\infty}(\ttB_{1/2})}\lesssim 1$}\,.
\end{cases}
$$
Then 
\begin{itemize} 
\item
$\ti{u}_{\rr}\in W^{s,2}(\er^n; \er^N)\cap W^{s+\delta_0,2}(\er^n; \er^N)$ with 
\eqn{cf.2}
$$
\nr{\ti{u}_{\rr}}_{W^{s,2}(\mathbb{R}^{n})}\le c\equiv c(\data), \qquad \nr{\ti{u}_{\rr}}_{W^{s+\delta_0,2}(\mathbb{R}^{n})}\le c\equiv c(\data)
$$
where $\delta_0\equiv \delta_0(\data)\in (0, 1-s)$. 
\item There exist $g_{\rr}\in L^{\infty}(\mathbb{R}^{n};\er^N)$, $\lambda_{\rr}\in L^{\infty}(\mathbb{R}^{n};\mathbb{R}^{N\times N})$\footnote{The maps $g_{\rr}, \lambda_{\rr}$ are defined in \eqref{ragnar} below.} such that $\ti{u}_{\rr}$ solves
$$
-\mathcal{L}_{a_{\rr}}\ti{u}_{\rr}=g_{\rr}+f_{\rr}+\lambda_{\rr}\ti{u}_{\rr}\qquad \mbox{in} \ \ \ttB_{1/8}
$$
in the sense of Definition \ref{def:weaksol}, i.e., 
\begin{flalign}
&\int_{\mathbb{R}^n} \int_{\mathbb{R}^n} \langle a_{\rr}(x,y,\ti{u}_{\rr}(x),\ti{u}_{\rr}(y))(\ti{u}_{\rr}(x)-\ti{u}_{\rr}(y)),\varphi(x)-\varphi(y) \rangle \frac{\dxy}{|x-y|^{n+2s}}\notag \\
& \qquad  = \int_{\ttB_{1/8}} \langle g_{\rr}+f_{\rr}+\lambda_{\rr}\ti{u}_{\rr}, \varphi \rangle\dx
\label{cutoffata}
\end{flalign}
holds for every $\varphi \in W^{s,2}(\er^n;\er^N)$ with compact support in $\ttB_{1/8}$. Moreover, 
\eqn{cf.1}
$$
\nr{\lambda_{\rr}}_{L^{\infty}(\mathbb{R}^{n})}\le c\omega_{\tx{E}}(1)\qquad \mbox{and}\qquad \nr{g_{\rr}}_{L^{\infty}(\mathbb{R}^{n})}\le c
$$
hold with $c\equiv c(\data)$.
\end{itemize}
\end{lemma}
\begin{proof} For \rif{cf.2} it is sufficient to check the second inequality. Using \rif{sc.0}-\rif{cutty}, and recalling \rif{uset}, we have 
\begin{flalign*} 
[\ti{u}_{\rr}]_{s+\delta_0,2;\mathbb{R}^{n}}^2&= [\ti{u}_{\rr}]_{s+\delta_0,2;\ttB_{3/4}}^2+ 2 \int_{\mathbb{R}^{n}\setminus \ttB_{3/4}}\int_{\ttB_{3/4}}\frac{\snr{\ti{u}_{\rr}(x)-\ti{u}_{\rr}(y)}^{2}}{\snr{x-y}^{n+2(s+\delta_0)}}\dxy \\
&\lesssim [\ti{u}_{\rr}]_{s+\delta_0,2;\ttB_{3/4}}^2+ \int_{\mathbb{R}^{n}\setminus \ttB_{3/4}}\int_{\ttB_{1/2}}\frac{\snr{\eta (x)u_{\rr}(x)}^{2}}{\snr{y}^{n+2(s+\delta_0)}}\dxy \nonumber \\
&\lesssim [\ti{u}_{\rr}]_{s+\delta_0,2;\ttB_{3/4}}^2+\nr{u_{\rr}}_{L^2(\ttB_1)}^2 \nonumber
\\
&\lesssim    [u_{\rr}]_{s+\delta_0,2;\ttB_{3/4}}^2 + \int_{\ttB_{1}}\int_{\ttB_{1}}\frac{\snr{u_{\rr}(x)}^{2}\snr{\eta (x)-\eta (y)}^{2}}{\snr{x-y}^{n+2(s+\delta_0)}}\dxy+1\nonumber \\
& \lesssim  \int_{\ttB_{1}}\int_{\ttB_{1}}\frac{\snr{u_{\rr}(x)}^{2}}{\snr{x-y}^{n-2(1-s-\delta_0)}}\dxy +1\\
& \lesssim \nr{u_{\rr}}_{L^2(\ttB_1)}^2+1 \lesssim_{\data} 1
\end{flalign*}
from which  $\ti{u}_{\rr}\in W^{s+\delta_0,2}(\er^n; \er^N)$ follows recalling \rif{sc.0}. To proceed, in the following we abbreviate 
\eqn{shorten}
$$
\aa(x, y):= a_{\rr}(x,y,u_{\rr}(x),u_{\rr}(y)) \,, \qquad \aat(x, y):= a_{\rr}(x,y,\ti{u}_{\rr}(x),\ti{u}_{\rr}(y))
$$
whenever $x, y\in \er^n$. 
Note that \rif{bs.2} implies $\aa(x, y)=\aa(y,x)$ and $\aat(x, y)=\aat(y,x)$. Taking this into account, by \rif{riscalata}, we write for any $\varphi \in W^{s,2}(\er^n;\er^N)$ with compact support in $\ttB_{1/8}$
\begin{flalign*}
&\hspace{-4mm} \int_{\mathbb{R}^{n}}\int_{\mathbb{R}^{n}}\langle \aat(x, y)(\ti{u}_{\rr}(x)-\ti{u}_{\rr}(y)),\varphi(x)-\varphi(y)\rangle\frac{\dxy}{\snr{x-y}^{n+2s}}\nonumber \\
&=\int_{\ttB_{1/2}}\int_{\ttB_{1/2}}\langle \aat(x, y)(\ti{u}_{\rr}(x)-\ti{u}_{\rr}(y)),\varphi(x)-\varphi(y)\rangle\frac{\dxy}{\snr{x-y}^{n+2s}}\nonumber \\
&\quad +2\int_{\mathbb{R}^{n}\setminus \ttB_{1/2}}\int_{\ttB_{1/2}}\langle \aat(x, y)(\ti{u}_{\rr}(x)-\ti{u}_{\rr}(y)),\varphi(x)-\varphi(y)\rangle\frac{\dxy}{\snr{x-y}^{n+2s}}\nonumber \\
&=\int_{\ttB_{1/2}}\int_{\ttB_{1/2}}\langle \aa(x, y)(\ti{u}_{\rr}(x)-\ti{u}_{\rr}(y)),\varphi(x)-\varphi(y)\rangle\frac{\dxy}{\snr{x-y}^{n+2s}}\nonumber \\
&\quad +\int_{\ttB_{1/2}}\int_{\ttB_{1/2}}\langle(\aat(x, y)- \aa(x, y))(\ti{u}_{\rr}(x)-\ti{u}_{\rr}(y)),\varphi(x)-\varphi(y)\rangle\frac{\dxy}{\snr{x-y}^{n+2s}}\nonumber \\
&\quad +2\int_{\mathbb{R}^{n}\setminus \ttB_{1/2}}\int_{\ttB_{1/2}}\langle \aa(x, y)(\ti{u}_{\rr}(x)-\ti{u}_{\rr}(y)),\varphi(x)\rangle\frac{\dxy}{\snr{x-y}^{n+2s}}\nonumber \\
&\quad +2\int_{\mathbb{R}^{n}\setminus \ttB_{1/2}}\int_{\ttB_{1/2}}\langle (\aat(x, y)-\aa(x, y))(\ti{u}_{\rr}(x)-\ti{u}_{\rr}(y)),\varphi(x)\rangle\frac{\dxy}{\snr{x-y}^{n+2s}}\nonumber \\
&=\int_{\ttB_{1/2}}\int_{\ttB_{1/2}}\langle \aa(x, y)(u_{\rr}(x)-u_{\rr}(y)),\varphi(x)-\varphi(y)\rangle\frac{\dxy}{\snr{x-y}^{n+2s}}\nonumber \\
&\quad +2\int_{\mathbb{R}^{n}\setminus \ttB_{1/2}}\int_{\ttB_{1/2}}\langle \aa(x, y)(u_{\rr}(x)-u_{\rr}(y)),\varphi(x)\rangle\frac{\dxy}{\snr{x-y}^{n+2s}}\nonumber \\
&\quad +\int_{\ttB_{1/2}}\int_{\ttB_{1/2}}\langle \aa(x, y)((\ti{u}_{\rr}(x)-u_{\rr}(x))-(\ti{u}_{\rr}(y)-u_{\rr}(y))),\varphi(x)-\varphi(y)\rangle\frac{\dxy}{\snr{x-y}^{n+2s}}\nonumber \\
&\quad +\int_{\ttB_{1/2}}\int_{\ttB_{1/2}}\langle(\aat(x, y)- \aa(x, y))(\ti{u}_{\rr}(x)-\ti{u}_{\rr}(y)),\varphi(x)-\varphi(y)\rangle\frac{\dxy}{\snr{x-y}^{n+2s}}\nonumber \\
&\quad +2\int_{\mathbb{R}^{n}\setminus \ttB_{1/2}}\int_{\ttB_{1/2}}\langle \aa(x, y)((\ti{u}_{\rr}(x)-u_{\rr}(x))-(\ti{u}_{\rr}(y)-u_{\rr}(y)),\varphi(x)\rangle\frac{\dxy}{\snr{x-y}^{n+2s}}\nonumber \\
&\quad +2\int_{\mathbb{R}^{n}\setminus \ttB_{1/2}}\int_{\ttB_{1/2}}\eta(x)\langle(\aat(x, y)- \aa(x, y))u_{\rr}(x),\varphi(x)\rangle\frac{\dxy}{\snr{x-y}^{n+2s}}\nonumber \\
&=:\int_{  \ttB_{1/8}}\langle f_{\rr},\varphi\rangle\dx+\mbox{(II)}+\mbox{(III)}+\mbox{(IV)}+\mbox{(V)}
\end{flalign*}
and we have used \eqref{riscalata} to recover the first term in the last line. In the following line we shall use that $\ti{u}_{\rr}=u_{\rr}$ in $\ttB_{1/4}$, $\ti{u}_{\rr}=0$ outside $\ttB_{1/2}$ and that $(1-\eta(y)) \varphi(y)=0$ for every $y\in \er^n$. Indeed, we obtain 
$$
\mbox{(II)}=2\int_{\ttB_{1/2}\setminus \ttB_{1/4}}\int_{\ttB_{1/8}}(1-\eta(y))\langle a_{\rr}(x,y,u_{\rr}(x),u_{\rr}(y))u_{\rr}(y),\varphi(x)\rangle\frac{\dxy}{\snr{x-y}^{n+2s}}=\int_{  \ttB_{1/8}}\langle g_{1;\rr},\varphi\rangle \dx
$$
where
$$
g_{1;\rr}(x):= \mathds{1}_{\ttB_{1/8}}(x)2\int_{\ttB_{1/2}\setminus \ttB_{1/4}}(1-\eta(y))\aa(x, y)u_{\rr}(y)\frac{\dy}{\snr{x-y}^{n+2s}}\,.
$$
For (III) write
\begin{flalign*}
\mbox{(III)}&=\int_{\ttB_{1/2}}\int_{\ttB_{1/2}}\langle(\aat(x, y)- a_{\rr}(x,y,\ti{u}_{\rr}(x),u_{\rr}(y)))(\ti{u}_{\rr}(x)-\ti{u}_{\rr}(y)),\varphi(x)-\varphi(y)\rangle\frac{\dxy}{\snr{x-y}^{n+2s}}\\
& \quad  +\int_{\ttB_{1/2}}\int_{\ttB_{1/2}}\langle(a_{\rr}(x,y,\ti{u}_{\rr}(x),u_{\rr}(y))- \aa(x, y))(\ti{u}_{\rr}(x)-\ti{u}_{\rr}(y)),\varphi(x)-\varphi(y)\rangle\frac{\dxy}{\snr{x-y}^{n+2s}}\\
& =: \mbox{(III)}_1+\mbox{(III)}_2\,.
\end{flalign*}
Noting that $\ti{u}_{\rr}\equiv u_{\rr}$ on $\ttB_{1/8}$ and that therefore
$$
\begin{cases}
\langle  ( \aat(x, y)- a_{\rr}(x,y,\ti{u}_{\rr}(x),u_{\rr}(y))) \xi,\varphi(y)\rangle =0\\[2 pt]
\langle(a_{\rr}(x,y,\ti{u}_{\rr}(x),u_{\rr}(y))- \aa(x, y)) \xi,\varphi(x)\rangle =0
\end{cases}
$$
holds for every $\xi \in \er^N$, $x, y \in \er^n$, we first find that $ \mbox{(III)}_1=\mbox{(III)}_2$ (for this exchange $x$ and $y$ and use the symmetry of 
$a_\rr(\cdot)$) and then that 
\begin{flalign*}
\mbox{(III)}&=2\int_{\ttB_{1/2}\setminus \ttB_{1/4}}\int_{\ttB_{1/8}}\langle (a_{\rr}(x,y,u_{\rr}(x),\ti{u}_{\rr}(y))-\aa(x, y))u_{\rr}(x),\varphi(x)\rangle\frac{\dxy}{\snr{x-y}^{n+2s}}\nonumber \\
&\quad -2\int_{\ttB_{1/2}\setminus \ttB_{1/4}}\int_{\ttB_{1/8}}\eta(y)\langle( a_{\rr}(x,y,u_{\rr}(x),\ti{u}_{\rr}(y))-\aa(x, y))u_{\rr}(y),\varphi(x)\rangle\frac{\dxy}{\snr{x-y}^{n+2s}}\nonumber \\
&=:\int_{\ttB_{1/8}}\langle \lambda_{1;\rr}u_{\rr}-g_{2;\rr},\varphi\rangle\dx
\end{flalign*}
holds with
$$
\begin{cases}
\displaystyle
\lambda_{1;\rr}(x):= \mathds{1}_{\ttB_{1/8}}(x)2\int_{\ttB_{1/2}\setminus \ttB_{1/4}}(a_{\rr}(x,y,u_{\rr}(x),\ti{u}_{\rr}(y))-\aa(x, y))\frac{\dy}{\snr{x-y}^{n+2s}}\\[12pt]
\displaystyle
g_{2;\rr}(x) :=\mathds{1}_{\ttB_{1/8}}(x) 2\int_{\ttB_{1/2}\setminus \ttB_{1/4}}\eta(y)(a_{\rr}(x,y,u_{\rr}(x),\ti{u}_{\rr}(y))-\aa(x, y))u_{\rr}(y)\frac{\dy}{\snr{x-y}^{n+2s}}
\end{cases}
$$
for every $x\in \er^n$. 
Concerning the mixed terms, we have, recalling that $\ti{u}_{\rr}\equiv 0$ outside $\ttB_{1/2}$ and that $\ti{u}_{\rr}\equiv u_{\rr}$ in $\ttB_{1/8}$, we find 
$$
\mbox{(IV)}=2\int_{\mathbb{R}^{n}\setminus \ttB_{1/2}}\int_{\ttB_{1/8}}\langle a_{\rr}(x,y,u_{\rr}(x),u_{\rr}(y))u_{\rr}(y),\varphi(x)\rangle\frac{\dxy}{\snr{x-y}^{n+2s}}\nonumber =\int_{\ttB_{1/8}}\langle g_{3;\rr},\varphi\rangle\dx
$$
and
\begin{flalign*}
\mbox{(V)}&=2\int_{\mathbb{R}^{n}\setminus \ttB_{1/2}}\int_{\ttB_{1/8}}\langle(a_{\rr}(x,y,u_{\rr}(x),0)-a_{\rr}(x,y,u_{\rr}(x),u_{\rr}(y)))u_{\rr}(x),\varphi(x)\rangle\frac{\dxy}{\snr{x-y}^{n+2s}}\\
& =\int_{\ttB_{1/8}}\langle \lambda_{2;\rr}u_{\rr},\varphi\rangle\dx,
\end{flalign*}
where
$$
\begin{cases}
\displaystyle g_{3;\rr}(x):=\mathds{1}_{\ttB_{1/8}}(x) 2\int_{\mathbb{R}^{n}\setminus \ttB_{1/2}}\aa(x, y)u_{\rr}(y)\frac{\dy}{\snr{x-y}^{n+2s}}\\[13pt]
\displaystyle
 \lambda_{2;\rr}(x):=\mathds{1}_{\ttB_{1/8}}(x) 2\int_{\mathbb{R}^{n}\setminus \ttB_{1/2}}(a_{\rr}(x,y,u_{\rr}(x),0)-a_{\rr}(x,y,u_{\rr}(x),u_{\rr}(y)))\frac{\dy}{\snr{x-y}^{n+2s}}
 \end{cases}
 $$
 for every $x\in \er^n$.  Setting 
 \eqn{ragnar}
$$g_{\rr}:=g_{1;\rr}-g_{2;\rr}+g_{3;\rr}, \qquad \lambda_{\rr}:=\lambda_{1;\rr}+\lambda_{2;\rr}$$ yields \rif{cutoffata}, again recalling that $u_{\rr}\equiv \ti{u}_{\rr}$ on $\ttB_{1/8}$. It remains to prove \rif{cf.1}. We use Jensen inequality (recall that $\omega(\cdot)$ is concave and therefore also $\omega_{\tx{E}}(\cdot)$ is) as follows:
\begin{eqnarray*}
\snr{\lambda_{\rr}(x)}&\le&\snr{\lambda_{1;\rr}(x)}+\snr{\lambda_{2;\rr}(x)}\nonumber \\
&\stackleq{bs.4bis}&c\mathds{1}_{\ttB_{1/8}}(x)\int_{\ttB_{1/2}\setminus \ttB_{1/4}}\omega_{\tx{E}}\left((1-\eta(y))\snr{u_{\rr}(y)}\right)\frac{\dy}{\snr{x-y}^{n+2s}}\\
&&\quad +c\mathds{1}_{\ttB_{1/8}}(x)\int_{\mathbb{R}^{n}\setminus \ttB_{1/2}}\omega_{\tx{E}}\left(\snr{u_{\rr}(y)}\right)\frac{\dy}{\snr{x-y}^{n+2s}}\nonumber \\
&\stackleq{uset}&c\int_{\mathbb{R}^{n}\setminus \ttB_{1/4}}\omega_{\tx{E}}\left(\snr{u_{\rr}(y)}\right)\frac{\dy}{\snr{y}^{n+2s}}\nonumber \\
& \stackrel{\mbox{\scriptsize Jensen}}{\leq} & c \omega_{\tx{E}}\left(\tail(u_{\rr};\ttB_{1/4})\right)\\
&\stackrel{\eqref{scatailggg}}{\le}&c\omega_{\tx{E}}\left(\nra{u_{\rr}}_{L^{2}(\ttB_{1})}\right)+c\omega_{\tx{E}}\left(\tail(u_{\rr};\ttB_{1})\right)\\
&\stackrel{\eqref{sc.0}}{\le}&c\omega_{\tx{E}}\left(1\right)\leq c(\data)\,,
\end{eqnarray*}
 from which the first inequality in \rif{cf.1} follows; we also used that $\omega(\cdot)\leq 1$. Of course we used the fact that $\omega(\cdot)$ is subadditive and non-decreasing. Similarly, also using \rif{uset}, we have
\begin{flalign*}
\snr{g_{\rr}(x)}&\le\snr{g_{1;\rr}(x)}+\snr{g_{2;\rr}(x)}+\snr{g_{3;\rr}(x)}\\ &\leq  c \mathds{1}_{\ttB_{1/8}}(x)\int_{\er^n\setminus \ttB_{1/4}}\frac{\snr{u_{\rr}(y)}}{\snr{x-y}^{n+2s}}\dy \leq  c \int_{\er^n\setminus \ttB_{1/4}}\frac{\snr{u_{\rr}(y)}}{\snr{y}^{n+2s}}\dy\\ & \leq  c\tail(u_{\rr};\ttB_{1/4})\leq c\nra{u_{\rr}}_{L^{2}(\ttB_{1})}+c\, \tail(u_{\rr};\ttB_{1}) \leq c(\data)
\end{flalign*}
and the proof is complete. 
\end{proof} 

\subsubsection{Step 3: Approximate $s$-harmonicity} As in Remark \ref{remarkino}, in the following we test \rif{cutoffata} 
by any $\varphi \in \mathbb{X}^{s,2}_{0}(\ttB_{1/16} ,\ttB_{1/8})$. In view of \rif{cf.2}-\rif{cutoffata} we aim at applying Lemma \ref{shar} to $v\equiv \ti{u}_{\rr}$  with the choices 
$
B\equiv \ttB_{1/16}$, $ \ti{B}\equiv \ttB_{1/8}$ and $g\equiv g_{\rr}$ and $a_{0}\equiv a_{\rr, \ttB_{R}}(0,0)$, where
\eqn{a0}
$$
 a_{\rr, \ttB_{R}}(0,0)\stackrel{\rif{avw}}{:=}\mint_{\ttB_{R}}\mint_{\ttB_{R}}a_{\rr}(x,y, 0,0)\dxy \,.
$$
Note that, a direct consequence of \eqref{bs.1}-\eqref{bs.2} and \rif{bs.4.0}$_2$ is that
$$
\Lambda^{-1}\snr{\xi}^{2} \leq \langle a_{\rr, \ttB_{R}}(0,0)\xi,\xi\rangle\qquad \mbox{and}\qquad \snr{ a_{\rr, \ttB_{R}}(0,0)}\le \Lambda\,,
$$
hold for all $\xi\in \er^N$. We need to verify \eqref{sh.0}-\eqref{sh.1}. For the upper bound in \eqref{sh.0} we have
\eqn{limitazioni}
$$
\nr{\ti{u}_{\rr}}_{W^{s,2}(\mathbb{R}^{n})}+\nr{\ti{u}_{\rr}}_{W^{s+\delta_{0},2}(\er^n)}+ \tail(\ti{u}_{\rr};\texttt {B}_{1/8}) + \nr{g_{\rr}}_{L^{\infty}(\mathbb{R}^{n})}\le c\equiv c (\data)\,.
$$
This easily follows by \eqref{cf.2} and \eqref{cf.1} (recall that  $\tail(\ti{u}_{\rr};\ttB_{1/8} )\lesssim 1$ as $\ti{u}_{\rr}$ is supported in $\ttB_{1/2}$). Once $\delta_{0}$ has been determined as a function of $\data$ as in \rif{cf.2}, following Lemma \ref{shar} we accordingly determine $t, p$ such that 
\eqn{tep}
$$
s< t \equiv t(\data)< 1\, \quad \mbox{and} \quad p\equiv p(\data)>2\,.
$$
As for \eqref{sh.1}, note that when verifying \rif{sh.1}, as in Remark \ref{semplice} it is sufficient to consider the case $[\varphi]_{0,t;\er^n}\leq 1$. This implies $\nr{\varphi}_{L^\infty(\ttB_1)}\leq 1$.
With $a_{0}\in \er^{N\times N}$ being defined in \eqref{a0}, and recalling the notation in  \rif{shorten}, given any $\varphi \in \mathbb{X}^{s,2}_{0}(\ttB_{1/16},\ttB_{1/8})\cap C^{0,t}(\er^n;\er^N)$, we control
\begin{flalign}
\notag &\left|\int_{\mathbb{R}^{n}}\int_{\mathbb{R}^{n}}\langle a_{\rr, \ttB_{R}}(0,0)(\ti{u}_{\rr}(x)-\ti{u}_{\rr}(y)),\varphi(x)-\varphi(y)\rangle\frac{\dxy}{\snr{x-y}^{n+2s}}-\int_{\ttB_{1/16}}\langle g_{\rr},\varphi\rangle\dx\right|\nonumber \\
\notag &  \qquad \stackrel{\eqref{cutoffata}}{\le}\int_{\ttB_{1/8}}\int_{\ttB_{1/8}}\snr{a_{\rr, \ttB_{R}}(0,0)-\aat(x, y)}\snr{\ti{u}_{\rr}(x)-\ti{u}_{\rr}(y)}\snr{\varphi(x)-\varphi(y)}\frac{\dxy}{\snr{x-y}^{n+2s}}\nonumber \\
\notag &  \qquad \qquad \quad +2\int_{\mathbb{R}^{n}\setminus \ttB_{1/8}}\int_{\ttB_{1/8}}\snr{a_{\rr, \ttB_{R}}(0,0)-\aat(x, y)}\snr{\ti{u}_{\rr}(x)-\ti{u}_{\rr}(y)}\snr{\varphi(x)-\varphi(y)}\frac{\dxy}{\snr{x-y}^{n+2s}}\nonumber \\
&  \qquad \quad\qquad  +\int_{\ttB_{1/16}}(\snr{f_{\rr}}+\lambda_{\rr}\snr{u_{\rr}})\snr{\varphi}\dx=:\mbox{(I)}+\mbox{(II)}+\mbox{(III)}\,.\label{merge0}
\end{flalign}
Note that, whenever $\texttt{p}\geq 1$
\begin{flalign*}
\mint_{\ttB_{1}}\omega_{\tx{E}}\left(\snr{\ti{u}_{\rr}}\right)^{\texttt{p}}\dx & \leq 
\mint_{\ttB_{1}}\omega_{\tx{E}}\left(\snr{\ti{u}_{\rr}}\right)\dx  \\& \leq 
\omega_{\tx{E}}\left(\nra{u_{\rr}}_{L^{1}(\ttB_{1})}\right) \leq \omega_{\tx{E}}\left(\nra{u_{\rr}}_{L^{2}(\ttB_{1})}\right) \stackrel{\rif{sc.0}}{\leq}  \omega_{\tx{E}}(1)  \stackrel{\rif{defiom}}{=}\omega (\tx{E}_{u}(x_{0},\rr)) 
\end{flalign*}
where we have also used, in turn, that $\omega(\cdot)$ does not exceed $1$, its concavity and Jensen's inequality. Recalling \rif{exx}, we conclude with 
\eqn{exx2}
$$
\nra{\omega_{\tx{E}}(\snr{\ti{u}_{\rr}})}_{L^{\texttt{p}}(\ttB_{1})}\leq \omega(\varepsilon_{*})^{1/\texttt{p}}\, \qquad \forall \ \texttt{p}\geq 1\,.
$$
Using H\"older inequality with conjugate exponents 
$$\left(\frac{2(n+s-t)}{t-s},2,\frac{2(n+s-t)}{n-2(t-s)}\right)$$ 
and that $[\varphi]_{0,t;\er^n}\leq 1$, we find
\begin{eqnarray}
\snr{\mbox{(I)}}&\le&\int_{\ttB_{1/8}}\int_{\ttB_{1/8}}\snr{a_{\rr, \ttB_{R}}(0,0)-a_{\rr}(x,y,0, 0)}\frac{\snr{\ti{u}_{\rr}(x)-\ti{u}_{\rr}(y)}}{\snr{x-y}^{n/2+s}}\frac{\dxy}{\snr{x-y}^{n/2+s-t}}\nonumber \\
&&\notag+\int_{\ttB_{1/8}}\int_{\ttB_{1/8}}
\snr{\aat(x, y)-a_{\rr}(x,y,0, 0)}\frac{\snr{\ti{u}_{\rr}(x)-\ti{u}_{\rr}(y)}}{\snr{x-y}^{n/2+s}}\frac{\dxy}{\snr{x-y}^{n/2+s-t}}\nonumber \\
&\notag\leq &c[\ti{u}_{\rr}]_{s,2;\ttB_{1/8}}\left(\int_{\ttB_{1/8}}\int_{\ttB_{1/8}}\snr{a_{\rr, \ttB_{R}}(0,0)-a_{\rr}(x,y,0, 0)}^{\frac{2(n+s-t)}{t-s}}\dxy\right)^{\frac{t-s}{2(n+s-t)}}\nonumber \\
&&\notag\qquad \qquad  \times\left(\int_{\ttB_{1/8}}\int_{\ttB_{1/8}}\frac{\dxy}{\snr{x-y}^{n+s-t}}\right)^{\frac{n-2(t-s)}{2(n+s-t)}}\nonumber \\
&\notag&+c[\ti{u}_{\rr}]_{s,2;\ttB_{1/8}}\left(\int_{\ttB_{1/8}}\int_{\ttB_{1/8}}[\omega_{\tx{E}}(\snr{\ti{u}_{\rr}(x)})+\omega_{\tx{E}}(\snr{\ti{u}_{\rr}(y)})]^{\frac{2(n+s-t)}{t-s}}\dxy\right)^{\frac{t-s}{2(n+s-t)}}\nonumber \\
&&\notag\qquad \qquad  \times\left(\int_{\ttB_{1/8}}\int_{\ttB_{1/8}}\frac{\dxy}{\snr{x-y}^{n+s-t}}\right)^{\frac{n-2(t-s)}{2(n+s-t)}}\nonumber \\
&\notag\stackrel{\eqref{cf.2}}{\le}&cR^{\frac{n(t-s)}{n+s-t}}\left(\mint_{\ttB_{R}}\mint_{\ttB_{R}}\snr{a_{\rr, \ttB_{R}}(0,0)-a_{\rr}(x,y,0, 0)}\dxy\right)^{\frac{t-s}{2(n+s-t)}}\nonumber\\
&& \qquad \quad +c\nra{\omega_{\tx{E}}(\snr{\ti{u}_{\rr}})}_{L^{\frac{2(n+s-t)}{t-s}}(\ttB_{1})}\nonumber \\
&  \stackrel{\eqref{smally},\eqref{exx2}}{\le}&cR^{\frac{n(t-s)}{n+s-t}}\delta^{\frac{t-s}{2(n+s-t)}}+c\omega(\varepsilon_{*})^{\frac{t-s}{2(n+s-t)}},\label{merge1}
\end{eqnarray}
with $c\equiv c(\data)$. Recalling that $\varphi\equiv 0$ outside $\ttB_{1/16}$, that allows us to use \rif{uset}, and that $\ti{u}_{\rr}$ is supported in $\ttB_{1/2}$, we have,  by means of \eqref{bs.4bis}
\begin{flalign*}
\snr{\mbox{(II)}}&\le c\int_{\mathbb{R}^{n}\setminus \ttB_{1/8}}\int_{\ttB_{1/16}}\snr{\aat(x, y)-a_{\rr}(x,y,0, 0)}\left(\snr{\ti{u}_{\rr}(x)}+\snr{\ti{u}_{\rr}(y)}\right)\frac{\dxy}{\snr{y}^{n+2s}}\nonumber \\
&\quad  +c\int_{\mathbb{R}^{n}\setminus \ttB_{1/8}}\int_{ \ttB_{1/16}}\snr{a_{\rr, \ttB_{R}}(0,0)-a_{\rr}(x,y,0, 0)}\left(\snr{\ti{u}_{\rr}(x)}+\snr{\ti{u}_{\rr}(y)}\right)\frac{\dxy}{\snr{y}^{n+2s}}\nonumber \\
&\leq c\int_{\mathbb{R}^{n}\setminus \ttB_{1/8}}\int_{ \ttB_{1/16}}\omega_{\tx{E}}(\snr{\ti{u}_{\rr}(x)})\snr{\ti{u}_{\rr}(x)}\frac{\dxy}{\snr{y}^{n+2s}}
\\ & \quad +c\int_{\mathbb{R}^{n}\setminus \ttB_{1/8}}\int_{ \ttB_{1/16}}\omega_{\tx{E}}(\snr{\ti{u}_{\rr}(y)})\snr{\ti{u}_{\rr}(x)}\frac{\dxy}{\snr{y}^{n+2s}}
\nonumber \\
&\quad  +c\int_{\mathbb{R}^{n}\setminus \ttB_{1/8}}\int_{ \ttB_{1/16}}\omega_{\tx{E}}(\snr{\ti{u}_{\rr}(x)})\snr{\ti{u}_{\rr}(y)}\frac{\dxy}{\snr{y}^{n+2s}}\\ & \quad  +c
\int_{\mathbb{R}^{n}\setminus \ttB_{1/8}}\int_{ \ttB_{1/16}}\omega_{\tx{E}}(\snr{\ti{u}_{\rr}(y)})\snr{\ti{u}_{\rr}(y)}\frac{\dxy}{\snr{y}^{n+2s}}
\nonumber \\
&\quad  +c\int_{\mathbb{R}^{n}\setminus \ttB_{1/8}}\int_{ \ttB_{1/16}}\snr{a_{\rr, \ttB_{R}}(0,0)-a_{\rr}(x,y,0, 0)}\snr{\ti{u}_{\rr}(y)}\frac{\dxy}{\snr{y}^{n+2s}}\nonumber \\
&\quad  +c\int_{\mathbb{R}^{n}\setminus \ttB_{1/8}}\int_{ \ttB_{1/16}}\snr{a_{\rr, \ttB_{R}}(0,0)-a_{\rr}(x,y,0, 0)}\snr{\ti{u}_{\rr}(x)}\frac{\dxy}{\snr{y}^{n+2s}}\\& =:  \sum_{1\leq k\leq 6}  \mbox{(II)}_{k}\,.
\end{flalign*}
By means of \rif{sc.0} and \rif{exx2} we have
$$
\mbox{(II)}_{1} \leq c  \int_{\mathbb{R}^{n}\setminus \texttt  B_{1/8}}\frac{\dy}{\snr{y}^{n+2s}} \nra{\omega_{\tx{E}}(\snr{\ti{u}_{\rr}})}_{L^2(\ttB_{1})}\nra{\ti{u}_{\rr}}_{L^2(\ttB_{1})}
\leq c \sqrt{\omega\left(\varepsilon_{*}\right)}\,.
$$
Recalling that $\ti{u}_{\rr}$ is supported in $\ttB_{1/2}$ we have, similarly
\begin{flalign*}
\mbox{(II)}_{2}&\leq c  \int_{\ttB_{1/2}\setminus \ttB_{1/8}}\omega_{\tx{E}}(\snr{\ti{u}_{\rr}(y)})\frac{\dy}{\snr{y}^{n+2s}}\nra{\ti{u}_{\rr}}_{L^1(\ttB_{1})}\\
&
\leq c  \nra{\omega_{\tx{E}}(\snr{\ti{u}_{\rr}})}_{L^2(\ttB_{1})}\nra{\ti{u}_{\rr}}_{L^2(\ttB_{1})}\\ &
\leq c\sqrt{\omega\left(\varepsilon_{*}\right)}\,.
\end{flalign*}
We continue with 
\begin{flalign*}
\mbox{(II)}_{3}+\mbox{(II)}_{4}
&\leq c  \int_{\ttB_{1/2}\setminus \ttB_{1/8}}\snr{\ti{u}_{\rr}(y)}\frac{\dy}{\snr{y}^{n+2s}}\nra{\omega_{\tx{E}}(\snr{\ti{u}_\rr})}_{L^1(\ttB_{1})}\\ & \quad +c \int_{\ttB_{1/2}\setminus \ttB_{1/8}}\omega_{\tx{E}}(\snr{\ti{u}_{\rr}(y)})\snr{\ti{u}_{\rr}(y)}\frac{\dy}{\snr{y}^{n+2s}}\\
&\leq c  \nra{\omega_{\tx{E}}(\snr{\ti{u}_{\rr}})}_{L^2(\ttB_{1})}\nra{\ti{u}_{\rr}}_{L^2(\ttB_{1})}\\& 
\leq c\sqrt{\omega\left(\varepsilon_{*}\right)}\,.
\end{flalign*}
Using \rif{smally}, \rif{sc.0} we find
\begin{flalign*}
\mbox{(II)}_{5} 
& \leq c\nr{\ti{u}_{\rr}}_{L^{2}(\ttB_{1/2})}\left(\int_{\ttB_{1/2}\setminus \ttB_{1/8}}\int_{ \ttB_{1/16}}\snr{a_{\rr, \ttB_{R}}(0,0)-a_{\rr}(x,y,0, 0)}^{2}\dxy\right)^{1/2}\\
& \leq c R^n \left(\mint_{\ttB_{R}}\mint_{ \ttB_{R}}\snr{a_{\rr, \ttB_{R}}(0,0)-a_{\rr}(x,y,0, 0)}\dxy\right)^{1/2}\leq c R^n\sqrt{\delta}\,.
\end{flalign*}
For $\mbox{(II)}_{6}$ we split 
$$
\mbox{(II)}_{6} = c \int_{\mathbb{R}^{n}\setminus \ttB_{R}}\int_{\ttB_{1/16}}[\dots]\, \frac{\dxy}{\snr{y}^{n+2s}} +  \int_{\ttB_{R}\setminus \ttB_{1/8}}\int_{\ttB_{1/16}}[\dots]\, \frac{\dxy}{\snr{y}^{n+2s}} =: \mbox{(II)}_{7} +\mbox{(II)}_{8}\,. 
$$
In turn, we have 
$$
\mbox{(II)}_{7} \leq c R^{-2s} \nra{\ti{u}_{\rr}}_{L^1(\ttB_1)} \leq c R^{-2s} \nra{\ti{u}_{\rr}}_{L^2(\ttB_1)}\leq c R^{-2s}
$$
and, by H\"older inequality and \rif{sc.0}
$$
\mbox{(II)}_{8} \leq c  R^{n}  \left(\mint_{\ttB_{R}}\mint_{ \ttB_{R}}\snr{a_{\rr, \ttB_{R}}(0,0)-a_{\rr}(x,y,0, 0)}^{2}\dxy\right)^{1/2} \nra{\ti{u}_{\rr}}_{L^2(\ttB_1)}\leq c R^{n}\sqrt{\delta}\,.
$$
All in all we have that 
\eqn{merge2}
$$
\snr{\mbox{(II)}} \leq c  \sqrt{\omega\left(\varepsilon_{*}\right)} + cR^n\sqrt{\delta} + cR^{-2s}
$$
holds with $c\equiv c (\data)$. Finally, we have 
\eqn{merge3}
$$
\snr{\mbox{(III)}}\stackrel{\eqref{cf.1}}{\le}c\nra{f_{\rr}}_{L^{\chi}(\ttB_{1})}+c\omega_{\tx{E}}\left(1\right)\nra{\ti{u}_{\rr}}_{L^2(\ttB_1)}\stackrel{\eqref{exx},\eqref{sc.0}}{\le}c\left(\varepsilon^{*}+\omega(\varepsilon_{*})\right)
$$
again with $c\equiv c (\data)$. Merging the content of displays \rif{merge0} and \rif{merge1}-\rif{merge3}, we conclude with
\begin{flalign}
&\notag \left|\int_{\mathbb{R}^{n}}\int_{\mathbb{R}^{n}}\langle a_{\rr, \ttB_{R}}(0,0)(\ti{u}_{\rr}(x)-\ti{u}_{\rr}(y)),\varphi(x)-\varphi(y)\rangle\frac{\dxy}{\snr{x-y}^{n+2s}}-\int_{ \ttB_{1/16}}\langle g_{\rr},\varphi\rangle\dx\right|\nonumber \\
&\qquad  \quad \qquad  \qquad \le c_*\left(R^{-2s}+R^{\frac{n(t-s)}{n+s-t}}\delta^{\frac{t-s}{2(n+s-t)}}+R^{n}\sqrt{\delta}+\varepsilon^{*}+\omega(\varepsilon_{*})^{\frac{t-s}{2(n+s-t)}}\right)\label{byby}
\end{flalign}
where $c_*\equiv c_*(\data)\geq 1$ and $t\equiv t(\data) \in (s,1)$. This holds whenever $\varphi \in \mathbb{X}^{s,2}_{0}(\ttB_{1/16},\ttB_{1/8})\cap C^{0,t}(\er^n;\er^N)$ with $[\varphi]_{0,t;\er^n}\leq 1$, where $t$ has been determined in \rif{tep}.

 \subsubsection{Step 4: Smallness conditions and $s$-harmonicity}\label{step4} Thanks to \rif{limitazioni} and \rif{byby} we are now in position to apply Lemma \ref{shar}. Note that the inequality in \rif{limitazioni} and the dependence of $\delta_0$ in \rif{cf.2} allow to determine the constant $\cchh$ appearing in Lemma \ref{shar} as a function of $\data$. We now fix $\eps_0 \in (0,1)$, which for the moment is left as a free parameter and that we will determine towards the end of the proof. Next, we start by choosing the first four parameters in \rif{sceltona} in order to verify
\eqn{ver} 
$$
\left|\int_{\mathbb{R}^{n}}\int_{\mathbb{R}^{n}}\langle a_{\rr, \ttB_{R}}(0,0)(\ti{u}_{\rr}(x)-\ti{u}_{\rr}(y)),\varphi(x)-\varphi(y)\rangle\frac{\dxy}{\snr{x-y}^{n+2s}}-\int_{ \ttB_{1/16}}\langle g_{\rr},\varphi\rangle\dx\right|\leq\left(\frac{\eps_0}{\cchh}\right)^{\frac{2p}{p-2}}
$$
whenever $\varphi$ has been chosen as above, i.e., $\varphi \in \mathbb{X}^{s,2}_{0}(\ttB_{1/16},\ttB_{1/8})\cap C^{0,t}(\er^n;\er^N)$ with  $[\varphi]_{0,t;\er^n}\leq 1$. The number $p\equiv p(\data)>2$ appears in \rif{tep}. Note that here we shall choose $\varepsilon_{*}, \varepsilon^{*}, R$ and $\delta$ still keeping a dependence on the constants that involves $\eps_0$. In Step 5 we shall trade the dependence on $\eps_0$ with a dependence on the number $\gamma_{0}$ fixed in \rif{gammazero} to finally get all the parameters in \rif{sceltona} with the prescribed dependence on the constants; this will be done choosing $\varepsilon_{0}\equiv \varepsilon_{0}(\data, \gamma_{0})$ as in \rif{sceltafinale} below. We fix $R$ large enough to get 
\eqn{dep.1.1}
$$c_*R^{-2s}< \frac 13 \left(\frac{\eps_0}{\cchh}\right)^{\frac{2p}{p-2}} \ \Longrightarrow \ R\equiv R(\data,\varepsilon_{0}),$$ 
where $c_*\equiv c_*(\data)$ appears in \rif{byby}, and then we fix $\delta,\varepsilon^{*},\varepsilon_{*}$ small enough to obtain
\eqn{dep.1}
$$
\begin{cases}
\displaystyle
\ c_*R^{\frac{n(t-s)}{n+s-t}}\delta^{\frac{t-s}{2(n+s-t)}}+c_*R^{n}\sqrt{\delta}<\frac{1}{3}\left(\frac{\eps_0}{\cchh}\right)^{\frac{2p}{p-2}} \ \Longrightarrow \ \delta\equiv \delta(\data,\varepsilon_{0})\\[16pt]\displaystyle
\ c_{*}\varepsilon^{*}+c_{*}\omega(\varepsilon_{*})^{\frac{t-s}{2(n+s-t)}}<\frac 13 \left(\frac{\eps_0}{\cchh}\right)^{\frac{2p}{p-2}} \ \Longrightarrow \ \varepsilon_{*}\equiv \varepsilon_{*}(\data,\omega(\cdot),\varepsilon_{0}), \ \ \varepsilon^{*}\equiv \varepsilon^{*}(\data,\varepsilon_{0}).
\end{cases}
$$
Inserting this in \rif{byby} yields \rif{ver}. Then we find $h\in \mathbb{X}^{s,2}_{\ti{u}_{\rr}}(\ttB_{1/16},\ttB_{1/8})$ solving 
$$
\begin{cases}
\ -\mathcal{L}_{a_{\rr, \ttB_{R}}(0,0)}h=g_{\rr}\quad &\mbox{in} \ \ \ttB_{1/16}\\
\ h=\ti{u}_{\rr}\quad &\mbox{in} \ \ \mathbb{R}^{n}\setminus \ttB_{1/16}
\end{cases}
$$
in the sense of \eqref{eqweaksol22}. 
Solvability is ensured by Lemma \ref{esiste}. Lemma \ref{shar}, \rif{sh.5}-\rif{sh.5.1} yields
\eqn{confronti}
$$
\begin{cases}
\, \nr{h}_{W^{s,2}(\ttB_{1/8})}+\tail(h-(h)_{\ttB_{1/16}};\ttB_{1/16})\le c\equiv c (\data) \\[5pt]
\, [\ti{u}_{\rr}-h]_{s,2;\ttB_{1/8}}<\varepsilon_0\,.
\end{cases}
$$

\subsubsection{Step 5: One scale global excess decay}\label{step3} Let $\tau\in (0,1/64)$ to be fixed later, and $\beta_{0}\in (s,\min\{2s,1\})$. We shall use following inequality connecting H\"older and Gagliardo seminorms, and whose proof is obvious:
\eqn{holga}
$$
\snra{h}_{s,2;  \ttB_{\lambda}}\lesssim \frac{\lambda^{\beta_0-s}}{(\beta_0-s)^{1/2}} [h]_{0,\beta_0; \ttB_{\lambda}} \,, \quad \mbox{for every  $\lambda>0$}\,.
$$
Recalling that $\ti{u}_{\rr}=u_{\rr}$ in $\ttB_{1/4} $, \rif{mediette}, and repeatedly using \rif{poinfrac} and \eqref{holga}, we have
\begin{eqnarray}
\frac{\tx{E}_{u}(x_{0},\tau\rr)}{\tx{E}_{u}(x_{0},\rr)}&=& \sqrt{\nra{u_{\rr}-(u_{\rr})_{\ttB_{\tau} }}_{L^{2}(\ttB_{\tau} )}^2+\tail(u_{\rr}-(u_{\rr})_{\ttB_{\tau} };\ttB_{\tau} )^2 }\nonumber \\
&\leq & \nra{u_{\rr}-(u_{\rr})_{\ttB_{\tau} }}_{L^{2}(\ttB_{\tau} )}+\tail(u_{\rr}-(u_{\rr})_{\ttB_{\tau} };\ttB_{\tau} ) \nonumber \\
&\stackrel{\eqref{scatail}}{\le}&c\tau^{s}\snra{u_{\rr}}_{s,2;\ttB_{\tau} }\nonumber +c\tau^{2s}\tail(u_{\rr}-(u_{\rr})_{\ttB_{1} };\ttB_{1} )+c\tau^{2s}\nra{u_{\rr}-(u_{\rr})_{\ttB_{1} }}_{L^{2}(\ttB_{1})}\nonumber \\
&&\quad +c\int_{\tau}^{1}\left(\frac{\tau}{\lambda}\right)^{2s}\nra{u_{\rr}-(u_{\rr})_{\ttB_{\lambda} }}_{L^{2}(\ttB_{\lambda} )}\frac{\dlam}{\lambda}\nonumber \\
&\leq &c\tau^{s}\snra{\ti{u}_{\rr}}_{s,2;\ttB_{\tau} }+c\tau^{2s}\tail(u_{\rr};\ttB_{1} )+c\tau^{2s}\nra{u_{\rr}}_{L^{2}(\ttB_{1})}\nonumber \\
&&\quad +c\int_{\tau}^{1/32}\left(\frac{\tau}{\lambda}\right)^{2s}\nra{\ti{u}_{\rr}-(\ti{u}_{\rr})_{\ttB_{\lambda} }}_{L^{2}(\ttB_{\lambda} )}\frac{\dlam}{\lambda}\nonumber \\
&\leq &c\tau^{s}\snra{\ti{u}_{\rr}-h}_{s,2;\ttB_{\tau} }+c\tau^{s}\snra{h}_{s,2;\ttB_{\tau} }\nonumber +c\tau^{2s}\tail(u_{\rr};\ttB_{1} )+c\tau^{2s}\nra{u_{\rr}}_{L^{2}(\ttB_{1})}\nonumber \\
&&\quad +c\int_{\tau}^{1/32}\left(\frac{\tau}{\lambda}\right)^{2s}\lambda^s\snra{\ti{u}_{\rr}}_{s,2;\ttB_{\lambda}}\frac{\dlam}{\lambda}\nonumber \\
&\stackrel{\eqref{sc.0}}{\le}&c\tau^{s-n/2}[\ti{u}_{\rr}-h]_{s,2;\ttB_{1/8}}+c\tau^{\beta_{0}}[h]_{0,\beta_{0};\ttB_{1/32}}+c\tau^{2s}\nonumber \\
&&\quad +c\int_{\tau}^{1/32}\left(\frac{\tau}{\lambda}\right)^{2s}\lambda^{s}\snra{\ti{u}_{\rr}-h}_{s,2;\ttB_{\lambda}}\frac{\dlam}{\lambda}+c\int_{\tau}^{1/32}\left(\frac{\tau}{\lambda}\right)^{2s}\lambda^{s}\snra{h}_{s,2;\ttB_{\lambda} }\frac{\dlam}{\lambda}\nonumber \\
&\stackrel{\eqref{confronti}}{\le}&c\tau^{s-n/2}\varepsilon_{0}+c\tau^{\beta_{0}}[h]_{0,\beta_{0};\ttB_{1/32} }+c\tau^{2s}\nonumber \\
&&\quad +c\tau^{2s}\varepsilon_{0} \int_{\tau}^{1/32}\lambda^{-s-n/2}\frac{\dlam}{\lambda}  +c\tau^{2s}[h]_{0,\beta_{0};\ttB_{1/32}} \int_{\tau}^{1/32}\lambda^{\beta_{0}-2s}\frac{\dlam}{\lambda} \nonumber \\
& \leq &c\tau^{s-n/2}\varepsilon_{0}+ c \tau^{\beta_{0}}[h]_{0,\beta_{0};\ttB_{1/32} }+c\tau^{2s}\nonumber \\
&\stackrel{\eqref{dg.2bis}}{\le}&c\tau^{\beta_{0}}\left(\nra{h-(h)_{\ttB_{1/16}}}_{L^{2}(\ttB_{1/16})}+\tail(h-(h)_{\ttB_{1/16}};\ttB_{1/16})+\nr{g_{\rr}}_{L^{\infty}(\ttB_{1/16})}\right)\nonumber \\
&& \quad +c\tau^{s-n/2}\varepsilon_{0}+c\tau^{2s}\label{piccige}
\end{eqnarray}
with $c\equiv c(\data, 2s- \beta_0, \beta_0-s)$\footnote{In order to pass from the third to the fourth line, here we have estimated 
\begin{flalign*}
\int_{\tau}^{1}\left(\frac{\tau}{\lambda}\right)^{2s}\nra{u_{\rr}-(u_{\rr})_{\ttB_{\lambda} }}_{L^{1}(\ttB_{\lambda} )}\frac{\dlam}{\lambda} =
\int_{\tau}^{1/32}[\ldots]\frac{\dlam}{\lambda}+\int_{1/32}^{1}[\ldots]\frac{\dlam}{\lambda} \lesssim \int_{\tau}^{1/32}[\ldots]\frac{\dlam}{\lambda}+\tau^{2s}\nra{u_{\rr}}_{L^{2}(\ttB_{1})}
\end{flalign*}
and then used that $\ti{u}_{\rr}=u_{\rr}$ in $\ttB_{1/4} $. 
}. A further application of \eqref{cf.1} and \eqref{confronti} in \rif{piccige} implies that
\eqn{scende00}
$$
\tx{E}_{u}(x_{0},\tau\rr) \leq \ti{c} \left(\tau^{\beta_{0}}+ \tau^{s-n/2}\varepsilon_{0}\right)\tx{E}_{u}(x_{0},\rr)
$$
holds whenever $\tau\in (0,1/64)$, $\beta_{0}\in (s,\min\{2s,1\})$. Note that the lower bound $\beta_0 >s$ is here required in order to apply \rif{holga} since the constant in \rif{scende00} blows-up when $\beta_0\to s$. On the other hand, note that if \rif{scende00} holds for a certain value of $\beta_0$, then it continues to hold for all the smaller values. We can therefore conclude that \rif{scende00} holds for the full range $\beta_{0}\in (0,\min\{2s,1\})$ and the relative constant becomes independent of the specific quantity $\snr{\beta_0-s}$ \footnote{Specifically, take $\beta_1:=(s+\min\{2s,1\})/2$ and derive \rif{scende00} with $\beta_0\equiv \beta_1$, so that dependence of the constant $c$ on $\beta_0$ becomes immaterial. Then estimate \rif{scende00} continues to hold for  $\beta_0<\beta_1$ no matter $|\beta_0-s|$ is small.}. We are now ready to finish the proof. With $\gamma_{0}$ as in  \rif{gammazero},  we take $\beta_0 \in (\gamma_0, \min\{2s,1\})$ and we write \rif{scende00} in the form 
\eqn{scende}
$$
\tx{E}_{u}(x_{0},\tau\rr) \leq \ti{c} \left(\tau^{\beta_{0}-\gamma_0}+ \tau^{s-n/2-\gamma_0}\varepsilon_{0}\right)\tau^{\gamma_0}\tx{E}_{u}(x_{0},\rr)\,.
$$ We select $\tau\in (0,1/64)$ such that
\eqn{dep.2}
$$
\ti{c}\tau^{\beta_{0}-\gamma_{0}}<\frac{1}{4} \ \Longrightarrow \tau\equiv \tau(\data,\gamma_{0}),
$$
and then reduce the size of $\varepsilon_{0}$ in such a way that
\eqn{sceltafinale}
$$
\ti{c}\tau^{s-n/2-\gamma_{0}}\varepsilon_{0}<\frac{1}{4} \ \Longrightarrow \ \varepsilon_{0}\equiv \varepsilon_{0}(\data,\gamma_{0})\,.
$$
These choices finally determine $R,\delta,\varepsilon^{*}\equiv R, \delta,\varepsilon^{*}(\data,\gamma_{0})$, and $\varepsilon_{*}\equiv \varepsilon_{*}(\data,\omega(\cdot),\gamma_{0})$, cf. \eqref{dep.1.1}-\eqref{dep.1} - observe that we incorporated any dependency on $\beta_{0}$ of the above quantities into a dependency on $s$, see Proposition \ref{cdg}. Inserting \rif{dep.2}-\rif{sceltafinale} in \rif{scende} leads to \eqref{exx.2}. 
\end{proof}
In the large potential regime \eqref{lp} the bound for the excess at the $\tau$-scale is automatic, as the proof of the next proposition shows.
\begin{proposition}[Excess decay in the large potential regime]\label{psp.1}
In the setting of Proposition \ref{psp}, and in particular with $\eps^*, \tau \equiv \eps^*, \tau(\data, \gamma_{0})$ determined in \eqref{sceltona}, but assuming \eqref{lp} instead of \eqref{exx}, it holds that
\eqn{exx.3}
$$
\tx{E}_{u}(x_{0},\tau\rr)\leq  c_{*}\pprhoo\,, \qquad c_{*}\equiv c_{*}(\data,\gamma_{0})\geq 1\,.
$$
\end{proposition}
\begin{proof}
We control
$$
\tx{E}_{u}(x_{0},\tau\rr)\stackrel{\eqref{scataildopoff}_1}{\leq} \frac{c}{\tau^{n/2}}\tx{E}_{u}(x_{0},\rr)  \stackleq{lp} \frac{c}{\varepsilon^{*} \tau^{n/2}}\pprhoo
$$
for $c\equiv c(\data,\gamma_{0})$. Setting $c_{*}=:2c/(\varepsilon^{*}\tau^{n/2})$ we obtain \eqref{exx.3}.  
\end{proof}
Combining Propositions \ref{psp} and \ref{psp.1} immediately yields
\begin{proposition}[Excess decay, one scale]\label{psp.2}
In the setting of Proposition \ref{psp} and \ref{psp.1}, if $\tx{E}_{u}(x_{0},\rr)<\varepsilon_{*}$ with $\eps_{*}$ being defined in \eqref{sceltona}, then 
\eqn{exx.333}
$$
\tx{E}_{u}(x_{0},\tau\rr)\leq  \frac{\tau^{\gamma_{0}}}{2}\tx{E}_{u}(x_{0},\rr)+c_{*}\pprhoo
$$
holds with $\tau \equiv \tau (\data, \gamma_0)$ and $c_{*}\equiv c_{*}(\data, \gamma_0)$ that are as in \eqref{sceltona} and \eqref{exx.3}, respectively. 
\end{proposition}

\section{The basic excess decay estimate and flatness improvement}\label{g.ex} In this section we give a first general consequence of Propositions \ref{psp} and \ref{psp.1}; we keep the notation used in Section \ref{sec6}. 
\begin{proposition}\label{cor.1}
Under assumptions \eqref{bs.1}-\eqref{bs.5}, let $u$ be a weak solution to \eqref{nonlocaleqn}, $B_{\rr}(x_{0})\Subset \Omega$ be a ball and $\gamma_{0}$ as in \eqref{gammazero}. There exist $R \equiv R(\data,\gamma_{0}) \geq 1$, $\delta \equiv \delta(\data,\gamma_{0})\in (0,1)$, as determined in \eqref{sceltona}, and a new smallness threshold $\epsb{b}\equiv \epsb{b}(\data,\omega(\cdot), \gamma_{0})\in (0,1)$,  such that if 
$a(\cdot)$ is $\delta$-vanishing in $B_{R\rr}(x_{0})\Subset \Omega$, and if 
\eqn{exx.42}
$$
\tx{E}_{u}(x_{0},\rr)<  \epsb{b}
\quad \mbox{and}\quad \sup_{t\le \rr}\, \pptto< \epsb{b}
$$
hold, then  the excess decay estimate
\eqn{exx.5}
$$
\tx{E}_{u}(x_{0},\sigma) \leq  c_{1}\left(\frac{\sigma}{r}\right)^{\gamma_{0}}\tx{E}_{u}(x_{0},r)+c_{1}\sup_{t\le r}\, \pptto
$$
is satisfied for all $0<\sigma\le  r\leq  \rr$, where $c_{1}\equiv c_{1}(\data,\gamma_{0})$. Moreover, 
\eqn{exx.72}
$$
\sup_{\sigma\le \rr}\tx{E}_{u}(x_{0},\sigma)  \leq c_{2}\epsb{b}
$$
holds with $c_{2}\equiv c_{2}(\data,\gamma_{0})$. Finally, 
\eqn{vmo.c}
$$\lim_{\sigma \to 0} \, \ppsixo =0 \Longrightarrow \lim_{\sigma\to 0}\tx{E}_{u}(x_{0},\sigma)=0\,.$$
The constants $ \epsb{b}$ and $c_1$ satisfy 
\eqn{leduecos}
$$
 \epsb{b} < \frac{\eps_{*}}{4c_*}< \frac{\eps_{*}}{4}\qquad \mbox{and} \qquad c_1 > 4 c_{*}\,,
$$
where $\eps_{*}\in (0,1)$ and $c_*\geq 1$ have been determined in \eqref{sceltona} and \eqref{exx.3}, respectively. 
\end{proposition}
\begin{proof} {\em Step 1.} Let us fix $\gamma_{0}$ such that $0< \gamma_{0} < \min\{2s,1\}$ as in Proposition \ref{psp}. Here we first assume that 
\eqn{exx.4}
$$
\tx{E}_{u}(x_{0},\rr)<\ti{\eps}\qquad \mbox{and}\qquad \sup_{t\le \rr}\, \pptto<\frac{\ti{\eps}}{2c_{*}}
$$
for some positive number $\ti{\eps}\leq \varepsilon_{*}$, where $\varepsilon_{*}\equiv \varepsilon_{*}(\data,\omega(\cdot),\gamma_{0})\in (0,1)$ is the same smallness threshold considered in $\eqref{exx}$, and $c_{*}\equiv c_{*}(\data,\gamma_{0})$ is the constant appearing in \eqref{exx.3}. Under such assumptions we prove that \rif{exx.5} holds for $r=\rr$ and that 
\eqn{exx.7}
$$
\sup_{\sigma\le \rr}\tx{E}_{u}(x_{0},\sigma)< c_{3}\ti{\eps}\,,\qquad c_{3}\equiv c_{3}(\data, \gamma_{0})\,.
$$
As $\tx{E}_{u}(x_{0},\rr)<\eps_{*}$ holds by \rif{exx.4}$_1$, we can apply Proposition \ref{psp.2}, that yields \rif{exx.333}. 
Then \eqref{exx.4} implies $\tx{E}_{u}(x_{0},\tau\rr)<\ti{\eps}$, so in view of $\eqref{exx.4}_{2}$ the same argument can be applied on $B_{\tau\rr}$ and eventually on all scales $B_{\tau^i\rr}$. Specifically, proceeding by induction we see that 
\eqn{together}
$$
\begin{cases}
\tx{E}_{u}(x_{0},\tau^i\rr)<\ti{\eps}\\[3pt]
\displaystyle \tx{E}_{u}(x_{0},\tau^{i+1}\rr)\leq  \frac{\tau^{\gamma_{0}}}{2}\tx{E}_{u}(x_{0},\tau^i\rr)+c_{*}
(\tau^i\rr)^{2s}\nra{f}_{L^\chi(B_{\tau^i\rr}(x_0))}
\end{cases}
$$ 
holds for every $i \in \en_0$. 
Again by induction, the above display implies
\begin{flalign}\label{exx.6}
\tx{E}_{u}(x_{0},\tau^{i+1}\rr)&\leq  \tau^{\gamma_{0}(i+1)}\tx{E}_{u}(x_{0},\rr)+c_{*}\sum_{\kappa=0}^{i}\tau^{\gamma_{0}(i-\kappa)}(\tau^{\kappa}\rr)^{2s}\nra{f}_{L^\chi(B_{\tau^{\kappa}\rr}(x_0))}\nonumber \\
&\leq \tau^{\gamma_{0}(i+1)}\tx{E}_{u}(x_{0},\rr)+\frac{c_{*}}{1-\tau^{\gamma_0}}\sup_{t\le \rr}\, \pptto,
\end{flalign}
for every $i \in \en_0$. 
Next, if $\sigma\in (0, \rr)$ we can find $i \in \en_0$ such that $\tau^{i+1}\rr\le \sigma<\tau^{i}\rr$, and
\eqn{sopra0}
$$
\tx{E}_{u}(x_{0},\sigma)\stackrel{\eqref{scataildopoff}}{\le} \ti{c}\tau^{-n/2}\tx{E}_{u}(x_{0},\tau^{i}\rr)=:  c_{3}\tx{E}_{u}(x_{0},\tau^{i}\rr)\stackrel{\eqref{together}}{<}c_{3} \ti{\eps}
$$
where $c_{3}\equiv c_{3}(\data,\gamma_{0}) := \ti{c}\tau^{-n/2}\geq 1$. Observe that $c_{3}$ is independent of $\ti{\eps}$. This completes the proof of \rif{exx.7}. 
When $i\geq 1$, we continue to estimate from the last display as follows:
\begin{eqnarray}
\tx{E}_{u}(x_{0},\sigma)&\stackrel{\eqref{exx.6}}{\le}&\ti{c}\tau^{-n/2+\gamma_{0}i}\tx{E}_{u}(x_{0},\rr)+\frac{c_{*}\ti{c}\tau^{-n/2}}{1-\tau^{\gamma_0}}\sup_{t\le \rr}\, \pptto\nonumber \\
&\le&\ti{c}\tau^{-n/2-\gamma_{0}}\left(\frac{\sigma}{\rr}\right)^{\gamma_{0}}\tx{E}_{u}(x_{0},\rr)+\frac{c_{*}\ti{c}\tau^{-n/2}}{1-\tau^{\gamma_0}}\sup_{t\le \rr}\, \pptto\nonumber 
\end{eqnarray}
and therefore we conclude with
\eqn{sopra}
$$
\tx{E}_{u}(x_{0},\sigma) \leq c_{1}\left(\frac{\sigma}{\rr}\right)^{\gamma_{0}}\tx{E}_{u}(x_{0},\rr)+c_{1}\sup_{t\le \rr}\, \pptto
$$
for
\eqn{lac1}
$$
c_{1}:=\frac{c_{*}\ti{c}\tau^{-n/2-\gamma_0}}{1-\tau^{\gamma_0}} \Longrightarrow c_1\equiv c_{1}(\data,\gamma_{0})> \max\{4c_*, c_3\}\,.
$$
Finally, the same estimate holds when $i=0$ too
$$
\tx{E}_{u}(x_{0},\sigma) \stackleq{scataildopoff} \ti{c}\tau^{-n/2}\left(\frac{\sigma}{\rr}\right)^{-\gamma_{0}}\left(\frac{\sigma}{\rr}\right)^{\gamma_{0}}\tx{E}_{u}(x_{0},\rr)\leq c_{1}\left(\frac{\sigma}{\rr}\right)^{\gamma_{0}}\tx{E}_{u}(x_{0},\rr)\,.
$$
Observe that $c_{1}$ is independent of $\ti{\eps}$. This, in particular, completes the proof of \eqref{exx.5} in the special case $r=\rr$ provided \rif{exx.4} holds. In fact, the constant $c_1$ found in \rif{lac1} fixes the one in \rif{exx.5} also in the case $r< \rr$, as we are going to see in the next step by choosing a suitable value of $\ti{\eps}$. 

{\em Step 2}. With $c_{1}, c_{3}\equiv c_{1}, c_{3}(\data,\gamma_{0})$ determined in Step 1, see \rif{sopra0} and \rif{lac1}, we now want to prove that choosing 
\eqn{eppi1}
$$
 \epsb{b} := \frac{\varepsilon_{*}}{4c_{3}(c_*+1)} \equiv \frac{\varepsilon_{*}}{c_{2}}\,, \qquad c_{2}:= 4c_{3}(c_*+1)
$$
in \rif{exx.42} yields the assertion of Proposition \ref{cor.1}. Indeed, note that the choice in \rif{eppi1} implies
$$
\tx{E}_{u}(x_{0},\rr)<\frac{\varepsilon_{*}}{c_{3}}\qquad \mbox{and}\qquad \sup_{t\le \rr}\, \pptt<\frac{\varepsilon_{*}}{2c_{*}c_{3}}\,.
$$
This is \rif{exx.4} with $\ti{\eps}\equiv \varepsilon_{*}/c_{3}\leq \varepsilon_{*}$ and therefore Step 1, estimate \rif{exx.7},  yields
$$
\sup_{\sigma\le \rr}\, \tx{E}_{u}(x_{0},\sigma)< c_{3}\ti{\eps}= \eps_{*} = c_{2} \epsb{b}
$$
that is \rif{exx.72}. Moreover, this, together with \rif{exx.42} and the choice in \rif{eppi1}, implies
$$
\tx{E}_{u}(x_{0},r)< \varepsilon_{*} \qquad \mbox{and}\qquad \sup_{t\le r}\, \pptt<\frac{\varepsilon_{*}}{2c_{*}}
$$
whenever $0 < r \leq \rr$, that is exactly the starting condition \rif{exx.4} with $\ti{\eps}=\eps_{*}$, but with $B_{r}$ as a starting ball instead of $B_{\rr}$. 
Step 1, and in particular \rif{sopra} with $r$ replacing $\rr$, now implies \rif{exx.5}. 

{\em Step 3}. We finally prove \rif{vmo.c}. Using \rif{exx.72} in \rif{exx.5}, and letting $\sigma \to 0$, yields
$$
\limsup_{\sigma \to 0}\, \tx{E}_{u}(x_{0},\sigma) \leq  c_{1}\sup_{t\le r}\, \pptto,
$$
and \rif{vmo.c} follows eventually letting $r\to 0$. 
\end{proof}
\subsection{Excess decay for linear systems} The basic excess decay for linear systems of the type \rif{eqweaksol-linear} avoids smallness conditions as in \rif{exx.42} and reads
\begin{proposition}\label{cor.1-linear}
Under assumptions \eqref{bs.1}-\eqref{bs.5}, let $u$ be a weak solution to \eqref{nonlocaleqn}, $B_{\rr}(x_{0})\Subset \Omega$ be a ball and $\gamma_{0}$ as in \eqref{gammazero}. There exist $R \equiv R(\data,\gamma_{0}) \geq 1$, $\delta \equiv \delta(\data,\gamma_{0})\in (0,1)$  such that if 
$a(\cdot)$ is $\delta$-vanishing in $B_{R\rr}(x_{0})\Subset \Omega$,  then  the excess decay estimate
\eqref{exx.5} 
holds as well as \eqref{vmo.c}. 
\end{proposition}
The proof of Proposition \ref{cor.1-linear}, which is identical to - and actually simpler than - the one of Proposition  \ref{cor.1}, and it relies on the fact that when the coefficient matrix $a(\cdot)$ does not depend on the solution then \rif{bs.4} is satisfied with $\omega(\cdot)\equiv 0$. This reflects in that in the proof of Proposition \ref{psp} the first condition in \rif{exx} is not needed and, as a consequence, also the second one is not required to iterate as in \rif{together}.

\section{Regular points}\label{rs.sec} Here we quantify the Hausdorff dimension of the set of points of $\Omega$ on which a smallness condition of type $\tx{E}_{w}(x_{0},\rr)<\varepsilon$ as in $\eqref{exx.4}$ does not take place. 
\begin{proposition}\label{regp1} Let $w$ be a $W^{t,p}_{\loc}(\Omega;\mathbb{R}^N) \cap L^{1}_{2s}$-regular map, with $0<t<1$, $p\geq 2$ and define $\Sigma\subset \Omega$ as
\eqn{sigmone}
$$
\Sigma:=\left\{x \in \Omega \, \colon \, \limsup_{\sigma\to 0}\,  \tx{E}_{w}(x,\sigma)>0\right\}\,.
$$
Then 
\eqn{accade}
$$
\begin{cases}
\ddim(\Sigma)\leq n-pt & \mbox{when}  \ \  pt<n\\
\mbox{$\Sigma$ is empty} &\mbox{when} \ \   pt\geq n\,.
\end{cases}
$$
\end{proposition}
\begin{proof} By a standard exhaustion argument, it is enough to prove the claim under the
stronger assumption $w\in W^{t,p}(\Omega;\mathbb{R}^N) \cap L^{1}_{2s}$. Let us first show that
\begin{flalign}\label{basic-inc}
\Sigma_{1} & :=\left\{x \in \Omega \, \colon \, \limsup_{\sigma\to 0}\,  \tail(w-(w)_{B_{\sigma}(x)};B_{\sigma}(x))>0\right\}\notag \\
&\subset \left\{x \in \Omega \, \colon \, \limsup_{\sigma\to 0}\,  \nra{w-(w)_{B_{\sigma}(x)}}_{L^p(B_{\sigma}(x))}>0\right\}=:\Sigma_0. 
\end{flalign}
For this we consider a generic point $x\in \er^n\setminus \Sigma_0$ and prove that $x\in \er^n\setminus \Sigma_1$, that is 
\eqn{hoppi}
$$
\lim_{\sigma\to 0}\,  \nra{w-(w)_{B_{\sigma}(x)}}_{L^p(B_{\sigma}(x))}=0 \Longrightarrow \lim_{\sigma\to 0}\,  \tail(w-(w)_{B_{\sigma}(x)};B_{\sigma}(x))=0\,.
$$
We take a ball $B_{\rr}(x) \subset \Omega$ and $\sigma \leq \rr$; 
by  \eqref{scatail} and H\"older inequality
\begin{align*}
&\tail(w-(w)_{B_{\sigma}(x)};B_{\sigma}(x))\leq c\left(\frac{\sigma}{\rr}\right)^{2s}\tail(w-(w)_{B_{\rr}(x)};B_{\rr}(x)) \nonumber \\
&\qquad +c\left(\frac{\sigma}{\rr}\right)^{2s}\nra{w-(w)_{B_{\rr}(x)}}_{L^{p}(B_{\rr}(x))}\nonumber +c \int_{\sigma}^{\rr}\left(\frac{\sigma}{\lambda}\right)^{2s}\nra{w-(w)_{B_{\lambda}(x)}}_{L^{p}(B_{\lambda}(x))}\frac{\dlam}{\lambda}\,.
\end{align*}
It is sufficient to show that the last term in the above display converges to zero when $\sigma \to 0$. For this note that we can assume, without loss of generality, that 
\eqn{ostro}
$$
\lim_{\sigma\to 0} \int_{\sigma}^{\rr}\left(\frac{1}{\lambda}\right)^{2s}\nra{w-(w)_{B_{\lambda}(x)}}_{L^{p}(B_{\lambda}(x))}\frac{\dlam}{\lambda}=\infty
$$
otherwise we are done. Note that the limit in the right-hand side of \rif{ostro} always exists as its argument is a non-increasing function of $\sigma$, and that the integrand is a non-negative, continuous function of $\lambda$. Applying de l'Hôpital criterion gives 
\begin{flalign*}
 \lim_{\sigma\to 0}\int_{\sigma}^{\rr}\left(\frac{\sigma}{\lambda}\right)^{2s}\nra{w-(w)_{B_{\lambda}(x)}}_{L^{p}(B_{\lambda}(x))}\frac{\dlam}{\lambda} &= 
   \lim_{\sigma\to 0}\frac{1}{\sigma^{-2s}}\int_{\sigma}^{\rr}\left(\frac{1}{\lambda}\right)^{2s}\nra{w-(w)_{B_{\lambda}(x)}}_{L^{p}(B_{\lambda}(x))}\frac{\dlam}{\lambda}\\
 &=\lim_{\sigma\to 0} \frac{1}{2s}\nra{w-(w)_{B_{\sigma}(x)}}_{L^{p}(B_{\sigma}(x))}=0 
\end{flalign*}
and \rif{hoppi} follows together with \rif{basic-inc}. To continue with the proof note that as $p\geq 2$, by H\"older's inequality and \rif{basic-inc} we have $\Sigma\subset \Sigma_0\cup\Sigma_1 = \Sigma_0$ and therefore to complete the proof we can prove \rif{accade} with $\Sigma$ replaced by $\Sigma_0$. We first consider the case $pt\geq n$ for which it is sufficient to observe that the emptiness of $\Sigma_0$ is an immediate consequence of \rif{poinfrac} (absolute continuity of the integral is needed in the borderline case $pt=n$). It remains to treat the case $pt<n$. For this denote
$$
\Sigma_3 :=\left\{x \in \Omega \, \colon \, \limsup_{\sigma\to 0}\,  \sigma^{t}  \snra{w}_{t,p;B_{\sigma}(x)}>0\right\}
$$
and use \cite[Lemma 4.2]{min03} in order to conclude that $\ddim(\Sigma_3)\leq n-pt$. As \rif{poinfrac} implies $\Sigma_{0} \subset \Sigma_3$, it also follows that $\ddim(\Sigma_{0})\leq n-pt$ and the proof is complete. \end{proof}
\begin{proposition}\label{regp2} Let $w\in W^{s,2}_{\loc}(\Omega;\mathbb{R}^N) \cap L^{1}_{2s}$ and let $\mathcal R\geq 1$ be an arbitrary number. The set 
$$
\ti{\Omega}_{\eps, \rr,  \mathcal R}(w):= \left\{x\in \Omega\, \colon  \, \tx{E}_{w}(x,\rr)<\varepsilon \  \mbox{and}\ B_{\mathcal R\rr}(x)\Subset \Omega\right\}
$$
is open for every choice of $\eps, \rr >0$. As a consequence,  for every choice of $\eps, \rr_{0}>0$ the set  
 \begin{flalign}\Omega_{\eps, \rr_0 , \mathcal R}(w) & :=  \bigcup_{0< \rr <  \rr_0} \ti{\Omega}_{\eps, \rr , \mathcal R}(w)\notag \\
 & =
\left\{x \in \Omega \, \colon \, \exists\,  B_{\mathcal R\rr_{x}}(x)\Subset \Omega\  \mbox{such that} \  \tx{E}_{w}(x, \rr_x) < \eps\  \mbox{and} \  \rr_x < \rr_0  \right\} \label{regolar}
 \end{flalign}
is open too. Moreover, if $w\in W^{t,p}_{\loc}(\Omega;\mathbb{R}^N)$  with $0<t<1$ and $p\geq 2$, then 
\eqn{regolar2}
$$ 
\begin{cases}
\, \ddim(\Omega \setminus \Omega_{\eps, \rr_0, \mathcal R}(w))\leq n-pt & \mbox{when}  \ \  pt<n\\[3pt]
\, \mbox{$\Omega \setminus \Omega_{\eps, \rr_0, \mathcal R}(w)$ is empty} &\mbox{when} \ \   pt\geq n\,.
\end{cases}
$$
\end{proposition}
\begin{proof}
The openness of $\ti{\Omega}_{\eps, \rr , \mathcal R}(w)$ stems from the continuity of $x\mapsto \tx{E}_{w}(x,\rr)$, that in turn follows from \rif{affini2}$_1$. Let $\Sigma$ be defined in \rif{sigmone}; note that since $\Omega$ is open, for every $x\in \Omega$ there exists $\ti{\rr}$  such that $B_{\mathcal R\rr}(x)\Subset \Omega$ provided $\rr < \ti{\rr}$. It follows that, if $x \not\in \Omega_{\eps, \rr_0,  \mathcal R}(w)$, then
$
\liminf_{\sigma\to 0}\,  \tx{E}_{w}(x,\sigma)\geq \eps
$
so that $\Omega \setminus \Omega_{\eps, \rr_0 , \mathcal R}(w)\subset \Sigma$ and \rif{regolar2} follows from \rif{accade}.  \end{proof}
\begin{proposition}\label{regp4} Let $u$ be a weak solution to \eqref{nonlocaleqn}. Then \eqref{regolar2}  holds with $w\equiv u$, for numbers $t\in (s,1)$ and $p\geq 2$, both depending on $\data$. 
\end{proposition}
\begin{proof}
This is a direct consequence of Theorem \ref{lamaggiore} and Proposition \ref{regp2}. \end{proof}

\section{\texorpdfstring{Review of the basic parameters and the regular set $\Omega_{u}$}{Review of the basic parameters and the regular set Omega u}}\label{parametri}
It is convenient to gather here some of the relevant parameters that will be used in the forthcoming proofs; these quantities are fixed up to the choice of the parameter $\gamma_{0} \in (0,\min\{2s,1\})$, and are 
\eqn{sceltona2}
$$
\begin{cases}
\,   \mbox{$\tau \equiv \tau(\data,\gamma_{0}) \in (0,1/64)$ is determined in \rif{sceltona}}\\[3pt]
\, \mbox{$\varepsilon_{*}\equiv \varepsilon_{*}(\data,\omega(\cdot),\gamma_{0})\in (0,1)$ is determined  in \rif{sceltona}}\\[3pt]
\, \mbox{$R \equiv R(\data,\gamma_{0}) \geq 1$ is determined  in \rif{sceltona}}\\[3pt]
\, \mbox{$\delta \equiv \delta(\data,\gamma_{0})\in (0,1)$ is determined  in \rif{sceltona}}\\[3pt]
\,  \mbox{$c_{*}\equiv c_{*}(\data,\gamma_{0})\geq 1$ is determined  in Proposition \ref{psp.1}}\\[3pt]
\,  \mbox{$\epsb{b} \equiv \epsb{b} (\data,\omega(\cdot),\gamma_{0}) \in (0,1)$  is determined  in Proposition \ref{cor.1}}\\[3pt]
\,  \mbox{$c_{1}, c_{2}\equiv c_{1}, c_{2}(\data,\gamma_{0})\geq 1$ are determined  in Proposition \ref{cor.1}}\,.
\end{cases}
$$
Moreover, in all the proofs of Theorems \ref{ureg1}-\ref{ureg5} we shall assume that $a(\cdot)$ is $(\delta,r_{0})$-\textnormal{BMO} regular for a certain radius $r_0$ (in fact in Theorem \ref{ureg3} we only prescribe that $a(\cdot)$ is  $\delta$-vanishing at the fixed ball $B_{R\rr}(x_{0})\subset \Omega$). Without loss of generality we shall assume that $r_0$ is the same in every case; the number $\delta$ will always be the one considered in \rif{sceltona} and \rif{sceltona2}. 
\subsection{Regular set}\label{regolarino2} We define the regular set $\Omega_{u}$ as follows. We consider a condition of the type 
\eqn{5,6}
$$
\sup_{B_{\sigma}\Subset \Omega, \sigma \leq \sigma_{0}} \ppsi<\epsb{b}
$$
already encountered in \rif{f.bmo} and the radius 
\eqn{raggiob0}
$$\rhob{b} := \min\{r_0/R, \sigma_{0},1\}$$
in such a way  that $a(\cdot)$ is $\delta$-vanishing in $B_{R\rr}(x_{0})$ whenever $\rr \leq \rhob{b}$ and $B_{R\rr}(x_{0})\Subset \Omega$. 
This said, in view of \rif{regolar}-\rif{regolar2}, we define 
\eqn{singolare}
$$
\Omega_{u}:= \Omega_{\epsb{b}, \rhob{b}, R}(u)= \left\{x \in \Omega \, \colon \, \exists\,  B_{R\rr_{x}}(x)\Subset \Omega\  \mbox{such that} \  \tx{E}_{u}(x, \rr_x) < \epsb{b}\  \mbox{and} \  \rr_x < \rhob{b}  \right\}\,.
$$
By Propositions \ref{regp2} and \ref{regp4} the set  $\Omega_{u}$ is open and 
\eqn{riduzione}
$$\ddim(\Omega \setminus \Omega_{u})\leq n-2s - \gamma \,, \qquad \mbox{for some positive} \ \gamma \equiv \gamma (\data)\in (0,n-2s]\,.$$ 
Specifically, we take $\gamma= \min\{n-2s,(p-2)s+(t-s)p\}$. 
\subsection{Reference neighbourhood}\label{refne}  For every $x_{0}\in \Omega_{u}$ we can find a positive radius $\rr_{x_{0}}<\rhob{b}\leq 1$ such that $B_{R\rr_{x_{0}}}(x_{0})\Subset \Omega$ and $\tx{E}_{u}(x_{0},\rr_{x_{0}})<\epsb{b}$ and therefore $x_{0}\in \tilde{\Omega}_{\epsb{b},\rr_{x_{0}}, R}(u)$. This last set is open by Proposition \ref{regp2}, so there exists a {\em reference neighbourhood} 
\eqn{referenza}
$$B_{r_{x_{0}}}(x_{0})\Subset \tilde{\Omega}_{\epsb{b},\rr_{x_{0}},R}(u)\subset \Omega_{u}$$ such that
\eqn{5,7}
$$
\tx{E}_{u}(x,\rr_{x_{0}})<\epsb{b} \quad \mbox{and} \quad 
B_{R\rr_{x_{0}}}(x)\Subset \Omega\,, \quad  \forall \  x \in B_{r_{x_{0}}}(x_{0})\,.
$$
For later use, see Section \ref{secbelow} below, we can without loss of generality assume that 
\eqn{raggio4}
$$
r_{x_{0}} \leq \rr_{x_{0}}/8\,.
$$
Note that with the definitions above the regular set $\Omega_{u}$ still depends on two parameters, that is $\gamma_{0}$ as in \rif{gammazero} (via $\epsb{b}$ and $R$) and $\sigma_{0}$ appearing in \rif{5,6}. These will be eventually chosen in the forthcoming proofs. Of course, $\Omega_u$ also depends on $\data$, which is of course a fixed set of parameters.
 
\section{Partial BMO/VMO regularity: Proof of Theorems \ref{ureg1}-\ref{ureg2}} 
\begin{proof}[Proof of Theorem \ref{ureg1}] Referring to Section \ref{parametri} for the relevant definitions, we take $\gamma_{0}=s$, $\sigma_{0}$ from \rif{f.bmo} and for the rest consider the parameters in \rif{sceltona2}. This fixes $\epsb{b}$ in \rif{f.bmo} and the regular set $\Omega_{u}$ in \rif{singolare}. 
By \rif{f.bmo} and \rif{5,7}, and recalling the notation established in Section \ref{refne}, we can apply Proposition \ref{cor.1} in each ball $B_{\rr_{x_{0}}}(x)$, $x\in B_{r_{x_{0}}}(x_{0})$, $x_0\in \Omega_{u}$ so that 
  \rif{exx.72} gives  
\eqn{recalling}
$$
\sup_{x \in B_{r_{x_{0}}}(x_{0}), \sigma \leq \rr_{x_{0}}}\tx{E}_{u}(x,\sigma)\le c_{2}\epsb{b}\,.
$$
This and a standard covering argument imply that $u\in \textnormal{BMO}_{\loc}(\Omega_{u};\er^N)$. Of course, here and in the following proofs, the Hausdorff dimension estimate of the singular set $\Omega \setminus \Omega_{u}$ follows from \rif{riduzione}. \end{proof}
\begin{proof}[Proof of Theorem \ref{ureg2}] Here the proof is essentially a refinement of the one for \rif{vmo.c} from Proposition \ref{cor.1} in that we ultimately want to show that the limit $ \lim_{\sigma\to 0}\tx{E}_{u}(x_{0},\sigma)=0$ in \rif{vmo.c} is locally uniform in the regular set $\Omega_{u}$. Still taking $\gamma_0=s$, we reconsider and complement the proof of Theorem \ref{ureg1} in view of assumption \rif{f.vmo}. By \rif{f.vmo} we find $\sigma_{0}\equiv \sigma_{0}(\data,\omega(\cdot),f(\cdot))>0$ such that \rif{5,6} is verified. 
This allows us to define the radius $\rhob{b}$ as in \rif{raggiob0}, that is 
\eqn{raggiob}
$$\rhob{b} \equiv \rhob{b} (\data,\omega(\cdot),f(\cdot), r_0)
:= \min\{r_0/R, \sigma_{0},1\}\,.$$
Accordingly, $\Omega_{u}$ is defined as in \rif{singolare} but with the current choice of $\sigma_{0}$. In particular, we can use \rif{recalling}  and the content of Proposition \ref{cor.1}, that is
\eqn{decadenza}
$$
\tx{E}_{u}(x,\sigma) \leq  c_{1}c_{2}\left(\frac{\sigma}{r}\right)^{s}\epsb{b}+c_{1}\sup_{t\le r}\, \ppttx,
$$
whenever $0<\sigma\le  r\leq  \rr_{x_0}$,  $x\in B_{r_{x_0}}(x_0)$, $x_{0}\in \Omega_{u}$; recall that $c_1, c_2$ only depend on $\data$.  
To proceed, fix $\texttt{s} \in (0,\epsb{b})$ and determine $\sigma_1 \in (0, \rhob{b})$ such that 
$$
c_{1}\sup_{B_{t}\Subset \Omega, t \leq \sigma_1}\, \pptt <  \frac{\texttt{s}}{2}\,.
$$ 
Note that $\sigma_1 \equiv \sigma_1 (\data,\omega(\cdot),f(\cdot), r_0, \texttt{s})$. 
For each $x_{0}\in \Omega_u$ we let $\ti{\rr}_{x_{0}}:=\min \{\rr_{x_{0}}, \sigma_1\}\leq 1$, and define $\sigma_{x_{0}}< \ti{\rr}_{x_{0}}$  by requiring that
\eqn{require}
$$
 c_{1}c_{2}\left(\frac{\sigma_{x_{0}}}{\ti{\rr}_{x_{0}}}\right)^{\gamma_{0}(=s)} \epsb{b} < \frac{\texttt{s}}{2}\,.
 $$
Matching the last two displays and \rif{decadenza} with $r\equiv \ti{\rr}_{x_{0}}$ we have 
\eqn{aquella}
$$\sigma \leq \sigma_{x_{0}}\equiv  \sigma_{x_{0}}(\data,\omega(\cdot),f(\cdot), r_0, \texttt{s}, \rr_{x_{0}}) \Longrightarrow \tx{E}_{u}(x,\sigma) < \texttt{s}, \quad \mbox{for every $x\in B_{r_{x_0}}(x_0)$}\,. $$
Note that \rif{aquella} provides a condition that in fact does not depend on the point $x$, and it is therefore uniform in the reference neighbourhood $B_{r_{x_0}}(x_0)$ defined in \rif{referenza}, i.e., 
\eqn{convy}
$$
\lim_{\sigma \to 0}  \tx{E}_{u}(x,\sigma) =0,  \quad \mbox{uniformly with respect to $x\in B_{r_{x_0}}(x_0)$}
$$
that immediately implies \rif{vazero}. 
By a standard covering argument this implies that $u\in \textnormal{VMO}_{\loc}(\Omega_{u};\er^N)$. (Just note that the covering argument uses a finite covering of reference neighbourhoods $B_{r_{x_{0}}}(x_{0})$ thereby determining $\sigma_{x_{0}}$ in \rif{require} corresponding to the smallest value of $\rr_{x_{0}}$, and therefore of  $\ti{\rr}_{x_{0}}$). 
\end{proof}  
\section{Bounds on the oscillations: Proof of Theorem \ref{ureg3}} 
Here everything is pointwise and we do not have to determine a regular set $\Omega_{u}$. Needless to say, we can assume that the right-hand side in \rif{oscosc} is finite and that therefore, by \rif{azero}
\eqn{azerodd}
$$ 
 \lim_{\sigma \to 0}\, \mathbf{I}^{f}_{2s,\chi}(x_{0},\sigma)=0\,.
$$ 
In the proof of Theorem \ref{ureg3} we assume that the hypotheses, and in particular \rif{osc.t0}, are verified in the ball $B_{R\rr}(x_{0})\subset \Omega$ with the choice of the parameters $\delta, R, \epsb{b}$ determined in Section \ref{parametri} with $\gamma_{0}= s$; we correspondingly determine all the remaining parameters in \rif{sceltona2} keeping $\gamma_0=s$.  
For any integer $i\geq -1$ we set 
$
\rr_{i}:=\tau^{i+1}\rr$, $B_{i}:=B_{\rr_{i}}(x_{0})
$ (note that $\rr_{-1}=\rr$ and $\rr_{0}=\tau\rr$), 
thus defining a sequence of shrinking balls $B_{i+1}\subset B_{i}\subset \cdots\subset B_{-1}\equiv B_{\rr}(x_{0})$. 
All the forthcoming balls will be centred at $x_{0}$. Note that 
\begin{flalign}
\notag \sum_{j\geq i}\ppjj& \le\frac{1}{\log(1/\tau)\tau^{n/\chi}}  \sum_{j\geq i}\int_{\rr_{j}}^{\rr_{j-1}}\pplam\frac{\dlam}{\lambda} \\
&\le\frac{1}{\tau^{n}}\mathbf{I}^{f}_{2s,\chi}(x_{0},\rr_{i-1})\,,\label{5,3}
\end{flalign}
for all $i\ge 0$. 
Via interpolation this gives
\eqn{5.7.2.0}
$$
\ppsi\le \frac{1}{\tau^{2n}}\mathbf{I}^{f}_{2s,\chi}(x_{0},\rr_{i-1})\qquad \mbox{for all} \ \ 0<\sigma\le \rr_{i}, \quad i\geq 0,
$$
so that \rif{azerodd} implies
\eqn{5.7.2.2}
$$
\lim_{\sigma\to 0}\, \ppsi=0\,.
$$
By \rif{osc.t0}, that is \eqref{exx.42} from Proposition \ref{cor.1}, and \rif{5.7.2.2}, we can apply \rif{vmo.c}; this yields
\eqn{5,0}
$$
\lim_{\sigma\to 0}\tx{E}_{u}(x_{0},\sigma)=0\,.
$$
Recalling that $\epsb{b}< \eps_*/(4c_*)$ by \rif{leduecos}, we see that \rif{exx.4} are verified with $\tilde \eps = \eps_{*}$ and therefore \rif{together} implies
$$
\begin{cases}
\, \tx{E}_{u}(x_{0},\rr_{j})<\eps_{*}\\[3pt]
 \displaystyle \, \tx{E}_{u}(x_{0},\rr_{j+1}) \leq   \tfrac{1}{2}\tx{E}_{u}(x_{0},\rr_{j})+c_{*}\rr_j^{2s}\nra{f}_{L^\chi(B_{j})}
\end{cases}
$$
for every integer $j\geq -1$. Summing up the above inequalities yields 
$$
\sum_{j=i}^{k+1}\tx{E}_{u}(x_{0},\rr_{j}) \leq \tx{E}_{u}(x_{0},\rr_{i})+   \frac{1}{2}\sum_{j=i}^{k}\tx{E}_{u}(x_{0},\rr_{j}) +c_{*}\sum_{j=i}^{k}\ppjj
$$
for any choice of integers $k\geq i\geq -1$ 
and therefore 
\eqn{therefore}
$$
\sum_{j=i}^{k+1}\tx{E}_{u}(x_{0},\rr_{j}) \leq   2 \tx{E}_{u}(x_{0},\rr_{i}) +2c_{*}\sum_{j=i}^{k}\ppjj\,.
$$
For integers $k>i\geq 0$ we now have 
\begin{align}
& \snr{(u)_{B_{k+1}}-(u)_{B_{i+1}}} \leq \sum_{j=i+1}^{k}\snr{(u)_{B_{j+1}}-(u)_{B_{j}}}\leq   \frac 1{\tau^{n/2}}\sum_{j=i+1}^{k}\nra{u-(u)_{B_{j}}}_{L^{2}(B_{j})}\nonumber \\
&\qquad \qquad \stackrel{\eqref{therefore}}{\le}  c\, \tx{E}_{u}(x_{0},\rr_{i}) +c\sum_{j\geq i}\ppjj
\stackrel{\eqref{5,3}}{\le} c\, \tx{E}_{u}(x_{0},\rr_{i})+c\, \mathbf{I}^{f}_{2s,\chi}(x_{0},\rr_{i-1})\label{5,4}
\end{align}
with $c\equiv c(\data)$; we have used that $\tau \equiv \tau(\data)$. The content of the previous display, together with \rif{azerodd} and \eqref{5,0} implies that $\{(u)_{B_{i}}\}$ is a Cauchy sequence, thus there exists $ u_0\in \er^N$ such that 
$(u)_{B_{i}} \to u_0$ in $\er^N$. 
Using this last piece of information, and letting $k\to \infty$ in \eqref{5,4}, we get
\eqn{5,4.1}
$$
\snr{u_0-(u)_{B_{i+1}}}\le c\, \tx{E}_{u}(x_{0},\rr_{i})+c\, \mathbf{I}^{f}_{2s,\chi}(x_{0},\rr_{i-1}),
$$
for $c\equiv c(\data)$ and every $i \in \en_0$. Let us now consider $\sigma \leq \rr$, and find an integer $i_{\sigma} \geq 1$ such that $\rr_{i_{\sigma}-1}<\sigma\le \rr_{i_{\sigma}-2}$. We then estimate, by means of \rif{scataildopoff}$_1$ and \rif{5,4.1}
\begin{flalign*}
\snr{u_0-(u)_{B_{\sigma}}}&\le \snr{u_0-(u)_{B_{i_{\sigma}+1}}}+\snr{(u)_{B_{i_{\sigma}+1}}-(u)_{B_{\sigma}}}\nonumber \\
&\le c\, \tx{E}_{u}(x_{0}, \rr_{i_{\sigma}})+c\, \mathbf{I}^{f}_{2s,\chi}(x_{0},\rr_{i_{\sigma}-1}) +c\tau^{-3n/2}\nra{u-(u)_{B_{\sigma}}}_{L^{2}(B_{\sigma})}\nonumber \\
&\le c\tau^{-2n}\, \tx{E}_{u}(x_{0}, \sigma)+c\, \mathbf{I}^{f}_{2s,\chi}(x_{0},\sigma)\,.
\end{flalign*}
This is exactly \rif{oscosc}, as $\tau \equiv \tau(\data)$. Moreover, 
thanks to \rif{azerodd} and \eqref{5,0}, letting $\sigma \to 0$ in the above display leads to \rif{lebp} with $u(x_0):= u_0$.

\section{Partial continuity: Proof of Theorem \ref{ureg4}} 
In the spirit of classical potential theory, here continuity of $u$ in $\Omega_u$ means that we shall first identify the solution $u$ with its precise representative \rif{lebp}, that will be shown to exist for every point $x_0\in \Omega_u$. Then, in a second step, we shall show that the precise representative is continuous so that the solution $u$ is continuous in $\Omega_u$ in the sense that it admits a continuous representative. To proceed, we again take $\gamma_{0}=s$ in Section \ref{parametri}. By \eqref{unic} and \rif{5.7.2.0}, 
\eqn{locale}
$$
\lim_{\sigma\to 0}\, \ppsix=0 \quad \mbox{uniformly with respect to $x$}.
$$
This is the main assumption \rif{f.vmo} in Theorem \ref{ureg2}. Therefore we determine the regular set $\Omega_{u}$ as in the proof of Theorem \ref{ureg2} and can use  \rif{convy} for every regular point $x_0\in \Omega_u$, where $B_{r_{x_0}}(x_0)$ denotes the canonical reference neighbourhood of $x_0$ in the sense of Section \ref{refne}. In particular, determining $\rhob{b}$ as in \rif{raggiob}, with $4r_{x_0} \leq \rr_{x_0}< \rhob{b}$ we have 
$$
 \tx{E}_u(x,\rr_{x_0})< \epsb{b} \quad \mbox{and}\quad \sup_{\sigma \leq \rr_{x_0}} \ppsix<\epsb{b}\,,\quad  \forall\  x \in B_{r_{x_0}}(x_0), \ \  \forall \  x_{0}\in \Omega_{u}
 $$
 and  that $a(\cdot)$ is $\delta$-vanishing in $B_{R\rr_{x_0}}(x)\subset \Omega$. 
These are in fact the assumptions \rif{osc.t0} from Theorem \ref{ureg3}, that imply 
\eqn{gsopra}
$$
\snr{u(x)-(u)_{B_{\sigma}(x)}}\le c\, \tx{E}_{u}(x,\sigma)+c\, \mathbf{I}^{f}_{2s,\chi}(x,\sigma)\,,\quad  \forall\  \sigma \leq \rr_{x_0},\quad  \forall\  x \in B_{r_{x_0}}(x_0), \ \  \forall \  x_{0}\in \Omega_{u}
$$
where
$
u(x):=\lim_{\sigma\to 0}(u)_{B_{\sigma}(x)}. 
$
Note that we have in fact proved in Theorem \ref{ureg3} that the precise representative of $u$ exists at every point $x \in B_{r_{x_0}}(x_0)$.  
We now consider a fixed regular point $x_0\in \Omega_u$. By \rif{unic} and \eqref{convy}, the right-hand side in \eqref{gsopra} converges to zero uniformly in $B_{r_{x_{0}}}(x_{0})$, thus proving that the continuous maps $B_{r_{x_{0}}}(x_{0})\ni x\mapsto (u)_{B_{\sigma}(x)}$ uniformly converge to $u$ as $\sigma\to 0$ in the ball $B_{r_{x_{0}}}(x_{0})$. This proves the continuity of $u$ in $B_{r_{x_{0}}}(x_{0})$ and therefore in $\Omega_{u}$, being $x_{0}$ an arbitrary point of $\Omega_{u}$.

\section{Partial H\"older regularity: Proof of Theorems \ref{ureg0} and \ref{ureg5}}\label{secbelow}

\subsection{Proof of Theorem \ref{ureg5}} We shall actually first derive the stronger estimate \rif{holdest-anc}. With $\beta$ as in \rif{asbeta}, we take $\gamma_{0}\equiv \gamma_{0}(\beta)$ as
\eqn{tantigamma}
$$\beta < \gamma_{0} := \frac{\min\left\{2s,1\right\} +\beta}{2} <\min\left\{2s,1\right\}\,.$$
This fixes  $\gamma_{0}$ in Section \ref{parametri}.  Applying \rif{marhol} with $p=\chi$ and  $d= n/(2s-\beta)$ gives 
\eqn{marci}
$$
\sigma^{2s}\nra{f}_{L^{\chi}(B_{\sigma})}  
  \lesssim_{n,s,\chi, \beta}\sigma^{\beta}\nr{f}_{\mathcal{M}^{\frac{n}{2s-\beta}}(B_{\sigma})}
$$
that in fact implies \rif{f.vmo}. It follows the existence of $\sigma_{0}$ such that \rif{5,6} holds and eventually of the radius $\rhob{b}$ defined in \rif{raggiob}, that now depends on $\data,\omega(\cdot),\beta,f(\cdot), r_0$. With $\gamma_{0}$ and $\rhob{b}$ fixed, we consider the set $\Omega_{u}$ described in \rif{singolare}, with \rif{5,7}-\rif{raggio4} being in force. By \eqref{5,6} and \eqref{5,7}  we are able to apply Proposition \ref{cor.1} on every ball of the type $B_{\rr_{x_{0}}}(x)$ for $x\in B_{r_{x_{0}}}(x_{0})$, that is 
\eqn{5,8}
$$
 \tx{E}_{u}(x,\texttt{t} r) \le c_{1}\texttt{t} ^{\gamma_{0}}\tx{E}_{u}(x,r)+c_{1}\sup_{t\le r}\ppttx
$$
that holds for all $0<r\le \rr_{x_{0}}$, $0 < \texttt{t} < 1$, where, after the choice of $\gamma_{0}$ in \rif{tantigamma}, it is $c_1\equiv c_1(\data, \beta)$. 
We now introduce the following fractional maximal type operators (recall \rif{defimax} and \rif{massimale}):
\eqn{Nsharp0} 
$$ 
\begin{cases}
\, \displaystyle  \MMM^{2s-\beta}_{\rr, \chi}(f)(x):= \sup_{0< \sigma \leq \rr }\, \sigma^{2s-\beta}\nra{f}_{L^\chi(B_{\sigma}(x))} \\[8pt]
 \, \displaystyle \MMM^{\#,\beta}_{\rr^{\star},\rr}(x)\equiv  \MMM^{\#, \beta}_{\rr^{\star},\rr}(u)(x):= \sup_{\rr^{\star}\leq  \sigma \leq \rr }\,\sigma^{-\beta} \tx{E}_{u}(x,\sigma)\\[8pt]
 \,  \displaystyle \MMM^{\#,\beta}_{\rr}(x)\equiv \MMM^{\#, \beta}_{ \rr}(u)(x):=  \sup_{0< \sigma \leq \rr }\, \sigma^{-\beta} \tx{E}_{u}(x,\sigma)
 \end{cases}
 \quad 0 < \rr^{\star} \leq \rr \leq \rr_{x_{0}}\,.
$$
Note that all the three quantities in the above display are non-decreasing functions of $\rr$. Moreover, using \rif{scataildopoff}$_1$, it easily follows that
\eqn{trisca}
$$
 \MMM^{\#,\beta}_{\rr^{\star},\rr}(x)\lesssim_{n,s} (\rr^{\star})^{-\beta}\left(\frac{\rr^{\star}}{\rr}\right)^{-n/2} \tx{E}_{u}(x,\rr)\,.
$$
Inequality \rif{5,8} now becomes 
$$
 \tx{E}_{u}(x,\texttt{t} r) \le c_{1}\texttt{t} ^{\gamma_{0}}\tx{E}_{u}(x,r)+c_1r^{\beta} \MMM^{2s-\beta}_{r, \chi}(f)(x)
$$
that now holds whenever $0< r \leq  \rr_{x_{0}}$.
Next, we determine $\texttt{t}\equiv \texttt{t} (\data,\beta)\in (0,1)$ such that $c_{1}\texttt{t} ^{\gamma_{0}-\beta}\leq 1/2$, so that 
$$
 \tx{E}_{u}(x,\texttt{t} r) \le  \frac{\texttt{t} ^{\beta}}{2}\tx{E}_{u}(x,r)+c_1r^{\beta}\MMM^{2s-\beta}_{r, \chi}(f)(x)\,.
$$
 Multiplying this last inequality by $(\texttt{t} r)^{-\beta}$, and using the resulting expression with $\rr_{x_{0}}/2^k \leq r \leq 
 \rr_{x_{0}}/2$, $2\leq k \in \en$, and finally keeping \rif{Nsharp0} in mind, we find
\begin{flalign*}
 \notag\MMM^{\#,\beta}_{\texttt{t} \rr_{x_{0}}/2^k, \texttt{t} \rr_{x_{0}}/2}(x) & \leq  \frac 12 
  \MMM^{\#,\beta}_{\rr_{x_{0}}/2^k,  \rr_{x_{0}}/2}(x) + \frac{c_1}{\texttt{t}^{\beta}}\MMM^{2s-\beta}_{\rr_{x_{0}}/2, \chi}(f)(x)\\
  & \leq \frac 12 
  \MMM^{\#,\beta}_{\texttt{t} \rr_{x_{0}}/2^k, \texttt{t} \rr_{x_{0}}/2}(x)+\frac 12\MMM^{\#,\beta}_{\texttt{t}  \rr_{x_{0}}/2,  \rr_{x_{0}}/2}(x)  + \frac{c_1}{\texttt{t}^{\beta}}\MMM^{2s-\beta}_{\rr_{x_{0}}/2, \chi}(f)(x)\,.
\end{flalign*}
By means of \eqref{trisca} we have 
\eqn{trisca22} 
$$
\MMM^{\#,\beta}_{\texttt{t}  \rr_{x_{0}}/2,  \rr_{x_{0}}/2}(x) \leq 
\frac{c}{\texttt{t}^{n/2+\beta}} \rr_{x_{0}}^{-\beta}\tx{E}_{u}(x,\rr_{x_{0}}/2)
$$
so that
\begin{flalign*}
 \MMM^{\#,\beta}_{\texttt{t} \rr_{x_{0}}/2^k, \texttt{t} \rr_{x_{0}}/2}(x) & \leq  \frac 12  \MMM^{\#,\beta}_{\texttt{t} \rr_{x_{0}}/2^k, \texttt{t} \rr_{x_{0}}/2}(x) \\
 & \quad +  \frac{c}{\texttt{t}^{n/2+\beta}} \rr_{x_{0}}^{-\beta}\tx{E}_{u}(x,\rr_{x_{0}}/2)+ \frac{c_1}{\texttt{t}^{\beta}}\MMM^{2s-\beta}_{\rr_{x_{0}}/2, \chi}(f)(x)
\end{flalign*}
holds, and reabsorbing terms yields
$$
 \MMM^{\#,\beta}_{\texttt{t} \rr_{x_{0}}/2^k, \texttt{t} \rr_{x_{0}}/2}(x) \leq c \rr_{x_{0}}^{-\beta}\tx{E}_{u}(x,\rr_{x_{0}}/2)+ c\MMM^{2s-\beta}_{\rr_{x_{0}}/2, \chi}(f)(x)
$$
that easily implies, again using \rif{trisca22}
$$
 \MMM^{\#,\beta}_{\texttt{t} \rr_{x_{0}}/2^k, \rr_{x_{0}}/2}(x) \leq c \rr_{x_{0}}^{-\beta}\tx{E}_{u}(x,\rr_{x_{0}}/2)+ c\MMM^{2s-\beta}_{\rr_{x_{0}}/2, \chi}(f)(x)\,.
$$
Letting $k \to \infty$ in the above inequality finally leads to 
$$
 \MMM^{\#,\beta}_{\rr_{x_{0}}/2}(x) \leq c \rr_{x_{0}}^{-\beta}\tx{E}_{u}(x,\rr_{x_{0}}/2)+ c\MMM^{2s-\beta}_{\rr_{x_{0}}/2, \chi}(f)(x)
$$
where  $c\equiv  c (\data,\beta)$ and for every $x \in B_{r_{x_0}}(x_0)$. Note that we have several times used that $\texttt{t}$ is a function of $\data$ and $\beta$. Recalling the definition in \rif{massimale} the above inequality implies
\eqn{biscotto}
$$
\textnormal{M}^{\#,\beta}_{\rr_{x_{0}}/2}(u;x) 
\leq c \rr_{x_{0}}^{-\beta}\tx{E}_{u}(x,\rr_{x_{0}}/2)+ c\MMM^{2s-\beta}_{\rr_{x_{0}}/2, \chi}(f)(x)\,.
$$
Recall that by \rif{raggio4} we have $r_{x_{0}}\leq \rr_{x_{0}}/8$ and therefore \rif{scataildopoff}$_2$ gives $\tx{E}_{u}(x,\rr_{x_{0}}/2) \lesssim_{n,s} \tx{E}_{u}(x_{0},\rr_{x_{0}})$. This and \rif{biscotto} imply
$$
\nr{\textnormal{M}^{\#,\beta}_{\rr_{x_{0}}/2}(u;\cdot) }_{L^\infty(B_{r_{x_{0}}}(x_{0}))}\leq c  \rr_{x_{0}}^{-\beta}\tx{E}_{u}(x_{0},\rr_{x_{0}})+c \nr{\MMM^{2s-\beta}_{\rr_{x_{0}}/2, \chi}(f) }_{L^\infty(B_{r_{x_{0}}}(x_{0}))}\,.
$$
Again using that $r_{x_0}\leq \rr_{x_0}/8$ inequality \rif{stimalfa2}  yields \rif{holdest-anc} with $c\equiv c (\data, \beta)$. Moreover 
\rif{marci}  implies that 
$$\nr{\MMM^{2s-\beta}_{\rr_{x_{0}}/2, \chi}(f)}_{L^{\infty}(B_{r_{x_{0}}}(x_{0}))}\leq \nr{\MMM^{2s-\beta}_{\rr_{x_{0}}/2, \chi}(f)}_{L^{\infty}(B_{\rr_{x_{0}}/2}(x_{0}))}  \lesssim_{n,s,\chi, \beta}
 \nr{f}_{\mathcal M^{\frac{n}{2s-\beta}}(B_{\rr_{x_{0}}}(x_{0}))}
$$
so that \rif{holdest} directly follows from \rif{holdest-anc}.

\vspace{2mm}

\begin{remark}\label{isingolari}{\em In Theorems \ref{ureg1}-\ref{ureg2} and \ref{ureg4} the regular set $\Omega_{u}$ has been defined in \rif{ilsingolare}, with $\gamma_{0}=s$ and $\sigma_0$ defined via the verification of \rif{5,6}, accordingly to the discussion made in Section \ref{parametri}. In Theorem \ref{ureg5} we still take $\Omega_{u}$ as in \rif{ilsingolare}, but with $\gamma_{0}\equiv \gamma_{0}(\beta)$ coming from \rif{tantigamma}. }\end{remark}

\subsection{Proof of Theorem \ref{ureg0}}\label{lasecagg} Thanks to \rif{hoppi}, we actually have 
\eqn{identity}
$$
\Omega_{u}= \{x \in \Omega\, \colon \, \lim_{\sigma\to 0 } \nra{u-(u)_{B_{\sigma}(x)}}_{L^{2}(B_{\sigma}(x))} =0\}=\{x \in \Omega\, \colon \, \lim_{\sigma\to 0 } \tx{E}_{u}(x,\sigma) =0\}\,.
$$
Moreover, note that the assumed uniform continuity of $a(\cdot)$ implies that for every $\delta>0$ there exists a positive radius $r_0\equiv r_0(\delta, \tilde{\omega}(\cdot), \Lambda)$ such that $a(\cdot)$ is $(\delta,r_{0})$-\textnormal{BMO} and therefore we are in a setting similar to the one of Theorem \ref{ureg5}. Fix $\beta$ such that $0<\beta <\min\{1, 2s\}$ and determine $\gamma_0$ as in \rif{tantigamma}; this allows us to fix $R, \delta \equiv R, \delta(\data, \beta)$ and $\epsb{b}\equiv \epsb{b}(\data, \tilde{\omega}(\cdot), \beta)$ by means of \rif{sceltona2}. Accordingly, we determine $r_0\equiv r_0(\data, \gamma_0, \tilde{\omega}(\cdot))$ relative to the value of $\delta$ we have just determined. We then consider the set $ \Omega_{\epsb{b}, \rhob{b}, R}(u)$ defined in \eqref{singolare} with 
$\rhob{b} := \min\{r_0/R, \sigma_{0},1\}$ as in \rif{raggiob0}, with $r_0$  we have just determined, and $\sigma_0:=[\epsb{b}/(1+\nr{f}_{L^\infty})]^{1/(2s)}$ via 
$$\sup_{B_{\sigma}\Subset \Omega, \sigma  \leq \sigma_{0}}\ppsi\leq  \sigma_0^{2s}\nr{f}_{L^\infty}<\epsb{b}$$ 
so that \rif{5,6} is verified.  
Note that \rif{identity} implies
$\Omega_u\subset \Omega_{\epsb{b}, \rhob{b}, R}(u).$ 
Therefore, with any $x_0 \in  \Omega_{u}$, proceeding as in the proof of Theorem \ref{ureg5} we deduce that 
$$
\rr_{x_{0}}^{\beta}[u]_{0,\beta;B_{r_{x_{0}}}(x_{0})}\lesssim \tx{E}_{u}(x_{0},\rr_{x_{0}})+\rr_{x_{0}}^{2s} \nr{f}_{L^{\infty}(B_{\rr_{x_{0}}}(x_{0}))}
$$
exactly as for \rif{holdest}. This obviously implies that 
$$\lim_{\sigma\to 0 } \nra{u-(u)_{B_{\sigma}(y)}}_{L^{2}(B_{\sigma}(y))} =0, \quad \mbox{for every $y \in B_{r_{x_{0}}}(x_{0})$}.
$$ We have proved that every point of $\Omega_u$ is a regular point and that $\Omega_u$ is open, that is, we have proved \rif{partclassic}$_1$, recalling that $0<\beta <\min\{1, 2s\}$ was arbitrary. In order to prove \rif{partclassic}$_2$ we again use Theorem \ref{lamaggiore} that yields $u \in W^{t,p}_{\loc}(\Omega;\er^N)$, for some $t\equiv t(\data)\in (s,1)$, $p\equiv p(\data)>2$, therefore we can apply  \cite[Lemma 4.2]{min03} that yields $\ddim(\Omega \setminus \Omega_{u})\leq n-2s - \gamma$  with $\gamma= \min\{n-2s,(p-2)s+(t-s)p\}$ and the proof is complete.

 \section{Partial Lipschitz continuity via potentials and Theorem \ref{gradreg1}}\label{final} 
In this section we prove Theorem \ref{gradreg1}. The proof goes along Sections \ref{identfinal}-\ref{sceltabeta} and involves arguments that will be useful for the proofs of Theorems \ref{gradreg2}-\ref{gradreg3} too. From now on, and for the rest of the paper, we shall always assume that $s \in (1/2,1)$ together with \rif{modco1}-\rif{modco2}. In particular, we shall always assume that $\alpha\in (0,2s-1)$.  
\subsection{Identification of the regular set $\Omega_{u}$}\label{identfinal}   Let us fix the main setting. We keep the choice made in \rif{sceltona2}, with $\epsb{b} \equiv \epsb{b} (\data,\alpha,\gamma_{0})$ in view of \rif{modco1}; in the following the dependence on $\omega(\cdot)$ will be replaced by dependence on $\alpha$. We first observe that, with $\delta \equiv \delta(\data,\gamma_{0})\in (0,1)$ determined in \rif{sceltona2}, assumption \eqref{modco2} clearly implies that $a(\cdot)$ is $(\delta,r_{0})$-\textnormal{BMO} for $r_{0}\equiv r_{0}(\data,\gamma_{0}, \omega_{\star}(\cdot))$ small enough; recall that $\omega_{\star}(\cdot)$ was defined in \rif{modco2}. Moreover, arguing as in \rif{5,3}-\rif{5.7.2.0}, we find
\eqn{maggioras}
$$
 \sigma^{2s-1}\nra{f}_{L^\chi(B_{\sigma}(x))}\lesssim \mathbf{I}^{f}_{2s-1,\chi}(x, 2\sigma)\,, \quad \mbox{whenever $B_{\sigma}(x)\subset \er^n$ and $\sigma \leq 1/2$}
$$
that together with \rif{assilipanc} implies \rif{locale} via
$$
 \sigma^{2s}\nra{f}_{L^\chi(B_{\sigma}(x))} \lesssim \left[\sup_{x\in \Omega}\, \mathbf{I}^{f}_{2s-1,\chi}(x, 1)\right]\sigma\,.
$$
This allows us to determine the radius $\sigma_{0}>0$ from \rif{5,6} as described in Section \ref{parametri}, towards the identification of the regular set $\Omega_{u}$.  
Indeed, we go back to $\Omega_{u}$ defined in \rif{singolare}, with $r_0, \sigma_{0}$ just determined; these in turn determine $\rhob{b}\equiv \rhob{b}(\data, \gamma_{0},  \omega_{\star}(\cdot), f)$ in \rif{raggiob0}, while $\epsb{b}\equiv \epsb{b}(\data, \gamma_{0})$ has been determined in \rif{sceltona2}. Observe $\Omega_{u}$ is at this stage determined essentially up to the choice of $\gamma_{0}$, that we here take as in \rif{tantigamma}, with $\beta $ to be determined later on. Therefore in $\epsb{b}$ and $\rhob{b}$ the dependence on $\gamma_{0}$ can be replaced with the dependence on $\beta$. Accordingly, $\Omega_{u}$ is determined up to the choice of $\beta$ such that $ 0< \beta <\min\left\{2s,1\right\}=1 $. Later on, in Section \ref{sceltabeta}, we shall finally choose one value of $\beta$, i.e., we shall take $\beta=1/2$.

\subsection{Fixing the reference ball $B_{r_{x_0}}(x_0)$, with estimates}\label{eidos}
Fix $x_{0}\in \Omega_{u}$, so that \rif{5,7} and \rif{raggio4} hold as described in Section \ref{parametri}. According to \rif{holdest-anc}, we have that 
$$
\snr{u(\ti{x})-u(\ti{y})}\le c\, \tx{E}_{u}(x_{0},\rr_{x_{0}})\left(\frac{\snr{\ti{x}-\ti{y}}}{\rr_{x_{0}}}\right)^{\beta}+c\nr{\MMM^{2s-\beta}_{\rr_{x_{0}}/2, \chi}(f)}_{L^{\infty}(B_{\rr_{x_{0}}/2}(x_{0}))}\snr{\ti{x}-\ti{y}}^{\beta}
$$
holds whenever $\ti{x}, \ti{y} \in B_{r_{x_{0}}}(x_{0})$ and $0< \beta <\min\left\{2s,1\right\} $. Recall that the definition of the maximal operator $\MMM^{2s-\beta}_{\rr_{x_{0}}/2, \chi}(f)$ is given in \rif{Nsharp0}$_1$. Note that \rif{maggioras} readily implies
$$
\nr{\MMM^{2s-\beta}_{\rr_{x_{0}}/2, \chi}(f)}_{L^{\infty}(B_{\rr_{x_{0}}/2}(x_{0}))} \lesssim  \nr{\mathbf{I}^{f}_{2s-1,\chi}(\cdot, 1)}_{L^\infty(\er^n)}
$$
and the right-hand side is finite thanks to \rif{assilipanc}. 
The content of the last two displays yields that
\eqn{holdest22}
$$
\snr{u(\ti{x})-u(\ti{y})}\le c H  \snr{\ti{x}-\ti{y}}^{\beta}\,, \quad \mbox{for all $\ti{x}, \ti{y} \in B_{r_{x_{0}}}(x_{0})$}
$$
and $0< \beta <\min\left\{2s,1\right\} $, where $c\equiv c (\data, \beta)$ and, recalling that $\rr_{x_{0}}^{-\beta} \leq \rr_{x_{0}}^{-1}$, also
\eqn{holdest222}
$$
H \equiv H(B_{\rr_{x_0}}(x_0)):= \rr_{x_{0}}^{-1}\nra{u}_{L^{2}(B_{\rr_{x_0}}(x_0))}+ \rr_{x_{0}}^{-1}\tail(u;B_{\rr_{x_0}}(x_0)) +\nr{\mathbf{I}^{f}_{2s-1,\chi}(\cdot, 1)}_{L^\infty(\er^n)}\,.
$$
\subsection{A comparison estimate}\label{sezi1} In the following we fix an arbitrary affine map $\ell$. Note that $u_\ell:=u-\ell$ still belongs to $W^{s,2}_{\loc}(\Omega;\er^N)\cap L^{1}_{2s}$, and this follows from \rif{tritri}$_2$. We then fix $x$ and a ball $B_{\rr_{\star}}(x)$ such that 
\eqn{recallradius}
$$
\begin{cases}
x \in B_{r_{x_{0}}/2}(x_{0})\\[7pt]
0< \rr_{\star} \leq \frac{r_{x_{0}}}{8} \Longrightarrow B_{4\rr_{\star}}(x)\subset B_{r_{x_0}}(x_0)\Subset \Omega\,.
\end{cases}
$$
Note that $\rr_{\star}\leq 1/64$ as from Section \ref{parametri} and \rif{raggio4} we have $r_{x_{0}}\leq 1/8$. A further restriction on the size $ \rr_{\star}$ will be taken later on, in \rif{meetme} below, making $\rr_{\star}$ a function of the objects $\data,\alpha,\beta,H,\omega_{\star}(\cdot)$. We consider a ball $B_{\rr}\equiv B_\rr(x)$, where 
\eqn{ilraggio}
$$\rr:=\frac{\rr_{\star}}{2^{N_\rr} }\quad \mbox{for some positive integer $N_\rr \geq 2$}$$ and set  
\eqn{aoao}
$$a_{0}:=a(x,x,u(x),u(x))\,.$$
In the rest of the proof, unless otherwise specified, all the balls will be centred at $x$, therefore we shall often abbreviate $B_r\equiv B_r(x)$, $r>0$.    
With $0 < \rr^\star\leq  \rr_{\star}$, we use the sharp maximal-type operators already considered in \rif{Nsharp0}
\eqn{Nsharp} 
$$
\begin{cases}
 \, \displaystyle \MMM^{\#,1}_{\rr^{\star},\rr_{\star}}(x)\equiv  \MMM^{\#,1}_{\rr^{\star},\rr_{\star}}(u)(x):= \sup_{\rr^{\star} \leq \sigma \leq \rr_{\star}}\, \sigma^{-1} \tx{E}_{u}(x,\sigma)\\[7pt]
 \,  \displaystyle \MMM^{\#,1}_{\rr_{\star}}(x)\equiv \MMM^{\#,1}_{\rr_{\star}}(u)(x):=  \sup_{0< \sigma \leq    \rr_{\star}}\, \sigma^{-1} \tx{E}_{u}(x,\sigma)\,.
 \end{cases}
$$
 By means of Lemma \ref{esiste} we now define $h\in \mathbb{X}^{s,2}_{u_\ell}(B_{\rr},B_{\rr_{\star}})$ as the unique solution to\footnote{Here observe again that every affine map belongs to $W^{s,2}_{\loc}(\Omega;\mathbb{R}^N) \cap L^{1}_{2s}$ when $s>1/2$.}
 \eqn{solvesh}
$$
\begin{cases}
\ -\mathcal{L}_{a_{0}}h=0\quad &\mbox{in} \ \ B_\rr\\
\ h= u_\ell\quad &\mbox{in} \ \ \mathbb{R}^{n}\setminus B_\rr\,.
\end{cases}
$$ 
Since $s>1/2$, every affine map belongs to $L^1_{2s}$ and is locally in
$W^{s,2}$. Hence $u_{\ell}$ is an admissible exterior datum. 
By Proposition \ref{phr} $h$ is smooth in $B_{\rr}$. 
As in \cite[Remark 3.4]{kns}, $\ell$ is a weak solution of $-\mathcal{L}_{a_{0}} \ell =0$ in $\er^n$ in the sense of Definition \ref{def:weaksol} so that 
\eqn{comparazione}
$$w \equiv w_{B_{\rr}, \ell}:=u_\ell- h \in \mathbb{X}^{s,2}_{0}(B_{\rr},B_{\rr_{\star}})$$ and solves
$$
\begin{cases}
\ -\mathcal{L}_{a_{0}}w=-\mathcal{L}_{a_{0}-a} u + f\quad &\mbox{in} \ \ B_\rr\\
\ w= 0\quad &\mbox{in} \ \ \mathbb{R}^{n}\setminus B_\rr
\end{cases}
$$ 
in the weak sense corresponding to Definition \ref{defidir}, i.e., 
		\begin{flalign} 
		\notag
			\noindent &\int_{\mathbb{R}^n} \int_{\mathbb{R}^n} \langle a_0(w(\ti{x})-w(\ti{y})),\varphi(\ti{x})-\varphi(\ti{y}) \rangle \frac{\dxyt}{|\ti{x}-\ti{y}|^{n+2s}} \\
			& \  =\int_{\mathbb{R}^n} \int_{\mathbb{R}^n} \langle (a_0-a(\ti{x},\ti{y}, u(\ti{x}), u(\ti{y})))(u(\ti{x})-u(\ti{y})),\varphi(\ti{x})-\varphi(\ti{y}) \rangle \frac{\dxyt}{|\ti{x}-\ti{y}|^{n+2s}}  + \int_{\Omega} \langle f, \varphi \rangle\d\ti{x} \label{weq}
		\end{flalign}
holds for every $\varphi \in W^{s,2}(\er^n;\er^N)$ such that $\varphi \equiv 0$ in $\er^n\setminus B_\rr$.
\begin{lemma}\label{comparisonestimate} With $w$ as in \eqref{comparazione}, the comparison estimate 
\eqn{compestuvx}
$$
		\nra{w}_{L^2(B_{\rr})} \leq c \rr \mathbf{R}(x,\rr)
$$
holds with $c\equiv c(\data,\alpha,\beta,H)$, where
\eqn{compestuvx2} 
$$
 \mathbf{R}(x,\rr):= \bigg [ \sum_{i=0}^{N_\rr} 2^{(1-2s)i} \omega_{\star} (2^i \rr) + \omega_{\star} (2\rr)+\left (\frac{\rr}{\rr_{\star}} \right )^{\alpha \beta} \bigg ] \MMM^{\#,1}_{\rr,\rr_{\star}}(x) + \rr^{2s-1} \nra{f}_{L^{\chi}(B_{4\rr})}\,.
$$	  
\end{lemma}
\begin{proof}
 According to the notation fixed in \rif{ballnotation}, for any $r \in (0,\rr_{\star}]$, we set $\mathcal B_{r}:= B_r(x)\times B_r(x)=B_r\times B_r$ and denote 
   $$||a-a_{0}||_{L^\infty(\mathcal B_{r})}:= \sup_{\ti{x},\ti{y} \in B_r(x)} |a(\ti{x},\ti{y},u(\ti{x}),u(\ti{y}))-a_{0}|\,.$$
We shall also denote $
B_{j} = B_{2^j \rr}(x)$ and  
$ \mathcal B_{j} =\mathcal B_{2^j \rr}=B_{2^j \rr}\times B_{2^j \rr}
$ for every integer $j\leq N_\rr$. Note that
\eqn{catena}
$$B_{\rr}=B_{0} \Subset B_1 \Subset \ldots \Subset  B_{N_{\rr}}=B_{\rr_{\star}}\,.$$ 
Thanks to \rif{bs.4} and \rif{modco2}, and recalling \rif{aoao}, for every $r \in (0,\rr_{\star}]$ we have
\begin{flalign}
	||a-a_{0}||_{L^\infty(\mathcal B_{r})} & \leq  \sup_{\ti{x},\ti{y} \in B_r(x)} |a(\ti{x},\ti{y},u(\ti{x}),u(\ti{y}))-a(\ti{x},\ti{y},u(x),u(x))| \nonumber\\
	& \quad \  + \sup_{\ti{x},\ti{y} \in B_r(x)} |a(\ti{x},\ti{y},u(x),u(x))-a_{0}| \nonumber\\
	& \leq  \Lambda \omega \left (2\sup_{B_{r}(x)} |u-u(x)| \right ) + \Lambda\omega_{\star} (r) \leq 2\Lambda \left (\osc_{B_{r}(x)} u \right )^\alpha + \Lambda\omega_{\star} (r)\,.\notag
	\end{flalign}
Appealing to \eqref{holdest22}, as $B_r(x)\subset B_{r_{x_0}}(x_0)$ by \rif{recallradius}$_2$, we conclude with
\eqn{osci}
$$
||a-a_{0}||_{L^\infty(\mathcal B_{r})}  \leq cr^{\alpha \beta} +c \omega_{\star} (r)\,, \qquad r \in (0,\rr_{\star}]\,,
$$ 
where the constant $c$ depends on $\data,\beta$ and $H$ accordingly to \rif{holdest22}-\rif{holdest222}. For later use we record the following direct consequence of \rif{osci} that holds for any number $\sigma \in [0,1]$:
\begin{flalign}
	\sum_{i=0}^{N_\rr} 2^{(\sigma-2s)i} ||a-a_{0}||_{L^\infty(\mathcal B_{i}) } & \leq c \sum_{i \geq 0} 2^{(\sigma+\alpha \beta-2s)i} \rr^{\alpha \beta} + c \sum_{i=0}^{N_\rr} 2^{(\sigma-2s)i} \omega_{\star} (2^i \rr) \nonumber\\
	& = c \rr^{\alpha \beta} + c\sum_{i=0}^{N_\rr} 2^{(\sigma-2s)i} \omega_{\star} (2^i \rr)
	\label{osci2}
\end{flalign}
again with $c\equiv c(\data,\alpha,\beta,H)$ where $\rr$ is in \rif{ilraggio}; here we used that $1+\alpha\beta -2s< 0$ by \rif{modco1}.
To proceed, testing \eqref{weq} by $\varphi\equiv w$, we find
\begin{eqnarray}
	[w]_{s,2;\er^n}^2& \stackrel{\eqref{bs.1}}{\le}& \Lambda \int_{\mathbb{R}^n} \int_{\mathbb{R}^n} \frac{ \langle a_{0} (w(\ti{x})-w(\ti{y})),w(\ti{x})-w(\ti{y}) \rangle}{|\ti{x}-\ti{y}|^{n+2s}}\dyxt \nonumber \\
	& \stackrel{\eqref{weq}}{=}& \Lambda \int_{\mathbb{R}^n} \int_{\mathbb{R}^n} \langle (a_{0}-a(\ti{x},\ti{y},u(\ti{x}),u(\ti{y})))(u(\ti{x})-u(\ti{y})),w(\ti{x})-w(\ti{y}) \rangle \frac{\dyxt}{|\ti{x}-\ti{y}|^{n+2s}} \nonumber \\
	&&+ \Lambda \int_{B_\rr} \langle f, w\rangle\d\ti{x} \nonumber \\
	& \leq& \Lambda \underbrace{\int_{B_{2\rr}} \int_{B_{2\rr}} |a(\ti{x},\ti{y},u(\ti{x}),u(\ti{y}))-a_{0}| \frac{|u(\ti{x})-u(\ti{y})||w(\ti{x})-w(\ti{y})|}{|\ti{x}-\ti{y}|^{n+2s}}\dyxt}_{=: \textnormal{I}_1} \nonumber \\
	&& + 2 \Lambda \underbrace{\int_{B_{\rr}} \int_{\mathbb{R}^n \setminus B_{2\rr}} |a(\ti{x},\ti{y},u(\ti{x}),u(\ti{y}))-a_{0}| \frac{|u(\ti{x})-(u)_{B_\rr}||w(\ti{x})|}{|\ti{x}-\ti{y}|^{n+2s}}\dyxt}_{=: \textnormal{I}_2} \nonumber \\
	&& + 2 \Lambda \underbrace{\int_{B_{\rr}} \int_{\mathbb{R}^n \setminus B_{2\rr}} |a(\ti{x},\ti{y},u(\ti{x}),u(\ti{y}))-a_{0}| \frac{|u(\ti{y})-(u)_{B_\rr}||w(\ti{x})|}{|\ti{x}-\ti{y}|^{n+2s}}\dyxt}_{=: \textnormal{I}_3}\nonumber  \\
	&& + \Lambda \underbrace{\int_{B_{\rr}} \snr{f}\snr{w}\d\ti{x}}_{=:\textnormal{I}_4} \label{combi}
\end{eqnarray}
and  proceed with the estimation of the terms $\textnormal{I}_1, \textnormal{I}_2, \textnormal{I}_3$ and $\textnormal{I}_4$. Thanks to \eqref{cacce}, and recalling that by \rif{ilraggio} we have $4\rr \le \rr_{\star}$, for $\textnormal{I}_1$ we have
\begin{flalign}
	\textnormal{I}_{1} & \leq c \rr^{n/2} ||a-a_{0}||_{L^\infty( \mathcal B_{2\rr})} \snra{u}_{s,2;B_{2\rr}} [w]_{s,2;\er^n} \nonumber\\
	& \leq  c \rr^{n/2} ||a-a_{0}||_{L^\infty( \mathcal B_{2\rr})}  
	\left(\rr^{-s}  \tx{E}_{u}(x,4\rr)+\rr^{s}\nra{f}_{L^{2_*}(B_{4\rr})}\right)[w]_{s,2;\er^n} \nonumber\\
		& \leq c \rr^{n/2} ||a-a_{0}||_{L^\infty( \mathcal B_{2\rr})} \left (\rr^{1-s} \MMM^{\#,1}_{\rr,\rr_{\star}}(x) + \rr^{s} \nra{f}_{L^{\chi}(B_{4\rr})} \right ) [w]_{s,2;\er^n}\nonumber \\
		& \leq c \rr^{n/2+1-s} ||a-a_{0}||_{L^\infty(\mathcal B_{2\rr})}  \MMM^{\#,1}_{\rr,\rr_{\star}}(x) [w]_{s,2;\er^n} 
		+ c \rr^{n/2+s} \nra{f}_{L^{\chi}(B_{4\rr})}   [w]_{s,2;\er^n}.\nonumber 
\end{flalign}
Therefore,  \rif{osci} implies
\eqn{sti1}
$$
\textnormal{I}_{1}  \leq c  \rr^{n/2+1-s} \left( \rr^{\alpha \beta} +\omega_{\star} (2\rr) \right) \MMM^{\#,1}_{\rr,\rr_{\star}}(x) [w]_{s,2;\er^n} + c\rr^{n/2+s} \nra{f}_{L^{\chi}(B_{4\rr})}   [w]_{s,2;\er^n}
$$ 
with $c\equiv c(\data,\alpha,\beta,H)$.
For $\textnormal{I}_{2}$ note that 
\eqn{bas}
$$
\begin{cases}
\, \mbox{$\snr{\ti{y}-\ti{x}} \approx \snr{\ti{y}-x} \gtrsim 2^{i}\rr$ when $\ti{y}\in \er^n\setminus B_{i-1}$ and $\ti{x} \in B_{\rr}$ for $i \geq 2$}\\[3pt]
\,  \nr{w}_{L^2(B_\rr)} \lesssim_{n,s}  \rr^{s}  [w]_{s,2;\er^n}
\end{cases}
$$ the second one being a consequence of \rif{poincf}. 
Also using the content of the above display, we then have 
\begin{flalign*}
	\textnormal{I}_{2} & \leq  2\Lambda \sum_{i=2}^{N_\rr} \int_{B_\rr} \int_{B_{i} \setminus B_{i-1}} |a(\ti{x},\ti{y},u(\ti{x}),u(\ti{y}))-a_{0}| \frac{|u(\ti{x})-(u)_{B_\rr}||w(\ti{x})|}{|\ti{y}-x|^{n+2s}}\dyxt \nonumber\\
	&\qquad  + 2\Lambda \int_{B_\rr} \int_{\mathbb{R}^n \setminus B_{\rr_{\star}}} |a(\ti{x},\ti{y},u(\ti{x}),u(\ti{y}))-a_{0}| \frac{|u(\ti{x})-(u)_{B_\rr}||w(\ti{x})|}{|\ti{y}-x|^{n+2s}}\dyxt \nonumber\\
	& \leq c  \left [ \rr^{-2s} \sum_{i=2}^{N_\rr} 2^{-2si} \nr{a-a_{0}}_{L^\infty( \mathcal B_{i})} + \rr_{\star}^{-2s} \right] \int_{B_\rr} |u-(u)_{B_\rr}||w|\d\ti{x} \nonumber\\
	& \leq c \rr^{n/2-2s} \left[ \sum_{i=0}^{N_\rr} 2^{-2si} \nr{a-a_{0}}_{L^\infty( \mathcal B_{i})} +  \left (\frac{\rr}{\rr_{\star}} \right)^{2s} \right]\nra{u-(u)_{B_\rr}}_{L^2(B_\rr)} \nr{w}_{L^2(B_\rr)} \nonumber\\
	& \leq c \rr^{n/2+1-s} \left[ \sum_{i=0}^{N_\rr} 2^{-2si} \nr{a-a_{0}}_{L^\infty( \mathcal B_{i})} + \left (\frac{\rr}{\rr_{\star}} \right)^{2s} \right] \MMM^{\#,1}_{\rr,\rr_{\star}}(x) [w]_{s,2;\er^n}\,.
\end{flalign*}
By finally using \rif{osci2} with $\sigma=0$ we conclude with 
\eqn{sti2}
$$  \textnormal{I}_{2}\leq c \rr^{n/2+1-s} \left[ \sum_{i=0}^{N_\rr} 2^{-2si} \omega_{\star} (2^i \rr) +\rr^{\alpha \beta} + \left (\frac{\rr}{\rr_{\star}} \right )^{2s}\right] \MMM^{\#,1}_{\rr,\rr_{\star}}(x) [w]_{s,2;\er^n}
$$
where $c\equiv c(\data,\alpha,\beta,H)$. 
For $\textnormal{I}_3$, again using \rif{bas} we estimate
\begin{flalign}
	\textnormal{I}_{3} & \leq 2\Lambda  \sum_{i=2}^{N_\rr} \int_{B_\rr} \int_{B_{i} \setminus B_{i-1}} |a(\ti{x},\ti{y},u(\ti{x}),u(\ti{y}))-a_{0}| \frac{|u(\ti{y})-(u)_{B_\rr}||w(\ti{x})|}{|\ti{y}-x|^{n+2s}}\dyxt \nonumber\\
	&\qquad  + 2\Lambda  \int_{B_\rr} \int_{\mathbb{R}^n \setminus B_{\rr_{\star}}} |a(\ti{x},\ti{y},u(\ti{x}),u(\ti{y}))-a_{0}| \frac{|u(\ti{y})-(u)_{B_\rr}||w(\ti{x})|}{|\ti{y}-x|^{n+2s}}\dyxt \nonumber\\
		& \leq c\rr^{-2s} \sum_{i=2}^{N_\rr} 2^{-2si} \nr{a-a_{0}}_{L^\infty( \mathcal B_{i})} \nra{u-(u)_{B_\rr}}_{L^1(B_{i})} \nr{w}_{L^1(B_\rr)} \nonumber\\
	&\qquad  + c \rr_{\star}^{-2s} \tail(u-(u)_{B_\rr};B_{\rr_{\star}}) \nr{w}_{L^1(B_\rr)} \nonumber\\
	& \leq c\rr^{n/2-2s}  \nr{w}_{L^2(B_\rr)}\sum_{i=2}^{N_\rr} 2^{-2si} \nr{a-a_{0}}_{L^\infty( \mathcal B_{i})} \nra{u-(u)_{B_\rr}}_{L^2(B_{i})}\notag \\
	&\qquad  + c \rr^{n/2}\rr_{\star}^{-2s} \tail(u-(u)_{B_\rr};B_{\rr_{\star}}) \nr{w}_{L^2(B_\rr)} \,.\label{prev0}
\end{flalign} 
In turn, for $2\leq i \leq N_{\rr}$ and recalling \rif{catena}, we have 
\begin{flalign*}
\nra{u-(u)_{B_\rr}}_{L^2(B_{i})} & \leq \nra{u-(u)_{B_{i}}}_{L^2(B_{i})} + \sum_{k=1}^{i} \snr{(u)_{B_{k}}-(u)_{B_{k-1}}}\\
& \leq \nra{u-(u)_{B_{i}}}_{L^2(B_{i})} + 2^n\sum_{k=1}^{i}\nra{u-(u)_{B_{k}}}_{L^2(B_{k})}\\
& \leq 2^{n+1}\sum_{k=1}^{i}\nra{u-(u)_{B_{k}}}_{L^2(B_{k})}
\end{flalign*}
and applying this with $i=N_{\rr}$, we find 
\begin{flalign*}
\tail(u-(u)_{B_\rr};B_{\rr_{\star}}) & \lesssim_{n,s}\, \tail(u-(u)_{B_{\rr_{\star}}};B_{\rr_{\star}})  + \snr{(u)_{B_{\rr_{\star}}}-(u)_{B_\rr}}\\
 & \lesssim_{n,s}\, \tail(u-(u)_{B_{\rr_{\star}}};B_{\rr_{\star}})  + \nra{u-(u)_{B_\rr}}_{L^2(B_{N_{\rr}})} \\
& \lesssim_{n,s} \, \tail(u-(u)_{B_{\rr_{\star}}};B_{\rr_{\star}}) + \sum_{k=1}^{N_{\rr}}\nra{u-(u)_{B_{k}}}_{L^2(B_{k})}\,.
\end{flalign*}
Inserting this last two estimates in \rif{prev0} and using \rif{bas}$_2$ yields
\begin{flalign}
	\textnormal{I}_{3} & \leq  c\rr^{n/2-2s}  \nr{w}_{L^2(B_\rr)}\sum_{i=2}^{N_\rr} \sum_{k=1}^{i} 2^{-2si} \nr{a-a_{0}}_{L^\infty( \mathcal B_{i})}\nra{u-(u)_{B_{k}}}_{L^2(B_{k})}\notag\\
	&\qquad  + c \rr^{n/2}\rr_{\star}^{-2s}\nr{w}_{L^2(B_\rr)} \sum_{k=1}^{N_{\rr}}\nra{u-(u)_{B_{k}}}_{L^2(B_{k})} \nonumber\\
&\qquad  + c \rr^{n/2}\rr_{\star}^{-2s}  \nr{w}_{L^2(B_\rr)}\tail(u-(u)_{B_{\rr_{\star}}};B_{\rr_{\star}}) \nonumber\\ 
& \leq  c\rr^{n/2+1-s}\sum_{i=2}^{N_\rr} \sum_{k=1}^{i} 2^{k-2si} \nr{a-a_{0}}_{L^\infty( \mathcal B_{i})}\MMM^{\#,1}_{\rr,\rr_{\star}}(x) [w]_{s,2;\er^n}\notag \\
&\qquad  + c \rr^{n/2+s}\rr_{\star}^{-2s} \left(\rr \sum_{k=1}^{N_{\rr}}2^k+\rr_{\star} \right)\MMM^{\#,1}_{\rr,\rr_{\star}}(x)[w]_{s,2;\er^n}\nonumber\\
& \leq c \rr^{n/2+1-s} \left[  \sum_{i=2}^{N_\rr} 2^{(1-2s)i} \nr{a-a_{0}}_{L^\infty( \mathcal B_{i})} + \left (\frac{\rr}{\rr_{\star}} \right )^{2s-1}\right] \MMM^{\#,1}_{\rr,\rr_{\star}}(x) [w]_{s,2;\er^n}\,. \label{ulti}
\end{flalign} 
Note that here we have used the obvious relations 
$$
\rr \sum_{k=1}^{N_{\rr}}2^k \approx 2^{N_{\rr}} \rr \stackrel{\eqref{ilraggio}}{=} \rr_{\star} \quad \mbox{and} \quad 
\sum_{k=1}^{i} 2^{k} \approx 2^i\,.
$$
In order to estimate the sum appearing in the last line of \rif{ulti} we use \rif{osci2} with $\sigma=1$, so that we conclude with 
\eqn{sti3}
$$
\textnormal{I}_{3}   \leq c \rr^{n/2+1-s} \left[ \sum_{i=0}^{N_\rr} 2^{(1-2s)i} \omega_{\star} (2^i \rr) +\rr^{\alpha \beta} + \left (\frac{\rr}{\rr_{\star}} \right )^{2s-1}\right] \MMM^{\#,1}_{\rr,\rr_{\star}}(x) [w]_{s,2;\er^n}\,.
$$
Finally, recalling that $2^*= 2n/(n-2s)$ and $2_*=2n/(n+2s)$ are conjugate exponents, for $\textnormal{I}_4$ we have 
\begin{eqnarray}
\notag	\textnormal{I}_4 & \leq  &c \nr{f}_{L^{2_{*}}(B_\rr)} \nr{w}_{L^{2^{*}}(\mathbb{R}^n)}\\ &\stackleq{poincf2}  &c \rr^{n/2+s} \nra{f}_{L^{2_*}(B_\rr)} [w]_{s,2;\er^n}\notag \\ &\stackleq{bs.5}  &c \rr^{n/2+s} \nra{f}_{L^{\chi}(B_\rr)} [w]_{s,2;\er^n}\,.
\label{sti4}
\end{eqnarray}
Combining the content of \rif{combi}, \rif{sti1}, \rif{sti2}, \rif{sti3} and \rif{sti4} yields
\begin{flalign}
	\notag 	\rr^{-n/2}[w]_{s,2;\er^{n}} & \leq c \rr^{1-s} \left[\sum_{i=0}^{N_\rr} 2^{(1-2s)i} \omega_{\star} (2^i \rr) + \omega_{\star} (2\rr)+\rr^{\alpha \beta} + \left (\frac{\rr}{\rr_{\star}} \right )^{2s-1}+ \left (\frac{\rr}{\rr_{\star}} \right )^{2s}\right]\MMM^{\#,1}_{\rr,\rr_{\star}}(x)\\
		& \quad  + c \rr^{s} \nra{f}_{L^{\chi}(B_{4\rr})}\,.\notag 
\end{flalign}
Finally using \rif{poincf} and that $2s-1> \alpha  > \alpha \beta$ by \rif{modco1}, we arrive at \rif{compestuvx} and the Lemma is proved. \end{proof}

\subsection{First-order excess decay}\label{sezi2}  We fix 
\eqn{intero}
$$
\mbox{$\texttt{t}=1/2^{m}$ with $m\geq 4$ being an integer to be determined}.
$$ 
We consider the affine map 
\eqn{defffie}
$$\ell_{x}(\ti{x}):=D h(x)(\ti{x}-x) + h(x) \qquad \tilde{x}\in \er^n\,,$$ where $h$ has been defined in \rif{solvesh}. 
Note that Proposition \ref{phr} and Mean Value Theorem provide
\eqn{eq:higherreg}
$$
\begin{cases}
\,\displaystyle   \nra{h-\ell_{x}}_{L^2(B_{\lambda}) }\leq c  \lambda^2   \nr{D^2h}_{L^\infty(B_{\rr/2}) }
 \\[5pt] \displaystyle   \qquad \leq c \left(\frac{\lambda}{\rr}\right)^2\left(\nra{h}_{L^{2}(B_{\rr})}+\tail(h;B_{\rr})\right), \quad 
 0 < \lambda \leq \frac{\rr}{2} \\[10pt]
\,  \displaystyle    \snr{D  h(x)} \rr + \snr{ h(x)}\leq c \nra{h}_{L^{2}(B_{\rr})}+c\, \tail(h;B_{\rr})
 \end{cases}
$$
where $c\equiv c (n,N,s,\Lambda)$. We shall also use
 \eqn{eq:higherreg22}
$$
\begin{cases}
\,  \nra{h}_{L^{2}(B_{\rr})}  + \tail(h;B_{\rr}) \lesssim \tx{E}_{u}(\ell;x,\rr)+ \nra{w}_{L^2(B_{\rr})} \lesssim \tx{E}_{u}(\ell;x,\rr)+ \rr \mathbf{R}(x,\rr)  \\[3pt]
\,   \tail(u_\ell;B_{\rr/2})\lesssim_{n,s}
\tail(u_\ell;B_{\rr})+\nra{u_{\ell}}_{L^2(B_\rr)}\,,
 \end{cases}
$$
where $c\equiv c(\data,\alpha,\beta,H)$. 
Note that \rif{eq:higherreg22}$_1$ follows using triangle inequality, \rif{compestuvx} and observing that  $\tail(u_\ell;B_{\rr})= \tail(h;B_{\rr})$. Inequality  \rif{eq:higherreg22}$_2$ instead follows just using \rif{scatailancora}. We can now estimate, using \rif{ovvia} repeatedly
\begin{eqnarray*}
	\tx{E}_{u}(\ell + \ell_{x};x,\texttt{t} \rr)
	&\leq & \nra{u_\ell- \ell_{x}}_{L^{2}(B_{\texttt{t} \rr})} + \tail(u_\ell- \ell_{x};B_{\texttt{t} \rr})\nonumber \\ 
	&\stackrel{\eqref{scatailancora}}{\le}&c \nra{w}_{L^2(B_{\texttt{t}\rr}) } +c \nra{h-\ell_{x}}_{L^2(B_{\texttt{t} \rr}) } \nonumber \\
	&& +c\texttt{t}^{2s}\tail(u_\ell;B_{\rr/2}) + c\texttt{t}^{2s}\tail(\ell_{x};B_{\rr/2}) \nonumber \\
	&&+c \texttt{t}^{2s} \nra{u_\ell}_{L^{2}(B_{\rr/2})} + c \texttt{t}^{2s} \nra{\ell_{x}}_{L^{2}(B_{\rr/2})} \nonumber \\
	&& +c\int_{\texttt{t} \rr}^{\rr/2}\left(\frac{\texttt{t} \rr}{\lambda}\right)^{2s}\nra{u_\ell- \ell_{x}}_{L^{1}(B_{\lambda} )}\frac{\dlam}{\lambda} \nonumber \\
	&\stackrel{\eqref{tritri}, \eqref{eq:higherreg22}}{\le}&c \nra{w}_{L^2(B_{\texttt{t}\rr}) } +c \nra{h-\ell_{x}}_{L^2(B_{\texttt{t} \rr}) } \nonumber \\
	&&+c \texttt{t}^{2s} \nra{u_\ell}_{L^{2}(B_{\rr})} +c\texttt{t}^{2s}\tail(u_\ell;B_{\rr}) + c\texttt{t}^{2s}\left(\snr{D  h(x)} \rr + \snr{ h(x)}\right)  \nonumber \\
		&& +c\int_{\texttt{t} \rr}^{\rr/2}\left(\frac{\texttt{t} \rr}{\lambda}\right)^{2s}\nra{w}_{L^{2}(B_{\lambda} )}\frac{\dlam}{\lambda}  \nonumber \\
	&& +c\int_{\texttt{t} \rr}^{\rr/2}\left(\frac{\texttt{t} \rr}{\lambda}\right)^{2s}\nra{h- \ell_{x}}_{L^{2}(B_{\lambda} )}\frac{\dlam}{\lambda}  \nonumber \\
	&\stackrel{\eqref{eq:higherreg}}{\le}&c\texttt{t}^{-n/2} \nra{w}_{L^2(B_\rr)} +c\texttt{t}^{2} \left ( \nra{h}_{L^{2}(B_{\rr})} + \tail(h;B_{\rr})\right ) \nonumber \\
	&&+ c\texttt{t}^{2s}\tx{E}_{u}(\ell;x,\rr) +c\texttt{t}^{2s}\left ( \nra{h}_{L^{2}(B_{\rr})} + \tail(h;B_{\rr})\right ) \nonumber \\ &&+c\texttt{t}^{2s}\rr^{2s}  \int_{\texttt{t} \rr}^{\rr/2} \lambda^{-2s}\nra{w}_{L^2(B_{\lambda}) } \frac{\dlam}{\lambda}\nonumber \\
	&&+c\texttt{t}^{2s} \rr^{2s-2} \left (\nra{h}_{L^{2}(B_{\rr})}  + \tail(h;B_{\rr}) \right ) \int_{\texttt{t} \rr}^{\rr/2} \lambda^{2-2s} \frac{\dlam}{\lambda}  \nonumber \\
	&\stackrel{\eqref{compestuvx},\eqref{eq:higherreg}}{\le} &c \texttt{t}^{-n/2} \rr\mathbf{R}(x,\rr) +c \texttt{t}^{2s} \left ( \nra{h}_{L^{2}(B_{\rr})} + \tail(h;B_{\rr}) \right ) \nonumber \\ &&+ c\texttt{t}^{2s}\tx{E}_{u}(\ell;x,\rr)   +c\texttt{t}^{2s} \rr^{n/2+2s+1} \int_{\texttt{t}\rr}^{\rr/2}\lambda^{-2s-n/2}\frac{\dlam}{\lambda} \, \mathbf{R}(x,\rr) \,.
\end{eqnarray*}
By yet using \eqref{compestuvx} and \eqref{eq:higherreg22} we conclude with 
\eqn{olga}
$$
\tx{E}_{u}(\ell + \ell_{x};x,\texttt{t} \rr) \leq \ti{c}\texttt{t}^{2s}\tx{E}_{u}(\ell;x,\rr) +c \texttt{t}^{-n/2} \rr\mathbf{R}(x,\rr)\,,
$$
where $\ti{c}, c=\ti{c}, c(\data,\alpha,\beta,H)>0$. Now, with $\gamma_{1}$ such that
$0 < \gamma_{1} < 2s$, we choose $\texttt{t}$ sufficiently small, that is we choose $m\geq 4$ large enough in \rif{intero}, in order to have 
\eqn{dep.2x}
$$
\ti{c}\texttt{t}^{2s-\gamma_{1}} <1/2\ \Longrightarrow \texttt{t}\equiv \texttt{t}(\data,\alpha,\beta, \gamma_{1}, H)>0.
$$
So far, for any $\gamma_{1} \in (0,2s)$ we have proved the first-order excess decay estimate
\eqn{firstorderex}
$$\tx{E}_{u}(\ell+\ell_{x};x,\texttt{t}\rr) \leq \frac{\texttt{t}^{\gamma_{1}}}{2}\tx{E}_{u}(\ell;x,\rr) +c  \rr \mathbf{R}(x,\rr),$$
holds for a constant $c\equiv c(\data,\alpha,\beta, \gamma_{1},H)\geq 1$. 
\begin{remark}\label{betterexpl}{\em 
Note that in \rif{firstorderex} we have used the explicit dependence of $\texttt{t}$ on the various constants determined in \rif{dep.2x} to establish the correct dependence of the involved constant $c$. Most importantly, note that, all in all, \rif{firstorderex} holds whenever $\ell$ is an arbitrary affine map, initially fixed at the very beginning of Section \ref{sezi1}. On the contrary, the affine map $\ell_x$ defined in \rif{defffie} depends on $\ell$ via the definition of the map $h$ in \rif{solvesh}, which takes into account the presence of $\ell$ in the boundary datum $u_{\ell}$. We keep in mind this to conclude the argument in the next lines. }\end{remark} For $r>0$, we denote 
\eqn{espressione}
$$ \mathcal{E}_{\gamma_{1}}(x,r):= r^{-\gamma_{1}} \inf_{\ell \textnormal{ affine}} \tx{E}_{u}(\ell;x,r)\,. $$
Note that the very definition of $\mathcal{E}_{\gamma_{1}}(x,r)$ implies
\eqn{hachiko}
$$
r^{\gamma_{1}}\mathcal{E}_{\gamma_{1}}(x,r) \leq \tx{E}_{u}(x,r)
\lesssim_{n,s} \nra{u}_{L^{2}(B_{r})}+ \tail(u;B_{r})\,.
$$
Then \eqref{firstorderex} implies 
$$ \mathcal{E}_{\gamma_{1}}(x,\texttt{t} \rr) \leq \frac{\rr^{-\gamma_{1}}}{2}\tx{E}_{u}(\ell;x,\rr) +c  \texttt{t}^{-\gamma_1}\rr^{1-\gamma_1} \mathbf{R}(x,\rr),$$
and, since $\ell$ is arbitrary and $\texttt{t}$ exhibits the dependence displayed in \eqref{dep.2x}, we conclude with 
\eqn{E1est}
$$ \mathcal{E}_{\gamma_{1}}(x,\texttt{t} \rr) \leq \tfrac 12 \mathcal{E}_{\gamma_{1}}(x,\rr) +c \rr^{1-\gamma_{1}} \mathbf{R}(x,\rr),
$$ 
where $c\equiv c(\data,\alpha,\beta, \gamma_{1},H)\geq 1$. 
Estimate \rif{E1est} holds whenever the radius $\rr$ is of the type in \rif{recallradius}-\rif{ilraggio} and whenever $x \in B_{r_{x_{0}}/2}(x_{0})$. 

\subsection{The chain}\label{thechain} We apply \rif{E1est} with $\rr\equiv \rr_{j}$
where for every $j \in \en_0$ it is
\eqn{iraggit}
$$\rr_{j}:=\frac{\texttt{t}^j \rr_{\star}}{8} =\frac{\rr_{\star}}{2^{mj+3}}\stackrel{\eqref{ilraggio}}{\Longrightarrow} N_{\rr_{j}}= mj+3\,,$$
and $\texttt{t}$ has been introduced in \rif{dep.2x}. We find 
\eqn{somme}  
$$ \mathcal{E}_{\gamma_{1}}(x,\rr_{j+1}) \leq \tfrac 12 \mathcal{E}_{\gamma_{1}}(x,\rr_j) +c \rr_j^{1-\gamma_{1}} \mathbf{R}(x,\rr_j),$$
so that for any $k \in \mathbb{N}$, summing over $j$ leads to
$$
\sum_{j=0}^{k+1} \mathcal{E}_{\gamma_{1}}(x,\rr_{j})  \leq \mathcal{E}_{\gamma_{1}}(x,\rr_{0})+\frac 12 \sum_{j=0}^{k} \mathcal{E}_{\gamma_{1}}(x,\rr_{j}) +c \sum_{j=0}^{k} \rr_{j}^{1-\gamma_{1}}  \mathbf{R}(x,\rr_{j}) ,
$$
where $c\equiv c(\data,\alpha,\beta,\gamma_{1},H)\geq 1$. Reabsorbing the first sum on the right-hand side we obtain
\begin{flalign} \sum_{j=0}^{k+1} \mathcal{E}_{\gamma_{1}}(x,\rr_{j}) &\leq 2\mathcal{E}_{\gamma_{1}}(x,\rr_0) +c \sum_{j=0}^{k} \rr_{j}^{1-\gamma_{1}} \mathbf{R}(x,\rr_{j})\notag\\& \leq c\, \rr_{\star}^{-\gamma_1}\tx{E}_{u}(x,\rr_{\star}) +c \sum_{j=0}^{k} \rr_{j}^{1-\gamma_{1}}  \mathbf{R}(x,\rr_{j})\,, \label{flatnesscontrol}
\end{flalign} 
for any $\gamma_{1} \in (0,2s)$, where $c\equiv c(\data,\alpha,\beta,\gamma_{1},H)\geq 1$. Note that in the last estimate we have used \rif{hachiko} and then \rif{scataildopoff}$_1$. 
\begin{remark}[A more general version - will be useful later]\label{amoregen}\footnote{The reader can skip the content of Remark \ref{amoregen} until Section \ref{secprova} below.}{\em In the setting of  Lemma \ref{comparisonestimate}, consider a radius $\rr_{\star \star}$ of the type 
$
\rr_{\star \star} := 2^M \rr$ with  $M \in \en$ being such that   $2\leq M \leq N_{\rr}$ so that 
$\rr <  \rr_{\star \star} \leq  \rr_{\star}$. 
Then 
\eqn{viene}
$$
		\nra{w}_{L^2(B_{\rr})} \leq c \rr \mathbf{R}_{\rr_{\star \star}}(x,\rr)
$$
holds with $c\equiv c(\data,\alpha,\beta,H)$, where
\eqn{compestuvx222re} 
$$
 \mathbf{R}_{\rr_{\star\star }}(x,\rr):= \bigg [ \sum_{i=0}^{M} 2^{(1-2s)i} \omega_{\star} (2^i \rr) + \omega_{\star} (2\rr)+\left (\frac{\rr}{\rr_{\star\star }} \right )^{\alpha \beta} \bigg ] \MMM^{\#,1}_{\rr,\rr_{\star \star}}(x) + \rr^{2s-1} \nra{f}_{L^{\chi}(B_{4\rr})}\,.
$$
The proof of this fact is the same as Lemma \ref{comparisonestimate} once we replace $\rr_{\star}$ by $\rr_{\star\star }$, as in fact it is $\mathbf{R}_{\rr_{\star}}=\mathbf{R}$ according to the definition given in \rif{compestuvx2}. As a consequence, we can apply this with the choice 
$\rr_{\star\star}= \rr_{k_0}$, $\rr= \rr_{j}$ with integers  $1\leq k_0 <j$, so that $M=m(j-k_0)$, and state the following version of \eqref{somme}:
\eqn{sommeancora}  
$$ \mathcal{E}_{\gamma_{1}}(x,\rr_{j+1}) \leq \tfrac 12 \mathcal{E}_{\gamma_{1}}(x,\rr_j) +c \rr_j^{1-\gamma_{1}} \mathbf{R}_{\rr_{k_0}}(x,\rr_j)$$
that holds with $c\equiv c(\data,\alpha,\beta,\gamma_{1},H)\geq 1$ and that follows using \rif{viene} instead of \rif{compestuvx} in the proof of \rif{somme}. 
 }
\end{remark}

\subsection{Proof of Theorem \ref{gradreg1} completed}\label{sceltabeta} In \rif{espressione} we take $\gamma_1=1$. For any $j \in \en_0$, denote by $\ell_{j,x}$ the affine map such that 
$$ \tx{E}_{u}(\ell_{j,x};x,\rr_j)= \sqrt{\nra{u-\ell_{j,x}}_{L^{2}(B_{\rr_j})}^2 + \tail(u-\ell_{j,x};B_{\rr_j})^2}=\inf_{\ell \textnormal{ affine}} \tx{E}_{u}(\ell;x,\rr_j)$$
and therefore attaining the infimum in the expression $\mathcal{E}_{1}(x,\rr_{j})$,  i.e., 
\eqn{minimizzano}
$$ \mathcal{E}_{1}(x,\rr_j)=\rr_{j}^{-1}  \sqrt{\nra{u-\ell_{j,x}}_{L^{2}(B_{\rr_j})}^2 + \tail(u-\ell_{j,x};B_{\rr_j})^2}= \rr_j^{-1} \inf_{\ell \textnormal{ affine}} \tx{E}_{u}(\ell;x,\rr_j)\,.$$
This map can obviously be written in the form $\ell_{j,x}(\tilde x)=  D \ell_{j,x}(\ti{x}-x) +\ell_{j,x}(x)$, $\ti{x} \in \er^n$ and its existence follows as in Remark \ref{esisteminimo}. Using Jensen inequality, \rif{ovvia} and \rif{tritri} we observe that
\begin{flalign*}
\rr_{j}^{-1} \tx{E}_{u}(x,\rr_{j})&  \lesssim \rr_{j}^{-1}  \left[ \nra{u-\ell_{j,x}(x)}_{L^{2}(B_{\rr_j})} + \tail(u-\ell_{j,x}(x);B_{\rr_j})\right]
+ \rr_{j}^{-1}\snr{\ell_{j,x}(x)-(u)_{B_{\rr_j}}}
\\
&  \lesssim\rr_{j}^{-1}  \left[ \nra{u-\ell_{j,x}(x)}_{L^{2}(B_{\rr_j})} + \tail(u-\ell_{j,x}(x);B_{\rr_j})\right]\\
&  \lesssim\rr_{j}^{-1}  \left[ \nra{u-\ell_{j,x}}_{L^{2}(B_{\rr_j})} + \tail(u-\ell_{j,x};B_{\rr_j})\right] \\
& \qquad + \rr_{j}^{-1} \left[\nra{\ell_{j,x}-\ell_{j,x}(x)}_{L^{2}(B_{\rr_j})} +  \tail(\ell_{j,x}-\ell_{j,x}(x);B_{\rr_j})\right] \\
 & \lesssim \mathcal{E}_{1}(x,\rr_{j}) + |D \ell_{j,x}| 
\end{flalign*}
holds for every $j\geq 0$ and the implied constant only depends on $n,s$. 
Now, fix an integer $k\geq 1$; also thanks to \rif{scataildopoff}$_1$, the content of the last display implies the following estimate on the maximal function  in \eqref{Nsharp}:
\eqn{maxfuncontrol}
$$	\MMM^{\#,1}_{\rr_k,\rr_{\star}}(x) \leq c\texttt{t}^{-n/2-1}  \max_{j \in \{0,...,k\}} \left(\mathcal{E}_{1}(x,\rr_{j}) + |D \ell_{j,x}| \right) + c  \rr_{\star}^{-1}\tx{E}_{u}(x,\rr_{\star})\,.$$
To estimate $|D \ell_{j,x}|$, we start by recalling that $\ti{x} \mapsto (D \ell_{j,x}-D \ell_{j+1,x})(\ti{x}-x) $ has null average on $B_{\rr_{j+1}}$, so that, also using H\"older inequality, for every $j \in \en_0$ we have 
\begin{eqnarray}
\snr{D \ell_{j,x}-D \ell_{j+1,x}} & \stackleq{tritri} & c \rr_{j+1}^{-1} \dashint_{B_{\rr_{j+1}}} \snr{(D \ell_{j,x}-D \ell_{j+1,x})(\ti{x}-x) }\d\ti{x}\qquad  \nonumber\\ & \stackleq{minav} & c \rr_{j+1}^{-1} \dashint_{B_{\rr_{j+1}}} \snr{(D \ell_{j,x}-D \ell_{j+1,x})(\ti{x}-x)+ \ell_{j,x}(x) - \ell_{j+1,x}(x)}\d\ti{x} \nonumber  \\ & =  &c \rr_{j+1}^{-1} \, \nra{\ell_{j,x}-\ell_{j+1,x}}_{L^1(B_{\rr_{j+1}})} \nonumber\\
&\leq   &c \rr_{j+1}^{-1} \, \nra{\ell_{j,x}-\ell_{j+1,x}}_{L^2(B_{\rr_{j+1}})} \nonumber\\
& \leq  & c \texttt{t}^{-n/2-1}\rr_{j}^{-1} \, \nra{u-\ell_{j,x}}_{L^2(B_{\rr_{j}})}+c\rr_{j+1}^{-1} \, \nra{u-\ell_{j+1,x}}_{L^2(B_{\rr_{j+1}})}\,.\nonumber 
\end{eqnarray}
Recalling the dependence of $\texttt{t}$ in \rif{dep.2x} (here it is $\gamma_1=1$), we conclude with 
\eqn{somme1} 
$$
|D \ell_{j,x}-D \ell_{j+1,x}|  \leq c\, \mathcal{E}_{1}(x,\rr_{j})+c\, \mathcal{E}_{1}(x,\rr_{j+1}) \,,
$$
where $c =c (\data,\alpha,\beta,H)\geq 1$. 
Summing up the previous inequalities yields 
$$
\max_{j \in \{0,...,k\}} |D \ell_{j,x}|   \leq |D \ell_{0,x}| +  \sum_{j=0}^{k-1} |D \ell_{j,x}-D \ell_{j+1,x}| \leq |D \ell_{0,x}| + c \sum_{j=0}^{k} \mathcal{E}_{1}(x,\rr_{j}) \,, $$
so that, thanks to \eqref{maxfuncontrol}, we arrive at
\eqn{semplice00}
$$	\MMM^{\#,1}_{\rr_k,\rr_{\star}}(x) \leq c|D \ell_{0,x}| +  c \sum_{j=0}^{k+1} \mathcal{E}_{1}(x,\rr_{j})+ c  \rr_{\star}^{-1}\tx{E}_{u}(x,\rr_{\star})\,.$$
For the term $|D \ell_{0,x}|$, using \rif{tritri} and the minimality of $\ell_{0,x}$ we have
\eqn{semplice111}
$$
\snr{D\ell_{0,x}}\rr_{\star}\lesssim \snr{D\ell_{0,x}}\rr_{0}+ \snr{\ell_{0,x}(x)} \approx  \nra{\ell_{0,x}}_{L^{2}(B_{\rr_0})} \lesssim \nra{u}_{L^{2}(B_{\rr_0})} + \tx{E}_{u}(x,\rr_0)\,.
$$
Combining \rif{semplice00} and \rif{semplice111} with \eqref{flatnesscontrol} applied with $\gamma_{1}=1$, and yet recalling \rif{scataildopoff}$_1$, now yields
\eqn{maxf1}
$$\MMM^{\#,1}_{\rr_k,\rr_{\star}}(x) \leq c_{4} \rr_{\star}^{-1} \left( \nra{u}_{L^{2}(B_{\rr_{\star}})}+ \tx{E}_{u}(x,\rr_{\star}) \right) +c_{4} \sum_{j=0}^{k} \mathbf{R}(x,\rr_{j})\,,$$
where $c_{4}=c_{4}(\data,\alpha,\beta,H)\geq 1$. Note that this constant does not depend on $\rr_{\star}$. Estimate \rif{maxf1} works whenever $k\geq 1$ is an integer. We proceed to find a bound for the last term on the right-hand side of \eqref{maxf1}, so that by means of \eqref{compestuvx2} we obtain
\begin{flalign} 
	\notag c_{4} \sum_{j=0}^{k} \mathbf{R}(x,\rr_{j}) &\leq 
	c_{4} \left(  \sum_{j\geq 0}  \sum_{i=0}^{N_{\rr_j}} 2^{(1-2s)i} \omega_{\star} (2^i \rr_j) \right) \MMM^{\#,1}_{\rr_k,\rr_{\star}}(x)+ c_4 \sum_{j\geq 0} \omega_{\star} (2\rr_j)
	\,  \MMM^{\#,1}_{\rr_k,\rr_{\star}}(x)\notag \\ \notag & \qquad +c_4\sum_{j=0}^{k}\left (\frac{\rr_j}{\rr_{\star}} \right )^{\alpha \beta}   \MMM^{\#,1}_{\rr_j,\rr_{\star}}(x) + c_4\sum_{j\geq 0}\rr^{2s-1}_j \nra{f}_{L^{\chi}(B_{4\rr_j})}\\ & =: \mathcal R_1 + \mathcal R_2 + \mathcal R_3 +\mathcal R_4\,.\label{combina00}
\end{flalign} 
Note that in the above estimation we have used that $i \mapsto \MMM^{\#,1}_{\rr_i,\rr_{\star}}(x)$ is non-decreasing to stimate $\MMM^{\#,1}_{\rr_j,\rr_{\star}}(x)\leq \MMM^{\#,1}_{\rr_k,\rr_{\star}}(x)$ for $j\leq k$. To proceed, as in \rif{5,3}, for every $k_1\geq0$ it holds that 
\begin{flalign}
\notag \sum_{j\geq k_1}\rr^{2s-1}_j \nra{f}_{L^{\chi}(B_{4\rr_j})} &\leq c \sum_{j> k_1}\int_{\rr_{j}}^{\rr_{j-1}}\lambda^{2s-1} \nra{f}_{L^{\chi}(B_{4\lambda})}\frac{\dlam}{\lambda} 
+ c \int_{\rr_{k_1}}^{2\rr_{k_1}}\lambda^{2s-1} \nra{f}_{L^{\chi}(B_{4\lambda})}\frac{\dlam}{\lambda} \\
&\leq c \, \int_{0}^{2\rr_{k_1}}\lambda^{2s-1} \nra{f}_{L^{\chi}(B_{4\lambda})}\frac{\dlam}{\lambda}\notag \\ & \leq  c  \mathbf{I}^{f}_{2s-1,\chi}(x,8\rr_{k_1})\,,\label{reacher1}
\end{flalign}
with $c\equiv c (n,\texttt{t})$. Similarly, we have 
\eqn{reacher2}
$$
\sum_{j\geq k_1} \omega_{\star} (2\rr_j)   
\leq c \sum_{j> k_1}\int_{\rr_{j}}^{\rr_{j-1}}\omega_{\star}(2\lambda)\frac{\dlam}{\lambda} 
+ c \int_{\rr_{k_1}}^{2\rr_{k_1}}\omega_{\star}(2\lambda)\frac{\dlam}{\lambda}
\leq c \int_0^{4\rr_{k_1}} \omega_{\star}(\lambda)\frac{\dlam}{\lambda}\,.
$$
Using \rif{reacher1} with $k_1=0$ and recalling that $8\rr_0=\rr_{\star}$ by \rif{iraggit} we have 
\eqn{5,3bissi}
$$
\mathcal R_4\leq c\,   \mathbf{I}^{f}_{2s-1,\chi}(x,\rr_{\star})\,,
$$
with $c\equiv c(\data,\alpha,\beta,H)$; here we are using that $\texttt{t}$ is determined in \rif{dep.2x} as a function of $\data,\alpha,\beta,H$. To proceed, using discrete Fubini, we have
\begin{align}
& \sum_{j \geq 0} \sum_{i=0}^{N_{\rr_{j}}} 2^{(1-2s)i} \omega_{\star} (2^i \rr_{j} ) \stackrel{\eqref{iraggit}} {=} \sum_{j \geq 0} \sum_{i=0}^{mj+3} 2^{(1-2s)i} \omega_{\star} (2^{i-mj-3} \rr_{\star} ) \nonumber\\
&\qquad \leq   \sum_{\mathfrak{m} \geq 0} 2^{(1-2s)\mathfrak{m}}  \sum_{i \geq 0} \omega_{\star} (2^{-i} \rr_{\star} ) \leq c(s)   \sum_{i \geq 0} \omega_{\star} (2^{-i} \rr_{\star} ) \nonumber\\
& \qquad\qquad\leq c(s) \sum_{i \geq 0} \int_{2^{-i} \rr_{\star} }^{2^{-i+1} \rr_{\star} }\omega_{\star} (\lambda) \,\frac{\dlam}{\lambda} =  c_{5} \int_0^{2\rr_{\star}} \omega_{\star}(\lambda)\frac{\dlam}{\lambda} \label{smalldini0}
\end{align}
where $c_{5}\equiv c_{5}(s)\geq 1$. Similarly, by \rif{reacher2} with $k_1=0$ we obtain
\eqn{smalldini0ancora}
$$
\sum_{j\geq 0} \omega_{\star} (2\rr_j) \leq c_{5} \int_0^{2\rr_{\star}} \omega_{\star}(\lambda)\frac{\dlam}{\lambda}\,.
$$
The starting radius $\rr_{\star}$ has been initially chosen to satisfy \rif{recallradius}. 
We now further reduce its size in order to meet
\eqn{meetme}
$$ c_{4} c_{5} \int_0^{2\rr_{\star}} \omega_{\star}(\lambda)\frac{\dlam}{\lambda} \leq \frac 1{8} \Longrightarrow  \rr_{\star}\equiv \rr_{\star}(\data,\alpha,\beta,H,\omega_{\star}(\cdot))$$ and therefore, thanks to \rif{smalldini0}-\rif{smalldini0ancora}, we have
\eqn{smalldini}
$$
\mathcal R_1 + \mathcal R_2\leq \frac 1{4}\MMM^{\#,1}_{\rr_k,\rr_{\star}}(x)\,.
$$
In order to estimate the remaining term $\mathcal R_3$ we again use that $i \mapsto \MMM^{\#,1}_{\rr_i,\rr_{\star}}(x)$ is non-decreasing to get that, for every choice of integers $j_0\geq 0$ and $k \geq 1$, it holds 
$$
\sum_{j=0}^{k} \left (\frac{\rr_{j}}{\rr_{\star}} \right )^{\alpha \beta} \MMM^{\#,1}_{\rr_{j},\rr_{\star}}(x)  \leq \sum_{j=0}^{j_0} \left (\frac{\rr_{j}}{\rr_{\star}} \right )^{\alpha \beta} \MMM^{\#,1}_{\rr_{j_{0}},\rr_{\star}}(x)  +
\sum_{j>j_0} \left (\frac{\rr_{j}}{\rr_{\star}} \right )^{\alpha \beta} \MMM^{\#,1}_{\rr_{k},\rr_{\star}}(x) \,.
$$
We then choose $\beta=1/2$ and define the integer $j_{0}=j_{0}(\data,\alpha,H) \in \mathbb{N}$ by
$$j_{0} :=\min \left \{\mathfrak m \in \mathbb{N} \, \colon \,  c_{4} \sum_{j>\mathfrak m}\left (\frac{\rr_{j}}{\rr_{\star}} \right )^{\alpha/2} \stackrel{\eqref{iraggit}}{=}\frac{c_{4} \texttt{t}^{\alpha(\mathfrak m+1)/2}}{8^{\alpha/2}(1-\texttt{t}^{\alpha/2})} \leq \frac 1{4} \right \}\,.$$
Note that using \rif{scataildopoff}$_1$ and the dependence of $\texttt{t}$ in \rif{dep.2x}, we can estimate 
$$
\MMM^{\#,1}_{\rr_{j_0},\rr_{\star}}(x) \leq c \texttt{t}^{-j_0(n/2+1)}\rr_{\star}^{-1}\tx{E}_{u}(x,\rr_{\star}) \leq c 
\rr_{\star}^{-1}\tx{E}_{u}(x,\rr_{\star})\,.
$$
Summarising the content of the last three displays yields 
\eqn{smalldini2}
$$
\mathcal R_3  \leq
\frac14 \MMM^{\#,1}_{\rr_{k},\rr_{\star}}(x) +\rr_{\star}^{-1-\alpha/2}\tx{E}_{u}(x,\rr_{\star})\,.
$$
Using in \rif{combina00} the estimates \rif{5,3bissi},\rif{smalldini} and \rif{smalldini2} found for the terms $\mathcal R_1, \dots, \mathcal R_4$ yields
\eqn{combina}
$$
	 c_{4} \sum_{j=0}^{k} \mathbf{R}(x,\rr_{j}) 
		\leq  \tfrac 12 \MMM^{\#,1}_{\rr_k,\rr_{\star}}(x) +c\rr_{\star}^{-1-\alpha/2}\, \tx{E}_{u}(x,\rr_{\star}) + c\, \mathbf{I}^{f}_{2s-1,\chi}(x,\rr_{\star})\,.
$$
Combining \rif{combina} with \eqref{maxf1} for any integer $k \geq 1$ we arrive at
$$ \MMM^{\#,1}_{\rr_k,\rr_{\star}}(x) \leq c\rr_{\star}^{-1} \nra{u}_{L^{2}(B_{\rr_{\star}})}+ c\rr_{\star}^{-1-\alpha/2}\, \tx{E}_{u}(x,\rr_{\star}) + c\, \mathbf{I}^{f}_{2s-1,\chi}(x,\rr_{\star}) \,,$$
and 
letting $k \to \infty$ yields the pointwise maximal function estimate
$$ \MMM^{\#,1}_{\rr_{\star}}(x) \leq c\rr_{\star}^{-1} \nra{u}_{L^{2}(B_{\rr_{\star}})}+ c\rr_{\star}^{-1-\alpha/2}\, \tx{E}_{u}(x,\rr_{\star}) + c\, \mathbf{I}^{f}_{2s-1,\chi}(x,\rr_{\star}) \,,$$
where $c\equiv c(\data,\alpha,H)\geq 1$. Note that this holds whenever the point $x$ and the radius $\rr_{\star}$ obey \rif{recallradius} and \rif{meetme}. Moreover, using \rif{scataildopoff}$_2$ and \rif{hachiko} we estimate 
$$
 \sup_{\rr_{\star}\leq \sigma < r_{x_0}/2}\, \sigma^{-1}\tx{E}_{u}(x,\sigma) \leq  c\rr_{\star}^{-n-1}\left[
 \nra{u}_{L^{2}(B_{r_{x_0}}(x_0))}+  \tail(u;B_{r_{x_0}}(x_0))
 \right]
$$
so that, recalling the definition in \rif{massimale}, we conclude with  
\begin{flalign} 
\notag & \nr{\textnormal{M}^{\#,1}_{r_{x_0}/2}(u;\cdot)}_{L^\infty(B_{r_{x_0}/2}(x_0))}+ \nr{\MMM^{\#,1}_{r_{x_0}/2}}_{L^\infty(B_{r_{x_0}/2}(x_0))}\\
& \qquad \leq  c \, \nra{u}_{L^{2}(B_{r_{x_0}}(x_0))}+ c\, \tail(u;B_{r_{x_0}}(x_0))  + c\,  \nr{\mathbf{I}^{f}_{2s-1,\chi}(\cdot, 1)}_{L^\infty(B_{r_{x_{0}}/2}(x_{0}))}\label{pwest}
\end{flalign}
where $c\equiv c(\data,\alpha,H,\omega_{\star}(\cdot))$ (for this recall the dependence of $\rr_{\star}$ in \rif{meetme}). Lemma \ref{campmax} now yields  
\eqn{finalissima}
$$  [u]_{0,1;B_{r_{x_0}/8}(x_{0})} \leq  c \, \nra{u}_{L^{2}(B_{r_{x_0}}(x_0))}+ c\, \tail(u;B_{r_{x_0}}(x_0))  + c\,  \nr{\mathbf{I}^{f}_{2s-1,\chi}(\cdot, 1)}_{L^\infty(B_{r_{x_{0}}/2}(x_{0}))}
$$
with  $c\equiv c(\data,\alpha,H,\omega_{\star}(\cdot))$.
In particular, a standard covering argument leads to conclude that $u \in C^{0,1}_{\loc}(\Omega_u;\mathbb{R}^N)$, so that the proof of Theorem \ref{gradreg1} is complete.

\section{Gradient oscillations and Theorems \ref{gradreg2}-\ref{gradreg3}}\label{final2} 
Theorem \ref{gradreg1} applies in the settings of Theorem \ref{gradreg2}-\ref{gradreg3} as the assumptions considered on $f$ in \rif{assilipul} and \rif{assilip2} imply \rif{assilipanc}. In the case of \rif{assilipul} this is obvious. In the case of \rif{assilip2} we simply apply \rif{marhol} with $d= n/(2s-1-\alpha)$, and this yields 
\eqn{marci222}
$$
\sigma^{2s-1}\nra{f}_{L^{\chi}(B_{\sigma})}  
  \lesssim_{n,s,\chi, \alpha}\sigma^{\alpha}\nr{f}_{\mathcal{M}^{\frac{n}{2s-1-\alpha}}(B_{\sigma})}
$$
for every ball $B_{\sigma} \subset \er^n$ so that, integrating, we find
$$
\mathbf{I}^{f}_{2s-1,\chi}(x,1)  \lesssim   \nr{f}_{\mathcal{M}^{\frac{n}{2s-1-\alpha}}(\er^n)}\int_{0}^{1}\sigma^{\alpha}\frac{\d \sigma }{\sigma} =c  \nr{f}_{\mathcal{M}^{\frac{n}{2s-1-\alpha}}(\er^n)} 
$$
where $c\equiv c(n,s,\chi, \alpha)$. Therefore, with a slight abuse of notation on the identity of the radius $r_{x_0}$ (which here will not anymore be the one identified in Section \ref{parametri} and used in Section \ref{eidos}), we can assume to be exactly in the same starting setting of Section \ref{final} with the agreement that now we have
\eqn{Lipest}
$$
\begin{cases}
 \, \nr{u}_{L^\infty(B_{r_{x_0}}(x_{0}))}+\nr{Du}_{L^{\infty}(B_{r_{x_0}}(x_{0}))} \leq H_2< \infty\\[5pt]
 \, \MMM^{\#,1}_{r_{x_0}}(x) \leq H_2 \,, \quad \mbox{for every $x\in B_{r_{x_0}}(x_0)$}
 \end{cases}
 $$
for some absolute, local constant $H_2\geq 1$\footnote{In fact, by \rif{pwest}-\rif{finalissima} we should denote $r_{x_0}/8$ instead of $r_{x_0}$ in \rif{Lipest}.  We prefer this slight abuse of notation - renaming $r_{x_0}/8$ by $r_{x_0}$ - as we are then able to use the machinery, and with indeed the same notation, developed in Section \ref{final}. Essentially, by \rif{Lipest}$_1$ we can perform all the steps of Section \ref{final}, with  \rif{holdest22} being replaced by its Lipschitz counterpart ($\beta=1$), i.e., 
$$
\snr{u(x)-u(y)}\le c H_2  \snr{x-y}\,, \quad \mbox{for all $x,y \in B_{r_{x_{0}}}(x_{0})$}\,.
$$ Also note that, throughout Sections \ref{identfinal}-\ref{thechain} the parameter $\rr_*$ remains a free number only obeying $\rr_* \leq r_{x_0}/8$, and we shall use precisely this freedom of choice here. Only later on, in \rif{meetme}, a restriction on the size of $\rr_*$ is assumed. This fact will not enter in this section; see the proof of Proposition \ref{finalprop}.}. The first inequality in \rif{Lipest} is essentially the content of Theorem \ref{gradreg1}, and specifically of \rif{finalissima}, while \rif{Lipest}$_2$ is a consequence of \rif{pwest}. 
We shall use, with suitable modifications, the content and the notation employed in Section \ref{final}. In particular, we start by \rif{recallradius}, where we take 
\eqn{raggiostella}
$$\rr_{\star} =\frac{r_{x_0}}{2^{\mfh}}\,, \quad \mbox{where}\ \en\ni\mfh=\begin{cases} \mbox{$\geq 3$ to be determined in the proof of Theorem \ref{gradreg2}}\\3 \ \mbox{in the proof of Theorem \ref{gradreg3}}\,.
\end{cases} $$
As in Sections \ref{final} in the following, unless differently specified, all the balls will be centred at $x\in B_{r_{x_0}/2}(x_{0})$. 
In fact, whenever the exponent $\beta$ will appear in the estimates of Section \ref{final}, this will be replaced by $1$. In the rest of the proofs of Theorems \ref{gradreg2}-\ref{gradreg3} the symbol $\lesssim$ will involve constants depending on  $\data,\alpha, H_2$ (and on $\nr{f}_{\mathcal{M}^{\frac{n}{2s-1-\alpha}}}$ in the proof of Theorem \ref{gradreg3}), but not on the radii denoted by $\rr$ or $\rr_j$ and neither on the point $x\in B_{r_{x_0}/2}(x_{0})$ considered. Occasionally, an additional dependence of the involved constants on $r_{x_0}$ will show up. As we are not going to provide final explicit estimates, we are not going to precisely keep track of the constants dependence as we did in the previous proofs, because this is not needed here.  As a final consequence of the above discussion, the regular set $\Omega_{u}$ for Theorems \ref{gradreg2}-\ref{gradreg3} will be the one already identified in Section \ref{identfinal} for Theorem \ref{gradreg1} (which is in turn the same one identified in Theorem \ref{ureg5} for the choice $\beta=1/2$). For clarity of exposition, we begin with the proof of Theorem \ref{gradreg3}, which follows from the arguments developed in Section \ref{final}. We then proceed to the more elaborate proof of Theorem \ref{gradreg2}.

\subsection{Proof of Theorem \ref{gradreg3}}  With reference to \rif{compestuvx2}, and with $\rr$ as in \rif{ilraggio}, note that \rif{assilip2}, $\alpha < 2s-1$ and \rif{marci222} imply
$$
\sum_{i=0}^{N_\rr} 2^{(1-2s)i} \omega_{\star} (2^i \rr) \lesssim \rr^\alpha \sum_{i=0}^{\infty} 2^{(1+\alpha-2s)i} \lesssim  \rr^\alpha\,, \qquad   \rr^{2s-1} \nra{f}_{L^{\chi}(B_{4\rr})} \lesssim \rr^{\alpha}\,.
$$
Using this and \rif{Lipest}$_2$ in \rif{compestuvx2} (used with $\beta=1$ as described after \rif{raggiostella}) yields
\eqn{compestuvx222} 
$$
\mathbf{R}(x,\rr)\lesssim (\rr^{\alpha}+ r_{x_0}^{-\alpha}\rr^{\alpha})H_2  \lesssim \rr^{\alpha}
$$
for every $x\in B_{r_{x_0}/2}(x_{0})$, where the involved constant is independent of $x,\rr$. Proceeding on, in Section \ref{sezi2} we choose $\gamma_{1}= 1+\alpha< 2s$ and determine $\texttt{t}\equiv \texttt{t}(\data,\alpha, H_2)$ as in \rif{dep.2x}. 
Using \rif{compestuvx222} in \rif{E1est} then yields 
$ \mathcal{E}_{1+\alpha}(x,\texttt{t} \rr) \leq    \mathcal{E}_{1+\alpha}(x,\rr)/2+c$
and therefore 
\eqn{E1estbissi}
$$ \mathcal{E}_{1+\alpha}(x, \rr_{j+1}) \leq \tfrac 12 \mathcal{E}_{1+\alpha}(x,\rr_{j}) +c$$
for every $j \in \en_0$ (recall the definition in \rif{iraggit}), where $c$ is independent of $j$ and $x$. In analogy to \rif{Nsharp}, it is therefore convenient to  use the fractional sharp maximal operators defined by 
\eqn{Nsharp2} 
$$
\begin{cases}
 \, \displaystyle \MMM^{\#,1+\alpha}_{k}(x)\equiv  \MMM^{\#,1+\alpha}_{k}(u)(x):= \sup_{0\leq j \leq k}\, \mathcal{E}_{1+\alpha}(x,\rr_{j})\\[12pt]
  \, \displaystyle \MMM^{\#,1+\alpha}_{\infty}(x)\equiv \MMM^{\#,1+\alpha}_{\infty}(u)(x):=  \sup_{j \geq 0}\,  \mathcal{E}_{1+\alpha}(x,\rr_j)
 \end{cases}
$$
for any $k \in \en_0$. 
With such definitions from \rif{E1estbissi} it follows that 
$$
\MMM^{\#,1+\alpha}_{k+1}(x) \leq \frac 12 \MMM^{\#,1+\alpha}_{k+1}(x) + \mathcal{E}_{1+\alpha}(x,\rr_{0}) + c 
$$
holds for every non-negative integer $k$, where $c$ is independent of $k$, so that 
$$\MMM^{\#,1+\alpha}_{k+1}(x) \leq   2\mathcal{E}_{1+\alpha}(x,\rr_{0}) + c $$
and, ultimately 
\eqn{5,900}
$$\MMM^{\#,1+\alpha}_{\infty}(x) \leq 2   \mathcal{E}_{1+\alpha}(x,r_{x_0}/2^6) + c \,,$$
where we recall that $\rr_0=\rr_{\star}/2^3=r_{x_0}/2^6$ by the definition in \rif{iraggit} and \rif{raggiostella}.   
By  \rif{Nsharp2}$_2$ and \rif{5,900} we have  
\begin{eqnarray*}
 \inf_{\ell \textnormal{ affine}} \tx{E}_{u}(\ell;x,\rr_j) & \lesssim & \left[\mathcal{E}_{1+\alpha}(x,r_{x_0}/2^6) +1\right] \rr_j^{1+\alpha}\\
 &  \stackrel{\eqref{hachiko}}{\lesssim} & r_{x_0}^{-1-\alpha} \left[\tx{E}_{u}(x,r_{x_0}/2^6)+1\right]  \rr_j ^{1+\alpha}  \\
&  \stackrel{\eqref{scataildopoff}}{\lesssim} & r_{x_0}^{-1-\alpha} \left[\tx{E}_{u}(x_{0},r_{x_{0}})+1\right]  \rr_j ^{1+\alpha}  \\
  &\stackrel{\eqref{hachiko}}{\lesssim} & r_{x_0}^{-1-\alpha}\left[
 \nra{u}_{L^{2}(B_{r_{x_0}}(x_0))}+ c\, \tail(u;B_{r_{x_0}}(x_0))+1
 \right]  \rr_j ^{1+\alpha} \\ &\lesssim  &  \rr_j ^{1+\alpha} 
\end{eqnarray*}
whenever $j \geq 0$ and $x \in B_{r_{x_{0}}/2}(x_{0})$ so that
$$
 \inf_{\ell \textnormal{ affine}}  \nra{u-\ell}_{L^2(B_{\rr_j})} \lesssim  \rr_j ^{1+\alpha}\,.
$$ 
Next, with $\sigma \in (0,\rr_0]$, find $j\geq 0$ such that $\rr_{j+1}< \sigma \leq \rr_j$, note that the above inequality implies
$$
 \inf_{\ell \textnormal{ affine}}  \nra{u-\ell}_{L^2(B_{\sigma})} \leq \texttt{t}^{-n/2}   
  \inf_{\ell \textnormal{ affine}}  \nra{u-\ell}_{L^2(B_{\rr_j})} \lesssim \texttt{t}^{-n/2-1-\alpha} \sigma^{1+\alpha}\,.
$$
Summarizing, we have proved that 
$$
 \inf_{\ell \textnormal{ affine}}  \nra{u-\ell}_{L^2(B_{\sigma}(x))} \lesssim  \sigma ^{1+\alpha}\,, \qquad \forall \ \sigma \leq r_{x_0}/2^6, \  x \in B_{r_{x_0}/2}(x_0)\,.
$$ 
Thanks to this last inequality, the standard integral characterization of gradient H\"older continuity due to Campanato and Meyers (see for instance \cite{cam00}) applies  and concludes the proof via a standard covering argument.

\subsection{Proof of Theorem \ref{gradreg2}}  We shall prove that $Du$ is continuous in the regular set $\Omega_{u}$ in the sense that its precise representative in
\rif{lebpdopo} can be extended in $\Omega_{u}$ to a continuous map. In other words, $Du$ is almost everywhere equal to a continuous map in $\Omega_{u}$. 
\subsubsection{Step 1: A \textnormal{VMO} type result}\label{secprova} That is 
\eqn{uniform2}
$$
\lim_{k \to \infty} \mathcal{E}_{1}(x,2^{-k}r_{x_{0}}) =0 \quad \mbox{uniformly w.r.t. $x\in B_{r_{x_0}/2}(x_0)$}\,.
$$ In Step 1 we use the machinery of Section \ref{final}, and start by \rif{recallradius}, with the choice $\rr_{\star}= 2^{-\mfh}r_{x_0}$ with  $\en\ni \mfh\geq 3$.  We go back to \rif{olga}, that can be restated as 
\eqn{olgavmo}
$$
\frac{\tx{E}_{u}(\ell + \ell_{x};x,\texttt{t} \rr)}{\texttt{t} \rr} \leq c\texttt{t}^{2s-1}\frac{\tx{E}_{u}(\ell;x,\rr)}{\rr} +c \texttt{t}^{-n/2-1} \mathbf{R}(x,\rr)\,,
$$
where $\rr$ and $\texttt{t}$ are as in \rif{ilraggio} and \rif{intero}, respectively, and are yet to be chosen. Being in \rif{olgavmo} $\ell$ an arbitrary affine map, using \rif{hachiko} and \rif{Lipest}$_2$ it follows 
\begin{flalign} \notag \mathcal{E}_{1}(x,\texttt{t} \rr) & \leq   c\texttt{t}^{2s-1} \mathcal{E}_{1}(x,\rr) +c \texttt{t}^{-n/2-1} \mathbf{R}(x,\rr)\\
\notag & \leq c\texttt{t}^{2s-1}\rr^{-1}\tx{E}_{u}(x,\rr)+c \texttt{t}^{-n/2-1} \mathbf{R}(x,\rr)\\
 & \leq c\texttt{t}^{2s-1} \MMM^{\#,1}_{r_{x_0}}(x)+c \texttt{t}^{-n/2-1} \mathbf{R}(x,\rr)
\notag \\ &  \leq c\texttt{t}^{2s-1}H_2+c \texttt{t}^{-n/2-1} \mathbf{R}(x,\rr) \,.\label{wanna}
 \end{flalign}
Estimating 
$$
\sum_{i=0}^{N_\rr} 2^{(1-2s)i} \omega_{\star} (2^i \rr)  \leq\left( \sum_{i=0}^{\infty} 2^{(1-2s)i} \right) \omega_{\star} (\rr_{\star}) \lesssim_{s} 
\omega_{\star} (\rr_{\star}) 
$$
by \rif{compestuvx2}, and choosing $\rr=2^{-k_1}\rr_{\star}=2^{-k_1-\mfh} r_{x_0}$, where $k_1\geq 4$ is an arbitrary integer, we find that 
 \eqn{wanna2}
 $$\mathbf{R}(x,\rr) =\mathbf{R}(x,2^{-k_1}\rr_{\star})\lesssim  \omega_{\star}(\rr_{\star})+ \left(\frac{\rr}{\rr_{\star}}\right)^{\alpha} + \rr^{2s-1} \nra{f}_{L^{\chi}(B_{4\rr})}\,.$$
Combining \rif{wanna} and \rif{wanna2} and recalling \rif{maggioras} and \rif{intero}, we arrive at
\begin{flalign*}
 \mathcal{E}_{1}(x, 2^{-k_1-m-\mfh}r_{x_0}) & \leq c_62^{-m(2s-1)}\\
 & \quad +c_7 2^{m(n/2+1)}\left[\omega_{\star}(2^{-\mfh}r_{x_0}) +2^{-\alpha k_1}+ \mathbf{I}^{f}_{2s-1,\chi}(x,2^{-k_1-\mfh+3}r_{x_0})\right]
\end{flalign*}
where $c_6, c_7\equiv c_6, c_7(\data,\alpha,H_2)$. This last inequality holds whenever 
$x \in B_{r_{x_0}/2}(x_0)$ and $m,k_1, \mfh\geq 4$ are integers.  We now fix $\eps>0$ and pick $m\equiv m (\data,\alpha,H_2, \eps)$ such that $c_62^{-m(2s-1)}< \eps/2$; once $m$ is fixed, thanks to \rif{assilipul}, we find $\ti{k}\equiv \ti{k} (\data,\alpha,H_2, \eps)$ and 
 $\mfh\equiv \mfh (\data,\alpha,\omega_{\star},H_2, \eps)$
 such that $$c_7 2^{m(n/2+1)}\left[\omega_{\star}(2^{-\mfh}r_{x_0}) + 2^{-\alpha k_1}+\mathbf{I}^{f}_{2s-1,\chi}(x,2^{-k_1-\mfh+3}r_{x_0})\right]< \frac{\eps}{2}$$ whenever $k_1 \geq \ti{k}$. We have therefore proved that $ \mathcal{E}_{1}(x, 2^{-k}r_{x_0}) <\eps$ provided $k > \ti{k}+m+\mfh+4$. This proves \rif{uniform2} noting that both $m,\tilde{k}, \mfh$ are independent of the point $x\in B_{r_{x_0}/2}(x_0)$.

\subsubsection{Step 2: A continuous limit}  From now on, and in the rest of the proof, in \rif{raggiostella} we take $\rr_{\star}=r_{x_0}/8$, that is, $\mfh=3$. We take the sequence of shrinking radii $\{\rr_j\}$ in \rif{iraggit} with the choice of $\texttt{t}$  in \rif{dep.2x} with $\beta =\gamma_1=1$ and $H\equiv H_2$, therefore it is $\texttt{t}\equiv \texttt{t}(\data,\alpha,H_2)$ (recall that $H_2$ is defined in \rif{Lipest}). The sequence of affine maps $\{\ell_{j,x}\}$ is defined in \rif{minimizzano}. We now consider three indices $2 \leq k_0 < k_1< k_2$ with corresponding radii defined through \rif{iraggit}
$$
\rr_{k_0}= \texttt{t}^{k_0} \rr_{0}=\frac{\rr_{\star}}{2^{mk_0+3}} \,, \quad \ \rr_{k_1}= \texttt{t}^{k_1} \rr_{0}=\frac{\rr_{\star}}{2^{mk_1+3}}\,, \quad \ 
\rr_{k_2}= \texttt{t}^{k_2} \rr_{0}=\frac{\rr_{\star}}{2^{mk_2+3}}\,.
$$
In Remark \ref{amoregen}, taking \rif{sommeancora} with $\gamma_1=1$, we obtain  
\eqn{E1estffgg}
$$ \mathcal{E}_{1}(x,  \rr_{j+1}) \leq \tfrac 12 \mathcal{E}_{1}(x,\rr_j) +c  \mathbf{R}_{\rr_{k_0}}(x,\rr_j)$$ 
for $j \geq k_1$, where, with $c\equiv c(\data,\alpha,H_2)\geq 1$ and thanks to \rif{Lipest}$_2$ and \rif{compestuvx222re}, it is 
\begin{flalign}
 \mathbf{R}_{\rr_{k_0}}(x,\rr_j) &\lesssim \sum_{i=0}^{m(j-k_0)} 2^{(1-2s)i} \omega_{\star} (2^i \rr_j) + \omega_{\star} (2\rr_j)\\
 & \quad +\left (\frac{\rr_j}{\rr_{k_0 }} \right )^{\alpha} + \rr^{2s-1}_j \nra{f}_{L^{\chi}(B_{4\rr_j})}\,.\label{compestuvx2222} 
\end{flalign}
Writing \rif{E1estffgg} with $j\in \{k_1, \ldots, k_2\}$ and arguing as for the proof of \rif{flatnesscontrol} we arrive at
\eqn{banjo}
$$ \sum_{j=k_1}^{k_2} \mathcal{E}_{1}(x,\rr_{j})\lesssim \mathcal{E}_{1}(x,\rr_{k_1}) + \sum_{j\geq k_1} \mathbf{R}_{\rr_{k_0}}(x,\rr_{j})\,.
$$  
We continue estimating the series in the above display; by \rif{compestuvx2222}  
it follows that 
\begin{flalign} 
	\notag \sum_{j\geq k_1} \mathbf{R}_{\rr_{k_0}}(x,\rr_{j}) &\lesssim 
	  \sum_{j\geq k_1}  \sum_{i=0}^{m(j-k_0)} 2^{(1-2s)i} \omega_{\star} (2^i \rr_j) +   \sum_{j\geq k_1} \omega_{\star} (2\rr_j)
	\notag \\  & \qquad +\sum_{j\geq k_1}\frac{1}{2^{m(j-k_0)\alpha}} + \sum_{j\geq k_1}\rr^{2s-1}_j \nra{f}_{L^{\chi}(B_{4\rr_j})}\,.\label{reacher3}
\end{flalign} 
As in \rif{smalldini0} we find
\begin{flalign*}
 \sum_{j \geq k_1} \sum_{i=0}^{m(j-k_0)} 2^{(1-2s)i} \omega_{\star} (2^i \rr_{j} )  &\leq  \sum_{j \geq k_0} \sum_{i=0}^{m(j-k_0)} 2^{(1-2s)i} \omega_{\star} (2^{i-m(j-k_0)} \rr_{k_0} )\nonumber\\
&\lesssim  \sum_{\mathfrak{m} \geq 0} 2^{(1-2s)\mathfrak{m}}  \sum_{i \geq 0} \omega_{\star} (2^{-i} \rr_{k_0} ) \\
& \lesssim   \int_0^{2\rr_{k_0}} \omega_{\star}(\lambda)\frac{\dlam}{\lambda} \,.
\end{flalign*}
Connecting the content of the last two displays and using \rif{reacher1}-\rif{reacher2} to estimate the second and the fourth sum in the right-hand side of \rif{reacher3}, we conclude with
$$
\sum_{j\geq k_1}\mathbf{R}_{\rr_{k_0}}(x,\rr_{j}) \lesssim  \int_0^{2\rr_{k_0}} \omega_{\star}(\lambda)\frac{\dlam}{\lambda}+ \frac{1}{\alpha 2^{m(k_1-k_0)\alpha}} +  \mathbf{I}^{f}_{2s-1,\chi}(x, 8\rr_{k_1})\,.
$$
This and \rif{banjo} give 
$$ \sum_{j=k_1}^{k_2} \mathcal{E}_{1}(x,\rr_{j})\lesssim \mathcal{E}_{1}(x,\rr_{k_1}) + \int_0^{2\rr_{k_0}} \omega_{\star}(\lambda)\frac{\dlam}{\lambda}+ \frac{1}{\alpha 2^{m(k_1-k_0)\alpha}} +  \mathbf{I}^{f}_{2s-1,\chi}(x, 8\rr_{k_1})\,.
$$  
Recalling \rif{somme1}, we
have 
$$ \snr{D \ell_{k_2,x}-D \ell_{k_1,x}} \leq 
\sum_{j=k_1}^{k_2-1} |D \ell_{j+1,x}-D \ell_{j,x}| 
\lesssim \sum_{j=k_1}^{k_2} \mathcal{E}_{1}(x,\rr_{j})$$
and therefore we conclude with
$$
 \snr{D \ell_{k_2,x}-D \ell_{k_1,x}} 
 \leq c_8 \mathcal{E}_{1}(x,\rr_{k_1}) + c_8\int_0^{2\rr_{k_0}} \omega_{\star}(\lambda)\frac{\dlam}{\lambda}+ \frac{c_8}{2^{m(k_1-k_0)\alpha}} + c_8 \mathbf{I}^{f}_{2s-1,\chi}(x, 8\rr_{k_1})\,,
$$
that holds whenever $k_2>k_1> k_0>2$ with $c_8 \equiv c_8 (\data,\alpha, H_2)\geq 1$. Fix $\eps >0$; thanks to \eqref{assilip} determine $k_0\equiv k_0(\data,\alpha,\omega_{\star}(\cdot), H_2, \eps)>2$ such that 
$$
\int_0^{2\rr_{k_0}} \omega_{\star}(\lambda)\frac{\dlam}{\lambda} < \frac{\eps}{2c_8}\,.
$$ 
Then by \rif{assilipul} and \rif{uniform2} determine $\texttt{k}\equiv \texttt{k}(\data,\alpha,\omega_{\star}(\cdot), H_2, \eps)> k_0$ such that $k_1>\texttt{k}$ implies 
$$
\mathcal{E}_{1}(x,\rr_{k_1})  +\frac{1}{2^{m(k_1-k_0)\alpha}}+  \mathbf{I}^{f}_{2s-1,\chi}(x, 8\rr_{k_1}) < \frac{\eps}{2 c_8}\,.
$$
Summarizing, we have proved that $k_2 > k_1 > \texttt{k}$ implies $ \snr{D \ell_{k_2,x}-D \ell_{k_1,x}} < \eps$ with the index $\texttt{k}$ being independent of the point $x \in B_{r_{x_0}/2}(x_0)$. 
This means that $\{D \ell_{j,x}\}_j$ is a uniformly Cauchy sequence in $B_{r_{x_0}/2}(x_0)$ (uniform with respect to $x$); therefore there exists a map $\mathcal D\colon B_{r_{x_0}/2}(x_0) \to \er^{N\times n}$ such that
\eqn{identifica}
$$
D \ell_{j,x} \to \mathcal D(x) \qquad \mbox{ uniformly w.r.t. $x\in B_{r_{x_0}/2}(x_0)$}\,.
$$
On the other hand, by Proposition \ref{continues}, for each fixed index $j$, and therefore for each fixed radius $\rr_j$,    
the map $x \mapsto D\ell_{j,x}$ is continuous; it follows that $x \mapsto \mathcal D(x)$ is continuous in $B_{r_{x_0}/2}(x_0)$ being a uniform limit of continuous maps. 

\subsubsection{Step 3: Pointwise convergence of the minimizing affine maps} Here we finally prove the gradient continuity. We do this by showing that $Du(x)\equiv \mathcal D(x)$ a.e. in $B_{r_{x_0}/2}(x_0)$, where $\mathcal D(\cdot)$  is the limit map identified in \rif{identifica}. We therefore conclude, via a standard covering argument, that $Du$ has a continuous representative in $\Omega_u$, as required. All this is achieved in the following 
\begin{proposition} \label{finalprop}Let $x\in B_{r_{x_0}/2}(x_0)$ be a point such that both the precise representatives of $u$ and $Du$ exist, i.e., a point such that 
\eqn{lebpdopo}
$$
u(x):=\lim_{\sigma \to 0}(u)_{B_{\sigma}(x)} \quad \mbox{and} \quad Du(x):=\lim_{\sigma \to 0}(Du)_{B_{\sigma}(x)}
$$
hold, and such that 
\eqn{graddi}
 $$
 \begin{cases} \displaystyle \lim_{\sigma\to 0} \,  \frac{\nra{u-\ell_{d}}_{L^{2}(B_{\sigma}(x))}}{\sigma}= 0\\[10pt]
\ell_{d}(\ti{x})\equiv \ell_{d,x}(\ti{x}) =  Du(x)( \ti{x}-x)
+u(x)\,, \quad \ti{x} \in \er^n\,.
\end{cases}
$$
Then
\eqn{convergenze}
$$\mbox{$D\ell_{j,x}\to Du(x)$ in $\er^{N\times n}$ \quad and \quad $\ell_{j,x}(x)\to u(x)$ in $\er^N$}\,.
$$
In particular, \eqref{convergenze} holds a.e.\,in $B_{r_{x_0}/2}(x_0)$ and therefore, as a consequence of \eqref{identifica}, $Du(x) \equiv \mathcal D(x)$ holds for a.e. $x \in B_{r_{x_0}/2}(x_0)$.
\end{proposition}
\begin{proof} By \rif{Lipest}$_1$ we have $u \in W^{1,\infty}(B_{r_{x_0}}(x_0))$ and therefore almost every point $x \in B_{r_{x_0}}(x_0)$ is such that \rif{lebpdopo}-\rif{graddi} hold \cite[Theorem 6.2]{eg} (in fact the first convergence in \rif{lebpdopo} holds for every $x\in B_{r_{x_0}}(x_0)$ as $u$ is Lipschitz continuous in $B_{r_{x_0}}(x_0)$). Now, we go back to the proof of \rif{combina} and note that it gives that 
\eqn{combinaaa}
$$
	 \sum_{j=0}^{k} \mathbf{R}(x,\rr_{j}) 
		\leq c \MMM^{\#,1}_{\rr_k,\rr_{\star}}(x) +c\rr_{\star}^{-1-\alpha/2}\, \tx{E}_{u}(x,\rr_{\star}) + c\, \mathbf{I}^{f}_{2s-1,\chi}(x,\rr_{\star})\,.
$$
for every choice of $\rr_*\leq r_{x_0}/8$, independently from the smallness condition in \rif{meetme}. Indeed, this was only needed to reabsorbe $\MMM^{\#,1}_{\rr_k,\rr_{\star}}(x)$, which is now already bounded by \rif{Lipest}. Using \rif{Lipest}$_2$ in \rif{combinaaa}, and eventually letting $k\to \infty$, we obtain
$$
 \sum_{j=0}^{\infty} \mathbf{R}(x,\rr_{j})\lesssim  H_2 +r_{x_0}^{-1-\alpha/2}\, \tx{E}_{u}(x,\rr_{\star}) +   \mathbf{I}^{f}_{2s-1,\chi}(x,r_{x_0}/8)\,.
$$
In particular, the series on the left-hand side converges.  In turn, using this last inequality in \rif{flatnesscontrol} with $\gamma_1=1$, and again letting $k \to \infty$, we obtain
$$
\sum_{j\geq 0} \mathcal{E}_{1}(x,\rr_{j}) < \infty \Longrightarrow \lim_{j\to \infty} \mathcal{E}_{1}(x,\rr_{j}) =0
\Longrightarrow \lim_{j\to \infty} \, \frac{\nra{u-\ell_{j,x}}_{L^{2}(B_{\rr_j}(x))}}{\rr_j}= 0\,.
$$
This, \rif{graddi} and triangle inequality give
$$ 
 \lim_{j\to \infty} \, \frac{\nra{\ell_{j,x}-\ell_{d}}_{L^{2}(B_{\rr_j}(x))}}{ \rr_j}=0
$$
so that \rif{tritri}$_1$ yields
$$
 \lim_{j\to \infty} \,\snr{D\ell_{j,x}-Du(x)}+\lim_{j\to \infty} \, \frac{\snr{\ell_{j,x}(x)-u(x)}}{\rr_j}
= 0
$$
from which \rif{convergenze} follows. 
\end{proof}

\section{Everywhere regularity for linear systems}\label{indicazioni} The proof of Theorem \ref{ppaolo} is identical to the one of Theorems \ref{gradreg1}-\ref{gradreg3} once we note that we do not need to start with the preliminary H\"older information \rif{holdest22} in the ball
$B_{r_{x_{0}}}(x_{0})$. This only serves later to control the oscillations of the coefficient matrix $a(\cdot)$ with respect to the solution and it is obviously not needed when such a dependence is not present. See for instance \rif{osci} that here can be replaced by 
$
||a-a_{0}||_{L^\infty(\mathcal B_{r})}  \lesssim \omega_{\star} (r)$. For the same reason, as we do not need any preliminary H\"older information on $u$, we do not need to start from a ball $B_{r_{x_{0}}}(x_{0})$ where \rif{holdest22} is satisfied, but we can replace it by any interior ball $B_r\Subset \Omega$, as it is visible in \eqref{finale-linear}. We now briefly comment on the adaptation to systems as \rif{eqweaksol-linear} of the  rest of the results, that is Theorems \ref{ureg0}-\ref{ureg5}. As already mentioned in Section \ref{everywhere}, they hold in the same way as the original statements/proofs, but in every case we can take $\Omega_u=\Omega$, i.e., all points are regular points. Moreover, all the a priori estimates involved are valid on every interior ball. This can be easily checked using the excess decay result of Proposition \ref{cor.1-linear} upon assuming that 
\eqn{reacher5}
$$ 
\sup_{B_{\sigma} \Subset \Omega; \sigma \leq \sigma_{0}} \ppsi<\infty
$$ 
so that the right-hand side in \eqref{exx.5} is finite and the estimate is effective. 
 In addition to this, in Theorem \ref{ureg3} condition \rif{osc.t0}$_1$ is not any longer needed and \rif{oscosc}-\rif{lebp} hold at every point $x_0\in \Omega$ provided 
 \eqn{reacher555}
 $$ \sup_{\sigma \leq \rr} \ppsixo<\infty.$$  On the other hand, note that when dealing with systems of the type in \rif{eqweaksol-linear} there is no loss of generality in assuming that smallness conditions as \rif{osc.t0}$_2$ are replaced by \rif{reacher5}-\rif{reacher555}. 
This follows considering $\ti{f}:=\mathfrak{A}^{-1}f$ and $\ti{u}:= \mathfrak{A}^{-1}u$, $\mathfrak{A}\in (0,\infty)$, so that $-\mathcal{L}_{a}\ti{u}=\ti{f}$  and therefore it is possible to satisfy \rif{f.bmo} for any $\epsb{b}>0$ by choosing $\mathfrak{A}\equiv \mathfrak{A}(\epsb{b})$ large enough. 
For the same reason, the BMO-regularity of Theorem \ref{ureg1} holds without requiring \rif{f.bmo} but just the finiteness condition \rif{reacher5}.


\begin{thebibliography}{}

\bibitem{BK} R.F. Bass, M. Kassmann,
Harnack inequalities for non-local operators of variable order. {\em Trans. Am. Math. Soc.} 357, 837-850 (2005).

\bibitem{bl} L. Brasco, E. Lindgren,
Higher Sobolev regularity for the fractional $p$-Laplace equation in the superquadratic case. {\em Adv. Math.} 304, 300-354 (2017).

\bibitem{bls} L. Brasco, E. Lindgren, A. Schikorra,
Higher H\"older regularity for the fractional $p$-Laplacian in the superquadratic case. {\em Adv. Math.} 338, 782-846 (2018).

\bibitem{brascoparini} L. Brasco, E. Parini,
The second eigenvalue of the fractional $p$-Laplacian. {\em Adv. Calc. Var.} 9, 323-355 (2016).

\bibitem{by} S.-S. Byun, Y. Youn,
Potential estimates for elliptic systems with subquadratic growth. {\em J. Math. Pures Appl.} 131, 193-224 (2019).

\bibitem{byunnon} S.-S. Byun, K. Kim, D. Kumar,
Regularity results for a class of nonlocal double phase equations with VMO coefficients. {\em Publ. Mat.} 68, 507-544 (2024).

\bibitem{cd} L. Caffarelli, G. D\'avila,
Interior regularity for fractional systems. {\em Ann. I. H. Poincar\'e - AN} 36, 165-180 (2019).

\bibitem{caff} L. Caffarelli, L. Silvestre,
An extension problem related to the fractional Laplacian. {\em Comm. PDE} 32, 1245-1260 (2007).

\bibitem{cam0} S. Campanato,
Propriet\`a di H\"olderianit\`a di alcune classi di funzioni. {\em Ann. Sc. Norm. Super. Pisa, Sci. Fis. Mat. (III)} 17, 175-188 (1963).

\bibitem{cam00} S. Campanato,
Propriet\`a di una famiglia di spazi funzionali. {\em Ann. Sc. Norm. Super. Pisa, Sci. Fis. Mat. (III)} 18, 137-160 (1964).

\bibitem{cam1} S. Campanato,
Equazioni ellittiche del II ordine e spazi $\mathcal L^{2, \lambda}$. {\em Ann. Mat. Pura Appl. (IV)} 69, 321-381 (1965).

\bibitem{cam2} S. Campanato,
{\em Sistemi ellittici in forma divergenza. Regolarit\`a all'interno}. Pubblicazioni della Classe di Scienze: Quaderni. Pisa: Scuola Normale Superiore (1980).

\bibitem{calderon} A. Calder\'on, R. Scott,
Sobolev type inequalities for $p>0$. {\em Studia Math.} 62, 75-92 (1978).

\bibitem{colombo1} M. Colombo, C. De Lellis, A. Massaccesi,
The generalized Caffarelli-Kohn-Nirenberg theorem for the hyperdissipative Navier-Stokes system. {\em Comm. Pure Appl. Math.} 73, 609-663 (2020).

\bibitem{cozzi} M. Cozzi,
Interior regularity of solutions of non-local equations in Sobolev and Nikol'skii spaces. {\em Ann. Mat. Pura Appl. (IV)} 196, 555-578 (2017).

\bibitem{dalio1} F. Da Lio, T. Rivi\`ere,
Three-term commutator estimates and the regularity of $1/2$-harmonic maps into spheres. {\em Anal. PDE} 4, 149-190 (2011).

\bibitem{dqc} C. De Filippis,
Quasiconvexity and partial regularity via nonlinear potentials. {\em J. Math. Pures Appl. (IX)} 163, 11-82 (2022).

\bibitem{dm} C. De Filippis, G. Mingione,
Gradient regularity in mixed local and nonlocal problems. {\em Math. Ann.} 388, 261-328 (2024).

\bibitem{dmn} C. De Filippis, G. Mingione,
Nonuniformly elliptic Schauder theory. {\em Invent. Math.} 234, 1109-1196 (2023).

\bibitem{dmn2} C. De Filippis, G. Mingione,
The sharp growth rate in nonuniformly elliptic Schauder theory. {\em Duke Math. J.} 174, 1775-1848 (2025).

\bibitem{follow} C. De Filippis, G. Mingione, S. Nowak,
Partial regularity in nonlocal systems II. {\em arXiv preprint} \href{https://arxiv.org/abs/2602.18848}{arXiv:2602.18848} (2026).

\bibitem{ds} C. De Filippis, B. Stroffolini,
Singular multiple integrals and nonlinear potentials. {\em J. Funct. Anal.} 285, 109952 (2023).

\bibitem{deg} E. De Giorgi,
Frontiere orientate di misura minima. {\em Seminario Mat. Scu. Normale Sup. Pisa}, 1960-61.

\bibitem{dgce} E. De Giorgi,
Un esempio di estremali discontinue per un problema variazionale di tipo ellittico. {\em Boll. Un. Mat. Ital. (IV)} 1, 135-137 (1968).

\bibitem{des} R.A. DeVore, R.C. Sharpley,
Maximal functions measuring smoothness. {\em Memoirs Amer. Math. Soc.} 47, 293 (1984).

\bibitem{des2} R.A. DeVore, R.C. Sharpley,
Besov spaces on domains in $\er^d$. {\em Trans. Am. Math. Soc.} 335, 843-864 (1993).

\bibitem{emmanuele} E. DiBenedetto,
{\em Real analysis}. Basel: Birkh\"auser (2002).

\bibitem{DKP} A. Di Castro, T. Kuusi, G. Palatucci,
Local behavior of fractional $p$-minimizers. {\em Ann. I. H. Poincar\'e - AN} 33, 1279-1299 (2016).

\bibitem{dnowak} L. Diening, K. Kim, H.S. Lee, S. Nowak,
Nonlinear nonlocal potential theory at the gradient level. {\em J. Eur. Math. Soc.}, doi:10.4171/JEMS/1706.

\bibitem{dms} L. Diening, J. M\'alek, M. Steinhauer,
On Lipschitz truncations of Sobolev functions (with variable exponent) and their selected applications. {\em ESAIM: COCV} 14, 211-232 (2008).

\bibitem{dn} L. Diening, S. Nowak,
Calder\'on-Zygmund estimates for the fractional $p$-Laplacian. {\em Ann. PDE} 11, 6 (2025).

\bibitem{dsv} L. Diening, B. Stroffolini, A. Verde,
The $\varphi$-harmonic approximation lemma and the regularity of $\varphi$-harmonic maps. {\em J. Diff. Equ.} 253, 1943-1958 (2012).

\bibitem{dpv} E. Di Nezza, G. Palatucci, E. Valdinoci,
Hitchhiker's guide to the fractional Sobolev spaces. {\em Bull. Sci. Math.} 136, 521-573 (2012).

\bibitem{dispa} S. Dispa,
Intrinsic characterizations of Besov spaces on Lipschitz domains. {\em Math. Nachr.} 260, 21-33 (2003).

\bibitem{dong1} H. Dong, H. Zhang,
Dini estimates for nonlocal fully nonlinear elliptic equations. {\em Ann. Inst. Henri Poincar\'e, Anal. Non Lin\'eaire} 35, 971-992 (2018).

\bibitem{dong2} H. Dong, T. Jin, H. Zhang,
Dini and Schauder estimates for nonlocal fully nonlinear parabolic equations with drifts. {\em Anal. PDE} 11, 1487-1534 (2018).

\bibitem{dudini} F. Duzaar, A. Gastel,
Nonlinear elliptic systems with Dini continuous coefficients. {\em Arch. Math.} 78, 58-73 (2002).

\bibitem{dumi} F. Duzaar, G. Mingione,
Harmonic type approximation lemmas. {\em J. Math. Anal. Appl.} 352, 301-335 (2009).

\bibitem{dust} F. Duzaar, K. Steffen,
Optimal interior and boundary regularity for almost minimizers to elliptic variational integrals. {\em J. Reine Angew. Math.} 546, 73-138 (2002).

\bibitem{eg} L.G. Evans, R.F. Gariepy,
{\em Measure theory and fine properties of functions. 2nd revised ed.}. Textbooks in Mathematics. Boca Raton, FL: CRC Press (2015).

\bibitem{fall} M. Fall,
Regularity results for nonlocal equations and applications. {\em Calc. Var.} 59, 181 (2020).

\bibitem{feros} X. Fern\'andez-Real, X. Ros-Oton,
Schauder and Cordes-Nirenberg estimates for nonlocal elliptic equations with singular kernels. {\em Proc. London Math. Soc.} (2024), doi:10.1112/plms.12629.

\bibitem{giaorange} M. Giaquinta,
{\em Multiple integrals in the calculus of variations and nonlinear elliptic systems}. Annals of Mathematics Studies 105 (1983).

\bibitem{giagreen} M. Giaquinta, L. Martinazzi,
{\em An introduction to the regularity theory for elliptic systems, harmonic maps and minimal graphs}. Edizioni della Normale (2012).

\bibitem{giusti} E. Giusti,
{\em Direct methods in the calculus of variations}. Singapore: World Scientific. vii, 403 p. (2003).

\bibitem{gmce} E. Giusti, M. Miranda,
Un esempio di soluzioni discontinue per un problema di minimo relativo ad un integrale regolare del calcolo delle variazioni. {\em Boll. Un. Mat. Ital. (IV)} 1, 219-226 (1968).

\bibitem{giumi} E. Giusti, M. Miranda,
Sulla regolarit\`a delle soluzioni deboli di una classe di sistemi ellittici quasi-lineari. {\em Arch. Ration. Mech. Anal.} 31, 173-184 (1968/69).

\bibitem{gu} S. Gutierrez,
Lusin approximation of Sobolev functions by H\"older continuous functions. {\em Bull. Inst. Math. Acad. Sin.} 3, 95-116 (2003).

\bibitem{hawi} P. Hartman, A. Winter,
On uniform Dini conditions in the theory of linear partial differential equations of elliptic type. {\em Amer. J. Math.} 77, 329-354 (1955).

\bibitem{finnish} J. Heinonen, T. Kilpel\"ainen, O. Martio,
{\em Nonlinear potential theory of degenerate elliptic equations}. Oxford Mathematical Monographs. Oxford: Clarendon Press (1993).

\bibitem{mazyadini} T. Jin, V. Maz'ya, J. Van Schaftingen,
Pathological solutions to elliptic problems in divergence form with continuous coefficients. {\em C. R. Math. Acad. Sci. Paris} 347, 773-778 (2009).

\bibitem{johnmaly} O. John, J. Mal\'y, J. Star\'a,
Nowhere continuous solutions to elliptic systems. {\em Comm. Math. Univ. Carol.} 30, 33-43 (1989).

\bibitem{john} F. John, L. Nirenberg,
On functions of bounded mean oscillation. {\em Comm. Pure Appl. Math.} 14, 415-426 (1961).

\bibitem{jordan} C. Jordan,
{\em Calculus of finite differences. 2nd ed.}. New York: Chelsea Co. XXI (1950).

\bibitem{katz} N.H. Katz, N. Pavlovic,
A cheap Caffarelli-Kohn-Nirenberg inequality for the Navier-Stokes equation with hyperdissipation. {\em Geom. Funct. Anal.} 12, 355-379 (2002).

\bibitem{KiWe} M. Kim, M. Weidner,
Optimal boundary regularity and Green function for nonlocal equations in divergence form. {\em J. Eur. Math. Soc.}, doi:10.4171/JEMS/1760 (2025).

\bibitem{kokupa} J. Korvenp\"a\"a, T. Kuusi, G. Palatucci,
The obstacle problem for nonlinear integro-differential operators. {\em Calc. Var.} 55, 63 (2016).

\bibitem{kovats} J. Kovats,
Fully nonlinear elliptic equations and the Dini condition. {\em Comm. PDE} 22, 1911-1927 (1997).

\bibitem{kmjfa} T. Kuusi, G. Mingione,
Universal potential estimates. {\em J. Funct. Anal.} 262, 4205-4269 (2012).

\bibitem{kumig} T. Kuusi, G. Mingione,
Guide to nonlinear potential estimates. {\em Bull. Math. Sci.} 4, 1-82 (2014).

\bibitem{KMdini} T. Kuusi, G. Mingione,
A nonlinear Stein theorem. {\em Calc. Var.} 51, 45-86 (2014).

\bibitem{KMS1} T. Kuusi, G. Mingione, Y. Sire,
Nonlocal equations with measure data. {\em Comm. Math. Phys.} 337, 1317-1368 (2015).

\bibitem{KMS} T. Kuusi, G. Mingione, Y. Sire,
Nonlocal self-improving properties. {\em Anal. \& PDE} 8, 57-114 (2015).

\bibitem{kumi} T. Kuusi, G. Mingione,
Partial regularity and potentials. {\em J. \'Ecole Pol. Math.} 3, 309-363 (2016).

\bibitem{KMSvec} T. Kuusi, G. Mingione,
Vectorial nonlinear potential theory. {\em J. Eur. Math. Soc.} 20, 929-1004 (2018).

\bibitem{kns} T. Kuusi, S. Nowak, Y. Sire,
Gradient regularity and first-order potential estimates for a class of nonlocal equations. {\em Amer. J. Math.}, to appear.

\bibitem{kronz} M. Kronz,
Partial regularity results for minimizers of quasiconvex functionals of higher order. {\em Ann. I. H. Poincar\'e - AN} 19, 81-112 (2002).

\bibitem{leoni} G. Leoni,
{\em A first course in Sobolev Spaces}. Graduate Studies in Mathematics 105. Amer. Math. Society (2009).

\bibitem{yanyanli} Y. Li,
On the $C^1$-regularity of solutions to divergence form elliptic systems with Dini-continuous coefficients. {\em Chin. Ann. Math., Ser. B} 38, 489-496 (2017).

\bibitem{mazo} K. Mazowiecka, A. Schikorra,
$W^{s,n/s}$-harmonic maps in homotopy classes. {\em J. Lond. Math. Soc. (II)} 108, 742-836 (2023).

\bibitem{mazo2} K. Mazowiecka, M. Mi\'skiewicz, A. Schikorra,
On the size of the singular set of minimizing harmonic maps. {\em Memoirs AMS} 302, no. 1519 (2024).

\bibitem{mazya} V. Maz'ya,
Examples of nonregular solutions of quasilinear elliptic equations with analytic coefficients. {\em Funct. Anal. Appl.} 2, 230-234 (1968).

\bibitem{HM} V. Maz'ya, M. Havin,
A nonlinear potential theory. {\em Russ. Math. Surveys} 27, 71-148 (1972).

\bibitem{schi4} T. Mengesha, A. Schikorra, S. Yeepo,
Calder\'on-Zygmund type estimates for nonlocal PDE with H\"older continuous kernel. {\em Adv. Math.} 383, Article 107692 (2021).

\bibitem{meng} T. Mengesha, J. Scott,
Self-improving inequalities for bounded weak solutions to nonlocal double phase equations. {\em Comm. Pure Appl. Anal.} 21, 183-212 (2022).

\bibitem{millot1} V. Millot, Y. Sire,
On a fractional Ginzburg-Landau equation and $1/2$-harmonic maps into spheres. {\em Arch. Ration. Mech. Anal.} 215, 125-210 (2015).

\bibitem{millot2} V. Millot, M. Pegon, A. Schikorra,
Partial regularity for fractional harmonic maps into spheres. {\em Arch. Ration. Mech. Anal.} 242, 747-825 (2021).

\bibitem{min03} G. Mingione,
The singular set of solutions to non-differentiable elliptic systems. {\em Arch. Ration. Mech. Anal.} 166, 287-301 (2003).

\bibitem{min06} G. Mingione,
Regularity of minima: an invitation to the dark side of the Calculus of Variations. {\em Appl. Math.} 51, 355-425 (2006).

\bibitem{min11} G. Mingione,
Gradient potential estimates. {\em J. Eur. Math. Soc.} 13, 459-486 (2011).

\bibitem{morrey} C.B. Morrey,
Partial regularity results for non-linear elliptic systems. {\em J. Math. Mech.} 17, 649-670 (1968).

\bibitem{sn} S. Nowak,
Regularity theory for nonlocal equations with $\textnormal{VMO}$ coefficients. {\em Ann. I. H. Poincar\'e - AN} 40, 61-132 (2023).

\bibitem{sn1} S. Nowak,
Improved Sobolev regularity for nonlocal equations with $\textnormal{VMO}$ coefficients. {\em Math. Ann.} 385, 1323-1378 (2023).

\bibitem{p} D.K. Palagachev,
Quasilinear elliptic equations with $\textnormal{VMO}$-coefficients. {\em Trans. Amer. Math. Soc.} 347, 2481-2493 (1995).

\bibitem{roberts} J.A. Roberts,
A regularity theory for intrinsic minimising fractional harmonic maps. {\em Calc. Var.} 57, 109 (2018).

\bibitem{sarason} D. Sarason,
Functions of vanishing mean oscillation. {\em Trans. Amer. Math. Soc.} 207, 391-405 (1975).

\bibitem{schi1} A. Schikorra,
Integro-differential harmonic maps into spheres. {\em Comm. PDE} 40, 506-539 (2015).

\bibitem{schi2} A. Schikorra,
$\eps$-regularity for systems involving non-local, antisymmetric operators. {\em Calc. Var.} 54, 3531-3570 (2015).

\bibitem{schi3} A. Schikorra,
Nonlinear commutators for the fractional $p$-Laplacian and applications. {\em Math. Ann.} 366, 695-720 (2016).

\bibitem{cornelia} C. Schneider,
Traces of Besov and Triebel-Lizorkin spaces on domains. {\em Math. Nachr.} 284, 572-586 (2011).

\bibitem{simon1} L. Simon,
{\em Lectures on geometric measure theory}. Proc. CMA 3, ANU Canberra (1983).

\bibitem{simon2} L. Simon,
{\em Theorems on regularity and singularity of energy minimizing maps}. Birkh\"auser-Verlag, Basel-Boston-Berlin (1996).

\bibitem{silve} L. Silvestre,
H\"older estimates for solutions of integro-differential equations like the fractional Laplace. {\em Indiana Univ. Math. J.} 55, 1155-1174 (2006).

\bibitem{steinbook} E. Stein,
{\em Singular integrals and differentiability properties of functions}. Princeton Mathematical Series 30. Princeton University Press. XIV (1970).

\bibitem{stein} E. Stein,
Editor's note: the differentiability of functions in $\er^n$. {\em Ann. Math. (II)} 113, 383-385 (1981).

\bibitem{TriebelI} H. Triebel,
{\em Interpolation theory, function spaces, differential operators}. Johann Ambrosius Barth, Heidelberg (1995).

\bibitem{Triebel3} H. Triebel,
{\em Theory of function spaces {III}}. Birkh\"auser Verlag, Basel (2006).

\bibitem{Triebel4} H. Triebel,
{\em Theory of function spaces {IV}}. Birkh\"auser/Springer (2020).

\end{thebibliography}
\end{document}